\documentclass[11pt,reqno]{amsart}

\usepackage{
  amsmath,
  amsfonts,
  amssymb,
  amsthm,
  amscd,
  todonotes,
  graphicx,
  mathtools,
  enumitem,
  mathdots,
  txfonts,
  stackrel
}
\usepackage[all]{xy}

\newcommand{\arxiv}[1]{{\tt arXiv:#1}}

\makeatletter
\@namedef{subjclassname@2020}{%
  \textup{2020} Mathematics Subject Classification}
\makeatother

\usepackage[vcentermath]{youngtab}

\usepackage{bbm}                     

\DeclareFontFamily{OT1}{pzc}{}
\DeclareFontShape{OT1}{pzc}{m}{it}{<-> s * [1.10] pzcmi7t}{}
\DeclareMathAlphabet{\mathpzc}{OT1}{pzc}{m}{it}

\usepackage[colorlinks=true, linkcolor=blue, citecolor=purple, urlcolor=blue, breaklinks=true]{hyperref}

\usepackage{tikz}
\usetikzlibrary{decorations.markings}
\usetikzlibrary{decorations.pathreplacing,calligraphy}
\usetikzlibrary{cd}
\usetikzlibrary{patterns}
\usetikzlibrary{shapes}

\tikzset{anchorbase/.style={baseline={([yshift=-0.5ex]current bounding box.center)}}}
\tikzset{wipe/.style={white,line width=4pt}}

\def\stringnumber#1{{\scriptstyle\color{blue}#1}}
\def\dt{{\color{white}\bullet}\!\!\!\circ}
\newcommand\opendot[1]{\node at (#1) {$\dt$}}

\usepackage{cleveref}

\crefname{definition}{Definition}{Definitions}
\crefname{example}{Example}{Examples}
\crefname{lemma}{Lemma}{Lemmas}
\crefname{corollary}{Corollary}{Corollaries}
\crefname{theorem}{Theorem}{Theorems}
\crefname{remark}{Remark}{Remarks}
\crefname{equation}{}{}
\crefname{enumi}{}{}
\crefname{section}{Section}{Sections}

\def\red{\color{purple}}

\def\sqi{{\mathbf{i}}}
\def\OGBim{\mathpzc{OGBim}}
\def\GSCAT{\mathpzc{gsCat}}

\def\tU{\widetilde{U}}
\def\tev{\widetilde{\ev}}
\def\tcoev{\widetilde{\coev}}
\def\tV{\widetilde{V}}
\def\tu{{\tilde u}}
\def\tv{{\tilde v}}
\def\tb{{\tilde b}}
\def\tc{{\tilde c}}
\def\UU{\mathfrak{U}}

\def\AA{\mathfrak{A}}
\def\even{\operatorname{even}}
\def\odd{\operatorname{odd}}
\newcommand\cA{\mathpzc{A}}
\newcommand\cR{\mathpzc{R}}
\newcommand\cI{\mathpzc{I}}
\newcommand\cB{\mathpzc{B}}
\newcommand\cC{\mathpzc{C}}

\newcommand\cV{\mathpzc{V}}

\newcommand\cW{\mathpzc{W}}

\newcommand\sVec{\mathpzc{gsVec}}
\newcommand\sVecunderline{\mathpzc{g}\!\underline{\mathpzc{sVec}}}
\newcommand\gsKarunderline{\operatorname{g\underline{sKar}}}
\newcommand\fdmod{{\operatorname{-mod}}}
\newcommand\sMod{{\operatorname{-gsMod}}}
\newcommand\psmod{{\operatorname{-pgsmod}}}
\newcommand\psMod{{\operatorname{-pgsMod}}}
\newcommand\Upsmod{{\operatorname{-pg\underline{smod}}}}
\newcommand\UpsMod{{\operatorname{-pg\underline{sMod}}}}

\newcommand\bisMod{{\operatorname{-gsMod-}}}
\newcommand\doMs{{\operatorname{gsMod-}}}
\newcommand\smod{{\operatorname{-gsmod}}}
\newcommand\Usmod{{\operatorname{-g\underline{smod}}}}
\newcommand\UsMod{{\operatorname{-g\underline{sMod}}}}

\newcommand{\sEnd}{\operatorname{gsEnd}}

\def\hatotimes{{\,\bar\otimes\,}}
\def\br{B}

\def\lround{(\!(}
\def\rround{)\!)}

\def\S{S\!}
\def\E{{E}}
\def\F{{F}}
\newcommand\C{\mathbb{C}}
\newcommand\Z{\mathbb{Z}}
\newcommand\Q{\mathbb{Q}}

\newcommand\N{\mathbb{N}}
\renewcommand\k{\mathbb{F}}

\def\sop{{\operatorname{sop}}}
\def\lex{{\operatorname{lex}}}
\def\oldomega{\varpi}

\def\rev{{\operatorname{rev}}}
\def\deg{\operatorname{deg}}
\def\h{{\operatorname{ht}}}
\def\parity{\operatorname{par}}
\def\sign{\operatorname{sgn}}

\def\diag{\operatorname{diag}}
\def\smiley{{\gamma}}
\def\ocirc#1{{\stackrel{_{\,\circ}}{#1}}}
\def\TT{{T}}

\def\o{o}
\def\c{c}
\newcommand{\sqbinom}[2]{\genfrac{[}{]}{0pt}{}{#1}{#2}}

\renewcommand{\rev}{{\operatorname{rev}}}

\def\transpose{\mathtt{t}}
\def\reverse{\mathtt{r}}
\def\0{{\bar{0}}}
\def\1{{\bar{1}}}
\def\eps{\varepsilon}
\def\Par{{\Lambda^{\!+}}}
\def\SPar#1{{\Lambda^{\!+}_{#1}}}
\def\GPar#1#2{{\Lambda^{\!+}_{#1\times#2}}}
\def\Comp{\Lambda}
\def\gsdim{\operatorname{dim}_{q,\pi}}
\def\gsrank{\operatorname{rk}_{q,\pi}}
\def\sl{\mathfrak{sl}}
\def\osp{\mathfrak{osp}}
\def\dash{{\text{-}}}

\def\fin{{\operatorname{fin}}}
\def\ev{\operatorname{ev}}
\def\Ch{\operatorname{Ch}}
\def\coev{\operatorname{coev}}

\DeclareMathOperator{\Hom}{Hom}
\DeclareMathOperator{\End}{End}
\DeclareMathOperator{\Aut}{Aut}

\def\CT{{\mathtt{CT}}}
\DeclareMathOperator{\SHOM}{\mathpzc{gsHom}}
\DeclareMathOperator{\SEND}{\mathpzc{gsEnd}}
\DeclareMathOperator{\END}{\mathpzc{End}}
\DeclareMathOperator{\id}{id}       
\DeclareMathOperator{\Id}{Id}       
\DeclareMathOperator{\im}{im}       
\DeclareMathOperator{\GSKar}{gsKar}     

\def\OSym{OS\!ym}

\def\OPol{O\hspace{-.15mm}Pol}
\def\Sym{S\!ym}
\DeclareMathOperator{\res}{Res}
\DeclareMathOperator{\ind}{Ind}
\DeclareMathOperator{\coind}{Coind}

\DeclareMathOperator{\tr}{tr}
\DeclareMathOperator{\Tr}{Tr}

\def\gamm{\eta}
\def\chi{\xi}
\def\Cone{\operatorname{Cone}}
\def\ONH{ONH}
\def\ROH{\overline{OH}}
\def\RONH{\overline{ONH}}   
\def\EONH{ONH}
\def\EOH{OH}
\def\sig{\operatorname{sh}} 
\def\NE{\overline{NE}}

\newtheorem{theorem}{Theorem}[section]
\newtheorem{lemma}[theorem]{Lemma}
\newtheorem*{lemma*}{Lemma}
\newtheorem{corollary}[theorem]{Corollary}

\theoremstyle{definition}
\newtheorem{definition}[theorem]{Definition}
\newtheorem{remark}[theorem]{Remark}
\newtheorem{example}[theorem]{Example}

\numberwithin{equation}{section}
\allowdisplaybreaks

\allowdisplaybreaks

\begin{document}

\title[Derived equivalences for spin symmetric groups]{Odd Grassmannian bimodules and derived equivalences for spin symmetric groups}

\author[J. Brundan]{Jonathan Brundan}
\address{Department of Mathematics,
University of Oregon, Eugene, OR 97403, USA}
\email{brundan@uoregon.edu}

\author[A. Kleshchev]{Alexander Kleshchev}
\address{Department of Mathematics,
University of Oregon, Eugene, OR 97403, USA}
\email{klesh@uoregon.edu}

\thanks{2020 {\it Mathematics Subject Classification}: 17B10, 18D25, 20C30.}
\thanks{Research supported in part by NSF grants DMS-2101783 (J.B.)
and DMS-2101791 (A.K.).}

\begin{abstract}
We prove odd analogs of results of Chuang and Rouquier 
on $\mathfrak{sl}_2$-categorification. 
Combined also with recent work of the second author with Livesey,
this allows us to complete the proof of Brou\'e's Abelian Defect Conjecture for the double covers of symmetric groups.
The article also develops the theory of odd symmetric functions initiated a decade ago by Ellis, Khovanov and Lauda. 
A key role in our approach is played by a 2-category
consisting of {\em odd Grassmannian bimodules} over superalgebras
which are odd analogs of equivariant cohomology algebras of Grassmannians. This is the odd analog of the category of Grassmannian bimodules which was at the heart of Lauda's 
independent approach to 
categorification of $\mathfrak{sl}_2$. We also construct an action of
the odd Kac-Moody 2-category $\UU(\sl_2)$ on the 2-category of odd Grassmannian
bimodules, and use this to give a new proof of its non-degeneracy.
\tableofcontents
\end{abstract}

\maketitle

\section{Introduction}

This paper establishes ``odd" analogs of results of Chuang and Rouquer \cite{CR}. The motivating 
problem is to prove Brou\'e's Abelian Defect Group Conjecture for the double covers of symmetric groups.
In the ordinary even theory, Brou\'e's conjecture for symmetric groups was proved in two steps. First came the work of Chuang and Kessar \cite{CK}, which established a Morita equivalence reducing the 
proof of Brou\'e's conjecture to proving that all blocks of
symmetric groups in characteristic $p > 0$ with the same defect
are derived equivalent.
Then the second part of the proof came in Chuang and Rouquier's work which deduced this assertion from a special case of a remarkable general
theory of $\mathfrak{sl}_2$-categorification. Their theory has had many
other significant applications and generalizations, especially following
the works of Rouquier \cite{Rou,Rou2} and Khovanov and 
Lauda \cite{KL3}, 
which upgraded from $\mathfrak{sl}_2$
to an arbitrary symmetrizable Kac-Moody algebra $\mathfrak{g}$.

The analogous story for the double covers of symmetric groups 
has an equally long history, being initiated of course by Schur soon after
the ordinary representation theory of symmetric groups was worked out.
In \cite{BK}, we uncovered a connection in the same spirit as Grojnowski's work
 \cite{Groj}---an important predecessor of \cite{CR}---between modular representations of spin symmetric groups 
in odd characteristic $p=2 l+1$ and the
Kac-Moody group of type $A_{2 l}^{(2)}$.
A few years later, the odd theory was 
given new life by work of Ellis, Khovanov and Lauda
\cite{EK, EKL, EL, E}, 
whose motivation came from the completely different direction of the categorification program related to odd Khovanov homology.
They developed a substantial theory of {\em odd symmetric functions} which plays a key role in this article.
Soon after the work of Ellis, Khovanov and Lauda, a major breakthrough was made in work of Kang, Kashiwara, Oh and Tsuchioka \cite{KKT, KKO1,KKO2}. They introduced
so-called {\em quiver Hecke superalgebras}, which are the odd analogs of the Khovanov-Lauda-Rouquier algebras that 
underpin all current approaches to 
categorification of Kac-Moody algebras. In fact, quiver Hecke superalgebras categorify the positive part of the
so-called {\em covering quantum group} $U_{q,\pi}(\mathfrak{g})$ associated to a
super Kac-Moody datum with underlying symmetrizable Kac-Moody algebra
$\mathfrak{g}$. These covering quantum groups were defined
independently and studied in
great detail by
Clark, Wang and Hill \cite{CW, CHW1, CHW2, Clark}.
Then there was a lull in activity, until work of the first author with Ellis
\cite{BE2} which simplified the 
definition of the odd analog of the 2-category associated to $\sl_2$
made originally by Ellis and Lauda \cite{EL}
and extended it to an arbitrary super Kac-Moody datum.
Recently, Dupont, Ebert and Lauda \cite{DEL} have 
used  ``rewriting theory" 
to prove that the odd $\sl_2$ 2-category from \cite{BE2} is non-degenerate, but this is still an open problem for other odd types.

It has in fact been expected for long time that there should exist a comprehensive odd analog of the Chuang and Rouquier theory, and that this should play a role in constructing the derived equivalences required to prove Brou\'e Conjecture for spin symmetric groups. However, due in part to the lack of an appropriate analog of the first part of the proof for symmetric groups---the part provided by the work of Chuang and Kessar---it was not investigated seriously until now. 
This analog has recently been established, in work  
of the second author with Livesey \cite{KLi}, and is in fact highly non-trivial. The arguments in \cite{KLi} depend essentially on the Morita equivalence between cyclotomic quiver Hecke superalgebras 
and group algebras of spin symmetric groups constructed in \cite{KKT},
and also rely on the new approach to the study of RoCK blocks developed by the second author and Evseev in \cite{EvK}.

This article completes the second step of this program for spin symmetric groups. In order to do this, one needs to be able to compute explicitly with some realization of the categorification of
the analog $V(-\ell)$ 
of the $\sl_2$-module of lowest weight $-\ell$
 for the covering quantum group
$U_{q,\pi}(\sl_2)$. We do this in this article by developing a non-trivial theory of {\em odd Grassmannian bimodules}. These are bimodules over pairs of algebras which we refer to as the
{\em equivariant odd Grassmannian cohomology algebras} since they are analogous to the $GL_\ell(\C)$-equivariant cohomology algebras of the usual
Grassmannian of $n$-dimensional subspaces of $\C^\ell$. 
The specialized versions of these algebras
with the word ``equivariant" removed were worked out already by
Ellis, Khovanov and Lauda \cite{EKL}, but the generalization to the equivariant setting is not obvious due to the non-commutativity of the algebra $\OSym$ of odd symmetric functions. 
The definition of equivariant odd Grassmannian cohomology algebras---which are purely algebraic in nature rather than coming from any known cohomology theory---is given in \cref{minuets}, and then the all-important
2-category $\OGBim_\ell$ of bimodules over these algebras is
introduced in \cref{shy}. The key property of this,
its {\em rigidity}, is established in \cref{cupsandcaps,secondadjunction}.

With this theory in place, in \cref{weekend},
we are able to write down the odd analog of the
{\em singular Rouquier complex} in the category $\OGBim_\ell$, proving
the necessary homological properties of this needed to be able to
obtain derived equivalences between the module categories over odd Grassmannian cohomology algebras. After that, we digress to explain the relationship between 
the odd $\sl_2$ 2-category $\UU(\sl_2)$ from \cite{EL, BE2}
(\cref{km2cat})
and the category $\OGBim_\ell$, namely, there is a 2-functor from the former to the latter (\cref{actiontheorem}).
This is the odd analog of the main result about the ordinary $\sl_2$
2-category obtained by Lauda in \cite{Lauda, Lauda2}. We use this 2-functor to give another proof of the
non-degeneracy of the odd $\sl_2$ 2-category established originally in
\cite{DEL}; see \cref{nondegthm}.
This implies that the Grothendieck ring of the super Karoubi envelope
of $\UU(\sl_2)$ is isomorphic to the appropriate integral form of the
covering quantum group $U_{q,\pi}(\sl_2)$.
Then, in \cref{2Repsec},
we develop some of the basic theory of
 2-representations of the odd $\sl_2$ 2-category, following 
\cite{Rou} quite closely. This is applied in the next section
to prove \cref{bigt}, which may be paraphrased as follows:

\vspace{2mm}

\noindent
{\bf Theorem.}
{\em
The bounded homotopy category $K^b(\cV)$
of any integrable Karoubian 2-representation $\cV$ of the odd $\sl_2$ 2-category $\UU(\sl_2)$
admits an auto-equivalence
categorifying the action of the simple reflection
in the associated Weyl group. 
}

\vspace{2mm}

In the final \cref{sapps},
we apply this, together with its even analog from \cite{CR},
to establish the key derived equivalences between blocks of double covers of symmetric and alternating groups predicted by Kessar and Schaps \cite{KS}.
In fact, we establish derived equivalences between 
the corresponding cyclotomic quiver Hecke superalgebras of type $A_{2l}^{(2)}$,
which were shown to be Morita equivalent to spin blocks of symmetric groups
up to Clifford twists in \cite{KKT}. 
Our arguments here also rest crucially on the results of
\cite{KKO1,KKO2} in order to check that representations of cyclotomic quiver Hecke superalgebras do admit the structure of a 
super Kac-Moody 2-representation.
Combined with the results in \cite{KLi}, this is sufficient to complete the proof of the Brou\'e Conjecture for double covers of symmetric and alternating groups. 

We would finally like to discuss some significant overlaps 
between the results of this article
and the independent work of Ebert, Lauda and Vera \cite{ELV}. 
Their work also introduces the equivariant odd Grassmannian cohomology algebras studied here, relating them to deformed odd cyclotomic nilHecke algebras in a similar way to  \cref{geordie} below, and they also establish the derived equivalences necessary to complete the proof of Brou\'e's Conjecture for spin symmetric groups. 
We view their general approach as complementary to ours, and it is reassuring to have an independent proof of this difficult place in the theory.
The strategy adopted by 
Ebert, Lauda and Vera is modelled on Vera's new treatment of derived equivalences in the ordinary even case developed in \cite{V}. 
In particular, it uses a version of the results of 
Kang, Kashiwara and Oh \cite{KKO1,KKO2} to construct the universal categorification of the $U_{q,\pi}(\sl_2)$-module $V(-\ell)$
in terms of representations of deformed odd cyclotomic
nil-Hecke algebras. This is the place where we use instead the theory of odd Grassmannian bimodules developed in this article, making our article more self-contained.


We expect the results here will have further applications, notably, to the representation theory of the Lie superalgebra $\mathfrak{q}_n(\C)$. This article also initiates the study of 2-representations of super Kac-Moody 2-categories in the spirit of Rouquier's 2-representation theory for ordinary Kac-Moody 2-categories.

\vspace{2mm}
\noindent
{\em Acknowledgements.}
We would like to thank Aaron Lauda for his generosity in discussing the results of \cite{ELV} and all of its authors for patiently waiting for our much less concise
text to be completed.

\vspace{2mm}
\noindent
{\em General conventions.}
With the exception of \cref{leavingdepoe}, we work over an algebraically closed field $\k$ of characteristic different from 2, and
all categories, functors, etc. are assumed to be $\k$-linear without further comment. The symbol $\otimes$ with no additional subscript denotes tensor product over $\k$.
We use the shorthand $X \in \cC$ to indicate that $X$ is an element of the object set of a category $\cC$.

Let $\Par$ be the set of all {\em partitions} $\lambda = (\lambda_1 \geq \lambda_2 \geq \cdots)$. We adopt standard notations such as $\lambda^\transpose = (\lambda_1^\transpose,\lambda_2^\transpose,\dots)$ 
for the transpose partition and $\h(\lambda)$ for the number $\lambda^\transpose_1$ of non-zero parts. The usual dominance ordering is denoted $\leq$. The lexicographic ordering 
$\leq_{\lex}$ is a total order refining $\leq$.
We use the English convention for Young diagrams and tableaux, so rows and columns are indexed like for matrices.
Let $\SPar{n} := \{\lambda \in \Par\:|\:\h(\lambda) \leq n\}$ be the set of partitions of height at most $n$ and 
$\GPar{m}{n}$ be the set of partitions $\lambda$
whose Young diagram fits into an $m \times n$ rectangle, i.e.
$\h(\lambda) \leq m$ and $\lambda_1 \leq n$. Note that
\begin{equation*}
\big|\GPar{m}{n}\big| = \binom{m+n}{n}.
\end{equation*}
For $\lambda\in\Par$, the following will be needed in various formulae for signs, following \cite[Sec.~2.2]{E}:
\begin{itemize} 
\item $N(\lambda)$ is the number of pairs $(A,B)$ such that $B$ is strictly north of $A$ (strictly above in any column);
\item $NE(\lambda)$ is the number of pairs of boxes $(A,B)$ such that $B$ is strictly northeast of $A$ (strictly above and strictly to the right);
\item $\NE(\lambda)$ is the number of pairs of boxes $(A,B)$ such that $B$ is weakly northeast of $A$ ([above or in the same row] {and} [to the right or in the same column]);
\item $dN(\lambda)$ is the number of pairs $(A,B)$
of boxes in the Young diagram of $\lambda$ such that $B$ is due north of $A$ (strictly above and in the same column);
\item $dE(\lambda)$
is the number of pairs of boxes $(A,B)$ in the Young diagram such that $B$ 
is due east of $A$ (strictly to the right and in the same row).
\end{itemize}
Some equivalent definitions: 
$N(\lambda) = \sum_{1 \leq i < j} \lambda_i \lambda_j$;
$dN(\lambda) = \sum_{i \geq 1} (i-1) \lambda_i$;
$dE(\lambda) = \sum_{i \geq 1} \binom{\lambda_i}{2}=dN(\lambda^\transpose)$;
$\NE(\lambda)=|\lambda|+dN(\lambda)+dE(\lambda)+NE(\lambda)$.

Let $\Comp(k,n)$ be the set of all {\em compositions} 
of $n$ with $k$ parts, that is, 
$k$-tuples $\alpha = (\alpha_1,\dots,\alpha_k)$ of non-negative integers
such that $|\alpha|:=\alpha_1+\cdots+\alpha_k = n$.
Let $N(\alpha) := \sum_{1 \leq i < j \leq k} \alpha_i\alpha_j$
(like for partitions).
The reversed composition is $\alpha^\reverse := (\alpha_k,\dots,\alpha_1)$.
Also $\alpha\sqcup\beta$ denotes concatenation of compositions
$\alpha$ and $\beta$.

We denote the symmetric group by $\S_n$ acting on the left on $\{1,\dots,n\}$.
The $i$th basic transposition is $s_i = (i\:\,i\!+\!1)$ and $\ell:\S_n \rightarrow \N$ is the associated length function.
We use the notation $w_n$ to denote the longest element of $\S_{n}$ of length $\ell(w_n) = \binom{n}{2}$.
We will often use the identities
$$
\binom{r+s}{2} = \binom{r}{2}+\binom{s}{2} + rs,
\qquad
\binom{-r}{2} = \binom{r+1}{2}
$$
for $r,s \in \Z$. 
For $\alpha\in \Comp(k,n)$, there is a corresponding parabolic subgroup $\S_\alpha$ of $\S_n$; it is the subgroup generated by all $s_i$ for $
i \in \{1,\dots,n\}-\{\alpha_1,\alpha_1+\alpha_2,\dots,\alpha_1+\cdots+\alpha_k\}.
$
We use $\left[\S_n / \S_\alpha\right]_{\min}$ and $\left[\S_\alpha \backslash \S_n\right]_{\min}$ to denote the sets of minimal length left and right coset representatives, respectively.
Also let $w_\alpha$ be the longest element of $\S_\alpha$ and $w^\alpha$ be the longest element of $\left[\S_n / \S_\alpha\right]_{\min}$, so that $w_n = w^\alpha w_\alpha$.
For example, $\S_{(n-1,1)}$ is $\S_{n-1}$ embedded into $\S_n$ as the permutations that fix $n$. These natural embeddings define a tower of subgroups
$\S_0 < \S_1 < \S_2 < \cdots.$
There is also the {\em shifted} embedding $\sig_{n}:\S_{n'} \hookrightarrow \S_{n+n'}, s_j \mapsto s_{n+j}$ for $n,n' \geq 0$.

We may implicitly identify $\lambda \in \Par$ with the ``dominant" composition $(\lambda_1,\dots,\lambda_k) \in \Comp(k,n)$ where $n := |\lambda|$ and $k := \h(\lambda)$.
Note then that $NE(\lambda)$ defined combinatorially above is the length of the unique minimal length
$\S_{\lambda^\transpose} \backslash \S_n / \S_\lambda$-double coset representative $w$ such that $\big|S_{\lambda^\transpose} \cap w S_\lambda w^{-1}\big| = 1$.
For $n \in \Z, r \geq 0$, we let
$$ 
\textstyle
n\# r := n+(n+1)+\cdots+(n+r-1) = nr + \binom{r}{2}.
$$

\section{Graded superalgebra}\label{sgradedsuperalgebra}

In this section, we review some basic language, referring the reader to the exposition in the introduction of \cite{BE} for more details; see also \cite[Sec.~6]{BE} which discusses gradings. 
For a commutative 
ring $R$, we write $R^\pi$ for the ring $R[\pi] / (\pi^2-1)$.
Assuming that 
$2$ is invertible in $R$, the Chinese Remainder Theorem
gives a ring isomorphism $R^\pi \stackrel{\sim}{\rightarrow} 
R \times R,
a \mapsto (a_+, a_-)$ for $a_{\pm} \in R$ defined 
by evaluating $\pi$ at $\pm 1$. Then we have that
$a \in (R^\pi)^\times$ if and only if both $a_+ \in R^\times$
and $a_- \in R^\times$.
For example,
$\pi q - q^{-1} \in \Q(q)^\pi$ is invertible because
both $q-q^{-1}$ and $-q-q^{-1}$ are invertible in $\Q(q)$.

A {\em graded vector superspace} is a
$\Z \times\Z/2$-graded vector space $V = \bigoplus_{d \in \Z, p \in \Z/2} V_{d,p}$. 
We may also write $V_p$ for $\bigoplus_{d \in \Z} V_{d,p}$, so $V_\0$ is the {\em even part} and $V_\1$ is the {\em odd part} of $A$.
For a homogeneous vector $v \in V_{d,p}$, we write $\deg(v)$ for its {\em degree} $d$ (the $\Z$-grading) and $\parity(v)$ for its {\em parity} $p$ (the $\Z/2$-grading). 
We write $\sVecunderline$ for the closed symmetric monoidal category of graded vector superspaces with morphisms that preserve both degree and parity of vectors.
Its symmetric braiding is defined on graded vector superspaces $V$ and $W$ by
\begin{align}
\br_{V,W}:&V \otimes W \rightarrow W \otimes V,
&v \otimes w \mapsto (-1)^{\parity(v) \parity(w)} w \otimes v.
\end{align}
This only makes sense if $v$ and $w$ are homogeneous, but we adopt the usual abuse of notation by suppressing this assumption. 
We use the notation $\Pi$ for the {\em parity switch functor} and $Q$ for the upward {\em grading shift functor}, using $\pi$ and $q$ for the induced maps at the level of Grothendieck groups.

A {\em graded superalgebra} is an associative, unital 
algebra in $\sVecunderline$.
Any graded superalgebra $A$ has the {\em parity involution}
\begin{align}
\operatorname{p}:A&\rightarrow A,&
a \mapsto (-1)^{\parity(a)} a.
\end{align}
For graded superalgebras $A$ and $B$, their tensor product $A \otimes B$ is the tensor product of the underlying graded vector superspaces viewed as a graded superalgebra so that
$$
(a_1  \otimes b_1) (a_2 \otimes b_2) = (-1)^{\parity(b_1)\parity(a_2)} a_1 a_2 \otimes b_1 b_2
$$
for $a_1,a_2 \in A, b_1, b_2 \in B$. 
We also write $A^{\sop}$
for the opposite superalgebra, whose multiplication is defined 
from $a \cdot b := (-1)^{\parity(a)\parity(b)} ba$.

Very important in this article is the graded superalgebra $\OPol_n$ of {\em odd polynomials}. By definition, this is the tensor product
\begin{equation}\label{opol}
\OPol_n := \underbrace{\OPol_1 \otimes \cdots \otimes \OPol_1}_{\text{$n$ times}},
\end{equation}
where $\OPol_1 := \k[x]$ is the usual commutative polynomial algebra in an indeterminate $x$, viewed as a graded
superalgebra so that $x$ is odd of degree 2. 
We let $x_i := 1 \otimes \cdots \otimes x \otimes \cdots \otimes 1$ where $x$ is in 
the $i$th tensor position from the left.
Here are some further observations.
\begin{itemize}
\item
For $\alpha \in \Comp(k,n)$, the tensor product $\OPol_{\alpha_1} \otimes \cdots\otimes\OPol_{\alpha_k}$ is canonically identified with $\OPol_n$
so that $1^{\otimes (j-1)} \otimes x_i\otimes 1^{\otimes(k-j)}\equiv x_{\alpha_1+\cdots+\alpha_{j-1}+i}$.
\item
The elements $x_1,\dots,x_n$ generate $\OPol_n$ subject to the relations
$x_j x_i  = - x_i x_j$
for $1 \leq i < j \leq n$. 
\item
There are no (non-zero) zero divisors in $\OPol_n$, although it does not have unique factorization, e.g., $(x_1-x_2)^2 = (x_1+x_2)^2$.
\item
The monomials
$x^\kappa := x_1^{\kappa_1}\cdots x_n^{\kappa_n} \in \OPol_n$
for $\kappa = (\kappa_1,\dots,\kappa_n) \in \N^n$
give a linear basis for $\OPol_n$. 
\end{itemize}
From the last point, it follows that
\begin{equation}\label{dimensionformula}
\gsdim \OPol_n = \frac{1}{(1-\pi q^2)^n} \in \Z\llbracket q\rrbracket^\pi,
\end{equation}
meaning that the coefficient of $q^d \pi^p$ in this generating 
function is the dimension of the homogeneous component in degree $d \in
\Z$ and parity $p \in \Z/2$.

A {\em graded supercategory} is a category enriched in $\sVecunderline$, and
a {\em graded superfunctor} 
is an enriched functor. In particular, graded superfunctors preserve degrees and parities of morphisms.
Given graded supercategories $\cA$ and $\cB$, 
we use the notation $\SHOM(\cA,\cB)$ for the graded supercategory of graded superfunctors and graded {\em super}natural transformations in the sense of \cite[Def.~1.1]{BE}. (The terminology ``supernatural transformation'' used here appeared earlier in \cite[Def.~2.16]{EL}; it is the appropriate notion of morphism in the $\sVec$-enriched Yoneda Lemma.) In particular, $\SEND(\cA) := \SHOM(\cA,\cA)$ is a strict {\em graded monoidal supercategory}; see \cite[Ex.~1.5(ii)]{BE}.
More generally, there is a strict {\em graded 2-supercategory} $\GSCAT$ 
consisting 
of (small) graded supercategories, graded superfunctors
and graded supernatural transformations; see \cite[Sec.~6]{BE}.
For graded superfunctors 
$F, G:\cA \rightarrow \cB$,
we denote the morphism space
$\Hom_{\SHOM(\cA,\cB)}(F,G)$, 
that is, the graded superspace consisting of all graded
supernatural transformations from $F$ to $G$, simply by $\operatorname{gsHom}(F,G)$. If $F=G$ we denote it by
$\operatorname{gsEnd}(F)$, 
this being a graded superalgebra.

For any graded supercategory $\cA$, we denote the {\em underlying ordinary category} consisting of the same objects and the morphisms that are even of degree 0 by $\underline{\cA}$. 
We will systematically write $\cong$ to denote the existence of an isomorphism that is not necessarily homogeneous, and $\simeq$ to denote the existence of an isomorphism that preserves parities and degrees.
For example,
the category $\sVecunderline$ is the underlying ordinary category of a graded monoidal supercategory $\sVec$.
In $\sVec$, a linear map $f:V \rightarrow W$ is homogeneous of degree $d \in \Z$ and parity $p \in \Z/2$ if $f(V_{d',p'}) \subseteq W_{d+d',p+p'}$
for all $d'\in \Z, p' \in \Z/2$. Then an arbitrary morphism in $\sVec$
is a finite sum of homogeneous linear maps of different degrees. 

In fact, $\sVec$ is a {\em graded monoidal
$(Q,\Pi)$-supercategory}; see
 \cite[Def.~1.12, Def.~6.5]{BE} for the formal definition.
The unit object is the field $\k$ viewed as a
graded superspace so that it is even in degree 0, the parity shift
functor is $\Pi \k \otimes-$ where $\Pi \k$ is $\k$ in degree 0
and odd parity, and grading shift functor is $Q\k \otimes-$ where 
$Q \k$ is $\k$ in degree one and even parity.
For any graded $(Q,\Pi)$-supercategory $\cV$, its 
underlying ordinary category is a {\em $(Q,\Pi)$-category} in the sense of \cite[Def.~6.12]{BE}. 
In fact,  $\sVecunderline$ is a 
{\em monoidal $(Q,\Pi)$-category} in the sense of \cite[Def.~1.14, Def.~6.14]{BE}.

Given two graded superalgebras $A$ and $B$, we write $A\bisMod B$
for the graded supercategory of graded $(A,B)$-superbimodules
and graded $(A,B)$-superbimodule homomorphisms; such a homomorphism is a finite sum of homogeneous $(A,B)$-superbimodule homomorphisms of different degrees and parities.
We adopt the usual sign convention for morphisms as in
\cite[Ex.~1.8]{BE}. So a morphism $f:V \rightarrow W$ in $A\bisMod B$
satisfies $f(avb) = (-1)^{\parity(f)\parity(a)}a f(v) b$ for $a \in A, b \in
B, v \in M$.
The category $A\bisMod B$ is a graded $(Q,\Pi)$-supercategory in the sense of
\cite[Def.~1.7, Def.~6.4]{BE}.
We use the notation $\Hom_{A\dash B}(V,W)$ to denote a morphism space in this category,
which is a graded vector superspace.
The parity switching functor $\Pi$ takes a
graded $(A,B)$-superbimodule $V$ to the same underlying graded 
vector space 
viewed as a superspace with the opposite parities
and actions of $a \in A$ and $b \in B$ on $v \in \Pi V$
defined in terms of the original action so that
\begin{equation}\label{clarifyparify}
a \cdot v\cdot  b := (-1)^{\parity(a)} avb.
\end{equation}
On a morphism $f:V\rightarrow W$, $\Pi f:\Pi V \rightarrow \Pi W$ 
is defined so that $(\Pi f) (v) = (-1)^{\parity(f)} f(v)$.
The grading shift functor $Q$ takes $V$ to the same underlying superbimodule
with the new grading $(Q V)_d := V_{d-1}$. This is less delicate since it does not introduce any additional signs. The definition of
graded $(Q,\Pi)$-supercategory also involves some additional data 
of supernatural isomorphisms $\zeta:\Pi\stackrel{\sim}{\Rightarrow}\Id, \sigma:Q\stackrel{\sim}{\Rightarrow}\Id$ and $\bar\sigma:Q^{-1}\stackrel{\sim}{\Rightarrow}\Id$, which in this case all come from the identity function on the underlying vector space. We will not need these in any significant way, so refer the reader to \cite{BE} for 
the details.

The graded $(Q,\Pi)$-supercategories $A\sMod$ and $\doMs B$ of graded left $A$-supermodules and right $B$-supermodules can now be defined
to be $A\bisMod \k$ and $\k \bisMod B$, respectively. We use the notation
$\Hom_{A\dash}(V,W)$ and $\Hom_{\dash B}(V,W)$ to denote the morphism spaces in these categories. We use
$A\smod$, $A\psMod$ and $A\psmod$
to denote the full subcategories of $A\sMod$ consisting of the finite-dimensional, projective and finitely generated projective left
$A$-supermodules, respectively.
The underlying ordinary categories are
$A\UsMod$, $A\Usmod$, $A\UpsMod$ and $A\Upsmod$.

For graded supercategories $\cA, \cB$, an {\em adjoint pair} $(E,F)$ of graded superfunctors 
$E:\cA \rightarrow \cB$ and $F:\cB \rightarrow \cA$ 
means an adjoint pair of $\k$-linear 
functors in the usual sense, such that in addition the unit and the counit of the adjunction are both even supernatural transformations of degree 0. It follows that the restrictions of $E$ and $F$
to the underlying ordinary categories also form an adjoint pair.

Now suppose that $A$ and $B$ are graded superalgebras such that 
$A$ is a (unital) subalgebra of $B$.
Then there is an adjoint triple of graded superfunctors
$(\ind_A^B, \res_A^B, \coind_A^B)$:
\begin{align}\label{at1}
\ind_A^B := B \otimes_A -&:A \sMod \rightarrow B \sMod,\\ 
\res_A^B := \Hom_{B \dash}(B,-)\simeq B \otimes_B -&:B\sMod \rightarrow A\sMod,\label{at2}\\
\coind_A^B := \Hom_{A \dash}(B, -)&:A\sMod \rightarrow B\sMod.\label{at3}
\end{align}
Following \cite{PS} (which explicitly treats graded superalgebras),
we say that $B$ is a {\em Frobenius extension} of $A$ of degree $d \in \Z$ and parity $p \in \Z/2$ if 
there exists a {\em trace map} $\tr:B \rightarrow A$ that 
is a homogeneous graded $(A,A)$-superbimodule homomorphism 
of degree $-d$ and parity $p$, together with homogeneous elements 
$b_1,\dots,b_m, b_1^\vee,\dots,b^\vee_m$ of $B$
such that
\begin{align}\label{traceprops1}
\deg(b_i)+\deg(b^\vee_i) &= d,&
\parity(b_i)+\parity(b^\vee_i) &= p,\\
\sum_{i=1}^m b_i \tr(b^\vee_i b)&=b, &\sum_{i=1}^m (-1)^{p \parity(b)+p\parity(b^\vee_i)}\tr(b b_i) b^\vee_i&=b\label{traceprops2}
\end{align} 
for any $b \in B$.
The associated {\em comultiplication} is
the homogeneous graded $(B,B)$-superbimodule homomorphism 
\begin{align}\label{generaltheory}
\Delta:B &\rightarrow B \otimes_A B,&1 \mapsto \sum_{j=1}^m (-1)^{p\parity(b_j^\vee)}b_j \otimes  b_j^\vee,
\end{align}
which is of degree $d$ and parity $p$.
In the adjoint pair $(\ind_A^B, \res_A^B)$,
the unit and counit of the canonical adjunction making
$(\ind_A^B, \res_A^B)$ into an adjoint pair
are induced by the superbimodule homomorphisms
$\eta:A \rightarrow B$ and $\mu:B \otimes_A B \rightarrow B$
given by the canonical inclusion and multiplication, respectively.
Assuming that we have a Frobenius extension, there is also an adjunction making
$(\res_A^B,  Q^{-d}\Pi^p \ind_A^B)$
into an adjoint pair, with
unit and counit of adjunction 
induced by the superbimodule homomorphisms
$\tr$ and $\Delta$ viewed now as homogeneous 
graded supermodule homomorphisms $\tr:Q^{-d} \Pi^p B \rightarrow A$
and $\Delta:B \rightarrow Q^d\Pi^p  B \otimes_A B$ that are even of
degree 0. This adjunction induces a canonical even degree 0 isomorphism
$\ind_A^B \simeq Q^{d}\Pi^p  \coind_A^B$.
Conversely, if there exists such an isomorphism of graded
superfunctors then 
 $B$ is a Frobenius extension of $A$ of degree $d \in \Z$ and parity
 $p \in \Z/2$; 
see \cite[Th.~6.2]{PS}.

We will only use the definitions from the previous paragraph in situations in which $B$ is positively graded and connected (i.e., $B_{0} = \k$). In that case,
the trace map is unique up to multiplication by a non-zero scalar;
see \cite[Prop.~4.7]{PS} for a more general uniqueness statement.
Moreover, 
the elements $b_1,\dots,b_m$ 
(resp., $b_1^\vee,\dots,b_m^\vee$) can be chosen so that they
give a basis for $B$ as a free right (resp., left) $A$-supermodule,
in which case \cref{traceprops2} can be replaced simply by the condition
\begin{equation}\label{newcond}
\tr(b_i^\vee b_j) = \delta_{i,j}
\end{equation}
for all $i,j$, i.e., the two bases are {\em dual} to each other.

Now suppose that $A$ is a {\em supercommutative} graded superalgebra,
i.e., its multiplication satisfies 
$a b = (-1)^{\parity(a)\parity(b)}ba$ for all $a,b \in A$.
A {\em graded $A$-superalgebra} $B$ is a graded 
superalgebra together with a 
structure map $\eta:A \rightarrow Z(B)$ which is a (unital) graded superalgebra homomorphism from $A$ to the {\em supercenter}
\begin{equation}
Z(B) = \big\{c \in B\:\big|\:bc = (-1)^{\parity(b)\parity(c)}cb\text{
  for all }b \in B\big\}.
\end{equation}
In particular, such a superalgebra $B$ is a {\em graded $A$-supermodule}, by
which we mean a graded
$(A,A)$-superbimodule
such that the left and right actions of $a \in A$ on a vector $v$ are related by
\begin{equation}
a v = (-1)^{\parity(a)\parity(v)} va.
\end{equation} 
We say that a graded $A$-superalgebra $B$ is a 
{\em graded Frobenius superalgebra} over $A$ 
of degree $d$ and parity $p$ if 
the structure 
map $\eta:A \rightarrow Z(B)$ is injective, and
$B$ is a Frobenius extension of $\eta(A)$ of degree $d$ and parity $p$ in the sense from the previous paragraph.

By the {\em graded super Karoubi envelope} $\GSKar(\cA)$
of a graded supercategory $\cA$
we mean the graded $(Q,\Pi)$-supercategory 
obtained by first passing to its
$(Q,\Pi)$-envelope $\cA_{q,\pi}$ (see the next paragraph),
then to the additive envelope of $\cA_{q,\pi}$, then finally
to the completion of that at all homogeneous idempotents.
The underlying ordinary category 
$\gsKarunderline(\cA)$, which is
an additive and idempotent complete $(Q,\Pi)$-category, is what 
is called the graded super Karoubi envelope in the final paragraph of \cite[Sec.~6]{BE}.
Any graded superfunctor $F:\cA \rightarrow \cB$ to an additive graded
$(Q,\Pi)$-supercategory $\cB$ 
whose underlying ordinary category is idempotent complete
induces a canonical graded superfunctor
$F_{q,\pi}:\GSKar(\cA) \rightarrow \cB$; this follows from \cite[Th.~6.3]{BE} combined with the usual universal properties of 
additive envelopes and idempotent completions.
The split Grothendieck group $K_0\big(\gsKarunderline(\cA)\big)$ is naturally a $\Z[q,q^{-1}]^\pi$-module.
For example, if $\cA$ is the graded supercategory with one object
whose endomorphisms are given by some graded superalgebra $A$
then, by Yoneda Lemma, $\GSKar(\cA)$ is contravariantly
graded superequivalent to $A\psmod$. In this case, we denote
$K_0\big(\gsKarunderline(\cA)\big)
\cong K_0(A\Upsmod)$ simply by $K_0(A)$.

The only part of the construction of $\GSKar(\cA)$ just outlined
which is not standard is the
notion of the {\em $(Q,\Pi)$-envelope} of $\cA$. According to \cite[Def.~6.8]{BE},
this is the graded supercategory $\cA_{q,\pi}$ with objects 
given by the symbols
$Q^d \Pi^p X$ for $d \in \Z, p \in \Z/2$ and $X \in \cA$,
with
$$
\Hom_{\cA_{q,\pi}}(Q^d \Pi^p X, Q^e \Pi^q Y)
:= Q^{e-d} \Pi^{p+q} \Hom_{\cA}(X,Y),
$$
where $Q$ and $\Pi$ on the right hand side are the grading and parity shift functors on $\sVec$.
Denoting $f \in \Hom_{\cA}(X,Y)$
viewed as a morphism in
$\Hom_{\cA_{q,\pi}}(Q^d \Pi^p X, Q^e \Pi^q Y)$
by $f^{e,q}_{d,p}$ as in \cite{BE}, the composition law in
$\cA_{q,\pi}$ is induced by the composition law in $\cA$
in the obvious way: 
\begin{equation}\label{vertcomp}
g_{d,q}^{e,r} \circ f_{c,p}^{d,q}
:= (g \circ f)_{c,p}^{e,r}.
\end{equation}
The grading and parity shift functors making
$\cA_{q,\pi}$ into a graded 
$(Q,\Pi)$ are defined on objects so that
$Q (Q^d \Pi^p X) = Q^{d+1} \Pi^p X$ and
$\Pi(Q^d \Pi^p X) = Q^d \Pi^{p+\1} X$,
and on morphisms so that $Q(f_{d,p}^{e,q}) = f_{d+1,p}^{e+1,q}$
and $\Pi f_{d,p}^{e,q} = (-1)^{\parity(f)+p+q} f_{d,p+\1}^{e,q+\1}$.
For a more complete account (including also the definitions of the
required supernatural isomorphisms $\zeta, \sigma$ and $\bar\sigma$)
we refer to \cite[Def.~6.8]{BE}.
We will always {\em identify} $\cA$ with the full subcategory
of $\cA_{q,\pi}$ via the obvious graded superfunctor
$\cA \rightarrow \cA_{q,\pi}, X \mapsto Q^0 \Pi^\0 X,
f \mapsto f_{0,\0}^{0,\0}$.
The graded supercategory $\cA$
is called {\em $(Q,\Pi)$-complete}
if this functor is a graded superequivalence;
equivalently, for every object $X \in \cA$,
$d \in \Z$ and $p \in \Z/2$ there is another
object $Y$ and an isomorphism
$f:Y\stackrel{\sim}{\rightarrow} X$
that is homogeneous of degree $d$ and parity $p$; cf. 
\cite[Lem.~4.1]{BE}.

It also makes sense to take the graded super Karoubi envelope 
$\GSKar(\AA)$ of
a graded 2-supercategory $\AA$, 
which is a graded $(Q,\Pi)$-2-supercategory.
The split Grothendieck ring $K_0\big(\gsKarunderline(\AA)\big)$ is naturally a locally unital $\Z[q,q^{-1}]^\pi$-algebra
equipped with a distinguished collection of mutually orthogonal
idempotents indexed by the objects of $\AA$.
We just 
explain the non-trivial step here, which is
the construction of the $(Q,\Pi)$-envelope $\AA_{q,\pi}$
of a graded 2-supercategory,
referring to \cite[Def.~6.10]{BE} and the subsequent discussion
for a fuller account.
In the case that $\AA$ is a strict graded 2-supercategory,
$\AA_{q,\pi}$ is a strict graded $(Q,\Pi)$-2-supercategory with 
the same objects as $\AA$ and morphism supercategories that 
 are the $(Q,\Pi)$-envelopes of the morphism supercategories in $\AA$. 
For 1-morphisms $X, Y: \lambda \rightarrow \mu$ in $\AA$, 
we represent the 2-morphism
$Q^d \Pi^p X \Rightarrow Q^e \Pi^q Y$
in $\AA_{q,\pi}$
associated to the 2-morphism $\alpha:X \Rightarrow Y$
by $\alpha_{d,p}^{e,q}$, which 
is of parity $\parity(\alpha)+p+q$ and degree $\deg(\alpha)+e-d$.
Vertical composition is defined in the obvious way as in
\cref{vertcomp}.
Horizontal composition is defined in terms of horizontal composition in $\AA$ but there are some surprising signs:
\begin{equation}
\alpha_{d,p}^{e,q} (\alpha')^{e',q'}_{d',p'}
:=
(-1)^{p(\parity(\alpha')+p'+q')+\parity(\alpha)q'}
(\alpha \alpha')_{d+d',p+p'}^{e+e',q+q'}
\label{horizcomp}
\end{equation}
These signs play an important role in the following lemma,
which when $p = \1$ is the idea of ``odd adjunction".

\begin{lemma}\label{oddadjunction}
Let $\AA$ be a strict graded 2-supercategory.
Suppose that $E:\lambda\rightarrow\mu$ and $F:\mu\rightarrow\lambda$ 
are 1-morphisms and
$\eps:E\circ F \Rightarrow 1_\mu$ and $\eta:1_\lambda \Rightarrow F\circ E$
are 2-morphisms of parity $p$
and degrees $d$ and $-d$, respectively, such that
$(\eps 1_E) \circ (1_E \eta) = 1_E$
and $(1_F \eps) \circ (\eta 1_F) = (-1)^p 1_F$.
\begin{enumerate}
\item
The degree 0 even 2-morphisms
$\eps_{d,p}^{0,\0}:E\circ (Q^d \Pi^p F) \Rightarrow \Id$ and $\eta_{0,\0}^{d,p}:\Id \Rightarrow (Q^{d} \Pi^p F) \circ E$ give the unit and counit of an adjunction
making $Q^{d} \Pi^p F$ into a right dual of $E$ in 
$\AA_{q,\pi}$.
\item
For a supernatural transformation 
$\alpha:E \Rightarrow E$ in $\AA$,
its mate $\Big(1_{Q^d \Pi^p F} \eps_{d,p}^{0,\0}\Big)\circ \Big(1_{Q^d \Pi^p F}
\alpha 1_{Q^d \Pi^p F}\Big)
\circ \Big(\eta_{0,\0}^{d,p} 1_{Q^d \Pi^p F}\Big):
Q^d \Pi^p F \Rightarrow Q^d \Pi^p F$ 
 with respect to the adjunction constructed in (1) is equal to 
$(-1)^p
\big((1_F \eps) \circ (1_F \alpha 1_F) \circ (\eta 1_F)\big)_{d,p}^{d,p}$.
\end{enumerate}
\end{lemma}

\begin{proof}
(1)
We need to show that
$\Big(\eps_{d,p}^{0,\0} 1_E\Big)\circ\Big(1_E \eta_{0,\0}^{d,p}\Big)
= 1_E$
and
$\Big(1_{Q^d\Pi^p F} \eps_{d,p}^{0,\0}\Big)\circ\Big(\eta_{0,\0}^{d,p}
1_{Q^d \Pi^p F}\Big)
= 1_{Q^d \Pi^p F}$.
We just check the second of these (the signs are more interesting!).
We have that $1_{Q^d \Pi^p F} = Q^d \Pi^p 1_F = (1_F)^{d,p}_{d,p}$.
So
\begin{align*}
\Big(1_{Q^d\Pi^p F} \eps_{d,p}^{0,\0}\Big)\circ\Big(\eta_{0,\0}^{d,p}
1_{Q^d \Pi^p F}\Big)
&=
\Big((1_F)_{d,p}^{d,p} \eps_{d,p}^{0,\0}\Big)
\circ \Big(\eta_{0,\0}^{d,p} (1_F)_{d,p}^{d,p}\Big)
=
(1_F \eps)_{2d,\0}^{d,p}
\circ \Big(\eta_{0,\0}^{d,p} (1_F)_{d,p}^{d,p}\Big)\\
&=
(-1)^p (1_F \eps)_{2d,\0}^{d,p}
\circ (\eta 1_F)_{d,p}^{2d,\0}
=
(-1)^p \big((1_F \eps)\circ (\eta 1_F)\big)_{d,p}^{d,p}\\
&=
(1_F)_{d,p}^{d,p}
= 1_{Q^d \Pi^p F}.
\end{align*}

\vspace{2mm}
\noindent
(2) This is another such explicit calculation.
We just note that $1_{Q^d \Pi^p F} \alpha 1_{Q^d \Pi^p F}
= (1_F \alpha 1_F)_{2d,\0}^{2d,\0}$ regardless of the parity of $\alpha$.
\end{proof}

\section{Combinatorics of \texorpdfstring{$(q,\pi)$-}{}binomial coefficients and odd quantum \texorpdfstring{$\mathfrak{sl}_2$}{sl(2)}}\label{leavingdepoe}

In this section, we recall briefly the definition of the enveloping algebra of ``odd quantum $\sl_2$'' discovered by Clark and Wang \cite{CW} and developed in much greater generality in \cite{CHW1}. 
We work initially over $\Q(q)^\pi$; cf. the opening paragraph
of \cref{sgradedsuperalgebra}.
The most significant difference compared to \cite{CHW1, Clark} 
is that our $q$ is $q^{-1}$ in  \cite{CHW1} and $\nu^{-1}$ in \cite{Clark}. We define the {\em $(q,\pi)$-integers} 
\begin{equation}\label{qinteger}
[n]_{q,\pi} := \frac{(\pi q)^n - q^{-n}}{\pi q - q^{-1}} = \left\{
\begin{array}{ll}
q^{1-n}+\pi q^{3-n}+\cdots+\pi^{n-1} q^{n-1}&\text{if $n \geq 0$,}\\ 
-\pi^n (q^{n+1}+\pi q^{n+3}+\cdots+\pi^{-n-1} q^{-n-1})&\text{if $n \leq 0$}
\end{array}\right.
\end{equation}
for any $n \in \Z$.
This is exactly the same as the definition given in
\cite[Sec.~1.6]{CHW1}---but because our $q$ is the $q^{-1}$ in
\cite{CHW1} it is actually a different convention!
For $n \neq 0$, the $(q,\pi)$-integer 
$[n]_{q,\pi}$ is invertible in the
ring $\Q(q)^\pi$; this follows because
the elements of $\Q(q)$ obtained from $[n]_{q,\pi}$
by setting $\pi = \pm 1$ are both non-zero.
Note also that
\begin{equation}\label{negativeintegers}
[-n]_{q,\pi} = -\pi^n [n]_{q,\pi}.
\end{equation}
There are corresponding $(q,\pi)$-factorials $[n]_{q,\pi}^!$ for $n \geq 0$:
\begin{align}\label{poincare}
[n]_{q,\pi}^!& := [n]_{q,\pi} [n-1]_{q,\pi}\cdots [1]_{q,\pi} = q^{-\binom{n}{2}} \sum_{w \in \S_n} (\pi q^{2})^{\ell(w)},
\end{align}
where the last equality is a consequence 
of the well-known factorization of the Poincar\'e polynomial for the symmetric group.
Then we have the $(q,\pi)$-binomial coefficients $\sqbinom{n}{r}_{q,\pi}$, which make sense as written for any $n \in \Z$ and $r \geq 0$:
\begin{align}
\sqbinom{n}{r}_{q,\pi} &:=\frac{[n]_{q,\pi} [n-1]_{q,\pi} \cdots [n-r+1]_{q,\pi}} {[r]^!_{q,\pi} }.
\end{align}
We also adopt the convention that $\sqbinom{n}{r}_{q,\pi}=0$ for any $n \in \Z$ and $r < 0$. Note by \cref{negativeintegers} that
\begin{equation}
\sqbinom{-n}{r}_{q,\pi} = (-1)^r \pi^{nr+\binom{r}{2}}
\sqbinom{n+r-1}{r}_{q,\pi}.
\end{equation}
We also need quantum {\em tri}nonomial coefficients for $n \in \Z$ and $r,s \geq 0$:
\begin{equation}
\sqbinom{n}{r,s}_{q,\pi} := \frac{[n]_{q,\pi} [n-1]_{q,\pi}\cdots [n-r-s+1]_{q,\pi}} {[r]_{q,\pi}^! [s]_{q,\pi}^!} = \sqbinom{n}{r}_{q,\pi} \sqbinom{n-r}{s}_{q,\pi}.
\end{equation}
Again we interpret $\sqbinom{n}{r,s}_{q,\pi}$ as zero if $r < 0$ or $s < 0$.
More generally, for $\alpha \in \Comp(k,n)$, let
\begin{equation}
\displaystyle\sqbinom{n}{\alpha}_{q,\pi}:= 
\frac{[n]^!_{q,\pi}}{[\alpha_1]^!_{q,\pi}\cdots [\alpha_k]^!_{q,\pi}}
\end{equation}
be the $(q,\pi)$-multinomial coefficient. The identity \cref{poincare} implies
that
\begin{equation}\label{multiidentity}
\sqbinom{n}{\alpha}_{q,\pi}
=
q^{-N(\alpha)}
\sum_{w \in [\S_n / \S_\alpha]_{\min}}
(\pi q^2)^{\ell(w)}.
\end{equation}

We let $-:\Q(q)^\pi\rightarrow \Q(q)^\pi$ be the $\Q^\pi$-algebra
involution with $\overline{q} = q^{-1}$. We use the word {\em
  anti-linear} for a $\Z$-module homomorphism  $f:V \rightarrow W$ 
between $\Q(q)^\pi$- or $\Z[q,q^{-1}]^\pi$-modules 
such that $f(c v) = \overline{c} f(v)$.
We have that
\begin{align}\label{baringpi}
\overline{[n]}_{q,\pi} &=\pi^{n-1} [n]_{q,\pi}, 
& \overline{[n]^!}_{\!\!q,\pi} &= \pi^{\binom{n}{2}}[n]_{q,\pi}^!,\\\label{goingtosunriver}
\overline{\sqbinom{n}{r}}_{q,\pi} &=\pi^{(n-r)r}\sqbinom{n}{r}_{q,\pi},
& \overline{\sqbinom{n}{r,\!s}}_{q,\pi}&= \pi^{(n-r)(r+s)+s} \sqbinom{n}{r,\!s}_{q,\pi}.
\end{align}
We stress that our bar involution is {\em  not} the same as the bar involution introduced in \cite{CHW1,Clark};
the latter takes $q$ to $\pi q^{-1}$ and {\em fixes} the
$(q,\pi)$-integers, $(q,\pi)$-factorials and $(q,\pi)$-binomial. 
Some 
further properties of $(q,\pi)$-binomial and trinomial coefficients are proved in the next two lemmas.

\begin{lemma}\label{bin}
The $(q,\pi)$-binomial and trinomial 
coefficients have the following properties.
\begin{enumerate}
\item For $n \in \Z$ and $r \geq 0$, we have that
\begin{equation*}
\displaystyle\sqbinom{n}{r}_{q,\pi} = q^{-r}\sqbinom{n-1}{r}_{q,\pi} + (\pi q)^{n-r} \sqbinom{n-1}{r-1}_{q,\pi} = (\pi q)^{r}\sqbinom{n-1}{r}_{q,\pi} +   q^{r-n} \sqbinom{n-1}{r-1}_{q,\pi}.
\end{equation*}
\item For $n \in \Z$ and $r, s \geq 0$, we have that
\begin{align*}
\sqbinom{n}{r,s}_{q,\pi} &= \pi^{s}q^{s-r} \sqbinom{n-1}{r,s}_{q,\pi} + (\pi q)^{n-r} \sqbinom{n-1}{r-1,s}_{q,\pi} + q^{s-n}
        \sqbinom{n-1}{r,s-1}_{q,\pi}.
\end{align*}
\item
For $n \in \Z$ and $r \geq 0$, we have that
\begin{align*}
\sum_{s+t=r} \pi^{\binom{t}{2}}(-q)^{-t} \sqbinom{n+s}{s,t}_{q,\pi} &= (\pi q)^{nr}.
\end{align*}
\end{enumerate}
\end{lemma}

\begin{proof}
\vspace{1mm}
\noindent
(1) 
The first equality follows from the definition of $(q,\pi)$-binomial
coefficient by replacing the $[n]_{q,\pi}$ in the numerator with
$q^{-r} [n-r]_{q,\pi} + (\pi q)^{n-r}[r]_{q,\pi}$ and then splitting
the result into two fractions. The second follows by replacing it
instead with
$\pi^r q^{r} [n-r]_{q,\pi} + q^{r-n}[r]_{q,\pi}$. 

 \vspace{1mm}
 \noindent
 (2) Using (1) twice, we have that
\begin{align*}
\pi^{s}q^{s-r} \sqbinom{n-1}{r,s}_{q,\pi}\!\!\!+ q^{s-n} \sqbinom{n-1}{r,s-1}_{q,\pi}
&=q^{-r}\sqbinom{n-1}{r}_{q,\pi} \left((\pi q)^{s} \sqbinom{n-r-1}{s}_{q,\pi}+ q^{r+s-n} \sqbinom{n-r-1}{s-1}_{q,\pi}\right)\\
&=   q^{-r} \sqbinom{n-1}{r}_{q,\pi}\sqbinom{n-r}{s}_{q,\pi}\!\!
=   \left(\sqbinom{n}{r}- (\pi q)^{n-r} \sqbinom{n-1}{r-1}_{q,\pi}\right)\sqbinom{n-r}{s}_{q,\pi}\\
&=\sqbinom{n}{r,s}_{q,\pi} - (\pi q)^{n-r} \sqbinom{n-1}{r-1,s}_{q,\pi}.
\end{align*}
    
\vspace{1mm}
\noindent
(3)
Let $a_n(r) := \sum_{s+t=r} \pi^{\binom{t}{2}} (-q)^{-t} \sqbinom{n+s}{s,t}_{q,\pi}$.
The goal is to show that $a_n(r) = (\pi q)^{nr}$. It is easy to check that this is true when $nr=0$. This gives the base of an induction.  For the induction step, we also need the identity
$$
a_n(r)=\pi^{nr}(\pi q)^{-r}\; \overline{a_{n-1}(r)}+ (\pi q)^{n}a_n(r-1) - \pi^{nr} (\pi q)^{1-n-r}\;\overline{a_{n-1}(r-1)},
$$
which will be verified in the next paragraph. Using this, it is easy to complete the proof for all $n \geq 0$ and $r \geq 0$ by induction on $n+r$. The proof for $n \leq 0$ and $r \geq 0$ goes instead by induction on
$r-n$ using  the following:
$$
a_{n}(r)=\pi^{nr}q^{-r} \;\overline{a_{n+1}(r)}- \pi^{nr} (\pi q)^{-n-1}q^{-r}\;\overline{a_{n+1}(r-1)} + (\pi q)^{n}a_{n}(r-1)
$$
This follows by applying $-$ to the previous identity, then replacing $n$ by $n+1$ and rearranging.

It remains to prove the first identity. Using (2), we have that 
\begin{align*}
a_n(r)&= \sum_{s+t=r} \pi^{\binom{t}{r}}(-q)^{-t}\left( \pi^t q^{t-s}  \sqbinom{n+s-1}{s,t}_{q,\pi} + (\pi q)^{n}\sqbinom{n+s-1}{s-1,t}_{q,\pi} + q^{t-s-n}\sqbinom{n+s-1}{s,t-1}_{q,\pi}\right).
\end{align*}
Moving the sum inside the parentheses produces three terms which we simplify separately, reindexing the second sum by replacing $s$ by $s+1$ and the third sum by replacing $t$ by $t+1$. We also use
\begin{align*}
\overline{a_n(r)} = \pi^{nr} \sum_{s+t=r}  \pi^{\binom{t+1}{2}} (-q)^{t} \sqbinom{n+s}{s,t}_{q,\pi},
\end{align*}
which follows by \cref{goingtosunriver}. In this way, the three terms become:
\begin{align*}
q^{-r} \sum_{s+t=r}\pi^{\binom{t+1}{2}}(-q)^{t}\sqbinom{n+s-1}{s,t}_{q,\pi} 
&= \pi^{(n-1)r} q^{-r} \;\overline{a_{n-1}(r)},\\
(\pi q)^{n}\sum_{s+t=r-1}\pi^{\binom{t}{2}} (-q)^{-t} \sqbinom{n+s}{s,t}_{q,\pi}&= (\pi q)^{n}a_n(r-1),\\
-q^{1-n-r} \sum_{s+t=r-1} \pi^{\binom{t+1}{2}}(-q)^{t}  \sqbinom{n+s-1}{s,t}_{q,\pi}
&= - q^{1-n-r}\pi^{(n-1)(r-1)} \overline{a_{n-1}(r-1)}.
\end{align*}
The sum of these three produces the right hand side of the identity we are proving.
\end{proof}

\begin{corollary}\label{qbinomialformula}
For $0 \leq r \leq n$, we have that
\begin{align*}
q^{(n-r)r}\sqbinom{n}{r}_{q,\pi} &=\sum_{\lambda\in \GPar{r}{(n-r)}} 
(\pi q^2)^{|\lambda|}.
\end{align*}
\end{corollary}

\begin{proof}
This is an induction exercise using \cref{bin}(1).
\end{proof}

Recall from the {\em General conventions}
that $n \# r$ denotes $n + (n+1)+\cdots+(n+r-1)$.

\begin{lemma}\label{numerology}
For $m, n \in \Z$ and $r \geq 0$, we let
\begin{align*}
b_{m,n}(r) &:=
(\pi q^{-2})^{(n-r)\# r}
\sum_{s=0}^{r-1} (\pi q^{2})^{n-r+m(r-s-1)}
q^{(m-n+r-1)(n-r+s+1)+(n-r)s}
\sqbinom{m+s}{n-r+s+1}_{q,\pi} \sqbinom{n-r+s}{s}_{q,\pi},
\\
c_{m,n}(r) &:= 
(\pi q^{-2})^{(n-r)\# r}
q^{(m-n+r)n+(n-r)r}\sqbinom{m+r}{n}_{q,\pi} \sqbinom{n}{r}_{q,\pi}.
\end{align*} 
Then we have that
$c_{m,n}(r) = b_{m,n}(r) + b_{m,n}(r+1)$
for any $m,n \in \Z$ and $r \geq 0$.
\end{lemma}

\begin{proof}
Proceed by induction on $r$. The base case $r=0$ is easily checked.
For the induction step, take $r > 0$.
We have that 
\begin{align*}
b_{m,n}(r) &= 
(\pi q^{-2})^{n-m-1}
b_{m,n-1}(r-1)
+
(\pi q^{-2})^{(n-r)\# r - n+r} 
q^{(m-n+r-1)n + (n-r)(r-1)}
\sqbinom{m+r-1}{n}_{q,\pi} \sqbinom{n-1}{r-1}_{q,\pi},\\
b_{m,n}(r+1) &= 
(\pi q^{-2})^{n-m-1}
b_{m,n-1}(r)
+
(\pi q^{-2})^{(n-r-1)\#(r+1) -n+r+1} 
q^{(m-n+r)n + (n-r-1)r}
\sqbinom{m+r}{n}_{q,\pi} \sqbinom{n-1}{r}_{q,\pi}.
\end{align*}
Note that 
$(n-r)\#(r-1) = (n-r)\# r -n+1$ and
$(n-r-1)\#(r+1) -n+r+1
= (n-r)\# r$.
Adding the above equations and using the induction hypothesis gives that
\begin{multline*}
b_{m,n}(r)+b_{m,n}(r+1)
=
(\pi q^{-2})^{(n-r)\# r - m}
q^{(m-n+r)(n-1)+(n-r)(r-1)}\sqbinom{m+r-1}{n-1}_{q,\pi} \sqbinom{n-1}{r-1}_{q,\pi}\\
\qquad\qquad\qquad\qquad
+(\pi q^{-2})^{(n-r)\# r - n+r} 
q^{(m-n+r-1)n + (n-r)(r-1)}
\sqbinom{m+r-1}{n}_{q,\pi} \sqbinom{n-1}{r-1}_{q,\pi}\\
+(\pi q^{-2})^{(n-r)\# r} 
q^{(m-n+r)n + (n-r-1)r}
\sqbinom{m+r}{n} \sqbinom{n-1}{r}_{q,\pi}.
\end{multline*}
Using the identity $(\pi q)^{m-n+r}\sqbinom{m+r-1}{n-1}_{q,\pi}+ q^{-n}\sqbinom{m+r-1}{n}_{q,\pi} = \sqbinom{m+r}{n}_{q,\pi}$
from \cref{bin}(1),
the first two terms combine into one leaving us with
$$
(\pi q^{-2})^{(n-r)\# r - n+r} 
q^{(m-n+r)n + (n-r)(r-1)}
\sqbinom{m\!+\!r}{n}_{q,\pi} \sqbinom{n\!-\!1}{r\!-\!1}_{q,\pi}
\!\!\!\!+(\pi q^{-2})^{(n-r)\# r} 
q^{(m-n+r)n + (n-r-1)r}
\sqbinom{m\!+\!r}{n}_{q,\pi} \sqbinom{n\!-\!1}{r}_{q,\pi}\!\!\!\!.
$$
Then we use the identity $(\pi q)^{n-r}\sqbinom{n-1}{r-1}_{q,\pi}+q^{-r}\sqbinom{n-1}{r} = \sqbinom{n}{r}_{q,\pi}$
to see finally that this is equal to $c_{m,n}((r)$.
\end{proof}

\begin{corollary}\label{numerologyc}
$\displaystyle 
q^{(n-r)r} \sqbinom{n}{r}_{q,\pi}\!\!=
\sum_{s=0}^{r} (\pi q^{2})^{(n-r)(r-s)}
q^{(n-r-1)s}
\sqbinom{n-r+s-1}{s}_{q,\pi}$ for $0 \leq r \leq n$.
\end{corollary}

\begin{proof}
Take $m=n-r$ in \cref{numerology}.
\end{proof}

Let $U_{q,\pi}(\sl_2)$ denote the locally unital $\Q(q)^\pi$-algebra with distinguished idempotents $\{1_k\:|\:k \in \Z\}$ and generators $\E 1_k = 1_{k+2} \E, \F 1_k = 1_{k-2} \F$ subject to the relations
\begin{align}\label{fridaylate}
\E\F 1_k -\pi \F\E 1_k &= \overline{[k]}_{q,\pi} 1_k
\end{align}
for all $k \in \Z$. Note there is some flexibility in writing the idempotents $1_k$---in any given monomial one just needs to include one idempotent somewhere in the word for the notation to be unambiguous. Let
\begin{align}\label{snow}
\E^{(d)}1_k = 1_{k+2d} E^{(d)} &:= \frac{\E^d 1_k}{[d]_{q,\pi}^!},
&1_k\F^{(d)} = \F^{(d)} 1_{k+2d} &:= \frac{\F^d 1_k}{[d]^!_{q,\pi}},\\
\label{snowtilde}
\overline{\E}^{(d)}1_k= 1_{k+2d} \overline{\E}^{(d)} &:= \frac{\E^d 1_k}{\overline{[d]^!}_{\!q,\pi}},&
1_k \overline{\F}^{(d)} = \overline{\F}^{(d)} 1_{k+2d} &:= \frac{\F^d 1_k}{\overline{[d]^!}_{\!q,\pi}}.
\end{align}
By \cref{goingtosunriver}, we have that
\begin{align}\label{snowier}
\overline{\E}^{(d)} 1_k &= \pi^{\binom{d}{2}} \E^{(d)} 1_k,&
1_k \overline{\F}^{(d)} &= \pi^{\binom{d}{2}} 1_k \F^{(d)}.
\end{align}
There is an anti-linear involution
\begin{align}
\oldomega:U_{q,\pi}(\sl_2)&\rightarrow U_{q,\pi}(\sl_2),
&1_k&\mapsto 1_{-k},\: \E 1_k\mapsto \F 1_{-k},\: \F 1_k\mapsto \E 1_{-k}.
\end{align}
This sends $\E^{(d)} 1_k \mapsto \overline{\F}^{(d)} 1_{-k}$ and $\F^{(d)} 1_k \mapsto \overline{\E}^{(d)} 1_{-k}$.
We warn the reader that this is different from the involution $\omega$ in \cite{CHW1}.

By a $U_{q,\pi}(\sl_2)$-module we mean a locally unital left module $V  =\bigoplus_{k \in \Z} 1_k V$. We call $1_k V$ the {\em $k$-weight space} of $V$. We say that $V$ is {\em integrable} if any weight vector $v \in 1_k V$ for $k\in \Z$ is annihilated by  $\E^n 1_k$ and $\F^n 1_k$ for $n \gg 0$ (depending on $v$). For $\ell \in \N$---a dominant weight for $\sl_2$---there is a $U_{q,\pi}(\sl_2)$-module $V(-\ell)$ which is free as a $\Q(q)^\pi$-module with basis $\{b_n^\ell\:|\:0 \leq n \leq \ell\}$ such that
\begin{itemize}
\item $b_n^\ell$ is of weight $2n-\ell$, i.e., $1_{2n-\ell} b_{n}^\ell = b_{n}^\ell$;
\item $b_{\ell}^\ell$ is a highest weight vector
and $b_0^\ell$ is a lowest weight vector, i.e., $\E b^\ell_\ell=\F b_0^\ell=0$;
\item for $0 \leq n < \ell$, we have that $\E b_n^\ell = [n+1]_{q,\pi} b_{n+1}^\ell$ and $\F b_{n+1}^\ell = \pi^n [\ell-n]_{q,\pi} b_{n}^\ell$.
\end{itemize}
We visualize the action with the familiar $\sl_2$-type picture showing how the operators $\E$ and $\F$ raise and lower basis vectors to multiples of basis vectors:
\begin{equation}\label{mmm}
\begin{tikzcd}
\phantom{a}& b_{\ell}^\ell\arrow[d,bend left,"{\pi^{\ell-1}[1]_{q,\pi}}" right]&\phantom{axyzw}\arrow[dddd,dashed, "\F" right]\\
\phantom{a}&  \arrow[u,bend left,"{[\ell]_{q,\pi}}" left] b_{\ell-1}^\ell \arrow[bend left,d,"{\pi^{\ell-2}[2]_{q,\pi}}" right]&\phantom{a}\\
\phantom{a}&  \arrow[u,bend left,"{[\ell-1]_{q,\pi}}" left]\vdots\arrow[d,bend left,"{\pi [\ell-1]_{q,\pi}}" right]& \phantom{a}\\
\phantom{a}&\arrow[u,bend left,"{[2]_{q,\pi}}" left]  b_{1}^\ell\arrow[d,bend left,"{[\ell]_{q,\pi}}" right]&\phantom{a}\\
\phantom{abb}\arrow[uuuu,dashed,"\E" left]&\arrow[u,bend left,"{[1]_{q,\pi}}" left]    b_0^\ell&\phantom{a}
\end{tikzcd}
\end{equation}
For $0 \leq n \leq \ell-d$, we have that
\begin{align}\label{pork}
  \E^{(d)} b_{n}^\ell &= \sqbinom{n+d}{d}_{q,\pi} b_{n+d}^\ell,&
 \F^{(d)} b_{n+d}^\ell &= \pi^{\binom{d}{2}+nd}\sqbinom{\ell-n}{d}_{q,\pi} b_{n}^\ell.
\end{align}
There is an anti-linear involution
\begin{align}\label{omegaV}
\oldomega:V(-\ell)&\rightarrow V(-\ell),& b_{n}^\ell&\mapsto \pi^{n(\ell-n)}b_{\ell-n}^\ell.
\end{align}
This has the key property that 
\begin{equation}\label{intertwiner}
\oldomega(uv) = \oldomega(u) \oldomega(v)
\end{equation}
for all $u \in U_{q,\pi}(\sl_2), v \in V(-\ell)$.

Let $V_{\pm}(-\ell) := \frac{1}{2}(1\pm \pi) V(-\ell)$. These are irreducible $U_{q,\pi}(\sl_2)$-modules generated by the highest weight vectors $\frac{1}{2}(1\pm \pi) b_{\ell}^\ell$ of weight $\ell$ on which $\pi$ acts by the scalar $\pm 1$. In particular, these modules are not isomorphic for different $\ell$ or different choices of sign.

\begin{theorem}[{\cite[Cor.~3.3.3]{CHW1}}]\label{integrablemodules}
Any integrable $U_{q,\pi}(\sl_2)$-module decomposes as a direct sum of the modules $V_\pm(-\ell)$ for $\ell \in \N$.
\end{theorem}

Now we can prove the main result of the section.

\begin{theorem}\label{oddreflection}
Let $V$ be an integrable $U_{q,\pi}(\sl_2)$-module. 
There is a linear automorphism $\TT:V\stackrel{\sim}{\rightarrow} V$ sending $1_{-k} V$ to $1_{k} V$ for each $k \in \Z$
such that 
\begin{align*}
\TT(v) &= 
\sum_{d \geq \max(0,-k)} (-q)^{d} \E^{(k+d)} \F^{(d)} v\\\intertext{
on a vector $v \in 1_{-k} V$.
The inverse is given explicitly by the formula}
\TT^{-1}(v_k) &= \sum_{d \geq \max(0,-k)} (-q)^{-d} \overline{\F}^{(k+d)} \overline{\E}^{(d)} v
\end{align*}
on a vector $v \in 1_k V$.
\end{theorem}

\begin{proof}
In view of \cref{integrablemodules}, it suffices to check this when $V = V(-\ell)$ for $\ell \in \N$. 
Take $-\ell \leq k \leq \ell$ with $k \equiv \ell\pmod{2}$ and set $n := \frac{\ell+k}{2}$ and $n' := \frac{\ell-k}{2}$, so $n'+n=\ell$ and $n'-n=k$.
The space $1_{-k} V(-\ell)$ is spanned by 
$b_{n}^\ell$ and $1_{k} V(-\ell)$ is spanned by $b_{n'}^\ell$.
Since $\F^{(d)} b_n^\ell = 0$ for $d > n$, we have by the definition in the statement of the theorem that
$\TT(b_{n}^\ell) = u b_{n}^\ell$ where
\begin{align*}
u &:= \sum_{d = \max(0,-k)}^n (-q)^{d} \E^{(k+d)} \F^{(d)} \in U_{q,\pi}(\sl_2).
\end{align*}
In the next paragraph, we show that 
\begin{equation}\label{rain2}
u b_n^\ell = (-1)^n \pi^{\binom{n}{2}+nn'} q^{n+n n'} b_{n'}^\ell.
\end{equation}
Assuming this, the proof can be completed as follows.
Applying $\oldomega$ to \cref{rain2} 
using \cref{omegaV,intertwiner}, we also have that
\begin{equation}\label{rain}
\oldomega(u) b_{n'}^\ell = (-1)^n\pi^{\binom{n}{2}+nn'}q^{-n-nn'} b_{n}^\ell.
\end{equation}
From \cref{rain,rain2}, it follows that $\oldomega(u) u b_{n}^\ell = b_{n}^\ell$ and $u \oldomega(u) b_{n'}^\ell = b_{n'}^\ell$.
Hence, $\TT:1_{-k} V(-\ell) \rightarrow 1_{k} V(-\ell)$ is an isomorphism with inverse $\TT^{-1}$ defined by multiplication by $\oldomega(u)$. Finally we observe that $\overline{\E}^{(d)} b_{n'}^\ell = 0$ for $d > n$ so
$$
\oldomega(u) b_{n'}^\ell = \sum_{d \geq \max(0,-k)} (-q)^{-d} 
\overline{\F}^{(k+d)} \overline{\E}^{(d)} b_{n'}^\ell,
$$
which agrees with the formula for $\TT^{-1}(b_{n'}^\ell)$ 
in the statement of the theorem.

It remains to prove \cref{rain2}. 
First we make some elementary computations using \cref{pork}:
\begin{align*}
u b_{n}^\ell &=
\sum_{d=\max(0,n-n')}^n (-q)^{d} 
\E^{(n'-n+d)} \F^{(d)}b_{n}^\ell\\
&=\sum_{d=\max(0,n-n')}^n
\pi^{\binom{d}{2}+(n-d)d}(-q)^d\sqbinom{n'+d}{n'}_{q,\pi} 
\E^{(n'-n+d)} 
b_{n-d}^\ell\\
&=\sum_{d=\max(0,n-n')}^n \pi^{\binom{d}{2}+(n-d)d} (-q)^{d} \sqbinom{n'+d}{n'}_{q,\pi} \sqbinom{n'}{n\!-\!d}_{q,\pi} b_{n'}^\ell\\
&=
\sum_{d = 0}^{n} \pi^{\binom{d}{2}+(n-d)d} (-q)^{d} \sqbinom{n'+d}{n'}_{q,\pi} \sqbinom{n'}{n-d}_{q,\pi} b_{n'}^\ell,
\end{align*}
noting in the last step that $\sqbinom{n'}{n-d}_{q,\pi} = 0$ if $d< n-n'$ so that we can remove the restriction on the summation. Then we 
switch to another variable $s := n-d$ and sum instead over $d,s \geq 0$ with $d+s=n$ to get that
\begin{align*}
u b_{n}^\ell &=
\sum_{d+s=n} \pi^{\binom{n-s}{2} + s(n-s)} (-q)^{n-s} \sqbinom{n'+d}{n'}_{q,\pi} \sqbinom{n'}{s}_{q,\pi} b_{n'}^\ell
 =
\pi^{\binom{n}{2}}(-q)^{n} \sum_{d+s=n} \pi^{\binom{s}{2}} (-q)^{-s} \sqbinom{n'+d}{d,s}_{q,\pi} v_{n',n}.
\end{align*}
Now an application of \cref{bin}(3) completes the proof of 
\cref{rain2}.
\end{proof}

\begin{remark}\label{specializing}
The specialization of $U_{q,\pi}(\sl_2)$ at
$\pi=1$, that is, the algebra $U_q(\sl_2) := U_{q,\pi}(\sl_2)
\otimes_{\Q(q)^\pi} \Q(q)$ where $\Q(q)$ is viewed here as a $\Q(q)^\pi$-module so
that $\pi$ acts as $1$, is the usual quantized enveloping algebra of
$\mathfrak{sl}_2$. Theorem~\ref{oddreflection} is well known in this case.
The specialization at $\pi = -1$ is the quantized enveloping algebra
$U_q(\osp_{1|2})$ of Clark and Wang \cite{CW}.
\end{remark}

The algebra $U_{q,\pi}(\sl_2)$ has a $\Z[q,q^{-1}]^\pi$-form we
denote by
$\mathbf{U}_{q,\pi}(\sl_2)$, namely, 
the $\Z[q,q^{-1}]^\pi$-algebra generated by the divided powers
$\E^{(r)} 1_k, \F^{(r)} 1_k$ for $r \geq 1, k \in \Z$.
The module $V(-\ell)$ is also defined over $\Z[q,q^{-1}]^\pi$, with its integral form $\mathbf{V}(-\ell)$ being the $\Z[q,q^{-1}]^\pi$-submodule of $V(-\ell)$ generated by the basis vectors chosen above.

\begin{theorem}[{\cite[Lem.~3.5]{Clark}}]
The algebra $\mathbf{U}_{q,\pi}(\sl_2)$ is free as a
$\Z[q,q^{-1}]^\pi$-module with basis given by
the monomials
$\left\{\F^{(r)} \E^{(s)} 1_k\:\big|\:r,s \geq 0, k \in \Z\right\}$.
\end{theorem}

We say that a $\Z[q,q^{-1}]^\pi$-free $\mathbf{U}_{q,\pi}(\sl_2)$-module $\mathbf{V}$ is {\em integrable} if any weight vector $v \in 1_k\mathbf{V}$ is annihilated by $\E^{(n)} 1_k$ and $\F^{(n)}1_k$ for $n \gg 0$ (depending on $v$). Equivalently, the $\Q(q)^\pi$-free $U_{q,\pi}(\sl_2)$-module $\Q(q)^\pi \otimes_{\Z[q,q^{-1}]^\pi}\mathbf{V}$ is integrable in the earlier sense. 
It is clear that the automorphism $\TT$ from \cref{oddreflection} descends to an automorphism of any integrable $\mathbf{U}_{q,\pi}(\sl_2)$-module that is free as a $\Z[q,q^{-1}]^\pi$-module.

\section{Odd symmetric functions}

This section is largely an exposition of results from \cite{EK,EKL}, and assumes the reader is already familiar with the classical theory of symmetric functions as in \cite{Mac}. However, we have made one substantial modification to the setup: instead of the elementary odd symmetric functions denoted $\eps_r$ in \cite{EKL},
we usually prefer to work with the renormalized odd elementary symmetric functions $e_r := (-1)^{\binom{r}{2}} \eps_r$. We will explain the implications of this more thoroughly as we proceed.
We also warn the reader that in \cite{EK} the notation $e_r$ is used
for the same thing as the element denoted 
$\eps_r$ in \cite{EKL}, so the $e_r$ of \cite{EK} is {\em not} 
the one here.

The algebra $\OSym$ of {\em odd symmetric functions}
is the graded superalgebra over the ground field $\k$ 
generated by elements $h_r\:(r \geq 1)$ of degree $2r$ and parity $r\pmod{2}$ subject to the relations of \cite[Cor.~2.13]{EK}:
\begin{align}
h_{r} h_{s}&= h_{s} h_{r}&&\text{if $r \equiv s \pmod{2}$}\label{osym1b}\\
h_r h_s+ (-1)^r h_s h_r &= (-1)^r h_{r+1}h_{s-1} + h_{s-1} h_{r+1}&&\text{if $r \not\equiv s \pmod{2}$}\label{osym2b}
\end{align}
for $r \geq 0, s \geq 1$, interpreting $h_0$ as $1$. We also define elements $e_r\:(r \geq 0)$ so that the {\em infinite Grassmannian relation}
\begin{equation}\label{oddgrassmannian}
\sum_{s=0}^r (-1)^s e_s h_{r-s} = \delta_{r,0}
\end{equation}
holds for all $r \geq 0$. 
The element $h_r$ is exactly the $r$th {\em complete odd symmetric function}
from \cite{EK}. We call $e_r$ the $r$th {\em elementary odd symmetric function}.

In \cite[Cor.~2.13, Prop.~2.10]{EK}, it is shown that their elements 
$\{\eps_r\:|\:r \geq 1\}$ 
generate $\OSym$ subject to exactly the same relations as the $h_r$.
Noting that $\binom{r}{2}+\binom{s}{2} \equiv \binom{r+1}{2} +\binom{s-1}{2}\pmod{2}$ when $r \not\equiv s\pmod{2}$, this means 
that our elements $\{e_r\:|\:r \geq 1\}$ also generate $\OSym$ subject to the same relations
\begin{align}
e_{r} e_{s}&= e_{s} e_{r}&&\text{if $r \equiv s \pmod{2}$}\label{osym1}\\
e_r e_s + (-1)^r e_s e_r &= (-1)^r e_{r+1} e_{s-1} +e_{s-1}e_{r+1}&&\text{if $r \not\equiv s \pmod{2}$}\label{osym2}
\end{align}
for $r \geq 0, s \geq 1$, again interpreting $e_0$ as $1$.
There are also mixed relations, which are
derived in \cite[Prop.~2.11]{EK}. These look slightly different with our modified odd elementary symmetric functions:
\begin{align}
e_r h_s &= h_s e_r&&\text{if $r \equiv s \pmod{2}$}\label{osym3}\\
e_r h_s + (-1)^r h_s e_r &= e_{r+1} h_{s-1} + (-1)^r h_{s-1} e_{r+1} &&\text{if $r \not \equiv s \pmod{2}$}\label{osym4}
\end{align}
for $r \geq 0, s \geq 1$.
The following is equivalent to \cite[(2.6)]{EK}:
\begin{equation}\label{etoh}
e_r = 
\operatorname{det} \left(h_{i-j+1}\right)_{i,j=1,\dots,r}
\end{equation}
where $\det$ should be interpreted as the usual Laplace expansion of determinant ordering monomials in the same way as the elements appear in the rows of the matrix.
For example: 
\begin{align*}
e_0 &= 1, &e_1 &= h_1, & e_2 &= h_1^2-h_2,& e_3 &= h_1^3 - h_1 h_2 - h_2 h_1 + h_3 = h_1^3-h_3.
\end{align*}
In fact, \cref{etoh} is a formal consequence of the infinite
Grassmannian relation which does not require any commutativity. The same thing holds for ordinary symmetric functions, indeed, \cref{oddgrassmannian} is the same relation as for the algebra $\Sym$ of symmetric functions from 
\cite[(I.2.6')]{Mac}, and \cref{etoh} is \cite[Ex.~I.2.8]{Mac}.

It is often useful to work with the generating functions
\begin{align}\label{genfuncs}
e(t)&= \sum_{r \geq 0} (-1)^r e_r t^{-r},&h(t)&= \sum_{r \geq 0} h_r t^{-r},
\end{align}
which are elements of $\OSym\llbracket t^{-1} \rrbracket$ for a formal even variable $t$. 
Now the infinite Grassmannian relation becomes the first of the following:
\begin{align}\label{powergr}
e(t) h(t) &= 1,&
h(t) e(t) &= 1.
\end{align}
Since $h(t)$ is invertible in the formal power series ring,
its left inverse $e(t)$ is also its right inverse, proving the second equality.
In other words, we have that 
\begin{equation}\label{oddgrassmannian2}
\sum_{s=0}^r (-1)^s h_{s} e_{r-s} = \delta_{r,0}
\end{equation}
for all $r \geq 0$. 
Consequently,
\begin{equation}\label{etoh2}
h_r = 
\operatorname{det} \left(e_{i-j+1}\right)_{i,j=1,\dots,r}.
\end{equation}
The evident symmetry between complete and elementary odd symmetric functions is best
expressed in terms of the algebra automorphism
\begin{align}\label{psidef}
\psi:&\OSym\rightarrow\OSym,
&
h_r &\mapsto (-1)^r e_r.
\end{align}
Extending $\psi$ trivially to $\OSym\llbracket t^{-1} \rrbracket$,
we have that $\psi(h(t)) = e(t)$.
As $e(t)$ is the two-sided inverse of $h(t)$,
it follows that $\psi(e(t)) = \psi^2(h(t))$ is the two-sided inverse of $\psi(h(t)) = e(t)$.
Hence, 
$\psi^2(h(t)) = h(t)$, and we have shown that $\psi$ is an involution.
So we also have that
\begin{equation}
\psi(e_r) = (-1)^r h_r.
\end{equation}

For $\lambda\in \Par$, we let 
\begin{align}
h_\lambda &:= h_{\lambda_1} h_{\lambda_2} \cdots,&e_\lambda &:= e_{\lambda_1}e_{\lambda_2}\cdots.
\end{align}
Similarly, we define $h_\alpha$ and $e_\alpha$
for a composition $\alpha \in \Comp(k,n)$.
As in \cite[(2.25)]{EKL}, the relations \cref{osym1b,osym2b} imply for $r < s$ that
\begin{align}\label{straighteningruleh}
h_r h_s &= \left\{
\begin{array}{ll}
h_s h_r&\text{if $r$ and $s$ have the same parity}\\
\displaystyle  h_s h_r+2\sum_{t=1}^r (-1)^{\binom{t}{2}} h_{s+t} h_{r-t}&\text{if $r$ is even and $s$ is odd}\\
\displaystyle -h_s h_r-2\sum_{t=1}^r (-1)^{\binom{t+1}{2}} h_{s+t} h_{r-t}&\text{if $r$ is odd and $s$ is even.}
\end{array}\right.
\end{align}
Similarly, by \cref{osym1,osym2}, we have for $r < s$ that
\begin{align}\label{straighteningrulee}
e_r e_s &= \left\{
\begin{array}{ll}
e_s e_r&\text{if $r$ and $s$ have the same parity}\\
\displaystyle  e_s e_r+2\sum_{t=1}^r (-1)^{\binom{t}{2}} e_{s+t} e_{r-t}&\text{if $r$ is even and $s$ is odd}\\
\displaystyle -e_s e_r-2\sum_{t=1}^r (-1)^{\binom{t+1}{2}} e_{s+t} e_{r-t}&\text{if $r$ is odd and $s$ is even.}
\end{array}\right.
\end{align}
Consequently, any monomial $h_\alpha$ 
or $e_\alpha$ for $\alpha \in \Comp(k,n)$
can be rearranged into decreasing order modulo a linear combination of lexicographically greater monomials of the same degree. This proves the easy spanning part of the next theorem. 

\begin{theorem}[{\cite[Cor.~2.12]{EK}}]\label{symbasis}
The set $\{h_\lambda\:|\:\lambda \in \Par\}$ is a linear basis for $\OSym$.
Equivalently, applying $\psi$, the set $\{e_\lambda\:|\:\lambda \in \Par\}$ is a basis.
\end{theorem}

From the construction in \cite[Sec.~2.1]{EL}, there is a comultiplication $\Delta^-:\OSym\rightarrow \OSym\otimes \OSym$, denoted simply by $\Delta$ in \cite{EL},
making $\OSym$ into a graded Hopf superalgebra such that
\begin{align}\label{deltah}
\Delta^-(h_r) &= \sum_{s=0}^r h_{s} \otimes h_{r-s}
\end{align}
for all $r \geq 0$. 
This can be written more concisely in terms of generating functions as
\begin{align}
\Delta^-(h(t)) &= h(t) \otimes h(t).
\end{align}
By \cite[Prop.~2.17]{EK}, the antipode $S^-:\OSym\rightarrow\OSym$, which we remind the reader is both a superalgebra anti-automorphism and a cosuperalgebra anti-automorphism, 
satisfies $S^-(h_r)= (-1)^r e_r$ or, equivalently, $S^-(h(t))= e(t)$. 

So far, apart from a lack of commutativity, 
there have been many similarities between $\OSym$ and the ordinary theory for $\Sym$,
but now some more significant differences come into view.
Unlike in the ordinary theory, $S^-$ is {\em not} an involution; indeed,
we have that $S^-(h_1) = -h_1$ and $S^-(h_2) = h_1^2-h_2$,
hence, $S^{-n}(h_2) = (-1)^n(h_2- n h_1^2)$ for any $n \geq 0$.
Another important point is that $\OSym$ is {\em not} a cocommutative cosuperalgebra,
e.g., $\Delta^-(h_2) = h_2 \otimes 1 + h_1 \otimes h_1 + 1 \otimes h_2$
is not invariant under the braiding $\br_{\OSym,\OSym}$. 
So the opposite comultiplication
\begin{align}
\Delta^+&:= \br_{\OSym,\OSym}\circ \Delta^-:\OSym \rightarrow \OSym \otimes \OSym
\end{align}
gives a second graded Hopf superalgebra structure on $\OSym$ (the multiplication is the same as before).
Remembering that our $e_r$ is $(-1)^{\binom{r}{2}} \eps_r$,
\cite[Prop.~2.5]{EK} implies that
\begin{equation}\label{nhs}
\Delta^+(e_r) = \sum_{s=0}^r e_{s} \otimes e_{r-s}
\end{equation}
or, equivalently,
\begin{equation}\label{nhs2}
\Delta^+(e(t)) = e(t) \otimes e(t).
\end{equation}
It follows that 
\begin{align}\label{psidelta}
\Delta^- \circ \psi &= (\psi \otimes \psi) \circ \Delta^+,&
\Delta^+ \circ \psi &= (\psi \otimes \psi) \circ \Delta^-,
\end{align}
because both sides of the left hand equation
agree on $e(t)$ and both sides of the right hand equation
agree on $h(t)$. This shows that $\psi$ is a cosuperalgebra {\em anti-}involution.
The antipode $S^+$ for the second Hopf superalgebra structure
is the inverse of $S^-$, so it takes $e_r$ to $(-1)^r h_r$.

We will use two more useful symmetries
\begin{align}\label{sillyboris1}
\smiley:&\OSym \rightarrow \OSym, & e_r &\mapsto (-1)^{\binom{r}{2}} e_r,\\
*:&\OSym \rightarrow \OSym,&  e_r &\mapsto e_r,\label{sillyboris2}
\end{align}
the first of which is an algebra involution, and the second is a
superalgebra {\em anti}-involution.
It is a routine check using \cref{osym1,osym2} to see that these make sense.
Note in particular that $\smiley$ takes our $e_r$ to the $\eps_r$
of \cite{EK, EKL}.
The symmetries $\smiley$ and $*$ commute.
Neither $\smiley$ nor $*$ commutes with $\psi$, but it is still 
true that
$*\circ \smiley$ commutes with $\psi$; see \cref{tigers} 
below for the proof of this.
We have that
\begin{align}\label{golly1}
\Delta^-\circ \smiley &=(\smiley\otimes \smiley) \circ \Delta^+,
&\Delta^+\circ \smiley &=(\smiley\otimes \smiley) \circ \Delta^-,\\
\Delta^+\circ * &=(*\otimes *) \circ \Delta^+,&
\Delta^-\circ * &=(*\otimes *) \circ \Delta^-.
\label{golly2}\end{align}
To justify these, it suffices to check the left hand equations,
then the right hand ones follow because
$\br_{\OSym,\OSym} \circ (\smiley \otimes \smiley) = (\smiley \otimes \smiley)\circ \br_{\OSym,\OSym}$
and $\br_{\OSym,\OSym} \circ (* \otimes *) = (* \otimes *)\circ \br_{\OSym,\OSym}$.
To check the left hand equation from \cref{golly1},
one instead shows that
$\br_{\OSym,\OSym} \circ \Delta^+ \circ \smiley = (\smiley \otimes \smiley) \circ \Delta^+$ by checking that both sides do the same thing on $e_r$.
The left hand equation form \cref{golly2} holds because both sides 
take $e(t)$ to $e(t)\otimes e(t)$.

\begin{lemma}\label{babies}
For $\lambda \in \Par$, 
$\smiley(e_\lambda)^* = (-1)^{dN(\lambda)+dE(\lambda)}
e_\lambda + $(a $\Z$-linear combination of $e_\mu$ for $\mu >_\lex \lambda$).
\end{lemma}

\begin{proof}
Let $k := \h(\lambda)$.
By the definitions, 
we have that $\smiley(e_\lambda)^* = (-1)^{\binom{|\lambda|}{2}}e_{\lambda_k} \cdots e_{\lambda_1}$.
Then we use \cref{straighteningrulee} to rewrite $e_{\lambda_k} \cdots e_{\lambda_1}$
as $\pm e_\lambda$ plus a sum of lexicographically higher $e_\mu$.
It remains to compute the sign. We get a sign change 
each time we commute $e_{\lambda_j}$ with $e_{\lambda_i}$ 
for $1 \leq i < j \leq k$ such that $\lambda_j$ is odd and $\lambda_i$ is even.
So the overall sign is 
$(-1)^{\binom{|\lambda|}{2}+
\sum_{1 \leq i < j \leq k} (\lambda_i-1) \lambda_j}$.
This simplifies to $(-1)^{dN(\lambda)+dE(\lambda)}$.
\end{proof}

\begin{remark}\label{eknotation}
In \cite[Sec.~2.3]{EK}, symmetries denoted $\psi_1,\psi_2$ and $\psi_3$ are introduced. 
These are related to our $\psi, \smiley$ and $*$ by
 $\psi_1 = \smiley\circ \operatorname{p}\circ \,\psi$ 
 (because the latter takes $h_r$
to $\eps_r = (-1)^{\binom{r}{2}} e_r$),
$\psi_2 = \psi \circ \operatorname{p}\circ\, \smiley \circ \psi$ 
(because the latter sends $h_r$ to $(-1)^{\binom{r+1}{2}} h_r$), and $\psi_3 = \psi
 \circ * \circ \psi$ (because
the latter is a superalgebra anti-involution taking $h_r$ to $h_r$). 
We emphasize that our $\psi=\psi_1\circ\psi_2$ is an involution, whereas $\psi_1$ is not.
\end{remark}

In \cite{EK}, the definition of $\OSym$ is motivated by the definition
of a non-degenerate symmetric\footnote{We really do mean symmetric
rather than supersymmetric here!} bilinear form
$(\cdot,\cdot)^-:\OSym\otimes \OSym \rightarrow \k$. Extending
the bilinear form $(\cdot,\cdot)^-$ on $\OSym$ to
$\OSym \otimes \OSym$ so that $(a_1 \otimes a_2, b_1 \otimes b_2)^- =
(a_1,b_1)^- (a_2,b_2)^-$, the form
is characterized uniquely by the following properties:
\begin{align}\label{charit}
(h_r, h_s)^- &= \delta_{r,s},& (ab,c)^- &= (a \otimes b, \Delta^-(c))^-
\end{align}
for $r,s \geq 0, a,b,c \in \OSym$. 
For symmetry's sake, one can also consider a form
$(\cdot,\cdot)^+$ which is defined in a similar way 
so that
\begin{align}\label{charit2}
(e_r, e_s)^+ &= \delta_{r,s},& (ab,c)^+ &= (a \otimes b, \Delta^+(c))^+
\end{align}
for $r,s \geq 0, a,b,c \in \OSym$. 
The forms $(\cdot,\cdot)_\pm$ are related by the first of the following properties:
\begin{align}\label{silverton}
\big(\psi(a),\psi(b))^{\pm} &= (a,b)^{\mp},&
\big(\smiley(a^*), \smiley(b^*)\big)^{\pm} &= (a, b)^\pm
\end{align}
for any $a,b \in \OSym$. This and the second property are both checked by induction on degree,
using \cref{psidef,psidelta,sillyboris1,sillyboris2,golly1,golly2}.
The first of the next two properties follows from \cite[(2.10)]{EK}, 
then the second follows by applying $\psi$:
\begin{align}\label{britain}
(e_r,e_s)^- &= (-1)^{\binom{r}{2}} \delta_{r,s},&
(h_r,h_s)^+ &= (-1)^{\binom{r}{2}} \delta_{r,s}.
\end{align}
The following allows $(h_\lambda,e_\mu)^\pm$ to be computed:
\begin{align}\label{brew}
(h_\lambda, e_r)^- &=(-1)^{\binom{r}{2}} \delta_{\lambda, (1^r)},
&
(h_r, e_\lambda)^+ &=(-1)^{\binom{r}{2}} \delta_{\lambda, (1^r)}.
\end{align}
The first equality here is established in \cite[Prop.~2.5]{EK}, again remembering that 
our normalization of the elements $e_r$ is different; then the second follows on
applying $\psi$.
In particular, \cref{brew} can be used to show that
\begin{align}\label{semiorthogonality}
\big(h_\lambda,e_{\mu}\big)^-
&= \left\{\begin{array}{ll}
(-1)^{NE(\lambda)+dN(\lambda)}&\text{if $\lambda = \mu^\transpose$}\\
0&\text{if $\lambda\not\leq_{\lex} \mu^\transpose$}
\end{array}\right.&
\big(h_\lambda,e_\mu\big)^+
&= \left\{\begin{array}{ll}
(-1)^{NE(\mu)+dN(\mu)}&\text{if $\lambda = \mu^\transpose$}\\
0&\text{if $\lambda\not\leq_{\lex} \mu^\transpose$}
\end{array}\right.
\end{align}
for $\lambda,\mu \in \Par$; see \cite[Prop.~2.14]{EK} for the first equality. 
This ``semi-orthogonality'' is used to complete the proof of \cref{symbasis} in \cite{EK}.

Recall that $\OPol_n$ is the algebra of odd
polynomials from \cref{opol}.
Define a
superalgebra involution $\smiley_n$ and a superalgebra anti-involution $*$
of $\OPol_n$
by
\begin{align}\label{borehamA}
\smiley_n:\OPol_n &\rightarrow \OPol_n,&
x_i &\mapsto x_{n+1-i},\\
*:\OPol_n &\rightarrow \OPol_n,&
x_i &\mapsto x_i.
\label{newtildeA}
\end{align}
The following theorem gives 
another way to motivate the definition of $\OSym$, as was explained originally in \cite{EKL}.

\begin{theorem}\label{mainfact}
There is a graded superalgebra homomorphism $\pi_n:\OSym \rightarrow \OPol_n$
taking $e_r$ and $h_r$ to the polynomials
\begin{align}\label{epoly}
e_r(x_1,\dots,x_n) &:= \sum_{1 \leq i_1 < \dots < i_r \leq n}
  x_{i_1} \cdots x_{i_r}\\
h_r(x_n,\dots,x_1) &:= \sum_{n \geq i_r \geq \cdots \geq i_1 \geq 1}
  x_{i_r} \cdots x_{i_1},\label{hpoly}
\end{align}
respectively.
Moreover, $\pi_n$ intertwines 
the involution $\smiley$ of $\OSym$ with the algebra involution
$\smiley_n$ of $\OPol_n$ from \cref{borehamA},
and it intertwines the 
anti-involution $*$ of $\OSym$ with the 
superalgebra anti-involution $*$ of $\OPol_n$ from \cref{newtildeA}.
Finally, if $n=a+b$, the diagram
\begin{equation}
\label{tensorconvention}
\begin{tikzcd}
\arrow[d,"\pi_n" left]\OSym\arrow[r,"\Delta^+" above]&\OSym\otimes \OSym\arrow[d,"\pi_a\otimes\pi_b" right]\\
\OPol_n\arrow[r,equal]&\OPol_a \otimes \OPol_b
\end{tikzcd}
\end{equation}
commutes,
where the identification at the bottom is as explained after \cref{opol}.
\end{theorem}

\begin{proof}
Note our $x_i$ is the variable $\tilde x_i = (-1)^{i-1} x_i$ in the notation of \cite{EKL}, hence, our $e_r(x_1,\dots,x_n)$ is the polynomial denoted $\eps_r(x_1,\dots,x_n)$ in \cite{EKL}. With this in mind, \cite[Lem.~2.3]{EKL} checks that the polynomials $e_r(x_1,\dots,x_n) \in \OPol_n$ satisfy the defining relations of $\OSym$ from \cref{osym1,osym2}. Hence, there is a unique homomorphism $\pi_n:\OSym\rightarrow \OPol_n$ 
such that $\pi_n(e_r) = e_r(x_1,\dots,x_n)$ for all $r \geq 0$.

The involution $\smiley$ of $\OSym$ takes $e_r$ to $(-1)^{\binom{r}{2}} e_r$.
The involution $\smiley_n$ of $\OPol_n$ takes $e_r(x_1,\dots,x_n)$ to
$$
e_r(x_n,\dots,x_1) = \sum_{n \geq i_r > \cdots > i_1 \geq 1}
x_{i_n} \cdots x_{i_2} x_{i_1}.
$$
Rearranging these monomials into increasing order of $x_i$ produces a sign of
$(-1)^{\binom{r}{2}}$. Hence, $\smiley_n$ takes $e_r(x_1,\dots,x_n)$
to $(-1)^{\binom{r}{2}} e_r(x_1,\dots,x_n)$.
This checks that $\pi_n \circ \smiley = \smiley_n \circ \pi_n$.
Similarly, we see that $\pi_n \circ * = * \circ \pi_n$ because
$*$ on $\OSym$ fixes $e_r$ and $*$ on $\OPol_n$ fixes $e_r(x_1,\dots,x_n)$.

In \cite[Lem.~2.8]{EKL}, again using that our $x_i$ is $\tilde x_i$ in \cite{EKL},
it is checked that the polynomials 
$$
h_r(x_1,\dots,x_n) = \sum_{1 \leq i_1 \leq \cdots \leq i_r \leq n}
x_{i_1} x_{i_2} \cdots x_{i_n}
$$
satisfy
$\sum_{s=0}^r (-1)^{\binom{s+1}{2}} e_s(x_1,\dots,x_n) 
h_{r-s}(x_1,\dots,x_n) = \delta_{r,s}$
for all $r \geq 0$.
Applying $\smiley_n$, it follows that
$\sum_{s=0}^r (-1)^{\binom{s+1}{2}} e_s(x_n,\dots,x_1)
h_{r-s}(x_n,\dots,x_1)=\delta_{r,s}.$
We already know that $e_s(x_n,\dots,x_1) = (-1)^{\binom{s}{2}} \pi_n(e_s)$,
so this shows that 
$\sum_{s=0}^r (-1)^s \pi_n(e_s) h_{r-s}(x_n,\dots,x_1)=\delta_{r,s}.$
Comparing with \cref{oddgrassmannian}, 
this proves that $\pi_n(h_r) = h_r(x_n,\dots,x_1)$.

Finally, to see that \cref{tensorconvention} commutes, use 
\cref{nhs} and the definition of $e_r(x_1,\dots,x_n)$.
\end{proof}

Now we {\em define} $\OSym_n$, the algebra of {\em odd symmetric polynomials}, to be the subalgebra of $\OPol_n$ that is the image of the homomorphism $\pi_n$ from \cref{mainfact}.
For any $a \in \OSym$, we use the notation $a^{(n)}$
to denote its canonical image in $\OSym_n$.
Note from \cref{epoly} that $e_r^{(n)} = 0$ for $r > n$.
For $\lambda \in \Par$, we have that
\begin{align}\label{sourceofne}
e_{\lambda^\transpose}^{(n)} &= 
\left\{
\begin{array}{ll}
(-1)^{NE(\lambda)} x^\lambda+\text{(a $\Z$-linear combination of $x^\kappa$ for
$\kappa \in \N^n$ with $\kappa < \lambda$)}&\text{if $\h(\lambda)\leq n$}\\
0&\text{if $\h(\lambda)> n$,}
\end{array}\right.
\end{align}
where $x^\kappa=x_1^{\kappa_1}\cdots x_n^{\kappa_n}$ and
$x^\lambda$ is defined similarly, identifying $\lambda\in \Par$ with $\h(\lambda) \leq n$ with $(\lambda_1,\dots,\lambda_n) \in \N^n$.
This is easily checked from the definition and gives the clearest explanation for the sign $NE(\lambda)$.

\begin{theorem}\label{whatekldo}
The set
$\Big\{e_{\lambda^\transpose}^{(n)}\:\Big|\:\lambda \in \SPar{n}\Big\}$ is a basis for $\OSym_n$.
\end{theorem}

\begin{proof}
The set
$\Big\{e_{\lambda^\transpose}^{(n)}\:\Big|\:\lambda \in \SPar{n}\Big\}$
spans $\OSym_n$ since 
the 
images of all other $e_\mu$ in the basis for $\OSym$ from \cref{symbasis} 
are zero.
Linear independence is clear from \cref{sourceofne}.
\end{proof}

\begin{corollary}
The quotient maps $\pi_n:\OSym \twoheadrightarrow \OSym_n$ induce an isomorphism
 $$\OSym \stackrel{\sim}{\rightarrow} \varprojlim \OSym_n,$$
where on the right we have the inverse limit of the inverse system
$\cdots \twoheadrightarrow \OSym_1 \twoheadrightarrow \OSym_0$ 
taken in the category of graded superalgebras,
with the map $\OSym_{n+1} \twoheadrightarrow \OSym_n$ taking $e_r^{(n+1)}$ to $e_r^{(n)}$. Moreover, $\OSym_n$ may be identified with the quotient
$\OSym \:\big/\: \langle e_r\:|\:r > n \rangle.$
\end{corollary}

\begin{corollary}[{\cite[(2.20)]{EKL}}]\label{bakeoffOSYM}
$\displaystyle \gsdim \OSym_n = \prod_{r=1}^n \frac{1}{1-(\pi q^2)^r}$.
\end{corollary}

\begin{proof}
\cref{whatekldo} shows that $\OSym_n$ has the same graded 
superdimension as a commutative polynomial algebra with generators $x_1,\dots,x_n$ such that $x_r$ is of degree $2r$ and parity $r\pmod{2}$.
\end{proof}

\begin{corollary}\label{bakeoffOPOL}
$\displaystyle \gsdim \OPol_n =  \gsdim \OSym_n\times q^{\binom{n}{2}} [n]^!_{q,\pi} = \gsdim \OSym_n\times \sum_{w \in \S_n} (\pi q^2)^{\ell(w)}$.
\end{corollary}

\begin{proof}
The second equality follows from the first by \cref{poincare}.
To obtain the first equality, we use \cref{bakeoffOSYM} to see that
\begin{align*}
\gsdim \OSym_n\times 
q^{\binom{n}{2}} [n]^!_{q,\pi} = &q^{\binom{n}{2}} [n]^!_{q,\pi} \prod_{r=1}^n\frac{1}{1-(\pi q^2)^r} = \;\frac{ q^{\binom{n}{2}} [n]^!_{q,\pi}}{(1-\pi q^{2})^n} \prod_{r=1}^n\frac{1-\pi q^{2}}{1-(\pi q^2)^r}\\
=\; &\frac{[n]^!_{q,\pi}}{(1-\pi q^2)^n} \prod_{r=1}^n \frac{\pi
      q-q^{-1}}{(\pi q)^r - q^{-r}}
=\;\frac{1}{(1-\pi q^2)^n} 
\stackrel{\cref{dimensionformula}}{=} \gsdim \OPol_n.
\end{align*}
\end{proof}

The next technical 
lemma about relations in $\OSym_{n+1}$
will be needed at a key place later on; see \cref{violins}
which is used to prove \cref{cupsandcaps}.

\begin{lemma}\label{badday}
For any $0 \leq p,q,k \leq n$ with $p+q\leq n$, we have that
$$
\sum_{m = 1}^{n-p-q} (-1)^{(n+k)(m+p+k)} e^{(n+1)}_{n-k+m} e^{(n+1)}_{n-p-q-m}
=\sum_{m = 1}^{n-p-q} (-1)^{(n+k)(m+q)} e^{(n+1)}_{n-p-q-m} e^{(n+1)}_{n-k+m}
$$
in $\OSym_{n+1}$.
\end{lemma}

\begin{proof}
If $p+q+k$ is even then $(-1)^{(n+k)(m+p+k)} = (-1)^{(n+k)(m+q)}$,
and 
$e^{(n+1)}_{n-k+m}$ commutes with $e^{(n+1)}_{n-p-q-m}$
by the relation \cref{osym1}. 
The result obviously follows in this situation
since corresponding terms on each side are equal.

Next, assume that $p+q+k$ is odd and $n \equiv p+q\pmod{2}$,
in which case $n+k$ is odd.
There are an even number of terms in the summations in the identity we are trying to prove. It suffices to show that sums of consecutive pairs of terms on each side are equal, i.e.,
\begin{multline*}
(-1)^{(n+k)(m+p+k)} e^{(n+1)}_{n-k+m} e^{(n+1)}_{n-p-q-m}
+(-1)^{(n+k)(m+1+p+k)} e^{(n+1)}_{n-k+m+1} e^{(n+1)}_{n-p-q-m-1}
=\\
(-1)^{(n+k)(m+q)} e^{(n+1)}_{n-p-q-m} e^{(n+1)}_{n-k+m}
+(-1)^{(n+k)(m+1+q)} e^{(n+1)}_{n-p-q-m-1} e^{(n+1)}_{n-k+m+1}
\end{multline*}
for every odd $m$ with $1 \leq m < n-p-q$.
Multiplying both sides by $(-1)^{(n+k)(m+p+k)} = (-1)^{(n+k)(m+q+1)}$, this simplifies to
$$
e^{(n+1)}_{n-k+m} e^{(n+1)}_{n-p-q-m}
- e^{(n+1)}_{n-k+m+1} e^{(n+1)}_{n-p-q-m-1}
=
- e^{(n+1)}_{n-p-q-m} e^{(n+1)}_{n-k+m}
+ e^{(n+1)}_{n-p-q-m-1} e^{(n+1)}_{n-k+m+1}.
$$
This follows because
when $m$ is odd, $n-k+m$ is even, so we have that
$$
e_{n-k+m} e_{n-p-q-m}
+ e_{n-p-q-m} e_{n-k+m}
=
e_{n-k+m+1} e_{n-p-q-m-1}
+ e_{n-p-q-m-1} e_{n-k+m+1}
$$
by \cref{osym2}.

Finally we treat the case that $p+q+k$ is odd and $n \not\equiv p+q\pmod{2}$,
when $n+k$ is even. Now the signs on both sides of 
the identity we are trying to prove
are always $+$ so can be omitted.
There are an odd number of terms in the summations. The $m=n-p-q$ terms in both
summations
are equal, indeed, they both equal 
$e^{(n+1)}_{2n-p-q-k}$ as $e^{(n+1)}_0 = 1$.
To see that the remaining terms in the summations are equal, we 
show that sums of consecutive pairs of terms are equal like in the previous paragraph, i.e.,
$$
e^{(n+1)}_{n-k+m} e^{(n+1)}_{n-p-q-m}
+ e^{(n+1)}_{n-k+m+1} e^{(n+1)}_{n-p-q-m-1}
=
 e^{(n+1)}_{n-p-q-m} e^{(n+1)}_{n-k+m}
+ e^{(n+1)}_{n-p-q-m-1} e^{(n+1)}_{n-k+m+1}
$$
for each odd $m$ with $1 \leq m < n-p-q-1$.
This follows because when $n-k+m$ is odd we have that
$$
e_{n-k+m} e_{n-p-q-m}
- e_{n-p-q-m} e_{n-k+m}
=
-e_{n-k+m+1} e_{n-p-q-m-1}+e_{n-p-q-m-1} e_{n-k+m+1}
$$
by \cref{osym2} again.
\end{proof}

It is time to say a little more about 
variants of the odd complete and elementary symmetric functions. The following lemma, which is another application of \cref{mainfact}, is helpful to understand the possibilities.

\begin{lemma}\label{tigers}
We have that
$\smiley(h_r) = (-1)^{\binom{r}{2}} h_r^*$ and
$\smiley(e_r) = (-1)^{\binom{r}{2}} e_r^*$
for any $r \geq 0$.
Hence, $*\circ \smiley \circ \psi = \psi \circ *\circ \smiley$.
\end{lemma}

\begin{proof}
It is immediate from the definitions that $\smiley(e_r) = (-1)^{\binom{r}{2}} e_r^*$.
To see the analogous thing for $h_r$,
it suffices to show that
$\smiley_n\big(h^{(n)}_r\big)^* = (-1)^{\binom{r}{2}} h_r^{(n)}$
for all $r \geq 0$.
This follows from 
the explicit descriptions of these polynomials and maps given in 
\cref{mainfact}.
To deduce finally that
$* \circ \smiley$ and $\psi$ commute, it suffices to check that
$(*\circ\smiley\circ\psi)(e_r) = (\psi\circ *\circ\smiley)(e_r)$
for all $r \geq 0$, which is clear at this point.
\end{proof}

As we have said before, our odd complete symmetric function
$h_r$ is the same as the $h_r$
in \cite{EK, EKL}, but our odd elementary symmetric function $e_r$
is different from the one there, which is 
\begin{equation}\label{toe}
\eps_r := \smiley(e_r) = (-1)^{\binom{r}{2}} e_r^*=
(-1)^{\binom{r}{2}} e_r,
\end{equation}
where the non-trivial equality follows by \cref{tigers}.
There is also a natural variant on the odd complete symmetric function $h_r$, namely,
\begin{equation}\label{tee}
\gamm_r := \smiley(h_r) = (-1)^{\binom{r}{2}} h_r^*.
\end{equation}
Since $h_r^* \neq h_r$
for $r > 1$, it is {\em not} the case that $\gamm_r = (-1)^{\binom{r}{2}} h_r$.
We call $\eps_r$ and $\gamm_r$ the 
{\em dual odd elementary and complete symmetric functions}.
Applying $\smiley$ to \cref{powergr} gives that
\begin{align}\label{powergr2}
\eps(t) \gamm(t) &= \gamm(t) \eps(t) = 1.
\end{align}
where $\eps(t):= \sum_{r \geq 0} (-1)^r \eps_r t^{-r}$ and 
$\eta(t) := \sum_{r \geq 0} \eta_r t^{-r}$.
These should make it clear that
$e$ and $h$ belong together as do $\eps$ and $\gamm$. It is
not so easy to relate $e$ to $\gamm$ or $h$ to $\eps$ in terms of
generating functions; cf. \cref{magic}. 

Consider again the truncation $\OSym_n$.
Let $\eps_r^{(n)}$ and $\gamm_r^{(n)}$ 
be the images $\eps_r$ and $\gamm_r$
in $\OSym_n$.
From \cref{mainfact}, it is clear that
\begin{align}\label{epolybetter}
  e_r^{(n)} &= \sum_{1 \leq i_1 < \dots < i_r \leq n}
  x_{i_1} \cdots x_{i_r}&
h^{(n)}
_r &= \sum_{n \geq i_r \geq \cdots \geq i_1 \geq 1}
  x_{i_r} \cdots x_{i_1}\\
 \eps_r^{(n)} &= \sum_{n \geq i_r > \dots > i_1 \geq 1}
  x_{i_r} \cdots x_{i_1},&
  \gamm^{(n)}_r &= \sum_{1 \leq i_1 \leq \cdots \leq i_r\leq n}
  x_{i_1} \cdots x_{i_r}.
\end{align}
This gives an explanation for the existence of the four basic familes of
odd symmetric functions $e_r, h_r, \eps_r$ and 
$\gamm_r$.
When working in $\OSym_n$ with generating functions, we prefer to modify
the definitions slightly, working not exactly with the images 
of $e(t), h(t), \eps(t)$ and $\gamma(t)$ 
in $\OSym_n\llbracket t^{-1}\rrbracket$, but incorporating a shift in $t$:
\begin{align}\label{clever}
e^{(n)}(t)&:= \sum_{r=0}^n (-1)^r e^{(n)}_r t^{n-r},
&h^{(n)}(t)&:= \sum_{r \geq 0} h^{(n)}_r t^{-n-r},\\
\eps^{(n)}(t)&:= \sum_{r=0}^n (-1)^r\eps^{(n)}_r t^{n-r},
&\gamm^{(n)}(t)&:= \sum_{r \geq 0} \gamm^{(n)}_r t^{-n-r}.\label{cleverer}
\end{align}
The advantage of this is that $e^{(n)}(t)$ 
and $\eps^{(n)}(t)$ are {\em polynomials} $\OSym_n[t]$,
indeed, we have that
\begin{align}\label{stupid}
e^{(n)}(t) &= (t-x_1) \cdots (t-x_n),
&
\eps^{(n)}(t) &= (t-x_n)\cdots (t-x_1).\\\intertext{Also,
noting that $(t-x)^{-1} = t^{-1}+x t^{-2}+x^2 t^{-3}+\cdots
\in \k[x]\llbracket t^{-1}\rrbracket$,
we have that}
h^{(n)}(t) &= (t-x_n)^{-1}\cdots (t-x_1)^{-1},
&
\gamm^{(n)}(t) &= (t-x_1)^{-1}\cdots (t-x_n)^{-1}\label{littlerock}
\end{align}
in $\OSym_n\llbracket t^{-1}\rrbracket$.
Now we have that
\begin{align}\label{newigr}
h^{(n)}(t)e^{(n)}(t) &= e^{(n)}(t) h^{(n)}(t) = 1,&
\gamm^{(n)}(t)\eps^{(n)}(t) &= \eps^{(n)}(t)\gamm^{(n)}(t) = 1,
\end{align}
equality in the ring
$\OSym_n\lround t^{-1} \rround$
of formal Laurent series in $t^{-1}$.

The next result is elementary but does not appear in the existing literature. Observe by \cref{oddgrassmannian} that
\begin{equation}\label{osym6}
z_{2r} := \sum_{s=0}^r e_{2s} h_{2r-2s} = \delta_{r,0}+\sum_{s=0}^{r-1} e_{2s+1}h_{2r-2s-1}.
\end{equation}
The element $z_{2r}$ is {\em central}: it commutes with all even $e_t$ by the first form of the definition, and it commutes with all odd $e_t$ by the second one. Also let
omicron be the special element
\begin{equation}\label{delement}
\o := e_1 = h_1,
\end{equation}
noting that $z_2=\o^2$. The relations \cref{osym2,osym2b} imply that
\begin{align}\label{osym5}
e_{2r+1}  &= \textstyle\frac{1}{2}\big(\o e_{2r} + e_{2r} \o\big),& h_{2r+1}  &= \textstyle\frac{1}{2}\big(\o h_{2r} + h_{2r} \o\big),
\end{align}
so that $\OSym$ is generated already by $\o$ and all even $e_{2r}\:(r \geq 1)$.

\begin{theorem}\label{tory}
The graded superalgebra $\OSym$ is generated by $\o$ and $e_{2r}\:(r \geq 1)$ subject only to the relations
\begin{align}
[e_{2r}, e_{2s}] &= 0\label{goingon}\\
[\o^2, e_{2r}] &= 0\label{goingon1}\\
[\o, e_{2r+2}]&=\textstyle\left[\frac{1}{2}\big(\o e_{2r}+e_{2r}\o\big), e_2\right]\label{goingon2}
\end{align}
for $r, s \geq 1$.
\end{theorem}

\begin{proof}
Let $A$ be the graded superalgebra generated by an odd element $\o$ 
of degree 2 and even elements
$e_{2r}\:(r \geq 1)$ of degree $4r$ subject to the relations \cref{goingon,goingon1,goingon2}. 
For $r \geq 0$, we set $e_1 := \o$ and  $e_{2r+1} := \frac{1}{2}\big(\o e_{2r} + e_{2r} \o\big) \in A$ for $r \geq 1$; cf. \cref{osym5}. 

We first construct a homomorphism $\alpha:A\rightarrow \OSym$ by mapping $\o \mapsto \o$ and $e_r \mapsto e_r$. To see that this makes sense, we need to check that the relations \cref{goingon,goingon1,goingon2} hold in $\OSym$. The first is immediate, and the second follows because we have observed already
that $z_2=\o^2$ is central in $\OSym$. For the third, in $\OSym$, we have that $[\o, e_{2r+2}] = [e_{2r+1},e_2]$ by \cref{osym2}, and
also $\frac{1}{2}( \o e_{2r} + e_{2r} \o) = e_{2r+1}$ by \cref{osym5}. Now the relation is clear.

Next we construct a homomorphism $\beta:\OSym \rightarrow A$ in the other direction so that on generators it sends $e_r \mapsto e_r$ for each $r \geq 1$. To show this is well defined, we must again check relations, this time showing that \cref{osym1,osym2} hold in $A$. The first one is immediate if $r$ and $s$ are both even.
When $r$ is odd and $s$ is even, \cref{osym2} is equivalent to the relation
\begin{equation}\label{osym2a}
[e_{2r-1}, e_{2s}] = [e_{2s-1},e_{2r}]
\end{equation}
for $r, s \geq 1$. To check it holds in $A$, we must show that the expression $[e_{2r-1},e_{2s}]\in A$ is symmetric in $r$ and $s$. We have that
\begin{align*}
4[e_{2r-1}, e_{2s}] &= 2[\o e_{2r-2} + e_{2r-2} \o, e_{2s}] \\
&= 2\o e_{2r-2} e_{2s} + 2e_{2r-2} \o e_{2s} - 2e_{2s} \o e_{2r-2} - 2e_{2s} e_{2r-2} \o\\
&= 2(\o e_{2s}- e_{2s} \o) e_{2r-2}  + 2e_{2r-2} (\o e_{2s} - e_{2s} \o)\\
&= 2[\o,e_{2s}] e_{2r-2}  + 2e_{2r-2} [\o,e_{2s}]\\
&= [\o e_{2s-2}+e_{2s-2}\o,e_2] e_{2r-2}  + e_{2r-2} [\o e_{2s-2}+e_{2s-2} \o,e_2]\\
&= [\o,e_2] e_{2s-2} e_{2r-2} + e_{2s-2} [\o,e_2] e_{2r-2} + e_{2r-2}[\o,e_2] e_{2s-2} + e_{2r-2} e_{2s-2} [\o,e_2],
\end{align*}
which is indeed symmetric in $r$ and $s$. When $r$ is even and $s$ is odd, \cref{osym2} is equivalent to the relation
\begin{equation}\label{osym2bb}
[e_{2r},e_{2s+1}]_+ = [e_{2s},e_{2r+1}]_+
\end{equation}
for $r, s \geq 0$, where $[x,y]_+$ here denotes $xy+yx$. So again we must show that $[e_{2r},e_{2s+1}]_+ \in A$ is symmetric in $r$ and $s$. We have that
$$
2[e_{2r}, e_{2s+1}]_+ = e_{2r}  \o e_{2s} + e_{2r} e_{2s} \o + \o e_{2s} e_{2r} + e_{2s} \o e_{2r}.
$$
This is symmetric in $r$ and $s$ because $e_{2r} e_{2s} = e_{2s} e_{2r}$. It remains to check the relation \cref{osym1} when $r$ and $s$ are both odd. Equivalently, we show that $[e_{2r+1},e_{2s+1}] = 0$ for $r,s \geq 0$ by induction on $s$. The base case $s=0$ follows because
$$
2[e_{2r+1}, \o] =  [\o e_{2r} + e_{2r} \o, \o]
= \o e_{2r} \o + e_{2r} \o^2 - \o^2 e_{2r} - \o e_{2r} \o = -\left[\o^2,e_{2r}\right] = 0.
$$
The following establishes the induction step: for $s > 0$ we have in $A$ that
\begin{align*}
2 [e_{2r+1}, e_{2s+1}]&= [e_{2r+1}, \o e_{2s} + e_{2s} \o]=\o [e_{2r+1}, e_{2s}] + [e_{2r+1}, e_{2s}] \o\\
&=\o [e_{2s-1}, e_{2r+2}] + [e_{2s-1}, e_{2r+2}] \o= [e_{2s-1}, \o e_{2r+2}+ e_{2r+2}\o]\\&= 2[e_{2s-1}, e_{2r+3}]=-2[e_{2r+3}, e_{2s-1}]= 0,
\end{align*}
using the $s=0$ case for the second and fourth equalities, \cref{osym2a} for the third equality, and the induction hypothesis for the final equality.

It remains to observe that $\alpha$ and $\beta$ are two-sided inverses. This is obviously the case on generators by the way we have defined the maps.
\end{proof}

In the corollaries, we use the following notation:
\begin{itemize}
\item
$\Sym$ is the algebra of symmetric
functions over $\k$ as in 
\cite{Mac} viewed as a graded superalgebra so that the $r$th elementary and complete symmetric functions are even of degree $4r$;
\item
$\Sym_n$ is the usual algebra of symmetric
polynomials in $n$ variables, i.e., it is the quotient of $\Sym$ obtained by setting the $r$th even elementary symmetric polynomials
to zero for all $r > n$;
\item
$\Sym[\c]$ and $\Sym_n[\c]$ are the supercommutative graded superalgebras
obtained from $\Sym$ and $\Sym_n$
by adjoining an odd element $\c$ of degree 2
with $\c^2 = 0$.
\end{itemize}

\begin{corollary}\label{dumbc1}
The largest supercommutative quotient
of $\OSym$ is the graded superalgebra
\begin{equation}\label{dumbr1}
R := \OSym \big/ \big\langle o^2, [o,e_2]\big\rangle.
\end{equation}
Writing $\dot a$ for the canonical image of $a \in \OSym$
in $R$, 
we have that
\begin{align}\label{tiger}
\dot e_{2r+1} &= \dot e_{2r} \dot o,&
\dot h_{2r+1} &= \dot h_{2r} \dot o,&
\dot \eps_{2r+1} &= \dot \eps_{2r} \dot o,&
\dot \gamm_{2r+1} &= \dot \gamm_{2r} \dot o,&
\dot z_{2r} &= \delta_{r,0}
\end{align}
for all $r \geq 0$.
Moreover,
there is an isomorphism of graded superalgebras
$\theta:
R \stackrel{\sim}{\rightarrow} \Sym[\c]$
taking $\dot o$ to $\c$, $\dot \eps_{2r} = (-1)^r \dot e_{2r}$ 
to the
$r$th even elementary symmetric function and $\dot h_{2r} = (-1)^r \dot \gamm_{2r}$ to the
$r$th even complete symmetric function.
\end{corollary}

\begin{proof}
In any supercommutative quotient of $\OSym$, we must have that $\o^2 = 0$ and $[\o,e_2] = 0$. Now we let $I$ be the two-sided ideal of $\OSym$ generated by $\o^2$ and $[\o,e_2]$ and show that $\OSym / I$ is supercommutative.
Since $\OSym$ is generated by the elements 
$e_{2r}\:(r \geq 1)$ and $\o$, the proof of this reduces at once to checking that 
$[\o, e_{2r}] \in I$ for all $r \geq 1$, which holds because
$$
2 [\o, e_{2r}]  = 2 [e_{2r-1}, e_2] = [e_{2r-2} \o + \o e_{2r-2}, e_2]= e_{2r-2} [\o,e_2] + [\o,e_2] e_{2r-2} \in I.
$$
Thus, we have shown that $R := \OSym / I$ is the largest supercommutative quotient of $\OSym$.
Next we observe that
$\dot e_{2r+1} = \dot e_{2r} \dot \o$,
$\dot h_{2r+1}= \dot h_{2r} \dot \o$
and $\dot z_{2r} = \delta_{r,0}$ for all $r \geq 0$.
The first two of these follow from \cref{osym5} and the supercommutativity of $\OSym / I$,
then the final equality follows using the first two together with the second form of the definition of $z_{2r}$ in  \cref{osym6}.
The superalgebra 
anti-involution $*:\OSym \rightarrow \OSym$ induces an anti-involution of
$R$. Since it fixes the generators 
$\dot e_r\:(r \geq 1)$ 
and $R$ is supercommutative, this induced anti-involution is
actually the identity. 
So by \cref{toe,tee},
we have that
$\dot \eps_{r} = (-1)^{\binom{r}{2}} \dot e_{r}$ 
and $\dot \gamm_{r} = (-1)^{\binom{r}{2}} \dot h_{r}$.
The remaining identities in \cref{tiger} follow 
using this.

Finally, we construct the isomorphism $\theta$.
We start by observing that there is a homomorphism
$\OSym \rightarrow \Sym[\c]$ taking $\o \mapsto \c$ and $\eps_{2r} = (-1)^r e_{2r}$ to the $r$th even
elementary symmetric function for $r \geq 1$. 
To see this, we apply \cref{tory} to reduce to checking that the relations \cref{goingon,goingon1,goingon2} all hold in $\Sym[\c]$, which is clear because it is supercommutative. 
Now this homomorphism factors through the quotient to induce $\theta:R \rightarrow \Sym[\c]$. 
Moreover, $R$ is spanned by the monomials $\dot e_\lambda$ and $\dot e_\lambda \dot \o$ for partitions $\lambda$ with all parts even. The images under $\theta$ of these elements give a linear basis for $\Sym[\c]$. This shows that $\theta$ is an isomorphism.
It just remains to check that $\theta$ takes $\dot h_{2r}
= (-1)^r \dot \gamm_{2r}$ to the $r$th 
even complete symmetric function. This follows from the usual 
infinite Grassmannian relation relating even complete symmetric functions to even elementary symmetric functions in $\Sym$
providing we can show that
\begin{equation}\label{baldy}
\sum_{s=0}^r \dot e_{2s} \dot h_{2r-2s} = \delta_{r,0}
\end{equation}
for all $r \geq 0$. 
This is true because the sum on the left hand side is $\dot z_{2r}$
by the first form of the definition  \cref{osym6},
which we have already shown is zero for $r \geq 1$.
\end{proof}

\begin{corollary}\label{dumbc3}
The largest supercommutative quotient of $\OSym_n$ is
\begin{equation}\label{dumbr3}
R_n := \OSym_n\: \Big/\: \Big\langle \big(o^{(n)}\big)^2, \big[o^{(n)},e_2^{(n)}\big]\Big\rangle.
\end{equation}
The isomorphism $\theta$ from \cref{dumbc1} induces an isomorphism $\theta_n$ from
$R_n$
to
$\Sym_{(n-1)/2}[\c]$ if $n$ is odd,
or to the quotient of $\Sym_{n/2}[\c]$ obtained by
setting the product of $\c$ and the $(n/2)$th even elementary symmetric polynomial to zero if $n$ is even.
\end{corollary}

\begin{proof}
This follows from \cref{dumbc1} since a supercommutative quotient of $\OSym_n$ is a supercommutative quotient of $\OSym$
in which the images of all $e_r\:(r > n)$ are zero. 
\end{proof}

\section{Odd nil-Hecke algebras}\label{secONH}

This section is largely an exposition of results from \cite{EKL}. 
The {\em odd nil-Hecke algebra} is the graded superalgebra $\ONH_n$ with generators
$x_i\:(i=1,\dots,n)$ and $\tau_j\:(j=1,\dots,n-1)$
which are odd of degrees 2 and $-2$, respectively,
subject to the following relations:
\begin{align}
  x_i x_j &= -x_j x_i&&(i \neq j)\label{ONH1}\\
  \tau_i \tau_j &= -\tau_j \tau_i&&(|i-j| > 1)\label{ONH3}\\
  x_i \tau_j &= -\tau_j x_i&&(i \neq j, j+1)\label{ONH5}\\
  \tau_j^2 &= 0\label{ONH2}\\
  \tau_j \tau_{j+1}\tau_j &= -\tau_{j+1}\tau_j\tau_{j+1}\label{ONH4}\\
  x_i \tau_i - \tau_i x_{i+1} &= 1 = \tau_i x_i - x_{i+1}\tau_i.\label{ONH6}
\end{align}
We warn the reader that the above is {\em not} the standard form of the presentation for this algebra which appears in all of the existing literature. The difference is in the relations \cref{ONH4,ONH6}, in which our minus signs become plus signs in the standard
presentation. To obtain the above presentation from the standard one, note that our generators $x_i$ and $\tau_j$ are equal to the elements denoted $(-1)^{i-1} x_i$ and $(-1)^{j-1} \tau_j$ elsewhere in the literature. This change certainly impacts many other formulae below, but it is usually straightforward to make the appropriate adaptation. 
One advantage of our modified sign convention can already be seen in the definitions \cref{epoly,hpoly} above---the corresponding formulae in \cite{EKL} involve some additional signs.

Let $\S_n$ act on the left on 
$\OPol_n$ by graded superalgebra automorphisms so that
${^w}\!x_i = (-1)^{\ell(w)+w(i)-i} x_{w(i)}$ for $w \in \S_n, 1 \leq i \leq n$.
In particular:
\begin{equation}
{^{s_j}}x_i = \left\{\begin{array}{ll}
 x_{j+1}&\text{if $i = j$}\\
 x_j&\text{if $i = j+1$}\\
-x_i&\text{otherwise.}
\end{array}\right.
\end{equation}
The {\em odd Demazure operator}
$\partial_j:\OPol_n \rightarrow \OPol_n$
is the linear map defined on $f \in \OPol_n$ by
\begin{equation}
  \partial_j(f) = \frac{(x_{j}+x_{j+1}) f -
    \left({^{s_j}}f\right)(x_{j}+x_{j+1})}{x_{j}^2-x_{j+1}^2},
\end{equation}
which makes sense because the denominator is central.
This formula first appeared in \cite[(4.10)]{KKO1} remembering, of course, our modified choice of signs. We actually never use this form of the definition of $\partial_j$, preferring the following
recursive definition: $\partial_j$ is the unique odd linear map of degree $-2$ such that
\begin{align}
\partial_j(x_i) &= \delta_{i,j} - \delta_{i,j+1},&
\partial_j(fg) &= \partial_j(f) g + \left({^{s_j}}f\right) \partial_j(g)
\label{betteract}
\end{align}
for $f, g \in \OPol_n$.
Now we make the graded vector superspace $\OPol_n$ 
into a left $\ONH_n$-supermodule so that  $x_i \in \ONH_n$ acts on $f \in \OPol_n$ by $x_i \cdot f := x_i f$,  and $\tau_j \in \ONH_n$ acts by $\tau_j \cdot f := \partial_j(f)$.
A tedious relation check shows that this definition makes sense.
It is straightforward to show by induction on $r$ that
 \begin{align}\label{rank22}
\tau_i \cdot x_{i}^{r+1} &= \sum_{s=0}^r x_{i+1}^{s} x_i^{r-s},
&\tau_i \cdot x_{i+1}^{r+1} &= -\sum_{s=0}^r
x_i^s x_{i+1}^{r-s}.
\end{align}
One can rewrite
\cref{rank22} as the generating function identities:
\begin{align}\label{rank23}
\tau_i \cdot (t-x_i)^{-1} &= (t-x_{i+1})^{-1} 
(t-x_i)^{-1},
&
\tau_i \cdot (t-x_{i+1})^{-1} &= -(t-x_i)^{-1} (t-x_{i+1})^{-1},
\end{align}
equalities in $\OPol_n\llbracket t^{-1}\rrbracket$.
From the former, we get that
\begin{equation}\label{slickproof}
\tau_{n-1} \cdots \tau_1 \cdot (t-x_1)^{-1}
= 
(t-x_n)^{-1} \cdots (t-x_1)^{-1}.
\end{equation}
Hence, recalling \cref{littlerock}, we get that
\begin{equation}\label{yesIneedit}
\tau_{n-1} \cdots \tau_1 \cdot x_1^{n+r-1} = h^{(n)}_r
\end{equation}
on computing $t^{-n-r}$-coefficients.

By the relations \cref{ONH1,ONH2,ONH3,ONH4,ONH5,ONH6}, $\ONH_n$ admits an algebra involution $\smiley_n$ and a superalgebra anti-involution $*$
defined by
\begin{align}\label{boreham}
\smiley_n:\ONH_n &\rightarrow \ONH_n,&
x_i &\mapsto x_{n+1-i}, & \tau_j &\mapsto -\tau_{n-j},\\
*:\ONH_n &\rightarrow \ONH_n,&
x_i &\mapsto x_i, & \tau_j &\mapsto -\tau_j.
\label{newtilde}
\end{align}
These definitions are consistent with \cref{borehamA,newtildeA}.
Note also that $\smiley_n$ and $*$ commute with each other.
Mirroring the notation for symmetric groups from {\em General conventions}, there is also a homomorphism
\begin{equation}\label{sigman}
\sig_n:\ONH_{n'} \rightarrow \ONH_{n+n'},\qquad x_i \mapsto x_{i+n}, \tau_j \mapsto \tau_{j+n}.
\end{equation}
We will show shortly that this is injective, but this is not yet clear.
Similarly, there is a homomorphism
$\sig_n:\OPol_{n'} \rightarrow \OPol_{n+n'},
x_i \mapsto x_{i+n}$, which is obviously injective. 
We have that
\begin{align}\label{smileyanddot}
\smiley_n(a) \cdot \smiley_n(f) &= \smiley_n(a\cdot f),&
\sig_n(a) \cdot \sig_n(f) &= \sig_n(a\cdot f),
\end{align}
for $a \in \ONH_n$, $f \in \OPol_n$,
or $a \in \ONH_{n'}$, $f \in \OPol_{n'}$, respectively.

For each $w \in \S_n$, we pick a reduced expression $w =
s_{j_1} \cdots s_{j_l}$ then set $\tau_w := \tau_{j_1} \cdots
\tau_{j_l}$. For the longest
element $w_n$, we choose the reduced expression
$(s_{n-1}s_{n-2}\cdots s_1) (s_{n-1} s_{n-2} \cdots s_2)\cdots (s_{n-1} s_{n-2}) s_{n-1}$,
and adopt the shorthands
\begin{align}\label{myshorthands}
\omega_n &:= \tau_{w_n}
= \tau_{n-1} \tau_{n-2}\cdots \tau_1 \sig_1(\omega_{n-1}),&
\chi_n &:= x_{n-1} x_{n-2}^2\cdots x_1^{n-1} = \sig_1(\chi_{n-1})x_1^{n-1}.
\end{align}
These elements have the following desirable property.

\begin{lemma}\label{normalization}
$\omega_n \cdot \chi_n = 1$.
\end{lemma}

\begin{proof}
When $n=1$, this is clear as $\omega_n = 1 = \chi_n$.
The result for $n > 1$ follows by induction:
\begin{align*}
\omega_n \cdot \xi_n &= \tau_{n-1} \cdots \tau_1
\sig_1(\omega_{n-1})\cdot \sig_1(\xi_{n-1}) x_1^{n-1}\\
&= \tau_{n-1} \cdots \tau_1
\cdot \sig_1(\omega_{n-1}\cdot \xi_{n-1}) x_1^{n-1}
= \tau_{n-1}\cdots \tau_1 \cdot x_1^{n-1} = 1,
\end{align*}
using \cref{yesIneedit} for the final equality.
\end{proof}

It is also clear that
\begin{align}\label{smileysign}
\smiley_n(\omega_n) &= \zeta_n \omega_n
\end{align}
for some $\zeta_n \in \{\pm 1\}$. 
One can verify explicitly
that $\zeta_n = (-1)^{\binom{n+1}{3}}$; the calculation is similar to the proof of \cite[Lem.~3.2]{EKL}. However, the only place this sign is used is in the proof of \cref{frobenius}, and in that place we actually do not need to know its acual value.
 
An important role will be played by the {\em odd Schubert polynomials}
\begin{equation}\label{oddschubert}
p^{(n)}_w := \tau_{w^{-1}w_n} \cdot \chi_n \in \OPol_n.
\end{equation}
For example, we have that
$p^{(3)}_{1} = 1$, 
$p^{(3)}_{s_1} = -x_1$,
$p^{(3)}_{s_2} = x_1+x_2$,
$p^{(3)}_{s_2 s_1} = x_1^2$,
$p^{(3)}_{s_1 s_2}= -x_2 x_1$ and
$p^{(3)}_{s_2 s_1 s_2} = x_2 x_1^2$.
In general, $p^{(n)}_w$ depends up to sign on the choice of reduced expression for $w^{-1} w_n$,
but we always have that $p^{(n)}_{w_n} = \chi_n$ and $p^{(n)}_{1} = 1$ thanks to \cref{normalization}.
Note also that $\deg\left(p^{(n)}_w\right) = 2\ell(w)$ and $\parity\left(p^{(n)}_w\right) \equiv \ell(w)\pmod{2}$.

\begin{theorem}[{\cite[Prop.~2.11]{EKL}}]\label{ONHbasis}
The elements $\big\{x^\kappa \tau_w =
x_1^{\kappa_1}\cdots x_n^{\kappa_n} \tau_w\:\big|\:w \in \S_n, \kappa \in \N^n\big\}$ give a basis for $\ONH_n$. Moreover, $\OPol_n$ is a faithful $\ONH_n$-module.
\end{theorem}

\begin{proof}
First one shows using \cref{ONH2,ONH3,ONH4} 
that any word in $\tau_j\:(j=1,\dots,n-1)$ can be reduced to $0$ or $\pm \tau_w$ for some $w \in \S_n$. It follows that the set in \cref{ONHbasis} spans $\ONH_n$. Then to establish the linear independence, 
suppose that we have some non-trivial linear relation 
$$
\sum_{w \in \S_n} \sum_{\kappa \in \N^n} c_{w,\kappa} x^\kappa \tau_w = 0$$ 
between the elements of this set.
Pick $w$ of minimal length such that $c_{w,\kappa} \neq 0$ for some $\kappa$.
Then we act on $p^{(n)}_{w}$.
For $w' \neq w$ with $\ell(w') \geq \ell(w)$, we have that $\tau_{w'} \cdot p^{(n)}_w = 0$
by the relations \cref{ONH2,ONH3,ONH4}, and $\tau_w \cdot p^{(n)}_{w} = \pm 1$ by \cref{normalization}.
So we deduce that $\sum_{\kappa \in \N^n} c_{w,\kappa} x^\kappa = 0$,
which is a contradiction. This also shows $\OPol_n$ is faithful.
\end{proof}

\begin{corollary}\label{bakeoffONH}
$\displaystyle\gsdim \ONH_n =\gsdim \OSym_n\times q^{\binom{n}{2}}[n]^!_{q,\pi}\times q^{-\binom{n}{2}}\overline{[n]}_{q,\pi}^!,$
\end{corollary}

\begin{proof}
The theorem gives that 
$$
\gsdim \ONH_n = \gsdim \OPol_n \times \sum_{w \in \S_n}  (\pi q^2)^{-\ell(w)}.
$$
Now replace $\gsdim \OPol_n$ by the first formula for it from \cref{bakeoffOPOL},
and replace the summation by a product using \cref{poincare}.
\end{proof}

\cref{ONHbasis}
implies that the obvious homomorphism $\OPol_n \rightarrow \ONH_n$ is injective. 
Henceforth, we identify $\OPol_n$ 
with a subalgebra of $\ONH_n$ via this map.
Another application of the basis theorem shows that the
homomorphism $\ONH_n \hookrightarrow \ONH_{n+1}$ taking $x_i$ to $x_i$ and $\tau_j$ to $\tau_j$ is injective. 
Thus, we have a tower of graded superalgebras
$\ONH_0 \subset \ONH_1 \subset \ONH_2 \subset\cdots$.
The basis theorem also shows that the homomorphism
$\sig_n:\ONH_{n'}\rightarrow \ONH_{n+n'}$
from \cref{sigman} is injective, as promised earlier.
For $\alpha\in\Comp(k,n)$, we let $\ONH_\alpha$ be the subalgebra
$\big\{x^\kappa\tau_w\:\big|\:w \in \S_\alpha, \kappa \in \N^n\big\}$ of $\ONH_n$.
Finally,
let 
$\ONH^{\fin}_n$ be the subalgebra of $\ONH_n$ with basis $\{\tau_w \:|\:w \in \S_n\}$. As an algebra, $\ONH^\fin_n$ is generated by the elements
$\tau_j\:(j=1,\dots,n-1)$ subject just to \cref{ONH2,ONH3,ONH4}. There is a unique way to make the ground field $\k$
into a purely even graded left $\ONH^\fin_n$-supermodule concentrated in degree 0; each $\tau_j$ acts as zero.
There is then a canonical isomorphism of graded $\ONH_n$-supermodules
\begin{equation}\label{polisinduced}
\ONH_n \otimes_{\ONH^\fin_n} \k \stackrel{\sim}{\rightarrow} \OPol_n,\qquad
x^\kappa \otimes 1 \mapsto x^\kappa.
\end{equation}
This isomorphism explains the origin of the polynomial representation of $\ONH_n$.

Now recall the subalgebra $\OSym_n$ of $\OPol_n$ which was defined just after \cref{mainfact}---it is the subalgebra of $\OPol_n$ generated by the odd symmetric polynomials $e_r^{(n)}$ from \cref{epoly}.
A different formulation of the definition of $\OSym_n$ was adopted in \cite{EKL}, 
where $\OSym_n$ was defined from the outset to be $\bigcap_{i=1}^{n-1} \ker \partial_i$, which is a subalgebra of $\OPol_n$. 
We will deduce the equality of $\OSym_n$ with this subalgebra in \cref{aaaah}, but one containment is obvious:
we have that
\begin{equation}\label{easyinc}
\displaystyle\OSym_n \subseteq \bigcap_{i=1}^{n-1} \ker \partial_i.
\end{equation}
To see this, it suffices to check that $\partial_i(e_r^{(n)}) = 0$ for all $i$ and $r=1,\dots,n$, which follows  from the definitions since $\partial_i(x_i+x_{i+1}) = \partial_i(x_i x_{i+1}) = 0$.

\begin{theorem}[{\cite[Prop.~2.13, Cor.~2.14]{EKL}}]\label{firstkeythm}
The graded right $\OSym_n$-supermodule $\OPol_n$ is free of
graded rank $q^{\binom{n}{2}} [n]^!_{q,\pi}$ with basis 
$\big\{p^{(n)}_w\:|\:w \in \S_n\}$
given by the odd Schubert polynomials from \cref{oddschubert}. So we
have that
\begin{equation}\label{bubbleandsqueek}
\OPol_n =\bigoplus_{w \in \S_n} p_w^{(n)} \OSym_n
\qquad\text{ with }\qquad p_w^{(n)} \OSym_n \simeq (\Pi Q^2)^{\ell(w)} \OSym_n
\end{equation}
as graded right $\OSym_n$-supermodules. Moreover, the action of $\ONH_n$ on $\OPol_n$ induces a graded superalgebra isomorphism
\begin{equation}\label{matrixiso}
\ONH_n \stackrel{\sim}{\rightarrow} \End_{\dash\OSym_n}(\OPol_n).
\end{equation}
\end{theorem}

\begin{proof}
We claim that the polynomials $p^{(n)}_w\:(w \in \S_n)$ are linearly independent over $\OSym_n$. To see this, take a non-trivial linear relation 
$$
\sum_{w \in \S_n} p^{(n)}_w b_w = 0
$$ 
for $b_w \in \OSym_n$. Choose $w$ of maximal length such that $b_w \neq 0$. Then we act with $\tau_{w}$.
We have that $\tau_{w} \cdot p^{(n)}_{w'} b_{w'} = 0$ for $w' \neq w$ by the relations \cref{ONH2,ONH3,ONH4}, 
and $\tau_{w} \cdot p^{(n)}_w b_w = \pm b_w$, so we deduce that 
$b_w = 0$, a contradiction.
The claim implies that the $\OSym_n$-submodule of $\OPol_n$ generated by $p^{(n)}_w\:(w \in \S_n)$
is of graded superdimension $\gsdim \OSym_n \times \sum_{w \in \S_n} (\pi q^{2})^{\ell(w)}$, which is equal to
$\gsdim \OPol_n$ by \cref{bakeoffOPOL}. 
Hence, the $p^{(n)}_w\:(w\in \S_n)$ also span $\OPol_n$ as an $\OSym_n$-module, 
and we have proved \cref{bubbleandsqueek}.

To establish \cref{matrixiso}, we first note by \cref{bubbleandsqueek} that
\begin{align*}
\gsdim \End_{\dash\OSym_n}(\OPol_n) &= \sum_{x,y \in \S_n} (\pi q^2)^{\ell(x)-\ell(y)}
= 
\left(\sum_{x \in \S_n} (\pi q^2)^{\ell(x)}\right)
\left(\sum_{y \in \S_n} (\pi q^{2})^{-\ell(y)}\right)
=q^{\binom{n}{2}}[n]^!_{q,\pi} \times q^{-\binom{n}{2}}\overline{[n]}^!_{\!q,\pi},
\end{align*}
applying \cref{poincare}.
The homomorphism $\rho:\ONH_n
\rightarrow \End_{\dash\OSym_n}(\OPol_n)$ is injective by
\cref{ONHbasis}.
Therefore it is an isomorphism because the graded superdimensions
are the same thanks to \cref{bakeoffONH}.
\end{proof}

\begin{corollary}\label{aaaah}
We have that $\displaystyle\OSym_n = \bigcap_{i=1}^{n-1} \ker \partial_i = \bigcap_{i=1}^{n-1} \im \partial_i.$
\end{corollary}

\begin{proof}
It is easy to see that $\ker \partial_i = \im \partial_i$ for each $i$, hence, the second equality holds. 
For the first one,  we have already noted in \cref{easyinc} that $\OSym_n \subseteq \bigcap_{i=1}^{n-1} \ker \partial_i$.
Conversely, take $f \in \bigcap_{i=1}^{n-1} \ker\partial_i$ and write it as
$f = \sum_{w \in \S_n} p^{(n)}_w b_w$ for $b_w \in \OSym_n$.
We need to show that $b_w = 0$ except when $w=1$. 
Suppose for a contradiction that this is not the case, and pick $w$ of maximal length such that $b_{w} \neq 0$. Then we act on $f$ with $\tau_{w}$ 
to see that $b_w = 0$, contradiction.
\end{proof}

\begin{remark}\label{aswell}
As well as the basis $F := \big\{p_w^{(n)}\:\big|\:w \in \S_n\big\}$ of odd Schubert polynomials from \cref{firstkeythm}, the monomials $G := \big\{x_n^{\kappa_n}\cdots x_1^{\kappa_1}\:\big|\:\kappa \in \N^n\text{ with }0 \leq \kappa_i \leq n-i\big\}$
form a basis for $\OPol_n$ as a free right $\OSym_n$-module. To see this, it suffices to show that $\k F = \k G$.
The elements of $F$ are linearly independent over $\k$ by \cref{firstkeythm}, so $\dim \k F = n!$. Also $\dim \k G = n!$ obviously. So we are reduced to checking that $\k F \subseteq \k
G$. To see this, we note first that $\k G$ is invariant under the 
action of each $\tau_i$, as may be seen directly using \cref{rank22}
plus $\tau_i \cdot x_{i+1}^r x_i^r = 0$.
Since $\chi_n \in \k G$, it follows that $p_w^{(n)} = \tau_{w^{-1}w_n} \cdot \chi_n \in \k G$ for each $w \in \S_n$ as claimed.
\end{remark}

Let $M_{q^{\binom{n}{2}}[n]^!_{q,\pi}}(\OSym_n)$ denote the usual algebra of matrices $A = (a_{w,w'})_{w,w' \in \S_n}$ with entries in $\OSym_n$ viewed as a graded superalgebra so that the matrix with $a \in (\OSym_n)_{i,p} $ in its $(w,w')$-entry and zeros elsewhere is of degree $i+2\ell(w)-2\ell(w')$ and parity $p+\ell(w)-\ell(w')\pmod{2}$. This graded superalgebra may be identified with $\End_{\dash \OSym_n}(\OPol_n)$ so that the matrix $A$ just described corresponds to the unique right $\OSym_n$-supermodule endomorphism of $\OPol_n$ taking $p_{w'}^{(n)}$ to $\sum_{w \in \S_n} p_w^{(n)} a_{w,w'}$ for every $w' \in \S_n$. Thus, \cref{firstkeythm} shows that $\ONH_n \cong M_{q^{\binom{n}{2}}[n]^!_{q,\pi}}(\OSym_n)$. It follows that the graded superfunctors
\begin{align}\label{Psin}
-\otimes_{\ONH_n} \OPol_n&:
\doMs\ONH_n \rightarrow
\doMs\OSym_n,\\
\Hom_{\ONH_n}(\OPol_n,-)&:\ONH_n\sMod \rightarrow 
\OSym_n\sMod\label{altPsin}
\end{align}
are equivalences of graded $(Q,\Pi)$-supercategories. 

\begin{theorem}[{\cite[Prop.~2.15]{EKL}}]\label{thecenter} 
The even center $Z(\ONH_n)_{\0}$ of $\ONH_n$ is the 
graded algebra  consisting of symmetric polynomials in $x_1^2,\dots,x_n^2$. This coincides with the even center $Z(\OSym_n)_{\0}$ of $\OSym_n$ embedded into $\ONH_n$ in the natural way.
\end{theorem}

\begin{proof}
This is proved in \cite{EKL} but 
we give a slightly different argument since there are some minor issues in the first paragraph of the original proof, which does not restrict attention to the {\em even} center.
Take $z \in Z(\ONH_n)$. Using \cref{ONHbasis}, we have that
$$
z = \sum_{w \in \S_n} f_w \tau_w
$$
for unique $f_w \in \OPol_n$. The first step is to show that $f_w = 0$ unless $w=1$. To see this, suppose for a contradiction that  it is not the case. Let $w$ be of maximal length such that $f_{w} \neq 0$. 
Pick $i \in \{1,\dots,n\}$ such that $j := w(i) \neq i$. We have that
$$
x_i z =  \sum_{w \in \S_n} x_i f_w \tau_w.
$$
Now we use the relations to express $z x_i$ as a linear combination $\sum_{v \in \S_n} g_v \tau_v$ for $g_v \in \OPol_n$.
Since $\tau_w x_i  = \pm x_j \tau_w $ plus a linear combination of $\tau_{w'}$ for $w'$ with $\ell(w') < \ell(w)$, 
we see that $g_w = \pm f_w x_j$. Thus, we must have that $x_i f_w = \pm f_w x_j$.
Since $i \neq j$, it is easy to see that this implies that $f_w=0$. So now we have proved that $z \in \OPol_n$.
Next, {\em assuming also that $z$ is even},  we show that $z$ is in fact in $\k[x_1^2,\dots,x_n^2]$ essentially following the idea from the proof in \cite{EKL}. Take any $1 \leq i \leq n$ and suppose that $z = \sum_{k \geq 0} f_k x_i^k$ for $f_k$ belonging to the subalgebra of $\OPol_n$ generated by $x_1,\dots,x_{i-1},x_{i+1},\dots,x_n$. Since $x_i z = z x_i$, we get that each $f_k$ must be even.
Since $z$ is even too it follows that $f_k = 0$ unless $k$ is even. This shows that $z$ only involves even powers of $x_i$. This is true for each $i$, so $z \in \k[x_1^2,\dots,x_n^2]$ as claimed.
To complete the proof that $z$ is actually a symmetric polynomial in $x_1^2,\dots,x_n^2$, and to show that any such polynomial is central, we can now refer the reader to the argument given in the second two paragraphs of the proof of \cite[Prop.~2.15]{EKL}.

Finally we explain how to see that $Z(\ONH_n)_{\0}$ coincides with $Z(\OSym_n)_{\0}$.
The supercenter of the matrix algebra $M_{q^{\binom{n}{2}}[n]^!_{q,\pi}}(\OSym_n)$ is isomorphic to $Z(\OSym_n)$ via the map
taking $z \in Z(\OSym_n)$ to the matrix 
$\diag(z,\dots,z)$. It follows that the even centers are isomorphic too. 
Given $z$ in the even center of $\ONH_n$, we have just shown that it is a polynomial
in $x_1^2,\dots,x_n^2$, so we have that $z p_w^{(n)} = p_w^{(n)} z$ for all $w \in \S_n$. It follows that $z$ acts on $\OPol_n$ in the same way as the matrix $\diag(z,\dots,z)$ under the identification of $\ONH_n$ with matrices described above. This shows that the natural embedding of $\OSym_n$ into $\ONH_n$
restricts to give an isomorphism between the even centers of $\OSym_n$ and $\ONH_n$.
\end{proof}

The idempotents in $\ONH_n$ corresponding to the diagonal matrix units $e_{w,w} \in M_{q^{\binom{n}{2}}[n]^!_{q,\pi}}(\OSym_n)$,
that is, the elements which act on $\OPol_n$ as the projections onto the indecomposable summands in \cref{bubbleandsqueek}, give a complete set of primitive idempotents in $\ONH_n$. 
It is clear from \cref{ONHbasis} that the component of $\ONH_n$ of
smallest degree is 1-dimensional
spanned by $\omega_n$. Since $\omega_n \chi_n \omega_n$ is of the same degree as $\omega_n$, it follows that 
$\omega_n \chi_n \omega_n$ is a scalar multiple of $\omega_n$.
Moreover, both $\omega_n \chi_n \omega_n$ and $\omega_n$
map $\chi_n$ to $1$ by \cref{normalization}, hence, we actually have that
\begin{equation}\label{skitime}
\omega_n \chi_n \omega_n = \omega_n.
\end{equation}
From this it follows that the following are both idempotents:
\begin{align}\label{idempotents}
(\chi\omega)_n &:= \chi_n \omega_n,
&(\omega\chi)_n &:= \omega_n \chi_n.
\end{align}
The first of these, $(\chi\omega)_n$, 
is exactly the matrix unit $e_{w_n,w_n}$ which projects $\OPol_n$ onto the top degree component $p_{w_n}^{(n)} \OSym_n$. 
This follows almost immediately since \cref{normalization} shows that
$\chi_n \omega_n\cdot p_{w_n}^{(n)} = p_{w_n}^{(n)}$ and $\chi_n
\omega_n\cdot  p_w^{(n)} = 0$ for all other $w \in \S_n$ as $\omega_n \cdot p_w^{(n)}= 0$ by degree considerations. 
In particular, this shows that $(\chi\omega)_n$ is a primitive idempotent.
The second one, $(\omega\chi)_n$, is also primitive
since we have that
\begin{equation}\label{idempotentstar}
(\omega\chi)_n=(\chi\omega)_n^*.
\end{equation}
To see this, it is clear from the definitions that
$(\chi\omega)_n^* = \pm (\omega\chi)_n$, and the sign must be plus since 
$(\omega\chi)_n$ is an idempotent.
Note also that $(\omega\chi)_n \OPol_n = \OSym_n$.
To see this, every $\partial_i$ annihilates 
$(\omega\chi_n) \cdot \OPol_n$, so $(\omega\chi)_n \cdot \OPol_n \subseteq
\OSym_n$ thanks to \cref{aaaah}, and 
it is easy to see directly that
$(\omega\chi)_n \cdot f = f$ for any $f \in \OSym_n$
giving the other containment.
Thus, we have shown that 
\begin{align}\label{newsnow}
(\chi\omega)_n \cdot \OPol_n &= \chi_n \OSym_n,& (\omega\chi)_n \cdot \OPol_n &= \OSym_n.
\end{align}
For $n \geq 2$,  $(\omega\chi)_n$ is {\em not} the idempotent corresponding to the matrix unit $e_{1,1}$ in the matrix algebra $M_{q^{\binom{n}{2}}[n]^!_{q,\pi}}(\OSym_n)$,
i.e., it is a projection of $\OPol_n$ onto $\OSym_n$, but along a different direct sum decomposition to \cref{bubbleandsqueek}.
It is convenient to work with since
left multiplication by $\omega_n$ defines a homogeneous isomorphism $(\chi\omega)_n \cdot \OPol_n \stackrel{\sim}{\rightarrow} (\omega\chi)_n \cdot \OPol_n$, with inverse defined by left multiplication by $\chi_n$.

\begin{lemma}\label{borisagain}
We have that
$\displaystyle ONH_n \simeq
\bigoplus_{w \in \S_n}  (\Pi Q^2)^{\ell(w)} \:(\omega\chi)_n \ONH_n
\simeq\bigoplus_{w \in \S_n}  (\Pi Q^2)^{-\ell(w)} \:(\chi\omega)_n \ONH_n$ as a graded right $\ONH_n$-supermodule.
\end{lemma}

\begin{proof}
Left multiplication by $\omega_n$ defines an isomorphism $(\chi\omega)_n \ONH_n \simeq (\Pi Q^2)^{\binom{n}{2}} (\omega\chi)_n \ONH_n$ with inverse given by left multiplication by $\chi_n$. Therefore it suffices to prove the first isomorphism. Since \cref{Psin} is a graded superequivalence, we can apply it to reduce the problem to proving that
$$
\OPol_n \simeq 
\bigoplus_{w \in \S_n}  (\Pi Q^2)^{\ell(w)} \:\OSym_n
$$
as graded right $\OSym_n$-supermodules,  where we have used that $(\omega\chi)_n \OPol_n = \OSym_n$ by \cref{newsnow}.
This follows from \cref{bubbleandsqueek}.
\end{proof}

\begin{lemma}\label{imathjmath}
The map $\imath:\OPol_{n} \rightarrow \ONH_n (\omega\chi)_n,
f \mapsto f (\omega\chi)_n$ is an even degree 0 isomorphism of graded left $\ONH_n$-supermodules.
The map $\jmath:\OSym_n \rightarrow (\omega\chi)_n \ONH_n (\omega\chi)_n$
defined by the composition of the natural inclusion of $\OSym_n$ into $\ONH_n$ followed by 
the projection $a \mapsto (\omega\chi)_n a (\omega\chi)_n$
is a graded superalgebra isomorphism.
Moreover, we have that $\imath(fa) = \imath(f) \jmath(a)$ for all $f \in \OPol_n$ and $a \in \OSym_n$.
\end{lemma}

\begin{proof}
Since $\tau_j (\omega\chi)_n = 0$ by degree considerations, there is a unique graded left $\ONH_n$-supermodule homomorphism $\OPol_n \rightarrow \ONH_n (\omega\chi)_n$ taking $1$ to $(\omega\chi)_n$ thanks to \cref{polisinduced}.
This is $\imath$. Also $(\omega\chi)_n \cdot 1 = 1$, so there is a supermodule homomorphism
$\ONH_n (\omega\chi)_n \rightarrow \OPol_n, a \mapsto a \cdot 1$. These two maps are mutual inverses, hence, 
$\imath$ is an isomorphism. 

The restriction of $\imath$ gives an isomorphism $(\omega\chi)_n \cdot \OPol_n \stackrel{\sim}{\rightarrow} (\omega\chi)_n \OPol_n (\omega\chi)_n$. 
Since $(\omega \chi)_n a (\omega\chi)_n = a (\omega\chi)_n$ for $a \in \OSym_n$ and $(\omega\chi)_n \cdot \OPol_n = \OSym_n$ by \cref{newsnow}, this restriction is the isomorphism $\jmath$ 
from the statement of the lemma. 
\end{proof}

\begin{corollary}\label{altPsi}
Using $\jmath$ to identify $\OSym_n$ with $(\omega\chi)_n \ONH_n (\omega\chi)_n$, the superfunctors 
$-\otimes_{\ONH_n} \OPol_n$ 
and $\Hom_{\ONH_n}(\OPol_n,-)$
from \cref{Psin,altPsin} are isomorphic to the idempotent truncation functors defined by right and left multiplication by the idempotent
$(\omega\chi)_n$, respectively.
\end{corollary}

We note finally that there is also a {\em right} action
of $\ONH_n$ on $\OPol_n$, and everything in this section
could be reformulated in terms of $\OPol_n$ viewed as an $(\OSym_n,\ONH_n)$-superbimdule.
This right action may be defined succinctly from
\begin{equation}\label{opolright}
f \cdot a := (-1)^{\parity(a)\parity(f)} \big(a^* \cdot f^*\big)^*
\end{equation}
for $f \in \OPol_n, a \in \ONH_n$.
The following more explicit description similar to \cref{betteract} 
can easily be derived from this:
\begin{align}
x_i \cdot \tau_j &= \delta_{i,j} - \delta_{i,j+1},&
(fg) \cdot \tau_j &= f (g \cdot \tau_j) + (f \cdot \tau_j) \left({^{s_j}}g\right).
\label{worseact}
\end{align}
for $f, g \in \OPol_n$.
The right action of $\ONH_n$
on $\OPol_n$ obviously commutes with the natural 
action of $\OSym_n$ by left multiplication.
\cref{firstkeythm} and \cref{idempotentstar} imply that
\begin{equation}\label{bubbleandsqueekier}
\OPol_n \simeq \bigoplus_{w \in S_n}
(\Pi Q^2)^{\ell(w)} \OSym_n
\end{equation}
as a graded left $\OSym_n$-supermodule,
with the ``bottom" summand that is $\OSym_n$ itself
being the image of the 
idempotent $(\chi\omega)_n$ acting on the right.
We stress that the action \cref{opolright} is different 
from the right action defined via
$f \cdot a := (-1)^{\parity(a) \parity(f)} a^* \cdot f$; the latter
action does not commute with the left action of $\OSym_n$.
When we talk about $\OPol_n$ as a right $\ONH_n$-supermodule,
we always mean the action defined via \cref{opolright}.

\begin{lemma}\label{trulycrazy}
$\displaystyle (t+x_2)^{-1} x_1^r \cdot \tau_1 = 
(t+x_2)^{-1} x_2^r \Big(t+(-1)^r x_1\Big)^{-1}
+ \sum_{q=0}^{r-1} (-1)^{qr}
(t+x_2)^{-1}x_2^{q}x_1^{r-q-1}.$
\end{lemma}

\begin{proof}
Similarly to \cref{rank22,rank23}, one shows that $(t+x_2)^{-1} \cdot \tau_1
= (t+x_2)^{-1} (t+x_1)^{-1}$ and $
x_1^r \cdot \tau_1 = \sum_{q=0}^{r-1} x_1^{r-q-1}x_2^q$.
These combine using \cref{worseact}
to give the final formula; one also needs to commute all $x_2$
to the left of all $x_1$ producing some additional signs.
\end{proof}

\section{Odd Schur polynomials}

Another important basis of $\OSym$ is introduced in \cite[Sec.~3.3]{EK}: the basis of
{\em odd Schur functions} $\{s_\lambda\:|\:\lambda \in \Par\}$.
As explained after \cite[Cor.~3.9]{EK},
this is the basis of $\OSym$ characterized uniquely by the properties
that $(s_\lambda, h_\mu)^- = 0$ if $\mu >_{\lex} \lambda$
and $s_\lambda = h_\lambda + $(a $\Z$-linear combination of
other $h_\mu$ for $\mu >_{\lex} \lambda$).
Some examples can be found in the appendix of \cite{EK}.
The key property of odd Schur functions is that they are 
signed-orthonormal: 

\begin{theorem}[{\cite[Cor.~3.9]{EK}}]\label{schurform}
For $\lambda,\mu \in \Par$, we have that
$(s_\lambda,s_\mu)^- = (-1)^{dN(\lambda)}\delta_{\lambda,\mu}$.
\end{theorem}

The {\em odd Kostka matrix} $(K_{\lambda,\mu})_{\lambda,\mu \in \Par}$ 
is the transition matrix defined from
\begin{equation}\label{kostka}
h_\mu = \sum_{\lambda\in\Par} K_{\lambda,\mu} s_\lambda.
\end{equation} 
There is an explicit formula for the entries of this matrix
derived in \cite[(3.7)]{EK}, as follows.
For a $\lambda$-tableau $T$ ($=$a function from the Young diagram of $\lambda$ to $\Z$), we let $N(T)$ be the number of pairs of boxes $(A,B)$ such that $B$ is strictly 
north of $A$ and also $T(B) \geq T(A)$.
For example, if $T$ is the 
unique semistandard $\lambda$-tableau of content $\lambda$ (so all entries on row $i$ are equal to $i$) then $N(T)=0$.
Then
\begin{equation}
K_{\lambda,\mu} = \sum_{T} (-1)^{N(T)}
\end{equation}
summing over semistandard $\lambda$-tableaux $T$ of content $\mu$.
Note from this description
that $K_{\lambda,\mu} = 0$ unless $\lambda \geq \mu$ in the dominance order.
So we actually have that
\begin{equation}\label{domnotlex}
s_\lambda = h_\lambda + \text{(a $\Z$-linear combination of
other $h_\mu$ for $\mu > \lambda$)}
\end{equation}
in the dominance rather than merely lexicographic ordering.

Since the involution $\psi_1 \psi_2$ in \cite{EK}
is our $\psi$ by \cref{eknotation}, 
\cite[Lem.~3.11]{EK} shows that
\begin{equation}
\label{niall}
\psi(s_\lambda) = (-1)^{NE(\lambda)+|\lambda|} s_{\lambda^\transpose}
\end{equation}
for any $\lambda \in \Par$. 
Hence, applying $\psi$ to \cref{domnotlex}, we have that
\begin{equation}\label{domnutlex}
s_{\lambda^\transpose} = (-1)^{NE(\lambda)} e_\lambda + \text{(a $\Z$-linear combination of
other $e_\mu$ for $\mu > \lambda$)}.
\end{equation}
From \cref{domnotlex,domnutlex}, we see in particular that
\begin{align}\label{hes}
s_{(r)} &= h_r,
&
s_{(1^r)}&=e_r.
\end{align}
Using also \cref{niall,silverton}, 
\cref{schurform} implies:

\begin{corollary}\label{schurformalt}
For $\lambda,\mu \in \Par$, we have that
$(s_\lambda,s_\mu)^+ = (-1)^{dE(\lambda)} \delta_{\lambda,\mu}$.
\end{corollary}

Applying \cref{niall} one more time, 
this time combined with the first identity from \cref{silverton}, 
it follows that 
$s_\lambda$ can also be characterized as the unique element of
$\OSym$ such that $\big(s_\lambda, e_\mu\big)^+ = 0$ for $\mu >_{\lex} \lambda^\transpose$
and $s_\lambda = (-1)^{NE(\lambda)} e_{\lambda^\transpose}+$(a $\Z$-linear combination of $e_{\mu}$ for $\mu >_\lex \lambda^\transpose$).
This characterization plus \cref{babies} and 
the second identity from \cref{silverton} implies that
\begin{equation}\label{bears}
\smiley(s_\lambda)^* = (-1)^{dN(\lambda)+dE(\lambda)} s_\lambda.
\end{equation}
We define the {\em dual odd Schur function} 
\begin{equation}\label{littlebears}
\sigma_\lambda := \smiley(s_\lambda) = (-1)^{dN(\lambda)+dE(\lambda)} s_\lambda^*.
\end{equation}
In particular, applying 
$\smiley$ to \cref{hes} and using the definitions \cref{toe,tee}, we have that
\begin{align}\label{hes2}
\sigma_{(r)} &= \eta_r,
&
\sigma_{(1^r)}&=\eps_r.
\end{align}
The {\em odd Schur polynomial}
 $s^{(n)}_\lambda$ and 
the {\em dual odd Schur polynomial} 
 $\sigma^{(n)}_\lambda$
 are the images of $s_\lambda$
and $\sigma_\lambda$
under the quotient map $\pi_n:\OSym \rightarrow \OSym_n$, respectively.
The dual odd Schur polynomials coincide with the polynomials 
introduced in \cite[Def.~4.10]{EKL} and 
play an important role in \cref{frobenius} below.

\begin{theorem}\label{alreadybasis}
The set $\big\{s^{(n)}_\lambda\:\big|\:\lambda \in \SPar{n}\big\}$ is a basis for $\OSym_n$. Moreover, for any $\lambda \in \Par$, we have that
\begin{equation}\label{triangularity}
s_{\lambda}^{(n)} = 
\left\{
\begin{array}{ll}
x^\lambda+\text{(a $\Z$-linear combination of $x^\kappa$ for
$\kappa \in \N^n$ with $\kappa < \lambda$)}&\text{if $\h(\lambda)\leq n$}\\
0&\text{if $\h(\lambda)> n$.}
\end{array}\right.
\end{equation}
\end{theorem}

\begin{proof}
This follows from \cref{domnutlex,whatekldo} plus \cref{sourceofne}.
\end{proof}

\begin{corollary}\label{rangoon}
The set $\big\{\sigma^{(n)}_\lambda\:\big|\:\lambda \in \SPar{n}\big\}$ is a basis for $\OSym_n$. Moreover, 
$\sigma^{(n)}_\lambda = 0$ for $\lambda \in \Par$ with $\h(\lambda) > n$.
\end{corollary}

\begin{proof}
Apply $\smiley_n$ to the results established in the theorem.
\end{proof}

\begin{corollary}
The set $\big\{h^{(n)}_\lambda\:\big|\:\lambda \in \SPar{n}\big\}$ is a basis for $\OSym_n$.
\end{corollary}

\begin{proof}
By graded dimension considerations, it suffices to show that
$\big\{h^{(n)}_\lambda\:\big|\:\lambda \in \SPar{n}\big\}$
spans $\OSym_n$. This follows from \cref{alreadybasis}
using \cref{domnotlex}
and also the observation that
$\mu > \lambda \in \SPar{n}\Rightarrow \mu \in \SPar{n}$.
\end{proof}

The next result was originally formulated as a conjecture in \cite[Conj.~5.3]{EKL}, and the conjecture was proved in \cite[Th.~3.8]{E}. 
However, we also need to reformulate it using our sign conventions, and for this we need a preliminary lemma.

\begin{lemma}\label{whysohard}
For $f \in \OSym_{n-1}$, $m \geq 0$
and $k=1,\dots,n$, we have that
$$
\tau_{k-1} \cdots \tau_1 x_1^{m+k-1} \cdot \sig_1(f)
= \sum_{i=0}^{k-1} \sig_i\big(h_{m+i}^{(k-i)}\big)
\tau_i \cdots \tau_1 \cdot \sig_1(f),
$$
equality in the $\ONH_n$-supermodule $\OPol_n$.
\end{lemma}

\begin{proof}
We prove this by induction on $k$, the case $k=1$ being trivial. 
For the induction step, we have by induction that
$
\tau_{k-1}\cdots \tau_1 x_1^{m+k} \cdot \sig_1(f)=
\sum_{i=0}^{k-1}\sig_i\big(h_{m+i+1}^{(k-i)}\big)\tau_i\cdots \tau_1 \cdot \sig_1(f).
$
Applying $\tau_k$ to both sides, we deduce that
$$
\tau_{k}\cdots \tau_1 x_1^{m+k}\cdot \sig_1(f)=
\sum_{i=0}^{k-1}\left(\tau_k \cdot \sig_i\big(h_{m+i+1}^{(k-i)}\big)\right)\tau_i\cdots \tau_1 \cdot \sig_1(f)
+
\sum_{i=0}^{k-1}
{^{s_k}}\sig_i\big(h_{m+1+i}^{(k-i)}\big)\tau_k \tau_i\cdots \tau_1 \cdot \sig_1(f).
$$
By \cref{rank22}, we have that $\tau_{k-i} \cdot h_{m+i+1}^{(k-i)}
= h_{m+i}^{(k-i+1)}$, so the $i$th term in the first summation becomes
$$
\sig_i\big(\tau_{k-i} \cdot h_{m+i+1}^{(k-i)}\big)
\tau_i\cdots \tau_1 \cdot \sig_1(f) = \sig_i \big(h_{m+i}^{(k+1-i)}\big)\tau_i \cdots \tau_1 \cdot \sig_1(f).
$$
The second summation gives zero except when $i=k-1$, when it gives
$\sig_k\big(h_{m+k}^{(1)}\big) \tau_k \cdots \tau_1 \cdot \sig_1(f)$.
In total, we obtain the desired $\sum_{i=0}^k \sig_i\big(h_{m+i}^{(k+1-i)}\big)\tau_i \cdots \tau_1 \cdot \sig_1(f)$.
\end{proof}

Recall for
$\lambda \in \N^n$ 
that $x^\lambda$ denotes $x_1^{\lambda_1} \cdots x_n^{\lambda_n}$.

\begin{theorem}[{\cite[Th.~3.8]{E}}]\label{schurtheorem}
For $\lambda \in \SPar{n}$, we have that
$s^{(n)}_\lambda = (\omega \chi)_n \cdot x^\lambda$.
\end{theorem}

\begin{proof}
The original formula from \cite{EKL} took the form 
\begin{equation}\label{theirs}
s_\lambda = (-1)^{\binom{n}{3}} \left[\partial_{w_0}\Big(x_1^{\lambda_1} \cdots  x_n^{\lambda_n} x_1^{n-1}\cdots x_{n-2}^2 x_{n-1}\Big)\right]^{w_0},
\end{equation}
using their notation everywhere. 
The result was proved 
in \cite{E} with exactly this in place of our $(\omega\chi)_n \cdot x^\lambda$.
In \cref{theirs}, the conjugation by $w_0$ corresponds up to a sign to an application of our
involution $\smiley_n$, which commutes with the action of $\partial_{w_0}$, again up to a sign,
due to \cref{smileyanddot,smileysign}.
This shows that the right hand side of \cref{theirs}
is equal to $(\omega \chi)_n \cdot  x^\lambda$ up to a sign.
Hence, $(\omega \chi)_n \cdot x^\lambda = \pm s^{(n)}_\lambda$.
Presumably, one could see that the sign is actually a plus by carefully keeping track of all of the sign changes in this translation. However, this is rather prone to error, so we give an alternative approach. It suffices by \cref{triangularity} 
to check that the $x^\lambda$-coefficient of 
$(\omega\chi)_n \cdot x^\lambda$ is 1.
From \cref{myshorthands}, 
we have that $\omega_n = \tau_{n-1} \cdots \tau_1 \sig_1(\omega_{n-1})$
and
$\chi_n = \sig_1(\chi_{n-1}) x_1^{n-1}$.
Also $x^\lambda = x_1^{\lambda_1} \sig_1(x^{\mu})$
where $\mu = (\lambda_2,\dots,\lambda_n)$.
Using these and induction on $n$, we get that
\begin{align*}
(\omega \chi)_n \cdot  x^\lambda
&=
\tau_{n-1} \cdots \tau_1 x_1^{\lambda_1+n-1}\cdot \sig_1\big(\omega_{n-1} \chi_{n-1} \cdot x^{\mu}\big)\\
&=
\tau_{n-1} \cdots \tau_1  x_1^{\lambda_1+n-1} \cdot\sig_1\big(s_{\mu}^{(n-1)}\big)
= 
\sum_{i=0}^{n-1} \sig_i\big(h_{\lambda_1+i}^{(n-i)}\big)
\tau_i \cdots \tau_1 \cdot \sig_1\big(s_{\mu}^{(n-1)}\big),
\end{align*}
the last equality being an application of \cref{whysohard}.
Now we express this in terms of the monomial basis for $\OPol_n$.
The only place a monomial whose $x_1$-exponent is $\geq \lambda_1$
can arise is from the $i=0$ term, which is $h_{\lambda_1}^{(n)} \sig_1\big(s_{\mu}^{(n-1)}\big)$.
This has leading term exactly $x^\lambda$, as required.
\end{proof}

Now we are going to discuss a graded superalgebra which may be interpreted as the odd analog of the equivariant cohomology algebra of the Grassmannian.
We set things up initially in greater generality.
Switching our default choice of variable from $n$ to $\ell$ for reasons that will become clear shortly, suppose that $\alpha \in \Comp(k,\ell)$. This represents the ``shape" of a partial flag variety,
Grassmannians being the special case that $k=2$. Let
\begin{equation}
\OSym_\alpha := \bigcap_{\substack{i \in \{1,\dots,\ell\}\\ i\notin\{\alpha_1,\alpha_1+\alpha_2,\dots,\alpha_1+\cdots+\alpha_k\}}} \ker \partial_i=
\bigcap_{\substack{i \in \{1,\dots,\ell\}\\ i \notin\{\alpha_1,\alpha_1+\alpha_2,\dots,\alpha_1+\cdots+\alpha_k\}}} \im \partial_i,
\end{equation}
which is a subalgebra of $\OPol_\ell$ containing $\OSym_\ell$.
We think of $\OSym_\alpha$ as being the odd analog of the ring of ``partial" invariants $\k[x_1,\dots,x_\ell]^{\S_\alpha}$. For example, we have that
$\OSym_\alpha = \OSym_\ell$ if $\alpha = (\ell)$,
and $\OSym_\alpha = \OPol_\ell$ if $\alpha = (1^\ell)$.
Note also that the
superalgebra anti-involution $*$ of $\OPol_\ell$
leaves $\OSym_\alpha$ invariant, whereas the
involution $\smiley_\ell$ takes $\OSym_\alpha$ to $\OSym_{w_k(\alpha)}$
where $w_k(\alpha) = (\alpha_k,\dots,\alpha_1)$ is the reversed composition.

Consider the following diagram:
\begin{equation}\label{gymnight}
\begin{tikzcd}
\arrow[d,twoheadrightarrow,"\pi_\ell" left]\OSym \arrow[rr,"{\Delta^+_{k}}" above]&& \overbrace{\OSym\otimes\cdots\otimes \OSym}^{k\text{ times}}
\arrow[d,twoheadrightarrow,"\pi_{\alpha_1}\otimes\cdots\otimes \pi_{\alpha_k}" right]\\
\OSym_\ell\arrow[r,hookrightarrow]\arrow[d,hookrightarrow]&\OSym_\alpha
\arrow[r,equals] &\OSym_{\alpha_1}\otimes\cdots\otimes\OSym_{\alpha_k}\arrow[d,hookrightarrow]\\
\OPol_\ell\arrow[rr,equals]&&\OPol_{\alpha_1}\otimes\cdots\otimes \OPol_{\alpha_k}
\end{tikzcd}
\end{equation}
The top horizontal map 
$\Delta^+_{k}$ is the $(k-1)$th iteration of the comultiplication
$\Delta^+:\OSym \rightarrow \OSym \otimes \OSym$.
The bottom equality is the canonical identification
explained just after \cref{opol}, and the outside square commutes thanks
to \cref{tensorconvention}.
In view of \cref{aaaah}, the subalgebra $\OSym_\alpha$ of $\OPol_\ell$ is identified with
the subalgebra $\OSym_{\alpha_1}\otimes\cdots\otimes \OSym_{\alpha_k}$
of $\OPol_{\alpha_1}\otimes\cdots\otimes \OPol_{\alpha_k}$.
This shows that the natural inclusion of
$\OSym_\ell$ into $\OSym_\alpha$ is induced by the comultiplication $\Delta^+$.

Recall
that $w_\alpha$ is the longest element of $\S_\alpha$ and $w^\alpha$ is the longest element of $[\S_\ell / \S_\alpha]_{\min}$, so that $w_\ell = w^\alpha w_\alpha$.
Noting that $w_\alpha = w_{\alpha_1}\sig_{\alpha_1}\big(w_{\beta}\big)$
where $\beta := (\alpha_2,\dots,\alpha_k)$, we recursively define
\begin{align}\label{pudding}
\omega_\alpha &:=  \omega_{\alpha_1}\sig_{\alpha_1}\big(\omega_{\beta}\big) 
\:\:(\:= \pm \tau_{w_\alpha}),&
\chi_\alpha &:= \sig_{\alpha_1}\big(\chi_{\beta}\big)\chi_{\alpha_1}.
\end{align}
We get from \cref{normalization} and induction on $k$ that
\begin{equation}\label{lunch}
\omega_\alpha \cdot \chi_\alpha = 1
\end{equation}
for any $\alpha$.
The following identity is proved in the same way as \cref{skitime}:
\begin{equation}
\omega_\alpha \chi_\alpha \omega_\alpha = \omega_\alpha.
\end{equation}
Similarly to \cref{idempotents}, it follows that the elements
\begin{align}\label{idempotents2}
(\chi\omega)_\alpha &:= \chi_\alpha \omega_\alpha,
&
(\omega\chi)_\alpha &:= \omega_\alpha\chi_\alpha
\end{align}
are primitive idempotents in $\ONH_\alpha$ such that
\begin{align}
(\chi\omega)_\alpha \cdot \OPol_\ell &= \chi_\alpha \OSym_\alpha,&
(\omega\chi)_\alpha \cdot \OPol_\ell &= \OSym_\alpha.
\end{align}
Also let
$\omega^\alpha := \pm \tau_{w^\alpha}$ for the particular sign chosen so that
\begin{equation}\label{omegafact}
\omega_\ell = \omega^\alpha \omega_\alpha
\end{equation}
and let
\begin{equation}\label{chimudef}
\chi^\alpha := \omega_\alpha \cdot \chi_\ell\in \OSym_\alpha.
\end{equation}
We have that
\begin{align}\label{chifact}
\omega^\alpha \cdot \chi^\alpha &= 1,
&\chi_\ell = \chi_\alpha\, \chi^\alpha.
\end{align}
The first of these equalities follows because $\omega^\alpha \cdot \chi^\alpha 
= \omega^\alpha \omega_\alpha \cdot \chi_\ell = \omega_\ell \cdot \chi_\ell = 1$
by \cref{omegafact,normalization}.
To establish the second, one first checks that
$\chi_\alpha \chi^\alpha = \pm \chi_\ell$ for some choice of sign,
and the sign is plus because $\omega_\ell \cdot 
\chi_\alpha \chi^\alpha = \omega^\alpha \omega_\alpha \cdot \chi_\alpha \chi^\alpha
= \omega^\alpha \cdot \chi^\alpha = 1 = \omega_\ell \cdot \chi_\ell$.
Finally, we have that
\begin{equation}\label{cloudy}
\omega_\alpha \chi_\alpha \omega_\ell = \omega_\ell = \omega_\ell \chi_\alpha \omega_\alpha.
\end{equation}
This follows because all three expressions act
in the same way on $\chi_\ell \in \OPol_\ell$ due to \cref{normalization,lunch,chimudef,chifact}.

\begin{theorem}\label{ordered}
For $\alpha \in \Comp(k,\ell)$, the graded superalgebra
$\OSym_\alpha$ is free 
as a right $\OSym_\ell$-supermodule
with basis 
$\big\{p^{(\ell)}_w\:|\:w \in [\S_\ell / \S_\alpha]_{\min}\}$.
Each $p^{(\ell)}_w$ in this basis belongs to
the subalgebra $\OSym_\alpha \cap \OPol_{\ell-\alpha_k}$.
\end{theorem}

\begin{proof}
By \cref{bakeoffOPOL} and \cref{multiidentity}, we have that
\begin{align}
\frac{\gsdim \OSym_\alpha}{\gsdim \OSym_\ell}
&=
\frac{q^{\binom{\ell}{2}}[\ell]^!_{q,\pi}}{
\prod_{i=1}^k q^{\binom{\alpha_i}{2}} [\alpha_i]^!_{q,\pi}}
=
q^{N(\alpha)}
\sqbinom{\ell}{\alpha}_{q,\pi}
= \sum_{w \in [\S_n / \S_\alpha]_{\min}}
(\pi q^2)^{\ell(w)}.\label{maybeuseful}
\end{align}
This is the graded rank of a free graded right
$\OSym_\ell$-supermodule with basis
$\big\{p^{(\ell)}_w\:|\:w \in [\S_\ell / \S_\alpha]_{\min}\}$.
So, to prove the theorem, it just remains to show that 
the elements  $p^{(\ell)}_w \:\left(w \in [\S_\ell / \S_\alpha]_{\min}\right)$
belong to $\OSym_{\alpha} \cap \OPol_{\ell-\alpha_k}$
and are linearly independent over $\OSym_\ell$.
The linear independence is immediate from \cref{firstkeythm}.

To show that $p_w^{(\ell)} \in \OSym_\alpha$, we 
need to show that $\partial_i(p^{(\ell)}_w) = 0$
 for all $i$ such that $s_i \in \S_\alpha$.
 We have that $w^{-1} w_\ell = w_\alpha w'$
for some $w' \in [\S_\alpha \backslash \S_\ell]_{\min}$.
So
$p_w^{(\ell)} = 
\tau_{w^{-1}w_\ell}\cdot \chi_\ell = \pm \omega_\alpha \tau_{w'}
\cdot \chi_\ell$.
Since $\ell(s_i w_\alpha) < \ell(w_\alpha)$ when
$s_i \in \S_\alpha$, the relations in $\ONH_\ell$
now imply that $\tau_i \cdot p_w^{(\ell)} = 0$.

To show that $p^{(\ell)}_w \in \OPol_{\ell-\alpha_k}$,
we again use $w^{-1} w_\ell = w_\alpha w'$
to deduce that
$\tau_{w^{-1}w_\ell} = 
\pm \sig_{\ell-\alpha_k}(\omega_{\alpha_k}) 
\tau_{w''}$ for some $w'' \in \S_\ell$.
By the argument explained in the last sentence of \cref{aswell}, it follows that
$p^{(\ell)}_w$ is a linear combination of terms of the form
$\sig_{\ell-\alpha_k}(\omega_{\alpha_k}) \cdot x^\kappa$
 for $\kappa \in \N^\ell$
with $0 \leq \kappa_i \leq \ell-i$ for all $i$.
It is now clear that $p^{(\ell)}_w \in \OPol_{\ell-\alpha_k}$
since $\sig_{\ell-\alpha_k}(\omega_{\alpha_k})
\cdot x_\ell^{\kappa_\ell} \cdots x_{\ell-\alpha_k+1}^{\kappa_{\ell-\alpha_k+1}}$ is a scalar by \cref{normalization} and
degree considerations.
\end{proof}

\begin{corollary}\label{frontandback}
Suppose that $\alpha \in \Comp(k,\ell)$ for $k \geq 1$.
\begin{enumerate}
\item
The graded superalgebra $\OSym_\alpha$ is a free as a graded right $\OSym_{(\alpha_1,\ell-\alpha_1)}$-supermodule
with basis $\Big\{\sig_{\alpha_1}(p_w^{(\ell-\alpha_1)})\:\Big|\:w \in (\S_{\ell-\alpha_1} / \S_{(\alpha_2,\dots,\alpha_k)})_{\operatorname{min}}\Big\}$.
\item
The graded superalgebra $\OSym_\alpha$ is free as a graded right $\OSym_{(\ell-\alpha_k,\alpha_k)}$-supermodule
with basis $\Big\{\smiley_{\ell-\alpha_k}\big(p_w^{(\ell-\alpha_k)}\big) \:\Big|\:w \in (\S_{\ell-\alpha_k} / \S_{(\alpha_{k-1},\dots,\alpha_{1})})_{\operatorname{min}}\Big\}$.
\end{enumerate}
All vectors in the bases described in (1)--(2) belong to the subalgebra
$\sig_{\alpha_1}\big(\OSym_{(\alpha_2,\dots,\alpha_{k-1})}\big)$.
\end{corollary}

\begin{proof}
(1) This follows immediately from the theorem.

\vspace{1mm}
\noindent
(2) This follows by applying the involution $\smiley_\ell$
to the the result from (1) with $\alpha$ replaced by
the reverse composition $\alpha^\reverse$.
\end{proof}

Continuing with $\alpha \in \Comp(k,\ell)$, we need a few more pieces of notation. For $i=1,\dots,k$, we 
define
\begin{align}
h_r^{(\alpha;i)}&:= \sig_{\alpha_1+\cdots+\alpha_{i-1}}\big(h_r^{(\alpha_i)}\big),&
e_r^{(\alpha;i)}&:= \sig_{\alpha_1+\cdots+\alpha_{i-1}}\big(e_r^{(\alpha_i)}\big).
\end{align}
Under the identication of $\OSym_\alpha$
with $\OSym_{\alpha_1}\otimes\cdots\otimes\OSym_{\alpha_k}$ from \cref{gymnight}, these 
are
$1^{\otimes(i-1)}\otimes h_r^{(\alpha_i)}\otimes 1^{\otimes(k-i)}$
and 
$1^{\otimes(i-1)}\otimes e_r^{(\alpha_i)}\otimes 1^{\otimes(k-i)}$, respectively.
We use similar notation for other elements of $\OSym_\alpha$
such as $e_\lambda^{(\alpha;i)}$, $h_\lambda^{(\alpha;i)}$ and $s_\lambda^{(\alpha;i)}$
for $\lambda \in \Par$. 
From \cref{epoly,hpoly}, we get that
\begin{align}\label{fabric}
e^{(\ell)}_r &= \sum_{\substack{r_1,\dots,r_k \geq 0 \\ r_1+\cdots+r_k = r}}
e^{(\alpha;1)}_{r_1} \cdots e^{(\alpha;k)}_{r_k},
&h^{(\ell)}_r &= \sum_{\substack{r_1,\dots,r_k \geq 0 \\ r_1+\cdots+r_k = r}}
h^{(\alpha;k)}_{r_k} \cdots h^{(\alpha;1)}_{r_1}.
\end{align}
These are more convenient when written in terms of
the generating functions
\begin{align}
e^{(\alpha;i)}(t)&:= 
\sum_{r = 0}^{\alpha_i} (-1)^r e^{(\alpha;i)}_r t^{\alpha_i-r},
&
h^{(\alpha;i)}(t)&:= 
\sum_{r \geq 0} h^{(\alpha;i)}_r t^{-\alpha_i-r}.
\end{align}
Now the identities \cref{fabric} become
\begin{align}\label{pie}
e^{(\ell)}(t)
&:= 
e^{(\alpha;1)}(t) e^{(\alpha;2)}(t) \cdots e^{(\alpha;k)}(t),&h^{(\ell)}(t)
&:= 
h^{(\alpha;k)}(t) \cdots h^{(\alpha;2)}(t) h^{(\alpha;1)}(t).
\end{align}
These identities, which generalize \cref{stupid,littlerock}, together with the infinite Grassmannian relation \cref{newigr}
are useful when moving between different families of generators,
as illustrated by the following lemma.

\begin{lemma}\label{notsoeasypeasy}
Suppose that $\ell = n+n'$ and $r \geq 0$.
\begin{enumerate}
\item We have that
$\displaystyle\sum_{s=0}^r (-1)^s h_{r-s}^{(n)} e_s^{(\ell)} 
= (-1)^r \sig_n\big(e_r^{(n')}\big)$,
which is zero for $r > n'$.
\item We have that
$\displaystyle\sum_{s=0}^r (-1)^s e_s^{(\ell)} \sig_n\big(h_{r-s}^{(n')}\big)
= (-1)^r e_r^{(n)}$,
which is zero for $r > n$.
\end{enumerate}
\end{lemma}

\begin{proof}
(1)
The first identity 
from \cref{pie} when $\alpha = (n,n')$ plus \cref{newigr}
gives that $$
\sig_n\big(e^{(n')}(t)\big) = h^{(n)}(t)e^{(\ell)}(t).
$$
Now
equate the coefficients of $t^{n'-r}$ on both sides.

\vspace{1mm}
\noindent
(2)
Similar.
\end{proof}

\begin{lemma}\label{neweasypeasy}
Suppose that $n \geq 0$. 
The following hold in $\OPol_{n+1}$ for any $m \geq 0$:
\begin{enumerate}
\item
$\displaystyle x_{n+1}^{m+n+1} = -\sum_{q=0}^{n} \left(\sum_{r=0}^m (-1)^{m+n+1-q-r} h_r^{(n+1)} e_{m+n+1-q-r}^{(n+1)}
\right) x_{n+1}^{q}$;
\item
$\displaystyle x_1^{m+n+1} = -\sum_{p=0}^{n} x_1^{p} \left(\sum_{s=0}^m
(-1)^{m+n+1-p-s} e^{(n+1)}_{m+n+1-p-s} h_s^{(n+1)} \right).$
\end{enumerate}
\end{lemma}

\begin{proof}
(1)
Induction on $m$.
The base case $m=0$ follows as $\sum_{q=0}^{n+1} (-1)^{n+1-q} e_{n+1-q}^{(n+1)} x_{n+1}^{q} = 0$
due to \cref{notsoeasypeasy}(2) with $\ell,n,n'$ and $r$ replaced with $n+1,n,1$ and $n+1$,
respectively.
Now take the identity we are trying to prove for some $m \geq 0$
and multiply on the right by $x_{n+1}$ to obtain
\begin{equation}\label{friday}
x_n^{m+1+n+1} = -\sum_{q=0}^{n} \left(\sum_{r=0}^m (-1)^{m+n+1-q-r} h_r^{(n+1)} e_{m+n+1-q-r}^{(n+1)}
\right) x_{n+1}^{q+1}.
\end{equation}
The $q=n$ term of the summation here is
$-\sum_{r=0}^m (-1)^{m+1-r} h_r^{(n+1)} e_{m+1-r}^{(n+1)} x_{n+1}^{n+1}$, which
equals $h_{m+1}^{(n+1)} x_{n+1}^{n+1}$ by the infinite Grassmannian relation 
\cref{oddgrassmannian2}. Using the $m=0$ case of the identity we are proving
this can then be rewritten as
$$
-\sum_{q=0}^{n} (-1)^{n+1-q} h_{m+1}^{(n+1)} e_{n+1-q}^{(n+1)} x_{n+1}^{q}.
$$
For the terms of \cref{friday} with $0 \leq q \leq n-1$, we reindex the summation 
replacing $q$ by $q-1$ to obtain
$$-\sum_{q=1}^{n} \left(\sum_{r=0}^m (-1)^{m+1+n+1-q-r} h_r^{(n+1)} e_{m+1+n+1-q-r}^{(n+1)}
\right) x_{n+1}^{q}.
$$
The expression in brackets is zero for $q=0$ since
$e_{m+1+n+1-r}^{(n+1)} = 0$ for all $0 \leq r \leq m$, so we can sum instead from $q=0$ to $n$.
Thus, we have shown that
\begin{align*}
x_{n+1}^{m+1+n+1}
&=
-\sum_{q=0}^{n} (-1)^{n+1-q} h_{m+1}^{(n+1)} e_{n+1-q}^{(n+1)} x_{n+1}^{q}
-\sum_{q=0}^{n} \left(\sum_{r=0}^m (-1)^{m+1+n+1-q-r} h_r^{(n+1)} e_{m+1+n+1-q-r}^{(n+1)}
\right) x_{n+1}^{q}\\
&=
-\sum_{q=0}^{n} 
\left((-1)^{n+1-q} h_{m+1}^{(n+1)} e_{n+1-q}^{(n+1)} +\sum_{r=0}^m (-1)^{m+1+n+1-q-r} h_r^{(n+1)} e_{m+1+n+1-q-r}^{(n+1)}
\right) x_{n+1}^{q}\\&=
-\sum_{q=1}^{n+1} 
\left(\sum_{r=0}^{m+1} (-1)^{m+1+n+1-q-r} h_r^{(n+1)} e_{m+1+n+1-q-r}^{(n+1)}
\right) x_{n+1}^{q}.
\end{align*}
This is just what is needed for the induction step.

\vspace{1mm}
\noindent
(2)
This is similar, or may be proved by applying $\smiley_{n+1}\circ *$ to (1).
\end{proof}

Now we focus on the most important case $k=2$, so 
$\alpha = (n,n') \in \Comp(2,\ell)$ for some $n,n' \geq 0$.
Then $\chi_\alpha = \chi_{(n,n')} = \sig_n(\chi_{n'}) \chi_n$.

\begin{theorem}\label{frobenius}
Suppose that $\ell=n+n'$ for $n,n' \geq 0$.
Then $\OSym_{(n,n')}$ has the following two bases as a free graded 
right $\OSym_\ell$-supermodule:
\begin{enumerate}
\item $\left\{s^{(n)}_{\lambda}\:\big|\:\lambda \in \GPar{n}{n'}\right\}$;
\item $\left\{\sig_n\big(\sigma^{(n')}_\mu\big)\:\big|\:\mu\in\GPar{n'}{n}\right\}$.
\end{enumerate}
Also let $\Tr:\OSym_{(n,n')} \rightarrow \OSym_\ell$ be the linear map
$a\mapsto \omega_\ell\: \chi_{(n,n')} \cdot a$.
This map is a homogeneous homomorphism of graded right $\OSym_\ell$-supermodules
of degree $-2nn'$ and parity $nn'\pmod{2}$, 
and the bases (1)--(2) satisfy
\begin{equation}\label{mainwork}
\Tr\left (s^{(n)}_{\lambda} \sig_n\big(\sigma^{(n')}_\mu\big)\right)
= \left\{
\begin{array}{ll}
\sign(\mu)&\text{if $\mu^\transpose_i = n'-\lambda_{n+1-i}\text{ for }i=1,\dots,n$}\\
0&\text{otherwise}
\end{array}\right.
\end{equation}
for $\sign(\mu) \in \{\pm 1\}$ with $\sign(\varnothing)=1$; see \cref{LR} below for a formula for $\sign(\mu)$ for general $\mu$.
\end{theorem}

\begin{proof}
The main work here is to prove \cref{mainwork}.
This turns out to be significantly harder than the analogous 
formula in the ordinary even theory;
see \cite[Rem.~4.12]{EKL} for an illuminating example.
Fortunately, the details are already worked out in 
\cite[Prop.~4.11]{EKL} up to an undetermined 
sign since our conventions are different.
To keep track of this sign, we repeat 
the first few steps of the proof in \cite{EKL} in our set up.
By \cref{schurtheorem,smileysign,smileyanddot,littlebears,cloudy}, we have that
\begin{align}\notag
\Tr\left (s^{(n)}_{\lambda} \sig_n\big(\sigma^{(n')}_\mu\big)\right)
&= 
\omega_\ell\; \chi_{(n,n')}\cdot s^{(n)}_{\lambda} \sig_n\Big(\smiley_{n'}\big(s^{(n')}_\mu\big)\Big)\\\notag
&=
\omega_\ell\;\chi_{(n,n')}
\cdot \big(\omega_n \chi_n \cdot x^\lambda\big)\;
\sig_n\big(\smiley_{n'}(\omega_{n'} \chi_{n'} \cdot x^\mu)\big)\\\notag
&=
 \zeta_{n'} \omega_\ell\;\chi_{(n,n')} \omega_{(n,n')}
\cdot \;
\sig_n\big(\smiley_{n'}(\chi_{n'})\big)
\chi_n x^\lambda
\sig_n\big(\smiley_{n'}(x^\mu)\big)\\\label{scary}
&=\zeta_{n'} \omega_\ell
 \cdot
\sig_n\big(\smiley_{n'}(\chi_{n'})\big) \chi_n x^\lambda \sig_n\big(\smiley_{n'}(x^\mu)\big).
\end{align}
Up to another sign, the monomial appearing after the $\cdot$ in \cref{scary} is as
considered in \cite[Lem.~4.9]{EKL}, so applying that lemma gives
that $\Tr\left (s^{(n)}_{\lambda} \sig_n\big(\sigma^{({n'})}_\mu\big)\right)$
is $\pm 1$ if $\mu^\transpose_i = n'-\lambda_{n+1-i}\text{ for }i=1,\dots,n$,
and it is zero otherwise.
It remains to check that \cref{scary} equals $+1$ in the special case that
$\lambda = ({n'}^n)$ and 
$\mu = \varnothing$.
To see this, one first checks that
$\chi_n x_1^{n'} \cdots x_n^{n'} = x_n^{n'} x_{n-1}^{n'+1}\cdots x_1^{n+n'-1}$. 
Hence,
letting $\omega_\ell = \tau \sig_n(\omega_{n'})$ for 
$\tau \in \ONH_\ell$,
\cref{scary} simplifies in this case to give
\begin{align*}
\Tr\left (s^{(n)}_{({n'}^n)}\right)
&= \zeta_{n'} \tau \sig_n(\omega_{n'}) \cdot
\sig_n\big(\smiley_{n'}(\chi_{n'})\big) x_n^{n'} x_{n-1}^{n'+1}\cdots x_1^{n+n'-1}
= \tau
\sig_n\big(\smiley_{n'}(\omega_{n'} \cdot \chi_{n'})\big)
\cdot x_n^{n'} x_{n-1}^{n'+1}\cdots x_1^{n'+n-1}\\
&= \tau
\cdot x_n^{n'} x_{n-1}^{n'+1}\cdots x_1^{n+n'-1}
=
\tau
\sig_n(\omega_{n'}) \cdot \sig_n(\chi_{n'})
x_n^{n'} x_{n-1}^{n'+1}\cdots x_1^{n+n'-1}
= \omega_\ell \cdot \chi_\ell = 1.
\end{align*}
Now \cref{mainwork} is proved.

It is clear from the definition that $\Tr$ is a homogeneous
homomorphism of graded right $\OSym_\ell$-supermodules
of degree $-2nn'$ and parity $nn'\pmod{2}$. 
It remains to show that the elements (1) and (2) are bases.
To see that the elements (1) are linearly independent over $\OSym_\ell$,
take a linear relation $\sum_{\lambda} s_\lambda^{(n)} a_\lambda$ for
$a_\lambda \in \OSym_\ell$. To see that $a_\lambda = 0$ for any given $\lambda$,
let $\mu\in\Par$ be defined so that 
$\mu^\transpose_i = n'-\lambda_{n+1-i}$ for $i=1,\dots,n$
and $\mu^\transpose_i = 0$ for $i > n$.
Then we have using \cref{mainwork} that 
$$
0=\Tr\bigg( \sig_n\big(\sigma_\mu^{(n')}\big)\sum_{\lambda'}s_{\lambda'}^{(n)} a_{\lambda'}
\bigg)=
\sum_{\lambda'}(-1)^{|\lambda'||\mu|}\Tr\Big(s_{\lambda'}^{(n)} \sig_n\big(\sigma_\mu^{(n')}\big)\Big) a_{\lambda'}=
(-1)^{|\lambda||\mu|} \sign(\mu) a_\lambda.
$$
This establishes the linear independence.
As $s_\lambda^{(n)}$ is of degree $2|\lambda|$ and parity $|\lambda|\pmod{2}$,
we deduce from \cref{qbinomialformula} that the elements (1) generate
a free graded right $\OSym_\ell$-supermodule of graded rank
$q^{nn'} \sqbinom{\ell}{n}_{q,\pi}$.
In view of \cref{ordered}, it follows that this submodule is
all of $\OSym_{(n,n')}$.
This proves that (1) is a basis. 
A similar argument gives that (2) is a basis too.
\end{proof}

\begin{corollary}\label{frontandback2}
Suppose that $\ell=n+d+n'$ for $n,d,n' \geq 0$.
Then
$\OSym_{(n,d,n')}$ is a free right $\OSym_{(n,n'+d)}$-supermodule
with basis
$\Big\{\sig_n\big(s_\lambda^{(d)}\big)\:\Big|\:
\lambda \in \GPar{d}{n'}\Big\}$,
and it is a free right $\OSym_{(n+d,n')}$-supermodule
with basis $\Big\{\sig_n\big(\sigma_\lambda^{(d)}\big)\:\Big|\:
\lambda \in \GPar{d}{n}\Big\}$.
\end{corollary}

\section{The odd analog of cohomology of Grassmannians}

Continue with $\ell=n+n'$. 
In the purely even theory, when all of the algebras involved are commutative,
the analog of the map $\Tr$ from \cref{frobenius} is actually a graded bimodule
homomorphism, so that it gives a trace making $\Sym_{(n,n')}$
into a graded Frobenius algebra over $\Sym_{\ell}$. 
However, in the odd case, $\OSym_{(n,n')}$ is usually 
{\em not} a Frobenius extension of $\OSym_{\ell}$, 
e.g., it is already false in the case $n=n'=1$ since 
one can check directly that the subalgebra
$\OSym_2$ of $\OPol_2$ has no complement as a graded $(\OSym_2,\OSym_2)$-superbimodule.
This is a significant obstruction to the development of the odd theory.
At this point in \cite[Sec.~5]{EKL}, the obstruction is avoided by passing 
to the finite-dimensional graded superalgebra
\begin{equation}\label{turn}
\ROH_n^\ell := \OSym_n \:\big/\: \big\langle h_r^{(n)}\:\big|\:r > n' \big\rangle
\:\,\left(\:= \OSym \:\big/\:\big\langle h_r, e_s\:\big|\:r > n', s > n\big\rangle\;\right).
\end{equation}
This is called the {\em odd Grassmannian cohomology algebra}
since it is an odd analog of the cohomology algebra $H^*\big(\operatorname{Gr}_n^\ell;\k\big)$ 
of the Grassmannian $\operatorname{Gr}_n^\ell$ of $n$-dimensional subspaces of $\C^\ell$.

We denote the image of $a \in \OSym_n$ in $\ROH_n^\ell$ by $\bar a$.
The first part of following theorem is
\cite[Prop.~5.4]{EKL}, but we give a different argument which gives
extra information.

\begin{theorem}\label{OH}
Suppose that $\ell = n+n'$.
The odd Schur polynomials
$\bar s_\lambda^{(n)}$ for
$\lambda \in \GPar{n}{n'}$
give a linear basis for $\ROH_n^\ell\,$,
and all other $\bar s_\lambda^{(n)}$ are zero. 
Moreover, viewing $\k$ as
a graded $\OSym_\ell$-supermodule in the obvious way,
there is a commuting diagram
\begin{equation}\label{OHCD}
\begin{tikzcd}
\ROH^\ell_n\arrow[rrr,"\bar a \mapsto a \otimes 1"  above]
\arrow[d,"\overline{\psi}_n^\ell" left]
&&&\OSym_{(n,n')} \otimes_{\OSym_\ell}\k\\
\ROH^\ell_{n'}\arrow[rrr,"\bar a \mapsto (-1)^{n\parity(a)} 1 \otimes \sig_n(a)" below]
&&&
\k\otimes_{\OSym_\ell} \OSym_{(n,n')}\arrow[u,"1\otimes a\mapsto a^*\otimes 1" right]
\end{tikzcd}
\end{equation}
of isomorphisms in which
\begin{enumerate}
\item
the top map is an even degree 0 isomorphism
of graded left $\OSym_n$-supermodules;
\item
the bottom map is an even degree 0
isomorphism
of graded right $\OSym_{n'}$-supermodules
for the action 
on 
$\k\otimes_{\OSym_\ell} \OSym_{(n,n')}$
defined by restriction along $\sig_n\circ\operatorname{p}^n:\OSym_{n'}
\rightarrow \OSym_{(n,n')}$;
\item
the right hand map is an even linear isomorphism of degree 0;
\item
the left hand map $\overline{\psi}_n^\ell$ is the graded superalgebra isomorphism induced
by the involution $\psi\circ\operatorname{p}^n$ of $\OSym$.
\end{enumerate}
\end{theorem}

\begin{remark}
The inclusion of the parity involution $\operatorname{p}^n$
in the definition of the left and bottom maps in \cref{OHCD} is 
hard to justify at this point---it could simply be omitted in both places
and the simplified result is also true. The signs have been included for consistency with \cref{EOH} below, in
which they are essential.
\end{remark}

\begin{proof}[Proof of \cref{OH}]
(1)
To construct the top map so that it is a homomorphism
of graded left $\OSym_n$-supermodules, we must show that
$\OSym_{(n,n')}\otimes_{\OSym_\ell} \k$ can be made into a 
graded left $\ROH^\ell_n$-supermodule
so that $\bar a \cdot (b \otimes 1) = ab \otimes 1$ for all $a \in \OSym_n,
b \in \OSym_{(n,n')}$.
Since it is already a graded left $\OSym_n$-supermodule,
and $\ROH^\ell_n$ is the quotient of $\OSym_n$
by the relations $h_r^{(n)} = 0$ for $r > n'$,
it suffices to check that
$h_r^{(n)}$ acts as zero on $\OSym_{(n,n')}\otimes_{\OSym_\ell} \k$
for all $r > n'$.
By \cref{frobenius}(2), any homogeneous element of
$\OSym_{(n,n')} \otimes_{\OSym_\ell} \k$ can be written as
$\sig_n(b) \otimes 1$ for $b \in \OSym_{n'}$.
Now we must show that
$h_r^{(n)} \sig_n(b) \otimes 1 = 0$ for $r > n'$.
This follows from the calculation
\begin{align*}
h_r^{(n)} \sig_n(b) \otimes 1 &=
(-1)^{r\parity(b)} \sig_n(b) h_r^{(n)} \otimes 1
= (-1)^{r \parity(b)} \sig_n(b) \sum_{s=0}^r (-1)^{s} h_{r-s}^{(n)}
e_{s}^{(\ell)}
\otimes 1\\
&= (-1)^{r \parity(b)} \sig_n(b) \sig_n\Big((-1)^r e_r^{(n')}\Big)
\otimes 1 = 0.
\end{align*}
The second equality here is just the observation that
$e^{(\ell)}_s \otimes 1$ is zero in $\OSym_{(n,n')}\otimes_{\OSym_\ell}\k$ for $s > 0$.
The penultimate equality
follows from \cref{notsoeasypeasy}(1).

Now consider $\bar s^{(n)}_\lambda \in \ROH^\ell_n$.
If $\h(\lambda) > n$, we already know this is zero by \cref{alreadybasis}.
We have by \cref{kostka} that $\bar s^{(n)}_\lambda = 
\bar h^{(n)}_\lambda + ($a linear combination of $\bar h^{(n)}_\mu$
for $\mu > \lambda$).
We deduce that $\bar s^{(n)}_\lambda = 0$ if $\lambda_1 > n'$
since $\bar h^{(n)}_\lambda$ and all of the $\bar h^{(n)}_\mu$ appearing in this expansion are zero by the defining relations of $\ROH^\ell_n$.
This shows that $\ROH^\ell_n$ is spanned by the elements
$\bar s^{(n)}_\lambda\:\big(\lambda \in \GPar{n}{n'}\big)$.
To see that these elements are linearly independent, hence, a basis
for $\ROH^\ell_n$, we act on the vector $1\otimes 1 \in\OSym_{(n,n')}\otimes_{\OSym_\ell} \k$ to obtain
the vectors
$s^{(n)}_\lambda\otimes 1\:(\lambda \in \GPar{n}{n'})$
which constitute a basis for
$\OSym_{(n,n')}\otimes_{\OSym_\ell} \k$ by \cref{frobenius}(1).
This argument also shows that the map (1) is an isomorphism.

\vspace{1mm}
\noindent
(2)
Similarly,
to construct the bottom map, we must make
$\k\otimes_{\OSym_\ell} \OSym_{(n,n')}$ into a graded right
$\ROH_{n'}^\ell$-supermodule
so that $(1 \otimes b) \cdot \bar a = (-1)^{n\parity(a)} 1 \otimes b\sig_n(a)$
for $b \in \OSym_{(n,n')}$ and $a \in \OSym_{n'}$.
To do this,
one first applies $*$ to \cref{frobenius} to deduce that $\k \otimes_{\OSym_\ell} \OSym_{(n,n')}$ is
spanned by vectors of the form $1 \otimes a$ for $a \in \OSym_n$.
This plus \cref{notsoeasypeasy}(2)
are then used
establish the well-definedness of the action. 
The fact that the bottom map 
is an isomorphism could be deduced using \cref{frobenius}(2) like in the previous paragraph, 
but it also follows 
once we have checked the commutativity of the diagram using that the other three maps (1), (3) and (4) are all isomorphisms.

\vspace{1mm}
\noindent
(3)
To obtain the map (3), we start with 
the isomorphism $\OSym_{(n,n')} \stackrel{\sim}{\rightarrow} \OSym_{(n,n')},
a \mapsto a^*$ where $*$ here is the restriction of the 
superalgebra anti-involution $*:\OPol_\ell \rightarrow \OPol_\ell$.
Since we have that $(ab)^* = (-1)^{\parity(a)\parity(b)}b^* a^*$ for any
$b \in \OSym_{(n,n')}$ and $a \in \OSym_\ell$
with $a^* \in \OSym_\ell$ again,
this induces the desired isomorphism $\k \otimes_{\OSym_\ell} \OSym_{(n,n')} \stackrel{\sim}{\rightarrow}
\OSym_{(n,n')}\otimes_{\OSym_\ell}\k$.

\vspace{1mm}
\noindent
(4)
By definition, $\ROH^\ell_n$ is the quotient of $\OSym$
by the two-sided ideal generated by
$\{e_r\:|\:r > n\}\cup\{h_r\:|\:r > n'\}$
and $\ROH^\ell_{n'}$ is the quotient of $\OSym$
by the two-sided ideal generated by
$\{h_r\:|\:r > n\}\cup\{e_r\:|\:r > n'\}$.
The involution $\psi\circ \operatorname{p}^n$ 
interchanges these two ideals
so it factors through the quotients to induce an isomorphism $\overline{\psi}_n^\ell:\ROH^\ell_n\stackrel{\sim}{\rightarrow}
\ROH^\ell_{n'}$. This gives the graded superalgebra isomorphism (4).

\vspace{1mm}
To complete the proof, it just remains to show that the diagram commutes.
Consider $\bar h_{r_1}^{(n)} \cdots \bar h_{r_k}^{(n)} \in \ROH^\ell_n$
for $r_1,\dots, r_k > 0$ and $k \geq 0$.
The map (1) takes it to
$h_{r_1}^{(n)} \cdots h_{r_k}^{(n)} \otimes 1$.
As in the opening paragraph of the proof, we have that
$$
f h_{r}^{(n)}\otimes 1 = (-1)^{r} f \sig_n\big(e_{r}^{(n')}\big) \otimes 1= (-1)^{r+\parity(f)r} \sig_n\big(e_{r}^{(n')}\big) f \otimes 1
$$ 
for any $r$ and $f \in \OSym_n$.
By induction, it follows that 
$$
h_{r_1}^{(n)} \cdots h_{r_k}^{(n)} \otimes 1
= (-1)^{r_1+\cdots+r_k+\sum_{i < j} r_i r_j} 
\sig_n\Big(e^{(n')}_{r_k} \cdots e^{(n')}_{r_1}\Big)\otimes 1
= 
(-1)^{r_1+\cdots+r_k}\sig_n\big(e^{(n')}_{r_1}\cdots e^{(n')}_{r_k}\big)^*\otimes 1.
$$
This is the same as the image of $\bar h_{r_1}^{(n)} \cdots \bar h_{r_k}^{(n)}$ going around the other three sides of the square.
\end{proof}

\begin{corollary}\label{zalatorispre}
For $\ell=n+n'$,
$\ROH_n^\ell$ is of graded superdimension
$q^{n n'} \sqbinom{\ell}{n}_{q,\pi}$.
\end{corollary}

\begin{proof}
This follows from the basis described in \cref{OH} plus \cref{qbinomialformula}.
\end{proof}

\begin{corollary}\label{hard}
In $\OSym_{(n,n')}\otimes_{\OSym_\ell} \k$,
we have that 
$\sig_n\big(\sigma_\mu^{(n')}\big)\otimes 1=
(-1)^{\NE(\mu)}  s_{\mu^\transpose}^{(n)} \otimes 1$
for every $\mu \in \Par$.
\end{corollary}

\begin{proof}
Note that $(-1)^{\NE(\mu)} 
s_{\mu^\transpose}^{(n)} \otimes 1$ is the image of
$(-1)^{\NE(\mu)} \bar s_{\mu^\transpose}^{(n)}$ under the top map in the commuting square \cref{OHCD}. Now we compute the image of $(-1)^{\NE(\mu)}
\bar s_{\mu^\transpose}^{(n)}$ around the other three edges of this square. 
Using that
$\NE(\mu) = |\mu|+dN(\mu)+dE(\mu)+NE(\mu)$,
it maps first to 
$(-1)^{dN(\mu)+dE(\mu)+n|\mu|}\bar s_{\mu}^{(n')}$
thanks to \cref{niall}, then to
$(-1)^{dN(\mu)+dE(\mu)} 1 \otimes \sig_n\big( s_\mu^{(n')}\big)$,
then to
$\sig_n\big(\sigma_\mu^{(n')}\big) \otimes 1$
thanks to \cref{littlebears}.
\end{proof}

\begin{corollary}\label{trace}
For $\ell = n+n'$,
there is a unique (up to scalars)
trace map 
$\overline{\tr}:\ROH_n^\ell\rightarrow \k$ making
$\ROH^\ell_n$ into a graded Frobenius superalgebra over $\k$ of degree
 $2nn'$ and parity $nn'\pmod{2}$.
Moreover, normalizing $\overline{\tr}$ so that 
$\overline{\tr}\big(\bar s_{({n'}^n)}^{(n)}\big)= 1$
and recalling the definition of $\Tr$ from \cref{frobenius},
we have that 
$\Tr(a) \otimes 1=1\otimes \overline{\tr}(\bar a)$
in $\OSym_{(n,n')} \otimes_{\OSym_\ell} \k$
for any $a \in \OSym_n$.
\end{corollary}

\begin{proof}
If it exists, the trace map is unique up to a non-zero scalar; cf. the discussion after \cref{traceprops2}.
Now we {\em define} $\overline{\tr}:\ROH_n^\ell\rightarrow \k$ so that
$\Tr(a) \otimes 1=1\otimes \overline{\tr}(\bar a)$ and check that is a trace 
map sending $\bar s_{({n'}^n)}^{(n)}$ to $1$. The latter statement follows because
we know in \cref{mainwork} that $\sign(\varnothing) = 1$.
To show that $\overline{\tr}$ is a trace map, 
we need to show that there exist linear bases
$b_1,\dots,b_m$ and $b_1^\vee,\dots,b_m^\vee$ for $\ROH_n^\ell$
satisfying \cref{traceprops1,newcond}.
We take $b^\vee_1,\dots,b^\vee_m$
and $b_1,\dots,b_m$ to be the elements
$\bar s_\lambda^{(n)}$ and
$(-1)^{\NE(\mu)}\sign(\mu) \bar s_{\mu^\transpose}^{(n)}$
for $\lambda\in \GPar{n}{n'}$ and $\mu \in \GPar{n'}{n}$, 
respectively,
enumerated  
so that 
$b^\vee_r = \bar s_\lambda^{(n)}$
and $b_r = (-1)^{\NE(\mu)}\sign(\mu) \bar s_{\mu^\transpose}^{(n)}$
if and only if $\mu^\transpose_i = n'-\lambda_{n+1-i}$ for each $i=1,\dots,n$.
The properties \cref{traceprops1} obviously hold,
and we have a pair of dual bases as in
\cref{newcond} thanks to \cref{hard,frobenius}.
\end{proof}

\begin{corollary}\label{LR}
Let $\lambda, \mu \in \Par$ be related as in the first case of \cref{mainwork}.
Then $\sign(\mu) = (-1)^{\NE(\mu)} LR_{\lambda,\mu^\transpose}^{\nu}$
where $\nu := ({n'}^n)$ and
 $LR_{\lambda,\mu^\transpose}^\nu$ denotes the odd Littlewood-Richardson coefficient,
 that is, the coefficient of $s_\nu$ when
$s_\lambda s_{\mu^\transpose}$ is expanded in terms of odd Schur functions.
\end{corollary}

\begin{proof}
By the previous two corollaries and the definition \cref{mainwork},
we have that 
$$
\sign(\mu) = (-1)^{\NE(\mu)}
\overline{\tr}\Big(\bar s_\lambda^{(n)} \bar s^{(n)}_{\mu^\transpose}\Big) = (-1)^{\NE(\mu)} LR_{\lambda,\mu^\transpose}^\nu.
$$
\end{proof}

Some very special odd Littlewood-Richardson coefficients arise in the odd analog of the {\em Pieri formula}
proved in \cite[(2.72)]{EKL}:
\begin{equation}\label{rowpieri}
s_\lambda h_r = \sum_{\mu} (-1)^{NE(\lambda)+NE(\mu)+S(\lambda,\mu)} s_{\mu}.
\end{equation}
The sum here is over all partitions $\mu$ whose Young diagram is obtained by adding one box to the bottom of $r$ 
different columns of the Young diagram of $\lambda$, and
$S(\lambda,\mu):=\sum_{1 \leq j \leq r} \sum_{k = i_j+1}^{\lambda_1} \lambda_k^\transpose$
assuming these columns are indexed by $i_1 < \cdots < i_r$.
The ghastly signs appearing in \cref{rowpieri} and in the next lemma fortunately play no significant role.

\begin{lemma}\label{thepierineededlater}
The inclusion $\OSym_n \hookrightarrow \OSym_{(n-1,1)}$
maps $$
s^{(n)}_\mu \mapsto
\sum_{r, \lambda}
(-1)^{NE(\lambda)+dE(\lambda)+NE(\mu)+
dE(\mu)+S(\lambda,\mu)+\binom{r}{2}} 
s^{(n-1)}_\lambda x_n^r
$$
where the sum is over all $r \geq 0$
and partitions $\lambda$ whose Young diagram is obtained by removing one box from the bottom of
$r$ different columns of the Young diagram of $\mu$,
including all boxes from its $n$th row since 
$s^{(n-1)}_\lambda = 0$ if $\lambda_n > 0$.
\end{lemma}

\begin{proof}
From \cref{gymnight}, it follows that 
the coefficient of $s^{(n-1)}_\lambda x_n^r$
when $s^{(n)}_\mu$ is expanded in terms of the Schur basis
for $\OSym_{(n-1,1)}$ 
is equal to the $s_\lambda \otimes h_r$-coefficient of
$\Delta^+(s_\mu)$.
Using \cref{schurformalt,hes}, this is $$
(-1)^{dE(\lambda)+\binom{r}{2}} (s_\lambda \otimes h_r, \Delta^+(s_\mu))^+
\stackrel{\cref{charit2}}{=}
(-1)^{dE(\lambda)+\binom{r}{2}} (s_\lambda h_r, s_\mu)^+.
$$
Now use \cref{rowpieri} plus \cref{schurformalt} once again.
\end{proof}

\begin{remark}\label{discussionhere}
A general formula for odd Littlewood-Richardson coefficients 
is derived in \cite[Th.~4.8]{E}, showing that they can be computed
by counting the same set of
semi-standard skew tableaux that appear in the ordinary Littlewood-Richardson rule, but counting each one with a sign $\pm$.
A useful consequence of this is that if an ordinary even Littlewood-Richardson coefficient is zero then so is the corresponding odd Littlewood-Richardson coefficient. This, together with the odd Pieri rule, is all that we actually use below.
\end{remark}

\section{Equivariant odd Grassmannian cohomology algebras}

Recall from \cref{dumbc3} that 
$R_\ell$ denotes
the largest supercommutative quotient of $\OSym_\ell$.
We are now going to work over this as our base ring.
We use the notation $\dot c$ to denote the image of $c \in \OSym_\ell$
in $R_\ell$.
Note that $\OSym_n \otimes R_\ell$ is a graded $R_\ell$-superalgebra
with structure map $\eta:R_\ell \rightarrow \OSym_n\otimes R_\ell,
\dot c \mapsto 1 \otimes \dot c$.

\begin{definition}\label{minuets}
For $\ell = n+n'$, the
{\em equivariant odd Grassmannian cohomology algebra}
 is
the graded $R_\ell$-superalgebra
\begin{align}\label{blue}
\EOH_n^\ell &:= 
\OSym_{n} \otimes R_\ell
 \:\Bigg/\: \Bigg\langle \sum_{s=0}^r (-1)^s h_{r-s}^{(n)} \otimes
 \dot e^{(\ell)}_{s}\:\Bigg|\:r > n'\Bigg\rangle.
\end{align}
For $a \in \OSym_n$
and $c \in \OSym_\ell$, we denote the canonical image 
of $a\otimes \dot c \in \OSym_n \otimes R_\ell$ 
in the quotient $\EOH_n^\ell$ by $a \hatotimes \dot c$.\end{definition}

As the name suggests, this is
an odd analog of the $GL_\ell(\C)$-equivariant cohomology algebra
of the Grassmannian of $n$-dimensional subspaces of $\C^\ell$.
Given any graded 
supercommutative $R_\ell$-superalgebra
$A$, one can specialize to obtain
the graded $A$-superalgebra $\EOH_n^\ell \otimes_{R_\ell} A$.
In particular, the ordinary
odd Grassmannian cohomology algebra $\ROH_n^\ell$ from \cref{turn} is naturally identified with the specialization $\EOH_n^\ell \otimes_{R_\ell} \k$.

\begin{example}
We have that
$\EOH_1^2 \cong 
\OPol_2 \,\big/\, \langle x_1^2+x_2^2,  x_1^3+x_1^2 x_2 \rangle$
via the isomorphism $h_r^{(1)}\hatotimes 1\mapsto x_1^r$,
$1 \hatotimes \dot e_1^{(2)}\mapsto x_1+x_2$,
$1\hatotimes \dot e_2^{(2)}\mapsto x_1x_2$.
A linear basis 
is given by the elements $\{x_1^r, x_2, x_1 x_2\:|\:r \geq 0\}$.
\end{example}

\begin{lemma}\label{alphamap}
For $\ell = n+n'$, there is a surjective graded superalgebra homomorphism
$\alpha_n^\ell:\EOH_n^\ell \twoheadrightarrow R_{n'}$ taking
$a \hatotimes 1$ to zero for $a \in \OSym_n$ of positive degree
and $1 \hatotimes \dot c^{(\ell)}$ to $\dot c^{(n')}$ for $c \in \OSym$.
\end{lemma}

\begin{proof}
This is clear from the nature of the defining relations \cref{blue}.
\end{proof}

\begin{lemma}\label{face}
The two-sided ideal
$\Big\langle \sum_{s=0}^r (-1)^s h_{r-s}^{(n)} \otimes
\dot e^{(\ell)}_{s}\:\Big|\:r > n'\Big\rangle$
of $\OSym_n \otimes R_\ell$ contains
the elements $\sum_{s=0}^r h_{r-s}^{(n)} \otimes \dot e^{(\ell)}_s$
for all $r > n'+1$.
Also the two-sided ideal
$\Big\langle \sum_{s=0}^r (-1)^{(r-n')s} h_{r-s}^{(n)} \otimes
\dot e^{(\ell)}_{s}\:\Big|\:r > n'\Big\rangle$
contains
the elements $\sum_{s=0}^r (-1)^{(r-n')s+s}h_{r-s}^{(n)} \otimes \dot e^{(\ell)}_s$
for all $r > n'+1$.
Hence, these two ideals are equal.
\end{lemma}

\begin{proof}
For the proof, we denote these two-sided ideals by $I$ and $J$, respectively.
By \cref{straighteningruleh}, we have that $h_1 h_r = h_r h_1$
if $r$ is odd and $h_1 h_r = - h_r h_1 + 2 h_{r+1}$ if $r$ is even.
Suppose that $r > n'+1$.
The first ideal $I$ contains
$a := \sum_{s=0}^{r-1} (-1)^s h_{r-1-s}^{(n)} \otimes \dot e^{(\ell)}_s$.
Hence, it contains 
$$
(-1)^r (h_1 \otimes 1) a - a (h_1 \otimes 1) = 
\sum_{s=0}^{r} (1-(-1)^{r-s}) h_{r-s}^{(n)} \otimes \dot e^{(\ell)}_s.
$$
Adding this to $\sum_{s=0}^r (-1)^{r-s} h_{r-s}^{(n)} \otimes \dot e^{(\ell)}_s$, which is also in $I$, gives the claimed elements
for the first assertion of the lemma.
The assertion about the second ideal $J$ is proved similarly.
To deduce that $I=J$, the facts established so far show that both are generated by the lowest degree generator
$\sum_{s=0}^{n'+1} (-1)^s h_{n'+1-s}^{(n)} \otimes \dot e^{(\ell)}_s$
together with the higher degree generators
$\sum_{s=0}^r (-1)^s h_{r-s}^{(n)} \otimes \dot e^{(\ell)}_s$
and 
$\sum_{s=0}^r h_{r-s}^{(n)} \otimes \dot e^{(\ell)}_s$
for all $r > n'+1$.
\end{proof}

For any $\alpha \in \Comp(k,\ell)$,
$\OSym_\alpha\otimes_{\OSym_\ell} R_\ell$
is a graded left $\OSym_\alpha$-supermodule, and it is a graded 
$R_\ell$-supermodule for the left action that is induced by the natural right action. Thus,
for $a \in \OSym_\alpha$ and 
$c \in \OSym_\ell$, we have that 
\begin{equation}\label{likehere}
ac \otimes 1  = a \otimes \dot c = (a \otimes 1) \cdot \dot c
= (-1)^{\parity(a)\parity(c)} \dot c \cdot (a \otimes 1).
\end{equation}
However this is in general {\em not} equal to $ca \otimes 1$.
Note also that $\OSym_\alpha\otimes_{\OSym_\ell} R_\ell$
is not itself a graded superalgebra in any apparent way.
Indeed, $R_\ell$ is the quotient of $\OSym_\ell$ by
the two-sided ideal $I_\ell$ generated by $\big(o^{(\ell)}\big)^2$
and $\big[o^{(\ell)},e_2^{(\ell)}\big]$,
so $$
\OSym_\alpha \otimes_{\OSym_\ell} R_\ell
=\OSym_\alpha\otimes_{\OSym_\ell} \OSym_\ell / I_\ell 
\simeq \OSym_\alpha / \OSym_\alpha I_\ell.
$$
However, in general, 
$\OSym_\alpha I_\ell$ is merely a left ideal, not
a two-sided ideal of $\OSym_\alpha$.
Similar remarks apply to $R_\ell\otimes_{\OSym_\ell} \OSym_\alpha$,
which is a graded right $\OSym_\alpha$-supermodule, and a graded
$R_\ell$-supermodule for the right action that
is induced by the natural left action.

\begin{theorem}\label{EOH}
For $\ell = n+n'$,
$\EOH_n^\ell$ is free as a graded $R_\ell$-supermodule with basis
given by the odd Schur polynomials
$s_\lambda^{(n)}\hatotimes 1$ for
$\lambda \in \GPar{n}{n'}$.
Moreover, 
there is a commuting diagram
\begin{equation}\label{EOHCD}
\begin{tikzcd}
\EOH^\ell_n\arrow[rrrr,"a\hatotimes \dot c \mapsto a \otimes \dot c"  above]
\arrow[d,"\psi_n^\ell" left]
&&&&\OSym_{(n,n')} \otimes_{\OSym_\ell}R_\ell\\
\EOH^\ell_{n'}\arrow[rrrr,"a \hatotimes \dot c \mapsto (-1)^{\parity(a)(n+\parity(c))} \dot c \otimes \sig_n(a)" below]
&&&&
R_\ell \otimes_{\OSym_\ell} \OSym_{(n,n')}\arrow[u,"\dot c \otimes a \mapsto (-1)^{\parity(a)\parity(c)}a^*\otimes \dot c" right]
\end{tikzcd}
\end{equation}
of isomorphisms in which
\begin{enumerate}
\item
the top map is an even degree 0 isomorphism
of graded $(\OSym_n, R_\ell)$-superbimodules;
\item
the bottom map is an even degree 0
isomorphism
of graded  $(R_\ell,\OSym_{n'})$-superbimodules
for the right action of $\OSym_{n'}$ on $R_\ell \otimes_{\OSym_\ell} \OSym_{(n,n')}$ defined by restricting
the natural right action of $\OSym_{(n,n')}$ along
$\sig_n \circ \operatorname{p}^n:\OSym_{n'} \rightarrow
\OSym_{(n,n')}$;
\item
the right hand map is an even degree 0
graded $R_\ell$-supermodule isomorphism;
\item
the left hand map $\psi_n^\ell$
is a graded $R_\ell$-superalgebra isomorphism 
such that 
\begin{align}
\label{newiphone}
\psi_n^\ell\big(a \hatotimes \dot c\big)
&=
\sum_{i=1}^p
a_i \hatotimes \dot c_i
\end{align}
for $a \in \OSym_n$, $c \in \OSym_\ell$
with
$ac=\sum_{i=1}^p (-1)^{n\parity(a_i)} \sig_n(a_i)^* c_i$
for
$a_i \in \OSym_{n'}, c_i
\in \OSym_\ell$.
\end{enumerate}
\end{theorem}

\begin{proof}
(1)
To construct the top map,
we must show that $\OSym_{(n,n')}\otimes_{OSym_\ell} R_\ell$ can be 
made into a graded left $\EOH_n^\ell$-supermodule so that
$$
a \hatotimes \dot c\cdot (b \otimes 1) = (-1)^{\parity(c)\parity(b)}ab \otimes \dot c
$$ 
for $a \in \OSym_n, b \in \OSym_{(n,n')}$ and $c \in \OSym_\ell$. 
To see this is well defined, we know by
\cref{frobenius}(2) that $\OSym_{(n,n')} \otimes_{\OSym_\ell} R_\ell$
is generated as a $R_\ell$-supermodule
by vectors of the form $\sig_n(b) \otimes 1$
for $b \in \OSym_{n'}$, so it suffices to check that
$\sum_{s=0}^r (-1)^s h_{r-s}^{(n)} \hatotimes \dot e_{s}^{(\ell)}$
acts as zero on $\sig_n(b) \otimes 1$ for $b \in \OSym_{n'}$
and $r > n'$.
This follows by
\cref{notsoeasypeasy}(1):
\begin{align*}
\sum_{s=0}^r (-1)^s h_{r-s}^{(n)} \hatotimes \dot e_{s}^{(\ell)}
\cdot (\sig_n(b) \otimes 1) &= 
\sum_{s=0}^r (-1)^{s+s\parity(b)} h_{r-s}^{(n)} \sig_n(b) \hatotimes \dot e_{s}^{(\ell)}\\ &= 
(-1)^{r \parity(b)}
\sig_n(b) \sum_{s=0}^r  (-1)^s h_{r-s}^{(n)}
e_s^{(\ell)}\otimes 1=
0.
\end{align*}
Thus, we have defined the top map.

Next, we show that the elements $\big\{s^{(n)}_\lambda\hatotimes 1\:\big|\:
\lambda \in \GPar{n}{n'}\big\}$ 
generate $\EOH^\ell_n$ as a graded $R_\ell$-supermodule.
This is not quite as easy as before since it is no longer the case that 
$s^{(n)}_\lambda \hatotimes 1= 0$ in $\EOH^\ell_n$ when $\lambda_1>n$.
 Instead, one shows by induction on $|\lambda|$
 that any $s^{(n)}_\lambda\hatotimes 1$ for
 $\lambda$ with $\lambda_1>n$ can be written as an $R_\ell$-linear
 combination of other $s^{(n)}_\mu \hatotimes 1$ for 
$\mu$ with $\mu_1 \leq n$ and $|\mu| < |\lambda|$.

Now the proof of the first part of the theorem can be completed.
The spanning set for $\EOH_n^\ell$
just constructed 
is also linearly independent
because it becomes the basis for $OSym_{(n,n')}\otimes_{\OSym_\ell}R_\ell$ arising from \cref{frobenius}(1) when we act on the vector $1 \otimes 1$.
This also shows that the top map (1) is an isomorphism.

\vspace{1mm}
\noindent
(2)
We need to make 
$R_\ell \otimes_{\OSym_\ell} \OSym_{(n,n')}$
into a graded right $\EOH_{n'}^\ell$-supermodule so that $$
(1 \otimes b)\cdot a\hatotimes  \dot c = 
(-1)^{(\parity(a)+\parity(b))\parity(c)+n\parity(a)}
\dot c \otimes b \sig_n(a)
$$
for $a \in \OSym_{n'}$,
$b \in \OSym_{(n,n')}$ and
$c \in \OSym_\ell$.
To check that this action is well defined, 
we know from \cref{face} that the ideal defining
$\EOH_{n'}^\ell$ is generated by the elements
$\sum_{s=0}^r (-1)^{(r-n)s} h_{r-s}^{(n')} \otimes \dot e_s^{(\ell)}$
for $r > n$.
It suffices to check that the image of each such element 
in $\EOH_{n'}^\ell$ acts as zero on
$1 \otimes b$ for $b \in \OSym_n$. 
This follows by
\cref{notsoeasypeasy}(2):
\begin{align*}
(1 \otimes b)\cdot
\sum_{s=0}^r (-1)^{(r-n)s} h_{r-s}^{(n')} \hatotimes \dot e_s^{(\ell)}
&= (-1)^{rn} \sum_{s=0}^r (-1)^{s+s\parity(b)}
\dot e_s^{(\ell)} \otimes b \sig_n\big(h_{r-s}^{(n')}\big)\\
&= (-1)^{r(\parity(b)+n)} 1 \otimes 
\sum_{s=0}^r (-1)^{s}
e_s^{(\ell)} \sig_n\big(h_{r-s}^{(n')}\big)b = 0.
\end{align*}

Applying $*$ to the basis from \cref{frobenius}(2), we deduce that $R_\ell
\otimes_{\OSym_\ell} \OSym_{(n,n')}$ is a free graded 
$R_\ell$-supermodule with basis $\Big\{1 \otimes \sig_n(s^{(n')}_{\mu})\:\Big|\:\mu \in \GPar{n'}{n}\Big\}$. Also 
the elements $\big\{s^{(n')}_{\mu}\hatotimes 1\:\big|\:\mu \in \GPar{n'}{n}\big\}$ span $\EOH_{n'}^\ell$ as in the previous paragraph. Acting on $1\otimes 1$
shows finally that these elements form a basis 
for $\EOH_{n'}^\ell$ and that the bottom map is an isomorphism.

\vspace{1mm}
\noindent
(3)
To construct the right hand map, we
start with the map
$$
R_\ell \otimes \OSym_{(n,n')}
\rightarrow \OSym_{(n,n')} \otimes_{\OSym_\ell} R_\ell,\qquad
\dot c \otimes a\mapsto
(-1)^{\parity(a)\parity(c)}a^* \otimes \dot c
$$ 
for $c \in \OSym_\ell, a \in \OSym_{(n,n')}$.
Using that $R_\ell$ is supercommutative, 
this is easily checked to be a graded
$R_\ell$-supermodule homomorphism.
Also this map is balanced. To see this, we need to show that the images of
$\dot c_1 \dot c_2 \otimes a$ and $\dot c_1 \otimes c_2 a$ are the same 
for $a \in \OSym_{(n,n')}$ and $c_1,c_2 \in \OSym_\ell$.
Note that the superalgebra anti-involution $*$ of $\OSym_\ell$
descends to a superalgebra anti-involution $*$ of $R_\ell$.
Since $R_\ell$ is supercommutative and each of its generators
$\dot e_r^{(\ell)}\:(r \geq 1)$ is fixed by $*$, this induced
anti-involution is equal to the identity.
So $\dot c_2 = \dot c_2^*$, and
the image of $\dot c_1\dot c_2 \otimes a$ is
\begin{align*}
(-1)^{\parity(a)(\parity(c_1)+\parity(c_2))}
a^* \otimes \dot c_1 \dot c_2
&=
(-1)^{\parity(a)\parity(c_1)+\parity(a)\parity(c_2)+\parity(c_1)\parity(c_2)}
a^* \otimes \dot c_2^* \dot c_1\\
&=(-1)^{\parity(a)\parity(c_1)+\parity(a)\parity(c_2)+\parity(c_1)\parity(c_2)}
a^* c_2^* \otimes \dot c_1\\
&=
(-1)^{(\parity(a)+\parity(c_2))\parity(c_1)}
(c_2 a)^* \otimes \dot c_1
\end{align*}
which is equal to the image of
$\dot c_1 \otimes c_2 a$.
So this map induces the required
graded $R_\ell$-supermodule isomorphism.

\vspace{1mm}
\noindent
(4)
We define $\psi_n^\ell$ to be the unique 
graded $R_\ell$-supermodule isomorphism making the diagram commute. 
If $ac = \sum_{i=1}^p (-1)^{n\parity(a_i)} \sig_n(a_i)^*c_i$ for 
$a_i \in \OSym_{n'}, c_i \in \OSym_\ell$
then the image of $a\hatotimes \dot c$ under the top map is
equal to $\sum_{i=1}^p (-1)^{n\parity(a_i)} 
\sig_n(a_i)^*\otimes\dot c_i$.
It follows that $\psi_n^\ell(a\hatotimes 1) 
= \sum_{i=1}^p a_i \otimes \dot c_i$
because the latter expression also maps to
$\sum_{i=1}^p (-1)^{n\parity(a_i)} \sig_n(a_i)^*\otimes\dot c_i$
when the bottom map followed by the right hand map is applied.

We still need to show that $\psi_n^\ell$ is actually a graded
$R_\ell$-superalgebra isomorphism.
For this, we take $a, b\in \OSym_n$
such that $a = \sum_{i=1}^p (-1)^{n\parity(a_i)} \sig_n(a_i)^* c_i$
and $b = \sum_{j=1}^q (-1)^{n\parity(b_j)} \sig_n(b_j)^* d_j$
for $a_i,b_j \in \OSym_{n'}$ and $c_i,d_j \in \OSym_\ell$.
We have that
\begin{align*}
ab &= \sum_{j=1}^q (-1)^{n\parity(b_j)} a \sig_n(b_j)^* d_j
= \sum_{j=1}^q(-1)^{(\parity(a)+n)\parity(b_j)} \sig_n(b_j)^*
a d_j\\
&=
\sum_{i=1}^p\sum_{j=1}^q (-1)^{n\parity(a_i)+(\parity(a_i)+\parity(c_i)+n)\parity(b_j)}
\sig_n(b_j)^* \sig_n(a_i)^* c_i d_j\\
&=
\sum_{i=1}^p\sum_{j=1}^q (-1)^{n(\parity(a_i)+\parity(b_j))+\parity(c_i)\parity(b_j)}
\sig_n(a_i b_j)^* c_i d_j.
\end{align*}
So
$$
\psi_n^\ell(ab) = 
\sum_{i=1}^p \sum_{j=1}^q
(-1)^{\parity(c_i) \parity(b_j)} a_i b_j\otimes \dot c_i \dot d_j
=
\Big(
\sum_{i=1}^p a_i \hatotimes \dot c_i\Big)\Big(\sum_{j=1}^q
b_j\hatotimes \dot d_j\Big)
= \psi_n^\ell(a) \psi_n^\ell(b).
$$
\end{proof}

\begin{corollary}\label{zalatoris}
For $\ell=n+n'$,
$\EOH_n^\ell$ is 
a free graded $R_\ell$-supermodule
of graded rank
$q^{n n'} \sqbinom{\ell}{n}_{q,\pi}$.
\end{corollary}

\begin{proof}
This follows by the basis theorem that is the first assertion of
\cref{EOH} plus \cref{qbinomialformula}.
\end{proof}

\begin{corollary}\label{Etrace}
For $\ell=n+n'$, there is a unique (up to scalars)
trace map 
$\tr:\EOH_n^\ell\rightarrow R_\ell$ making
$\EOH^\ell_n$ into a graded Frobenius superalgebra over 
$R_\ell$ of degree
 $2nn'$ and parity $nn'\pmod{2}$.
Moreover, normalizing $\tr$ so that 
$\tr\big(s_{({n'}^n)}^{(n)}\hatotimes 1\big)= 1$
and recalling the definition of $\Tr$ from \cref{frobenius},
we have that 
$\Tr(a) \otimes 1=1\otimes \tr(a\hatotimes 1)$
in $\OSym_{(n,n')} \otimes_{\OSym_\ell} R_\ell$
for any $a \in \OSym_n$.
\end{corollary}

\begin{proof}
The uniqueness follows from the general principles discussed after 
\cref{traceprops2} since $\big(\EOH_n^\ell\big)_{0} = \k$.
For existence, define $\tr:\EOH_n^\ell \rightarrow R_\ell$
to be the unique graded $R_\ell$-supermodule homomorphism
such that
$\Tr(a) \otimes 1=1\otimes \tr(a\hatotimes 1)$
for $a \in \OSym_n$.
Noting that $\binom{n}{2}+\binom{n'}{2} - \binom{\ell}{2} = -nn'$,
the definition of $\Tr$ implies that
this is a homogeneous linear map
of degree $-2nn'$ and parity $nn'\pmod{2}$ such that
$\tr\big(s_{({n'}^n)}^{(n)}\hatotimes 1\big)= 1$.
It remains to check \cref{traceprops1,newcond}.
We take $b^\vee_1,\dots,b_m^\vee, b_1,\dots,b_m \in \EOH_n^\ell$
to be the elements
$s_\lambda^{(n)}\hatotimes 1$ 
and 
$
\big(\psi_n^\ell\big)^{-1}\Big((-1)^{dN(\mu)+dE(\mu)+n|\mu|}\sign(\mu)s_{\mu}^{(n')}\hatotimes 1\Big)$
for $\lambda\in \GPar{n}{n'}$ and $\mu \in \GPar{n'}{n}$, 
respectively,
enumerated in such a way that 
$b^\vee_r = s_\lambda^{(n)}\hatotimes 1$
and $\psi_n^\ell(b_r) = 
(-1)^{dN(\mu)+dE(\mu)+n|\mu|}\sign(\mu) s_{\mu}^{(n')}\hatotimes 1$
if and only if $\mu^\transpose_i = n'-\lambda_{n+1-i}$ for each $i=1,\dots,n$.
By \cref{littlebears} and the commutativity of the diagram \cref{EOHCD},
the top horizontal map in the diagram takes
$b^\vee_r$ to $s_\lambda^{(n)} \otimes 1$
and $b_r$ to 
$(-1)^{dN(\mu)+dE(\mu)}\sign(\mu)\sig_n\Big(s_{\mu}^{(n')}\Big)^* \otimes 1=
\sign(\mu)\sig_n\big(\sigma_{\mu}^{(n')}\big) \otimes 1$.
Now \cref{frobenius} implies that
$b^\vee_1,\dots,b_m^\vee$ and $b_1,\dots,b_m$ give dual bases for
$\EOH_n^\ell$ as a free graded $R_\ell$-supermodule,
as required for \cref{newcond}.
\end{proof}

The next lemma investigates the
graded $R_\ell$-superalgebra isomorphism $\psi_n^\ell:\EOH_n^\ell\stackrel{\sim}{\rightarrow}
\EOH_{n'}^\ell$ constructed in \cref{EOH}(4).

\begin{lemma}\label{uglyc}
For $\ell=n+n'$, the isomorphism
$\psi_n^\ell:\EOH_n^\ell\rightarrow \EOH_{n'}^\ell$
maps
\begin{align}\label{po1}
h^{(n)}_r\hatotimes 1
&\mapsto \sum_{s=0}^r (-1)^{(n+1)s} e_{s}^{(n')} \hatotimes
\dot h_{r-s}^{(\ell)},&
e^{(n)}_r\hatotimes 1
&\mapsto \sum_{s=0}^r (-1)^{(n+r)s} h_{s}^{(n')} \hatotimes
\dot e_{r-s}^{(\ell)}
\end{align}
for $r \geq 0$.
The inverse isomorphism
$\big(\psi_n^\ell\big)^{-1}:\EOH_{n'}^\ell\rightarrow \EOH_{n}^\ell$
maps
\begin{align}\label{po2}
h^{(n')}_r\hatotimes 1
&\mapsto (-1)^{(n+r)r}\sum_{s=0}^r (-1)^{rs} e_{r-s}^{(n)} \hatotimes
\dot h_{s}^{(\ell)},&
e^{(n')}_r\hatotimes 1
&\mapsto (-1)^{(n+1)r}\sum_{s=0}^r (-1)^{s} h_{r-s}^{(n)} \hatotimes
\dot e_{s}^{(\ell)}
\end{align}
for $r \geq 0$.
\end{lemma}

\begin{proof}
We first observe that
$h^{(n)}(t) = 
\sig_n\Big(e^{(n')}\big(t\big)\Big) h^{(\ell)}(t)$ and
$\eps^{(n)}(t) = 
\sig_n\Big(\gamm^{(n')}\big(t\big)\Big) \eps^{(\ell)}(t)$
in $\OSym_{(n,n')}$. 
This follows from \cref{newigr} 
using the identities
$h^{(\ell)}(t) = \sig_n\big(h^{(n')}(t)\big) h^{(n)}(t)$
and
$\eps^{(\ell)}(t) = \sig_n\big(\eps^{(n')}(t)\big)\eps^{(n)}(t)$.
Hence, using \cref{newiphone}, we deduce that $\psi_n^\ell$ maps
\begin{align*}
h^{(n)}(t) \hatotimes 1&\mapsto (-1)^{nn'} e^{(n')}\big((-1)^{n} t\big)^*\hatotimes \dot h^{(\ell)}(t),
&\eps^{(n)}(t) \hatotimes 1&\mapsto (-1)^{n n'} \gamm^{(n')}\big((-1)^{n}t\big)^*\hatotimes \dot \eps^{(\ell)}(t).
\end{align*}
The first formula in \cref{po1} follows from the first of these identities 
on equating $t^{-n-r}$-coefficients,
remembering that $e^{(n')}_r$ is fixed by $*$.
Similarly, equating $t^{n-r}$-coefficients in the second identity
gives that
$$
\psi_n^\ell\Big(\eps^{(n)}_r\hatotimes 1\Big)
= \sum_{s=0}^r (-1)^{(n+1)(s)} \big(\gamm_s^{(n')}\big)^* \hatotimes
\dot \eps_{r-s}^{(\ell)}.
$$
Replacing $\eps^{(n)}_{r}$ with $(-1)^{\binom{r}{2}}
e^{(n)}_{r}$,
$\big(\gamm^{(n')}_{s}\big)^*$
with $(-1)^{\binom{s}{2}} h^{(n')}_{s}$
and $\dot \eps_{r-s}^{(\ell)}$ with $(-1)^{\binom{r-s}{2}} \dot e_{r-s}^{(\ell)}$
using \cref{toe,tee}
gives the second formula in \cref{po1}.
The formulae in \cref{po2} are easily deduced from \cref{po1} together
with the infinite Grassmannian relation in $R_\ell$.
\end{proof}

The following composition defines a graded $R_\ell$-superalgebra automorphism:
\begin{equation}\label{deltadef}
\delta_n^\ell := \psi_{n'}^\ell \circ
\psi_n^\ell:\EOH_n^\ell 
\stackrel{\sim}{\rightarrow}
\EOH_n^\ell.
\end{equation}
Using \cref{uglyc}, this can be described explicitly on generators, as follows.

\begin{corollary}\label{overchicago}
The automorphism $\delta_n^\ell$
maps
\begin{align}\label{yuk}
h^{(n)}_r \hatotimes 1&\mapsto
(-1)^{\ell r} h^{(n)}_r \hatotimes 1
+(-1)^{\ell(r-1)} \big(1+(-1)^{n+r}\big) h_{r-1}^{(n)}\hatotimes \dot o^{(\ell)},\\
e^{(n)}_r \hatotimes 1&\mapsto
(-1)^{\ell r} e^{(n)}_r \hatotimes 1
+(-1)^{(\ell+1)(r-1)} \big(1+(-1)^{n+r}\big) e_{r-1}^{(n)}\hatotimes
                        \dot o^{(\ell)},\label{yuky}\\\label{wuss}
\eps^{(n)}_r \hatotimes 1&\mapsto
(-1)^{\ell r} \eps^{(n)}_r \hatotimes 1
+(-1)^{\ell(r-1)} \big(1+(-1)^{n+r}\big) \eps_{r-1}^{(n)}\hatotimes \dot o^{(\ell)},\\
\eta^{(n)}_r \hatotimes 1&\mapsto
(-1)^{\ell r} \eta^{(n)}_r \hatotimes 1
+(-1)^{(\ell+1)(r-1)} \big(1+(-1)^{n+r}\big) \eta_{r-1}^{(n)}\hatotimes \dot o^{(\ell)}
\label{yuky2}
\end{align}
for any $r \geq 1$.
In particular, $\delta_n^\ell\big(o^{(n)}\hatotimes 1\big)
= (-1)^\ell o^{(n)}\hatotimes 1 + \big(1-(-1)^n\big) 1 \hatotimes \dot o^{(\ell)}$.
\end{corollary}

\begin{proof}
We first prove \cref{yuk}.
Using \cref{po1}, one gets that
$$
(\psi_{n'}^\ell \circ \psi_n^\ell)\big(h_r^{(n)}\hatotimes 1\big)
=
\sum_{s=0}^r 
(-1)^{\ell r+(n'+r)s}
h_{r-s}^{(n)} \hatotimes
\bigg(
\sum_{t=0}^{s}
(-1)^{(n+r-s+1)t}
\dot e_{s-t}^{(\ell)} \dot h_{t}^{(\ell)}
\bigg)\;.
$$
Now we claim that the expression in the brackets here is equal to
$\delta_{s,0}$ if $n+r-s+1$ is odd, 
and it is equal to $\delta_{s,0}+2\delta_{s,1}\dot o^{(\ell)}$
is $n+r-s+1$ is even.
The formula \cref{yuk} is easily deduced using this claim.
To prove the claim, it follows immediately for odd $n+r-s+1$ by the
 infinite
Grassmannian relation. When $n+r-s+1$ is even, it follows from the relation
\begin{equation}
\sum_{t=0}^s \dot e_{s-t} \dot h_t = \delta_{s,0}+2\delta_{s,1} \dot o
\end{equation}
in $R$, which is easily checked in the cases $s=0$ and $s=1$
and follows for $s \geq 2$ using \cref{tiger,baldy}.

The proof of \cref{yuky} is a similar calculation. 

Then \cref{wuss} follows easily
since $\eps_r^{(n)} = (-1)^{\binom{r}{2}} e_r^{(n)}$ and
$\eps_{r-1}^{(n)} = (-1)^{\binom{r}{2}+r+1} e_{r-1}^{(n)}$.

Finally, to deduce \cref{yuky2}, we first write \cref{wuss} in terms
of generating functions:
\begin{equation}\label{eyes}
\delta_n^\ell\big(\eps^{(n)}(t)\hatotimes 1\big)
=
(-1)^{\ell n} \eps^{(n)}\big((-1)^\ell t\big) \hatotimes 1
-
(-1)^{\ell n} t^{-1} \Big[\eps^{(n)}\big((-1)^\ell
t\big)-\eps^{(n)}\big((-1)^{\ell+1} t\big)\Big] \hatotimes \dot o^{(\ell)}.
\end{equation}
This can be checked by equating coefficients of $t^{n-r}$ on both
sides.
Then we formally invert both sides of \cref{eyes} using \cref{newigr} to
obtain the identity
\begin{equation}\label{eyes2}
\delta_n^\ell\big(\eta^{(n)}(t)\hatotimes 1\big)
=
(-1)^{\ell n} \eta^{(n)}\big((-1)^\ell t\big) \hatotimes 1
+
(-1)^{(\ell-1) n} t^{-1} \Big[\eta^{(n)}\big((-1)^{\ell+1}
t\big)-\eta^{(n)}\big((-1)^{\ell} t\big)\Big] \hatotimes \dot o^{(\ell)}.
\end{equation}
To see that the right hand of this is indeed the inverse of the right
hand side of \cref{eyes}, one just has to multiply together
using $\big(\dot o^{(\ell)}\big)^2 = 0$ 
then simplify the result to see that it equals 1, which is straightforward. 
Finally, we equate coefficients of $t^{-n-r}$ on both sides of
\cref{eyes2} to obtain \cref{yuky2}.
\end{proof}

\section{Deformed odd cyclotomic nil-Hecke algebras}

The {\em odd cyclotomic nil-Hecke algebra} $\RONH_n^\ell$ is the quotient of $\ONH_n$ by the two-sided ideal generated by the element $x_1^{\ell}$. 
This algebra was introduced originally in \cite[Sec.~5]{EKL}.
In particular, in \cite[Prop.~5.2]{EKL}, it is shown that
$\RONH_n^\ell$ is zero unless $0 \leq n \leq \ell$, in which case it
is isomorphic to the graded matrix superalgebra
$M_{q^{n/2}[n]^!_{q,\pi}}\big(\ROH_n^\ell\big)$,
notation as explained after \cref{aswell}.

\begin{definition}\label{defo}
The  {\em deformed odd cyclotomic nil-Hecke algebra} is the quotient algebra
\begin{equation}\label{list}
\EONH_n^\ell := 
\ONH_n \otimes R_\ell \:\Bigg/\:
\left\langle \sum_{r=0}^\ell (-1)^r x_1^{\ell-r} \otimes \dot e_r^{(\ell)}
\right\rangle
\end{equation}
for $n > 0$.
We also let $\EONH_0^\ell := \ONH_0 \otimes R_\ell \cong R_\ell$
so that $\EONH_n^\ell$ makes sense for all $n \geq 0$.
We denote the image of $a \otimes \dot c$ in $\EONH_n^\ell$
by $a \hatotimes \dot c$. 
\end{definition}

The graded $R_\ell$-superalgebra $\EONH_n^\ell$  is
the odd analog of the algebras
defined for $\sl_2$ in \cite[Sec.~5.2.1]{Rou}
and for other Cartan types in \cite[Sec.~4.4.1]{Rou2};
our definition \cref{list} looks more like the latter formulation.

\begin{theorem}\label{geordie}
The deformed odd cyclotomic nil-Hecke algebra $\EONH^\ell_n$ is zero unless
$0 \leq n \leq \ell$. Assuming $0 \leq n \leq \ell$,
the natural left action of $\ONH_n \otimes R_\ell$ on $\OPol_n\otimes_{\OSym_n} \EOH_n^\ell$
factors through the quotient $\EONH_n^\ell$
to make $\OPol_n\otimes_{\OSym_n} \EOH_n^\ell$ into a
graded $(\EONH_n^\ell,\EOH_n^\ell)$-superbimodule.
The associated representation
$$
\rho:\EONH_n^\ell\rightarrow
  \End_{\dash\EOH_n^\ell}\left(\OPol_n\otimes_{\OSym_n} \EOH_n^\ell\right)
$$
is an isomorphism of graded superalgebras.
Moreover, 
$\OPol_n \otimes_{\OSym_n} \EOH_n^\ell$
is free as a graded right $\EOH_n^\ell$-supermodule
with basis $\big\{x_n^{\kappa_n} \cdots x_1^{\kappa_1} \hatotimes 1\:\big|\:\kappa \in K_n\}$ where
\begin{equation*}
K_n := \big\{\kappa = (\kappa_1,\dots,\kappa_n)
 \in \N^n\:\big|\: 0 \leq \kappa_i \leq n-i\text{ for }i=1,\dots,n\big\}.
\end{equation*}
\end{theorem}

\cref{geordie} shows that 
$\EONH_n^\ell$ is isomorphic to the graded matrix superalgebra $M_{q^{\binom{n}{2}}[n]^!_{\!q,\pi}}\Big(\EOH_n^\ell\Big)$.
We will prove the theorem later in the section, 
ultimately deducing it from \cref{firstkeythm} which showed that
 $\ONH_n$ is isomorphic to
$M_{q^{\binom{n}{2}}[n]^!_{\!q,\pi}}\Big(\OSym_n\Big)$.
First, we state some corollaries.
The first is the analog of \cref{altPsi} for $\EONH_n^\ell$.

\begin{corollary}\label{veryslow}
The element $(\overline{\omega\chi})_n := (\omega\chi)_n \hatotimes 1$ 
is a primitive idempotent in $\EONH_n^\ell$. 
Moreover, the maps $\imath$ and $\jmath$ from \cref{imathjmath} 
induce an isomorphism
$\EOH_n^\ell \cong (\overline{\omega\chi})_n \EONH_n^\ell (\overline{\omega\chi})_n$,
of graded superalgebras and an isomorphism
$\OPol_n\otimes_{\OSym_n} \EOH_n^\ell\simeq 
\EONH_n^\ell (\overline{\omega\chi})_n$ of graded $(\EONH_n^\ell, \EOH_n^\ell)$-superbimodules. Making these identifications, the idempotent truncation functor
$$
(\overline{\omega\chi})_n-:
\EONH^\ell_n\sMod
\rightarrow \EOH_n^\ell\sMod 
$$
is an equivalence of graded $(Q,\Pi)$-supercategories.  
\end{corollary}

\begin{corollary} \label{geordiecor}
For $0 \leq n \leq \ell$, 
$\EONH_n^\ell$ is a free graded $R_\ell$-supermodule of graded rank
$$
q^{n(\ell-n)} [\ell]^!_{q,\pi} \overline{[n]^!}_{\!q,\pi} 
\big/  [\ell-n]^!_{q,\pi}.
$$
\end{corollary}

\begin{proof}
This follows using the final part of the theorem and \cref{zalatoris}.
\end{proof}

\begin{corollary}\label{hucor}
For $0 \leq n \leq \ell$, the monomials $$
\big\{ x^\kappa \tau_w \hatotimes 1 \:\big|\:w \in \S_n, \kappa \in \N^n \text{ with }0 \leq  \kappa_i \leq \ell-i\text{ for }i=1,\dots,n\big\}
$$ 
form a basis for $\EONH_n^\ell$ as a free $R_\ell$-supermodule.
\end{corollary}

\begin{proof}
The free graded $R_\ell$-supermodule
with basis given by elements of the same degrees and parities as these monomials 
is graded rank
$q^{n(\ell-n)} [\ell]^!_{q,\pi} \overline{[n]}^!_{q,\pi} /
[\ell-n]^!_{q,\pi}$, which is the same as the graded rank
of $\EONH_n^\ell$ according to \cref{geordiecor}.
Therefore it suffices to show that the monomials
$\big\{ x^\kappa \tau_w \hatotimes 1\:\big|\:w \in \S_n, \kappa \in \N^n \text{ with }0 \leq  \kappa_i \leq \ell-i\text{ for }i=1,\dots,n\big\}$ 
are linearly independent over $R_\ell$.

We first prove this linear independence in the special case that $n=\ell$,
in which case we have simply that $\EOH_\ell^\ell = R_\ell$.
Suppose we have a linear relation
$\sum_{w, \kappa} x^\kappa \tau_w \hatotimes \dot c_{w,\kappa} = 0$
in $\EONH_\ell^\ell$
for $\dot c_{w,\kappa} \in \OSym_\ell$ that are not all zero,
summing over $w \in S_\ell, \kappa \in \N^\ell$
with $0 \leq \kappa_i \leq \ell-i$ for all $i=1,\dots,\ell$.
Pick $w$ of minimal length such that $\dot c_{w,\kappa} \neq 0$ for some $\kappa$.
Then we act on the vector 
$p_w^{(\ell)} \hatotimes 1 \in \OPol_\ell \otimes_{\OSym_\ell}
R_\ell$ to deduce as in the proof of \cref{ONHbasis}
that $\sum_\kappa (-1)^{\parity(c_{w,\kappa}) \ell(w)} 
x^\kappa \hatotimes \dot c_{w,\kappa} = 0$. 
The elements $x^\kappa\hatotimes 1$ for 
$\kappa \in \N^\ell$
with $0 \leq \kappa_i \leq \ell-i$ for all $i=1,\dots,\ell$
are linearly independent over $R_\ell$ thanks to \cref{aswell}.
So this implies that $\dot c_{w,\kappa} = 0$ for all $\kappa$, which is a contradiction.

Now we treat the general case.
The inclusion $\ONH_n\otimes R_\ell \hookrightarrow \ONH_\ell
\otimes R_\ell$
induces an $R_\ell$-superalgebra homomorphism
$\imath:\EONH^\ell_n \rightarrow \EONH^\ell_\ell$.
The monomials in $\EONH^\ell_n$ which we are trying to show are linearly independent map to a subset of the monomials shown to be linearly independent in the previous paragraph. This completes the proof (and also shows that 
$\imath$ is injective).
\end{proof}

\begin{corollary}\label{buckley}
For $0 \leq n < \ell$,
the graded superalgebra homomorphism $\EONH_n^\ell \rightarrow \EONH_{n+1}^\ell$
induced by the natural embedding 
$\ONH_n\otimes R_\ell \hookrightarrow \ONH_{n+1}\otimes R_\ell$ is injective.
\end{corollary}

\begin{proof}
This follows immediately from the basis theorem in the previous corollary.
\end{proof}

\begin{corollary}\label{forconsistency}
The graded superalgebra 
$\EONH_n^\ell$ is a graded Frobenius superalgebra over $R_\ell$ of degree
$2n(\ell-n)$ and parity $n(\ell-n) \pmod{2}$.
\end{corollary}

\begin{proof}
This follows from \cref{geordie} and \cref{Etrace}.
\end{proof}

\begin{remark} \label{choices}
Since
$\RONH_n^\ell \cong \EONH_n^\ell \otimes_{R_\ell} \k$
and $\ROH_n^\ell \cong \EOH_n^\ell \otimes_{R_\ell} \k$,
\cref{geordiecor} implies that $\RONH_n^\ell$
is isomorphic to the graded matrix superalgebra $M_{q^{\binom{n}{2}}[n]^!_{q,\pi}}\Big(\ROH_n^\ell\Big)$,
recovering \cite[Prop.~5.2]{EKL}.
Also \cref{hucor} implies that $\RONH_n^\ell$
has basis
given by the canonical images of the monomials
$$
\big\{ x^\kappa \tau_w\:\big|\:w \in \S_n, \kappa \in \N^n \text{ with }0 \leq  \kappa_i \leq \ell-i\text{ for }i=1,\dots,n\big\},
$$
recovering \cite[Th.~4.10]{HS}. The proof that these monomials span given
in \cite{HS} gives an explicit algorithm to ``straighten" arbitrary monomials into this form.
\end{remark}

In the remainder of the section, we prove \cref{geordie}. 
The approach is based on the proof of \cite[Prop.~5.2]{EKL} (the result which we are generalizing).
First we record some preliminary lemmas.

\begin{lemma}\label{companion}
Suppose that $n \geq 1$.
Let $y$ be a non-zero homogeneous element of $\sig_1(\OPol_{n-1})$
and consider the (free) right $\OSym_n$-submodule of $\OPol_n$ with 
basis
$v_1,\dots,v_n$ defined from
$v_i := y x_1^{i-1}$.
The matrix of the endomorphism of this subspace defined by the left action of $(-1)^{\parity(y)} x_1$ is
equal to the (non-commutative) companion matrix
\begin{equation}\label{thematrix}
\begin{pmatrix}
0&0&\cdots&0&(-1)^{n-1}e^{(n)}_n\\
1&0&\cdots&0&(-1)^{n-2} e^{(n)}_{n-1}\\
\vdots&\ddots&\ddots          &\vdots&\vdots\\
0&\cdots&1&0&-e^{(n)}_{2}\\
0&\cdots&0&1&e^{(n)}_1
\end{pmatrix}
\end{equation}
of the polynomial $(t-x_1) \cdots (t-x_n) \in \OPol_n[t]$.
\end{lemma}

\begin{proof}
We have that $(-1)^{\parity(y)} x_1 y x_1^{i-1} = y x_1^i$. 
This gives all but the last column of the matrix already.
For the last column, use 
the identity $\sum_{r=0}^n (-1)^s x_1^{n-s} e_s^{(n)}=0$, which is
\cref{notsoeasypeasy}(1) taking $n$ and $n'$
there to be $1$ and $n-1$ in the current setup.
\end{proof}

\begin{lemma}\label{comppower}
Let $C$ be the $n \times n$ companion matrix from \cref{thematrix}.
For any $1 \leq i,j \leq n$ and $k \geq 0$, the 
$(i,j)$-entry of $C^k$ is equal to 
\begin{equation}\label{expression}
c_{i,j;k} := \sum_{t=0}^{\min(k+j-i,n-i)}
(-1)^t e_t^{(n)}h_{k+j-i-t}^{(n)},
\end{equation}
which is zero if $k< i-j$ and $1$ if $k=i-j$.
\end{lemma}

\begin{proof}
This goes by induction on $k=0,1,\dots$. When $k=0$, we have that
$$
c_{i,j;k} = \sum_{t=0}^{j-i} (-1)^t e_t^{(n)} h_{j-i-t}^{(n)}
$$ 
which is zero if $i > j$ as it is the empty sum,
and it is $\delta_{i,j}$ if $i \leq j$ by the infinite Grassmannian relation.
This checks the induction base. Then we take $k \geq 0$ and consider 
the $(i,j)$-entry of
$C^{k+1} = C C^k$.
Since $C$ has at most two non-zero entries in its $i$th row,
namely, its $(i,i-1)$-entry $1$ if $i > 1$ and
its $(i,n)$-entry $(-1)^{n-i} e_{n+1-i}^{(n)}$, 
we get by induction that
the $(i,j)$-entry of $C^{k+1}$ is equal to
$c_{i-1,j;k} + (-1)^{n-i} e_{n+1-i}^{(n)}
c_{n,j;k}$ where the first term should be omitted in the case that $i=1$.
This is equal to
$$
\sum_{t=0}^{\min(k+1+j-i,n+1-i)}
(-1)^t e_t^{(n)}h_{k+1+j-i-t}^{(n)}
+
(-1)^{n-i} e_{n+1-i}^{(n)} h_{k+j-n}^{(n)},
$$
interpreting $h_{k+j-n}^{(n)}$ as zero if $k < n-j$ and
noting in the case that $i=1$ that the first term here is zero by the infinite Grassmannian relation (so there
is no need to omit it). 
To complete the proof, we need to show that
$$
\sum_{t=0}^{\min(k+1+j-i,n-i)}
(-1)^t e_t^{(n)}h_{k+1+j-i-t}^{(n)}=
\sum_{t=0}^{\min(k+1+j-i,n+1-i)}
(-1)^t e_t^{(n)}h_{k+1+j-i-t}^{(n)}
+
(-1)^{n-i} e_{n+1-i}^{(n)} h_{k+j-n}^{(n)}.
$$
If $k+1+j-i \leq n-i$ both sums are over $0 \leq t \leq k+1+j-i$ and the second term on the right hand side is zero by convention since $k+j-n< 0$, so the equality is true.
If $k +1+j-i > n-i$ then the sum on the right hand side has one extra term
when $t = n+1-i$ compared to the sum on the left hand side. But this extra term
is $(-1)^{n+1-i} e_{n+1-i}^{(n)} h_{k+j-n}^{(n)}$ which cancels with the final term on the right hand side.
\end{proof}

\begin{proof}[Proof of \cref{geordie}]
Let $A := \ONH_n\otimes R_\ell$,
$B := \OSym_n \otimes R_\ell$ and
$V := \OPol_n \otimes R_\ell$,  
which is a graded $(A,B)$-superbimodule.
By \cref{aswell}, $V$ is free as a graded right $B$-supermodule
with basis $\big\{x_n^{\kappa_n}\cdots x_1^{\kappa_1}
\otimes 1\:\big|\:\kappa \in K_n\big\}$.
Also, by \cref{firstkeythm}, $A \cong \End_{\dash B}(V)$ so that
$A$ can be identified with the graded matrix superalgebra
consisting of matrices $(a_{\kappa,\kappa'})_{\kappa,\kappa' \in K_n}$,
for $a_{\kappa,\kappa'} \in B$, this matrix representing the endomorphism
$x^{\kappa'} \otimes 1 \mapsto \sum_{\kappa \in K_n}
(x^\kappa \otimes 1)a_{\kappa,\kappa'}$.
In this situation, Morita theory implies that
there are bijections between the sets of graded superideals of $A$,
graded sub-superbimodules of $V$
and graded superideals of $B$
so that $I \leftrightarrow I V = V J \leftrightarrow J$.
For $I \unlhd A$ corresponding to $J \unlhd B$ in this way,
a set of generators for $J$ is given by the matrix entries of a set of generators of $I$, and 
we have that $A / I \cong \End_{\dash B/J} (V / IV)$.
Thus, to prove the theorem, we 
start from the two-sided ideal $I$ of $A$ from \cref{list}, 
which may be described equivalently as the two-sided ideal generated
by the elements 
$\sum_{s=0}^{r} (-1)^s x_1^{r-s} \otimes \dot e^{(\ell)}_{s}\:(r \geq \ell)$.
We must show that the two-sided ideal $J$ of $B$ generated
by the matrix entries of the generators of $I$ is equal to 
$B$ if $n > \ell$ and it is equal to the two-sided ideal
of $B$ from \cref{blue} if $n \leq \ell$.

Consider the matrix associated to the generator 
$\sum_{s=0}^r (-1)^s x_1^{r-s} \otimes \dot e^{(\ell)}_s$ of $I$
for $r \geq \ell$.
\cref{companion} 
implies that it is a block diagonal matrix with $n \times n$ blocks
parametrized by all $\kappa \in K_n$ with $\kappa_1 = 0$,
the $(i,j)$-entry of this block being the $x_n^{\kappa_n} \cdots x_2^{\kappa_2} x_1^{i-1}\otimes 1$-coefficient
of $\sum_{s=0}^r (-1)^s x_1^{r-s} \otimes \dot e_s^{(\ell)} \cdot (x_n^{\kappa_n}\cdots x_2^{\kappa_2} x_1^{j-1} \otimes 1)$ for $1 \leq i,j \leq n$.
Denote this matrix entry by $f_{i,j;r}^{(\kappa)}$.
Using \cref{companion,comppower}, 
we have that
\begin{equation}\label{thenastything}
f_{i,j;r}^{(\kappa)}
 = 
(-1)^{r |\kappa|} \sum_{s=0}^r (-1)^{sj} \sum_{t=0}^{\min(r+j-i-s,n-i)}
(-1)^t e_t^{(n)} h_{r+j-i-s-t}^{(n)} \otimes \dot e_s^{(\ell)}.
\end{equation}
Apart from the leading sign which is irrelevant for the problem in hand, 
this does not depend on $\kappa$, so we may as well assume
from now on that $\kappa = \mathbf{0}=(0,\dots,0)$.
If $n > \ell$, we take $r=\ell,j=1$ and $i=\ell+1$, in which case 
the summations collapse and we deduce that $f_{\ell+1,1;\ell}^{(\mathbf{0})} =1 \in J$.
So $J = B$ as required in this case.

Now assume that $n \leq \ell$ and set $n' := \ell-n$.
It remains to show that 
the elements 
\begin{align*}F &:= \left\{f_{i,j;r}^{(\mathbf{0})}\:\Big|\:1 \leq i,j \leq n, r \geq \ell\right\},&
G &:= \left\{\sum_{s=0}^r (-1)^s h_{r-s}^{(n)} \otimes \dot e_s^{(\ell)}\:\Bigg|\:r > n'\right\}
\end{align*}
generate the same two-sided ideal of $B$.
We switch the summations in \cref{thenastything} to deduce that
$$
f_{i,j;r}^{(\mathbf{0})}=
\sum_{t=0}^{\min(r+j-i,n-i)}
(-1)^t e_t^{(n)}
\left(\sum_{s=0}^{\min(r,r+j-i-t)}
(-1)^{sj} 
 h_{r+j-i-t-s}^{(n)} \otimes \dot e_s^{(\ell)}\right).
 $$
 We have that $r+j-i \geq \ell + 1 - i > n-i$ so the first summation is over
 $0 \leq t \leq n-i$.
 Since $\dot e_s^{(\ell)} = 0$ for $s > \ell$
 we can change the second summation so that it 
 is over $0 \leq s \leq r+j-i-t$.
 Taking $i=n$ and $j=1$ gives us the elements
 $\sum_{s=0}^{r-n+1} (-1)^{s} h_{r-n+1-s}^{(n)} \otimes \dot e_s^{(\ell)}$
 for all $r \geq \ell$. Since we have all $r$ so that $r-n+1 > n'$,
these already give us all of the elements of $G$, demonstrating one containment.
 It remains to show that all $f_{i,j;r}^{(\mathbf{0})}$
 for $1 \leq i,j \leq n$ and $r\geq \ell$ also lie in $\langle G \rangle$.
 In fact, given that $t \leq n-i$, we have that
$$
\sum_{s=0}^{r+j-i-t}
(-1)^{sj} 
 h_{r+j-i-t-s}^{(n)} \otimes \dot e_s^{(\ell)} \in \langle G \rangle
 $$
 because 
 $r+j-i-t \geq n' + j$,
the sign is $(-1)^s$ if $j \geq 1$ is odd or $1$ if $j \geq 2$ is
even,
and we know all of these elements lie in $\langle G \rangle$ by \cref{face}.
 \end{proof}

\section{The 2-category 
\texorpdfstring{$\OGBim_\ell$}{}
of odd Grassmannian bimodules}

Throughout the section, $\ell$ is fixed.
We will work always over the ground ring $R_\ell$, 
but note that everything in this section 
also makes sense 
on base change from $R_\ell$ to any graded
supercommutative $R_\ell$-superalgebra, including the most important case when the ground ring is the field $\k$.
For $0 \leq n \leq \ell$, we denote the element of $\EOH_n^\ell$
formerly denoted
by $a \hatotimes 1$ simply by $\bar a\:\:(a \in \OSym_n)$, and identify
$R_\ell$ with a subalgebra of $\EOH_n^\ell$ via the embedding
$R_\ell\hookrightarrow Z(\EOH_n^\ell), \dot c
\mapsto1\hatotimes \dot c\:(c \in \OSym_\ell)$. 

Suppose that $n,d,n' \geq 0$ with $n+d+n' =\ell$ and 
$\alpha \in \Comp(k,d)$.
The cases $\alpha = (d)$ and $\alpha = (1^d)$ 
will be particularly important. 
Let
\begin{align}\label{vbim}
\tV_{n;\alpha}^\ell &:= \OSym_{(n,\alpha_1,\dots,\alpha_k,n')} \otimes_{\OSym_\ell} R_\ell,&
U_{\alpha;n}^\ell&:= 
R_\ell \otimes_{\OSym_\ell} \OSym_{(n',\alpha_1,\dots,\alpha_k,n)}.
\end{align}
These are graded $R_\ell$-supermodules, with the action of $R_\ell$ on
$\tV_{n;\alpha}^\ell$ coming from the natural right action
and the action of $R_\ell$
on $U_{\alpha;n}^\ell$ coming from the natural left action.
We refer to $\tV^\ell_{n;\alpha}$ and $U^\ell_{\alpha;n}$ as {\em odd Grassmannian bimodules}.
According to the following lemma, they are graded superbimodules over equivariant odd Grassmannian cohomology algebras.

\begin{lemma}\label{lemma0}
Let $\ell=n+d+n'$ and $\alpha \in \Comp(k,d)$ be fixed as above.
\begin{enumerate}
\item
There is a unique way to make $\tV^\ell_{n;\alpha}$ into a
graded $\big(\EOH^\ell_n,\EOH_{n+d}^\ell\big)$-superbimodule
so that the left action of $\bar a\:\:(a \in \OSym_n)$
is defined by $\bar a (b \otimes 1) := ab  \otimes 1$
for $b \in \OSym_{(n,\alpha_1,\dots,\alpha_k,n')}$, 
and the right action of
$\bar a\:\:(a \in \OSym_{n+d})$
is defined by $(b \otimes 1) \bar a := ba \otimes 1$ 
for $b \in \OSym_{(n,\alpha_1,\dots,\alpha_k)}$.
Moreover:
\begin{enumerate}
\item
$\tV^\ell_{n;\alpha}$
has basis $\Big\{\smiley_{n+d}\big(p_w^{(n+d)}\big)\otimes 1\:\Big|\:w \in [\S_{n+d} / \S_{(\alpha_k,\dots,\alpha_1,n)}]_{\min}\Big\}$
as a right $\EOH_{n+d}^\ell$-supermodule;
\item
$\tV^\ell_{n;\alpha}$ has basis
$\Big\{\sig_{n}\big(p_w^{(n'+d)}\big)\otimes 1\:\Big|\:w \in [\S_{n'+d} / \S_{(\alpha_1,\dots,\alpha_k,n')}]_{\min}\Big\}$
as a left $\EOH^\ell_{n}$-supermodule.
\end{enumerate}
\item
There is a unique way to make $U^\ell_{\alpha;n}$ into a graded
$\big(\EOH_{n+d}^\ell,\EOH_{n}^\ell\big)$-superbimodule
so that the right action of $\bar a\:\:(a \in \OSym_{n})$
is defined by $(1 \otimes b) \bar a := (-1)^{(n'+d)\parity(a)} 1 \otimes b \sig_{n'+d}(a)$
for $b \in \OSym_{(n',\alpha_1,\dots,\alpha_k,n)}$, 
and the left action of $\bar a\:\:(a \in \OSym_{n+d})$ is defined by
$\bar a (1 \otimes \sig_{n'}(b)) := (-1)^{n'\parity(a)} 1\otimes \sig_{n'}(ab)$
for $b \in \OSym_{(\alpha_1,\dots,\alpha_k,n)}$.
Moreover:
\begin{enumerate}
\item
$U_{\alpha;n}^\ell$ 
has basis
 $\Big\{1\otimes \smiley_{n'+d}\big(p_w^{(n'+d)}\big)^*
\:\Big|\:w \in [\S_{n'+d} / \S_{(\alpha_k,\dots,\alpha_1,n')}]_{\min}\Big\}$
as a right $\EOH_{n}^\ell$-supermodule;
\item
$U_{\alpha;n}^\ell$ 
has basis
$\Big\{1\otimes\sig_{n'}\big(p_w^{(n+d)}\big)^*\:\Big|\:w \in [\S_{n+d} / \S_{(\alpha_1,\dots,\alpha_k,n)}]_{\min}\Big\}$
as a left $\EOH^\ell_{n+d}$-supermodule.
\end{enumerate}
\end{enumerate}
\end{lemma}

\begin{proof}
(1)
By \cref{ordered}, $\OSym_{(n,\alpha_1,\dots,\alpha_k,n')}$
is generated by the elements of $\OSym_{(n,\alpha_1,\dots,\alpha_k)}$
as a right $\OSym_\ell$-supermodule.
Hence, $\tV^\ell_{n;\alpha}$ is generated as an $R_\ell$-supermodule
by elements of the form $b \otimes 1$
for $b \in \OSym_{(n,\alpha_1,\dots,\alpha_k)}$.
In view of this, provided that it is well defined, there is a unique
way to make $\tV^\ell_{n;\alpha}$ into a graded right $\EOH_{n+d}^\ell$-supermodule
such that $(b \otimes 1) \bar a = ba \otimes 1$
for all $b \in \OSym_{(n,\alpha_1,\dots,\alpha_k)}$ and $a \in \OSym_{n+d}$.
It is also clear that the 
right action of $\EOH_{n+d}^\ell$ defined in this way 
and the left action of $\EOH_{n}^\ell$ from the statement of the lemma commute with each other, again assuming that both actions are well defined.

To see that the right action is well defined,
we have that
$$
\tV^\ell_{n;\alpha} \simeq \OSym_{(n,\alpha_1,\dots,\alpha_k,n')} \otimes_{\OSym_{(n+d,n')}} \OSym_{(n+d,n')}\otimes_{\OSym_\ell} R_\ell.
$$
By \cref{frontandback}(2), we deduce that
$$
\tV^\ell_{n;\alpha} = \bigoplus_{w \in \big[\S_{n+d} / \S_{(\alpha_k,\dots,\alpha_1,n)}\big]_{\min}}
\smiley_{n+d}\big(p_w^{(n+d)}\big) \otimes \big(\OSym_{(n+d,n')}\otimes_{\OSym_\ell} R_\ell\big)
$$
with each summand being a
copy of $\OSym_{(n+d,n')} \otimes_{\OSym_\ell} R_\ell$
shifted in degree and parity. Hence,
each of these subspaces is isomorphic to
$\EOH_{n+d}^\ell$
via the isomorphism from \cref{EOH}(1).
The right action of $\EOH_{n+d}^\ell$ 
we are defining
is just the natural right action of $\EOH_{n+d}^\ell$
on itself transported through these isomorphisms.
So it is well defined.
We have also proved (1a).

For the left action,
we have that 
$$
\tV_{n;\alpha}^\ell \simeq
\OSym_{(n,\alpha_1,\dots,\alpha_k,n')} \otimes_{\OSym_{(n,n'+d)}}
\OSym_{(n,n'+d)} \otimes_{\OSym_\ell} R_\ell.
$$
By \cref{frontandback}(1), we deduce that
$$
\tV_{n;\alpha}^\ell = \bigoplus_{w \in \big[\S_{n'+d} / \S_{(\alpha_1,\dots,\alpha_k,n')}\big]_{\min}}
\sig_{n}\Big(p_w^{(n'+d)}\Big) \otimes \big(\OSym_{(n,n'+d)} \otimes_{\OSym_\ell} R_\ell\big)$$
with each summand being
a graded left $\OSym_{n} \otimes R_\ell$-submodule isomorphic to $\OSym_{(n,n'+d)}\otimes_{\OSym_\ell} R_\ell$
(shifted in parity and degree).
It remains to apply \cref{EOH}(1) to see that the left action of 
$\EOH_{n}^\ell$ is well defined. This also establishes (1b).

\vspace{2mm}
\noindent
(2)
This is similar to the proof of (1), using
the isomorphism from \cref{EOH}(2) in place of the one from \cref{EOH}(1), and the left supermodule analogs of
\cref{ordered,frontandback} obtained by
applying $\smiley_\ell \circ *$ to those assertions.
\end{proof}

More often than not, we will work with a degree- and parity-shifted version of $\tV_{n;\alpha}^\ell$
and, very occasionally, of $U_{\alpha;n}^\ell$. 
These have been chosen to ensure that the adjunctions in
\cref{cupsandcaps,secondadjunction} below are even of degree 0, and
to eliminate any additional shifts in the definition of the singular
Rouqiuer complex in the next section.
Recall that
$n\# d$ denotes $n + (n+1) + \cdots + (n+d-1)$.
For $\ell=n+d+n'$
and $\alpha \in \Comp(k,d)$ as before, we define
\begin{align}\label{evil}
V_{n;\alpha}^\ell &:= (\Pi Q^{-2})^{n\# d}
\tV_{n;\alpha}^\ell,&
\tU_{\alpha;n}^\ell &:= (\Pi Q^{-2})^{n'\# d}
U_{\alpha;n}^\ell,
\end{align}
recalling \cref{clarifyparify}.
We also refer to these as odd Grassmannian bimodules.
Of course, the bases for $\tV_{n;\alpha}^\ell$ and
$U_{\alpha;n}^\ell$ from \cref{lemma0} are also bases for
$V_{n;\alpha}^\ell$ and $\tU_{\alpha;n}^\ell$.

We proceed to develop the properties of odd Grassmannian 
bimodules in a systematic way.
Take $\ell=n+d+n'$ and $\alpha \in \Comp(k,d)$.
There
are even degree 0 isomorphisms of graded $R_\ell$-supermodules
\begin{align}
*:\tV^\ell_{n;\alpha}&\stackrel{\sim}{\rightarrow}U^\ell_{\alpha;n'}\qquad\text{and}
&*:V^\ell_{n;\alpha}&\stackrel{\sim}{\rightarrow}\tU^\ell_{\alpha;n'},&
a\otimes \dot c&\mapsto (-1)^{\parity(a)\parity(c)} \dot c  \otimes a^*,
\label{starry2}\\
*:U^\ell_{\alpha;n}&\stackrel{\sim}{\rightarrow}\tV^\ell_{n';\alpha}\,\qquad\text{and}
&*:\tU^\ell_{\alpha;n}&\stackrel{\sim}{\rightarrow} V^\ell_{n';\alpha},&
\dot c \otimes a &\mapsto  (-1)^{\parity(a)\parity(c)} a^* \otimes \dot c.
\label{starry1}
\end{align}
The first pair of these with the roles of $n$ and $n'$ switched
are two-sided inverses of the second pair.

\begin{lemma}\label{lemma0b}
Continue with $\ell=n+d+n'$ and $\alpha \in \Comp(k,d)$.
\begin{enumerate}
\item
The isomorphisms from \cref{starry2} 
satisfy
$$
\left(\bar a_1 v \bar a_2\right)^* = \psi_n^\ell(\bar a_1) v^* \psi_{n+d}^\ell(\bar a_2)
$$
for $a_1 \in \OSym_n$, $a_2 \in \OSym_{n+d}$ and $v \in
\tV^\ell_{n;\alpha}$ or $v\in V^\ell_{n;\alpha}$.
\item
The isomorphisms from \cref{starry1} 
satisfy
$$
\left(\bar a_1 u \bar a_2\right)^* = \big(\psi_{n'}^\ell\big)^{-1} (\bar a_1) u^* \big(\psi_{n'+d}^\ell\big)^{-1}(\bar a_2)
$$
for $a_1 \in \OSym_n$, $a_2 \in \OSym_{n+d}$ and $u \in
U^\ell_{\alpha;n}$ or
$u\in\tU^\ell_{\alpha;n}$.
\end{enumerate}
\end{lemma}

\begin{proof}
(1) 
It suffices to prove this for $v \in \tV_{n;\alpha}^\ell$, then the
identity for $v$ viewed instead as a vector in $V_{n;\alpha}^\ell$ 
follows as the same parity shifts are used to define
$V_{n;\alpha}^\ell$ 
from $\tV_{n;\alpha}^\ell$
as $\tU_{\alpha;n'}^\ell$ from $U_{\alpha;n'}^\ell$.
We first consider right actions. Take $a \in \OSym_{n+d}$.
By \cref{newiphone},
we have that 
$\psi_{n+d}^\ell(\bar a) = \sum_{i=1}^p \bar a_i \dot c_i$
where $a = \sum_{1=1}^p (-1)^{(n+d)\parity(a_i)}\sig_{n+d}(a_i)^* c_i$.
We saw in the proof of \cref{lemma0}(1) that $\tV^\ell_{n;\alpha}$ is spanned as an $R_\ell$-supermodule
by vectors of the form $b\otimes 1$ for 
$b \in \OSym_{(n,\alpha_1,\dots,\alpha_k)}$.
So we may assume that $v = b \otimes 1$ for such a $b$.
We have that
\begin{align*}
(v \bar a)^* &= 
(ba \otimes 1)^* = \left(\sum_{i=1}^p (-1)^{(n+d)\parity(a_i)} 
b \sig_{n+d}(a_i)^* c_i \otimes 1\right)^*
= 
\left(\sum_{i=1}^p (-1)^{(n+d+\parity(b))\parity(a_i)}\sig_{n+d}(a_i)^* b \otimes \dot c_i\right)^*\\
&=\sum_{i=1}^p (-1)^{(n+d)\parity(a_i)+(\parity(b)+\parity(a_i))\parity(c_i)} \dot c_i
\otimes b^* \sig_{n+d}(a_i) =
(1 \otimes b^*) \cdot \sum_{i=1}^p \bar a_i \dot c_i
=v^* \psi^\ell_{n+d}(\bar a).
\end{align*}
For left actions,
take $a \in \OSym_{n}$ with
$a = \sum_{i=1}^p (-1)^{n\parity(a_i)}
\sig_{n}(a_i)^* c_i$, 
so that $\psi^\ell_{n}(\bar a) = \sum_{i=1}^p \bar a_i \dot c_i$.
We may assume that $v =\sig_n(b)\otimes 1$
for $b \in \OSym_{(\alpha_1,\dots,\alpha_k,n')}$.
Then we have that
\begin{align*}
(\bar a v)^* &= 
(a \sig_n(b) \otimes 1)^*
= (-1)^{\parity(a)\parity(b)} (\sig_n(b) a \otimes 1)^*\\
&=
\left(
\sum_{i=1}^p (-1)^{(\parity(a_i)+\parity(c_i))\parity(b) + n \parity(a_i)}
\sig_n(b a_i^*) \otimes \dot c_i
\right)^*=
\sum_{i=1}^p (-1)^{\parity(a_i) (n+ \parity(c_i))}
\dot c_i \otimes \sig_n(a_i b^*)\\
&= \left(\sum_{i=1}^p a_i \dot c_i\right)
\cdot (1 \otimes \sig_n(b^*))
= \psi_n^\ell(\bar a) v^*.
\end{align*}

\vspace{1mm}
\noindent
(2)
By (1), we have that $\Big[\Big(\big(\phi_{n'}^\ell\big)^{-1}(\bar
a_1))\Big) u^* \Big(\big(\phi_{n'+d}^\ell\big)^{-1}(\bar a_2)\Big)\Big]^*
= \bar a_1 u \bar a_2$ for $a_1 \in \OSym_n, a_2 \in \OSym_{n+d}$ and
$u \in U_{\alpha;n}^\ell$ or $u \in \tU_{\alpha;n}^\ell$. Now apply $*$ to both sides.
\end{proof}

Assuming still that
$\ell = n+d+n'$ and $\alpha \in \Comp(k,d)$, we next introduce a
convenient shorthand for special 
elements of odd Grassmannian bimodules.
Recall that $\alpha^\reverse$ denotes the reversed composition
$(\alpha_k,\dots,\alpha_1)$,
and note that the involution $\smiley_d:\OSym_d\rightarrow \OSym_d$
interchanges the subalgebras
$\OSym_\alpha$ and $\OSym_{\alpha^\reverse}$. 
For $f \in \OSym_\alpha$ 
and $g \in \OSym_{\alpha^\reverse}$, we let
\begin{align}
\label{fancynotation1}
\tv_{n;\alpha}(f) &:= 
\sig_n(f) \otimes 1 \in \tV^\ell_{n;\alpha},
&
v_{n;\alpha}(f) &:= 
\sig_n(f) \otimes 1 \in V^\ell_{n;\alpha},\\
\label{fancynotation2}
u_{\alpha;n}(g) &:= (-1)^{n' \parity(g)} 1 \otimes \sig_{n'}\big(\smiley_d(g)\big) \in U_{\alpha;n}^\ell,
&\tu_{\alpha;n}(g) &:= (-1)^{n' \parity(g)} 1 \otimes \sig_{n'}\big(\smiley_d(g)\big) \in \tU_{\alpha;n}^\ell.
\end{align}
The vectors $\tv_{n;\alpha}(f)$ and $u_{\alpha;n}(g)$ are
of the same degrees and parities as $f$ and $g$, respectively.
The vector $v_{n;\alpha}(f)$ is equal to $\tv_{n;\alpha}(f)$ but
it is being viewed as an element of a different superbimodule---the 
left action of $\EOH_n^\ell$ is different due to the parity shift in \cref{evil}.
Also, $v_{n;\alpha}(f)$ is of degree $\deg(f) - 2(n\# d)$
and parity $\parity(f)+n\# d\pmod{2}$.
Similarly,
$\tu_{\alpha;n}(g)$ is equal to $u_{\alpha;n}(g)$ but with
a different left action of $\EOH_{n+d}^\ell$, and
$\tu_{\alpha;n}(g)$ is of degree $\deg(g)-2 (n'\# d)$ and parity
$\parity(g)+n'\#d\pmod{2}$.
The isomorphisms $*$ from \cref{starry2,starry1} satisfy
\begin{align}
\tv_{n;\alpha}(f)^* &= (-1)^{n \parity(f)}
                           u_{\alpha;n'}\big(\smiley_d(f)^*\big),&
v_{n;\alpha}(f)^* &= (-1)^{n \parity(f)} \tu_{\alpha;n'}\big(\smiley_d(f)^*\big),\label{reallystarry1}\\\label{reallystarry2}
u_{\alpha;n}(g)^* &=(-1)^{n'\parity(g)}\tv_{n';\alpha}
\big(\smiley_d(g)^*\big),
&\tu_{\alpha;n}(g)^* &=(-1)^{n'\parity(g)} v_{n';\alpha}
\big(\smiley_d(g)^*\big).
\end{align}
We point out also that the basis vectors in both of the bases constructed in \cref{lemma0}(1)
are
of the form $\tv_n(f)$ for $f \in \OSym_\alpha$.
So the vectors $\tv_n(f)$ (resp., $v_n(f)$) for all $f \in \OSym_\alpha$ 
generate $\tV_{n;\alpha}^\ell$ (resp., $V_{n;\alpha}^\ell$) either as a left
$\EOH_n^\ell$-supermodule or as a right $\EOH_{n+d}^\ell$-supermodule.
Similarly, the bases vectors in \cref{lemma0}(2) are of the form
$u_n(g)$
for $g \in \OSym_{\alpha^\rev}$. 
So the vectors $u_n(g)$ (resp., $\tu_n(g)$) for all $g \in
\OSym_{\alpha^\rev}$ 
generate
$U_{\alpha;n}^\ell$
(resp., $\tU_{\alpha;n}^\ell$) either as a left
$\EOH_{n+d}^\ell$-supermodule or as a right $\EOH_n^\ell$-supermodule.

\begin{lemma}\label{lemma1}
Let $\ell=n+d+d'+n'$, $\alpha\in \Comp(k,d)$
and $\alpha' \in \Comp(k',d')$.
\begin{enumerate}
\item
There are unique isomorphisms
of graded $\big(\EOH^\ell_n, \EOH^\ell_{n+d+d'}\big)$-superbimodules
\begin{align}
\tc_{\alpha,\alpha'}:\tV^\ell_{n;\alpha\sqcup\alpha'}
&\stackrel{\sim}{\rightarrow}
\tV^\ell_{n;\alpha} \otimes_{\EOH^\ell_{n+d}} \tV_{n+d;\alpha'}^\ell,\label{itchy}\\\notag
\tv_{n;\alpha\sqcup\alpha'}\big(f\sig_d(f')\big)&\mapsto
\tv_{n;\alpha}(f)\otimes \tv_{n+d;\alpha'}(f')\\\intertext{and}
\label{itchyandscratchy}
c_{\alpha,\alpha'}: V^\ell_{n;\alpha\sqcup\alpha'}
&\stackrel{\sim}{\rightarrow}
V^\ell_{n;\alpha} \otimes_{\EOH^\ell_{n+d}} V_{n+d;\alpha'}^\ell,\\\notag
v_{n;\alpha\sqcup\alpha'}\big(f\sig_d(f')\big)
&\mapsto
(-1)^{((n+d)\# d')\parity(f)} 
v_{n;\alpha}(f)\otimes v_{n+d;\alpha'}(f')
\end{align}
for $f \in \OSym_\alpha, f' \in \OSym_{\alpha'}$.
\item
There are unique isomorphisms
of 
graded $\big(\EOH^\ell_{n+d+d'}, \EOH^\ell_{n}\big)$-superbimodules
\begin{align}\label{scratchy}
b_{\alpha',\alpha}:
U^\ell_{\alpha'\sqcup\alpha;n}
&\stackrel{\sim}{\rightarrow}
U^\ell_{\alpha';n+d} \otimes_{\EOH^\ell_{n+d}} U^\ell_{\alpha;n}\\\notag
u_{\alpha'\sqcup\alpha;n}\big(\sig_d(f') f\big)&\mapsto (-1)^{d' \parity(f)} u_{\alpha';n+d}(f') \otimes u_{\alpha;n}(f)\\\intertext{and}\label{scratchyitchy}
\tb_{\alpha',\alpha}:
\tU^\ell_{\alpha'\sqcup\alpha;n}
&\stackrel{\sim}{\rightarrow}
\tU^\ell_{\alpha';n+d} \otimes_{\EOH^\ell_{n+d}} \tU^\ell_{\alpha;n}\\\notag
\tu_{\alpha'\sqcup\alpha;n}\big(\sig_d(f') f\big)&\mapsto
                                                                                           (-1)^{d'
                                                                                           \parity(f)+((n'+d')\#
                                                        d) \parity(f')}
                                                                                           \tu_{\alpha';n+d}(f')
                                                                                           \otimes
                                                                                           \tu_{\alpha;n}(f)
\end{align}
for 
$f \in \OSym_{\alpha^\reverse}, f' \in \OSym_{(\alpha')^\reverse}$.
\end{enumerate}
\end{lemma}

\begin{proof}
In each case, the uniqueness is clear since the vectors specified are
superbimodule generators.

\vspace{1mm}
\noindent
(1)
Let $\gamma := (n,\alpha_1,\dots,\alpha_k,\alpha'_1,\dots,\alpha'_{k'},n')$.
Consider the surjective $\k$-linear map
$$
\OSym_{\gamma}
\rightarrow \tV_{n;\alpha}^\ell \otimes_{\EOH_{n+d}^\ell} 
\tV^\ell_{n+d;\alpha'},
\qquad
b_1 \sig_{n+d}(b_2) \mapsto (b_1 \otimes 1) \otimes (\sig_{n+d}(b_2) \otimes 1)
$$
for
$b_1 \in \OSym_{(n,\alpha_1,\dots,\alpha_k)}, b_2 \in \OSym_{(\alpha'_1,\dots,\alpha'_{k'},n')}$.
It is easy to check that it is a right $\OSym_\ell$-supermodule homomorphism,
so it induces a surjective $R_\ell$-supermodule homomorphism
$\tc_{\alpha,\alpha'}:\OSym_\gamma\otimes_{\OSym_\ell} R_\ell
\rightarrow \tV_{n;\alpha}^\ell \otimes_{\EOH^\ell_{n+d}} \tV_{n+d;\alpha'}^\ell$.
This is the map \cref{itchy}.

The domain and range of $\tc_{\alpha,\alpha'}$ are free graded $R_\ell$-supermodules, so 
to see that $\tc_{\alpha,\alpha'}$ is an isomorphism it suffices to check that they have the same graded ranks.
By \cref{lemma0}(1a) and \cref{multiidentity}, 
$\tV_{n;\alpha}^\ell$ is a free graded right $\EOH_{n+d}^\ell$-supermodule
of graded rank
$q^{N(\alpha)+nd}
\sqbinom{n+d}{(n,\alpha_1,\dots,\alpha_k)}_{q,\pi}$.
By \cref{lemma0}(1b) and \cref{multiidentity},
$\tV_{n+d;\alpha'}^\ell$ is a free graded left $\EOH_{n+d}^\ell$-supermodule
of graded rank
$q^{N(\alpha')+n'd'}
\sqbinom{n'+d'}{(\alpha_1',\dots,\alpha'_{k'},n')}_{q,\pi}$.
By \cref{zalatoris}, $\EOH_{n+d}^\ell$ is a free graded $R_\ell$-supermodule
of graded rank $q^{(n+d)(n'+d')} \sqbinom{\ell}{n+d}_{q,\pi}$.
Multiplying these together and using the identity
$$
N(\gamma) = N(\alpha)+N(\alpha')+nd+n'd'+(n+d)(n'+d')
$$
gives that
$\tV_{n;\alpha}^\ell \otimes_{\EOH_{n+d}^\ell} \tV^\ell_{n+d;\alpha'}$ is a free graded $R_\ell$-supermodule
of graded rank
$q^{N(\gamma)} \sqbinom{\ell}{\gamma}$.
This is also the graded rank of
$\tV^\ell_{n;\alpha\sqcup\alpha'}$ as 
a free graded $R_\ell$-supermodule, as follows from \cref{ordered} and \cref{maybeuseful}.

We still need to show that $\tc_{\alpha,\alpha'}$
is a graded $\big(\EOH_n^\ell,
\EOH_{n+d+d'}^\ell\big)$-supermodule homomorphism. 
We just go through the details for the right action whose definition is slightly more
complicated than the left action. We restrict to
considering just to vectors $b_1 \sig_{n+d}(b_2)\otimes 1 \in \tV_{n;\alpha\sqcup \alpha'}^\ell$
for $b_1 \in \OSym_{(n,\alpha_1,\dots,\alpha_{k})}$ and $b_2 \in \OSym_{(\alpha'_1,\dots,\alpha'_{k'})}$.
We can do this because these vectors generate $\tV_{n;\alpha\sqcup\alpha'}^\ell$ as an $R_\ell$-supermodule.
Then we take $a \in \OSym_{n+d+d'}$, write it as 
$a = \sum_{i=1}^p a_i' \sig_{n+d}(a_i'')$
for $a_i' \in \OSym_{(n,\alpha_1,\dots,\alpha_k)}$ and $a_i'' \in \OSym_{(\alpha'_1,\dots,\alpha'_{k'})}$,
and calculate:
\begin{align*}
\tc_{\alpha,\alpha'}\left(
(b_1 \sig_{n+d}(b_2) \otimes 1) \bar a
\right)
&=\sum_{i=1}^p
\tc_{\alpha,\alpha'}\left(b_1 \sig_{n+d}(b_2) a'_i\sig_{n+d}(a_i'')\otimes 1\right)\\
&=\sum_{i=1}^p(-1)^{\parity(b_2)\parity(a_i')}\tc_{\alpha,\alpha'}\left(b_1 a_i'  \sig_{n+d}(b_2a_i''))\otimes 1\right)\\
&=\sum_{i=1}^p
(-1)^{\parity(b_2)\parity(a_i')}(b_1 a_i' \otimes 1) \otimes (  \sig_{n+d}(b_2)\sig_{n+d}(a_i'')\otimes 1),\\
\tc_{\alpha,\alpha'}\left(b_1 \sig_{n+d}(b_2) \otimes 1\right)
\bar a &=
((b_1 \otimes 1) \otimes (\sig_{n+d}(b_2) \otimes 1)) 
\bar a
= \sum_{i=1}^p (b_1 \otimes 1) \otimes (\sig_{n+d}(b_2) a_i' \sig_{n+d}(a_i'') \otimes 1)\\
&= \sum_{i=1}^p (-1)^{\parity(b_2)\parity(a_i')}
(b_1 \otimes 1) \otimes \bar a_i' (\sig_{n+d}(b_2) \sig_{n+d}(a_i'') \otimes 1)\\
&= \sum_{i=1}^p(-1)^{\parity(b_2)\parity(a_i')}
(b_1 \otimes 1) \bar a_i'\otimes (\sig_{n+d}(b_2) \sig_{n+d}(a_i'') \otimes 1)\\
&= \sum_{i=1}^p(-1)^{\parity(b_2)\parity(a_i')}
(b_1 a_i' \otimes 1)\otimes (\sig_{n+d}(b_2) \sig_{n+d}(a_i'') \otimes 1).
\end{align*}
These are equal so $\tc_{\alpha,\alpha'}$ is a right $\EOH_{n+d+d'}^\ell$-supermodule homomorphism.
We have now established the existence of \cref{itchy}. 

To obtain \cref{itchyandscratchy},
we define a superbimodule isomorphism $c_{\alpha,\alpha'}$
so that the following diagram commutes:
$$
\begin{tikzcd}
V^\ell_{n;\alpha\sqcup\alpha'}\arrow[d,"\id" left]
\arrow[r,"c_{\alpha,\alpha'}" above]&V^\ell_{n;\alpha}\otimes_{\EOH_{n+d}^\ell}
V^\ell_{n+d;\alpha'}\\
\tV^\ell_{n;\alpha\sqcup\alpha'}\arrow[r,"\tc_{\alpha,\alpha'}" below]&
\tV^\ell_{n;\alpha}\otimes_{\EOH_{n+d}^\ell}
\tV^\ell_{n+d;\alpha'}
\arrow[u,"\id \otimes \id" right]\end{tikzcd}
$$
The vertical maps here arise from identity maps on the underlying 
vector spaces, which are graded superbimodule isomorphisms but they are not 
even of degree 0. 
It remains to compute $c_{\alpha,\alpha'}\big(v_{n;\alpha\sqcup\alpha'}(f\sig_d(f'))\big)$
explicitly by tracing it around the other three sides of the square
to see that it is exactly the map $c_{\alpha,\alpha'}$ from \cref{itchyandscratchy}.
The complicated sign arises because the right hand map takes
$\tv_{n;\alpha}(f) \otimes \tv_{n+d;\alpha'}(f')$
to $(-1)^{((n+d)\#d')\parity(f)} v_{n;\alpha}(f) \otimes v_{n+d;\alpha'}(f')$
since $\id:\tV_{n+d;\alpha'}^\ell \rightarrow V_{n+d;\alpha'}^\ell$
is of parity $(n+d)\# d'$.

\vspace{1mm}
\noindent
(2)
Writing $*$ for the appropriate one of the isomorphisms
from \cref{starry1,starry2},
we define $b_{\alpha',\alpha}$ to be
the composition $(*\otimes *) \circ \tc_{\alpha',\alpha}\circ *$.
The appropriate diagram is
$$
\begin{tikzcd}
U_{\alpha'\sqcup\alpha;n}^\ell\arrow[r,"b_{\alpha',\alpha}" above]
\arrow[d,"*" left]&
U_{\alpha';n+d}^\ell \otimes_{\EOH_{n+d}^\ell}
U_{\alpha;n}^\ell\\
\tV_{n';\alpha'\sqcup\alpha}^\ell
\arrow[r,"\tc_{\alpha',\alpha}" below]
&\tV_{n';\alpha'}^\ell \otimes_{\EOH_{n'+d'}^\ell} \tV_{n'+d';\alpha}^\ell
\arrow[u,"* \otimes *"]
\end{tikzcd}
$$
The resulting isomorphism $b_{\alpha',\alpha}$ is a 
graded $\big(\EOH^\ell_{n+d+d'}, \EOH^\ell_{n}\big)$-superbimodule homomorphism
thanks to \cref{lemma0b}(1).
It remains to compute $b_{\alpha',\alpha}\big(\sig_d(f') f\big)$
for $f \in \OSym_{\alpha^\reverse}, f' \in \OSym_{(\alpha')^\reverse}$.
Note that $\sig_d(f') f = (-1)^{\parity(f)\parity(f')}
f \sig_d(f') \in \OSym_{(\alpha'\sqcup\alpha)^\reverse}$.
Using \cref{reallystarry1,reallystarry2}, we have that
\begin{align*}
\tc_{\alpha',\alpha}\left(u_{\alpha'\sqcup\alpha;n}\big(\sig_d(f') f\big)^*\right)^{*\otimes *}
&=
(-1)^{n' \parity(f)+n'\parity(f') + \parity(f)\parity(f')}
\tc_{\alpha',\alpha}\left(\tv_{n';\alpha'\sqcup\alpha}\Big(\smiley_{d+d'}\big(
f^* \sig_d((f')^*)\big)\Big)\right)^{* \otimes *}\\
&=
(-1)^{n' \parity(f)+n'\parity(f')}
\tc_{\alpha',\alpha}\left(\tv_{n';\alpha'\sqcup\alpha}\Big(\smiley_{d'}(f')^*
\sig_{d'}\big(\smiley_{d}(f)^*\big)\Big)\right)^{* \otimes *}\\
&=
(-1)^{n' \parity(f)+n'\parity(f')}
\left(
\tv_{n';\alpha'}\big(\smiley_{d'}(f')^*\big)\otimes 
\tv_{n'+d';\alpha}\big(\smiley_d(f)^*\big)
\right)^{*\otimes *}\\
&=
(-1)^{d' \parity(f)} u_{\alpha';n+d}(f')\otimes u_{\alpha;n}(f),
\end{align*}
which is the formula for
$b_{\alpha',\alpha}\Big(u_{\alpha'\sqcup\alpha;n}\big(\sig_d(f') f\big)\Big)$
from \cref{scratchy}. This establishes the existence of
$b_{\alpha',\alpha}$.
Finally, the existence of \cref{scratchyitchy} can now be deduced 
in the same way that \cref{itchyandscratchy} was
deduced from \cref{itchy} in the proof of (1).
\end{proof}

The next lemma gives ``Schur bases" for $V_{n;(d)}^\ell$
and $U_{(d);n}^\ell$, and for various specializations
in which we are viewing $\k$ as a graded supermodule 
concentrated in degree 0 and even parity in the unique possible way.
Similar statements hold for $\tV_{n;(d)}^\ell$ and
$\tU_{(d);n}^\ell$, but these will not be needed.

\begin{lemma}\label{lemma01}
Suppose that $\ell =n+d+n'$.
\begin{enumerate}
\item
The  supermodule $V_{n;(d)}^\ell$ has basis
$\Big\{v_{n;(d)}\big(s_\lambda^{(d)}\big)\:\Big|\:\lambda \in \GPar{d}{n'}\Big\}$ as a free left $\EOH_n^\ell$-supermodule, 
and basis $\Big\{v_{n;(d)}\big(\sigma_\lambda^{(d)}\big)\:\Big|\:
\lambda \in \GPar{d}{n}\Big\}$
as a free right $\EOH_{n+d}^\ell$-supermodule.
Hence, the vectors
$\Big\{v_{n;(d)}\big(\sigma_\lambda^{(d)}\big)\otimes 1\:\Big|\:\lambda \in \GPar{d}{n}\Big\}$ give a linear basis for
$V_{n;(d)}^\ell\otimes_{\EOH_{n+d}^\ell}\k$.
Moreover, for $\lambda \in \Par$ and any $f \in \OSym_d$,
we have that
\begin{equation}\label{trains}
v_{n;(d)}\big(f \sigma_\lambda^{(d)}\big)\otimes 1
=
(-1)^{\NE(\lambda)+|\lambda|(\parity(f)+n\#d)} \bar s^{(n)}_{\lambda^\transpose}
v_{n;(d)}(f) \otimes 1
\end{equation}
in $V_{n;(d)}^\ell \otimes_{\EOH_{n+d}^\ell} \k$.
In particular,
$v_{n;(d)}\big(\sigma_\lambda^{(d)}\big) \otimes 1 = 0$
unless $\lambda \in \GPar{d}{n}$.
\item
The supermodule $U_{(d);n}^\ell$ has basis
$\Big\{u_{(d);n}\big(s_\lambda^{(d)}\big)\:\Big|\:
\lambda \in \GPar{d}{n}\Big\}$
as a free left $\EOH_{n+d}^\ell$-supermodule,
and basis 
$\Big\{u_{(d);n}\big(\sigma_\lambda^{(d)}\big)\:\:\Big|\:\:
\lambda \in \GPar{d}{n'}
\Big\}$ 
as a free right $\EOH_{n}^\ell$-supermodule.
Hence,
the vectors
$\Big\{1 \otimes u_{(d);n}\big(s_\lambda^{(d)}\big)
\:\Big|\:\lambda \in \GPar{d}{n}\Big\}$ give a linear basis for
$\k\otimes_{\EOH_{n+d}^\ell} U_{(d);n}^\ell$.
Moreover, for $\lambda \in \Par$ and any $f \in \OSym_d$,
we have that
\begin{equation}\label{trains2}
1\otimes u_{(d);n}\big(s_\lambda^{(d)} f\big)
=
(-1)^{\NE(\lambda)+|\lambda|(\parity(f)+d)} 1\otimes 
u_{(d);n}(f) \bar s^{(n)}_{\lambda^\transpose}
\end{equation}
in $\k\otimes_{\EOH_{n+d}^\ell} U_{(d);n}^\ell$.
In particular,
$1\otimes u_{(d);n}\big(s_\lambda^{(d)}\big) = 0$
unless $\lambda \in \GPar{d}{n}$.
\end{enumerate}
\end{lemma}

\begin{proof}
(1)
The existence of the two families of 
Schur bases for $V_{n;(d)}^\ell$ 
follow in the same way as the bases in \cref{lemma0}(1)
were constructed, using \cref{frontandback2} in place of \cref{frontandback}.
To prove \cref{trains},
we have in $\OSym_{(n,d)} \otimes_{\OSym_{n+d}}\k$
that 
$\sig_n\big(\sigma_\lambda^{(d)}\big) \otimes 1 =
(-1)^{\NE(\lambda)} s_{\lambda^\transpose}^{(n)} \otimes 1$
thanks to \cref{hard}.
Multiplying this identity on the left by $\sig_n(f)$
for $f\in \OSym_d$
gives that
$$
\sig_n\big(f \sigma_\lambda^{(d)}\big) \otimes 1 =
(-1)^{\NE(\lambda)+|\lambda|\parity(f)} 
s_{\lambda^\transpose}^{(n)} \sig_n(f) \otimes 1.
$$
This implies \cref{trains}.

For the final assertion, it remains to observe that if $\lambda \notin \GPar{d}{n}$
then we either have that $\lambda_1 > n$, in which case
$s_{\lambda^\transpose}^{(n)} = 0$ by \cref{alreadybasis},
or $\h(\lambda) > d$, in which case
$\sigma_{\lambda}^{(d)} = 0$ by \cref{rangoon}.
Hence,
$\tv_{n;(d)}\big(\sigma_\lambda^{(d)}\big)\otimes 1 = 
\pm
\tv_{n;(d)}(1) \otimes 1 =
0$
for such $\lambda$.

\vspace{1mm}
\noindent
(2)
The existence of the two families of bases follows
by applying $*$ to the bases in (1),
using also \cref{lemma0b}(1), \cref{reallystarry1} and \cref{bears}.
To prove \cref{trains2} (which is {\em not} what one gets by applying
$*$ to \cref{trains}),
we start from the identity
$s_\lambda^{(d)} \otimes 1= (-1)^{\NE(\lambda)} 
\sig_d\big(\sigma_{\lambda^\transpose}^{(n)}\big) \otimes 1$
in $\OSym_{d,n} \otimes_{\OSym_{n+d}} \k$
from \cref{hard}.
Applying the isomorphism $*$
that is the right hand map of \cref{OHCD}
gives
$1 \otimes \sigma_\lambda^{(d)}= (-1)^{\NE(\lambda)} 
1 \otimes \sig_d\big(s_{\lambda^\transpose}^{(n)}\big)$
in $\k\otimes_{\OSym_{n+d}} \OSym_{d,n}$.
This we multiply on the right by $\smiley_d(f) \in \OSym_d$ to obtain
$1 \otimes \sigma_\lambda^{(d)} \smiley_d(f) = (-1)^{\NE(\lambda) + 
\parity(f)|\lambda|} 
1 \otimes \smiley_d(f) \sig_d\big(s_{\lambda^\transpose}^{(n)}\big)$.
This implies that
$$
1 \otimes \sig_{n'}\big(\sigma_\lambda^{(d)} \smiley_d(f)\big)
= (-1)^{\NE(\lambda)+\parity(f)|\lambda|}
1 \otimes \sig_{n'}(\smiley_d(f)) \sig_{n'+d}\big(s_{\lambda^\transpose}^{(n)}\big)
$$
in $\k \otimes_{\EOH_{n+d}^\ell} U^\ell_{(d);n}$.
Using the definition of the right action in \cref{lemma0}(2) and \cref{fancynotation2}, 
this implies \cref{trains2}.

The final assertion follows as in (1).
\end{proof}

From this point onwards, we will 
denote $\tV_{n;(1)}^\ell, V_{n;(1)}^\ell, U_{(1);n}^\ell$
and $\tU_{(1);n}^\ell$ by
$\tV_n^\ell, V_n^\ell,
U_n^\ell$ and $\tU_n^\ell$, respectively.
We denote the generators
$\tv_{n;(1)}(f), v_{n;(1)}(f), u_{(1);n}(g)$
and $\tu_{(1);n}(g)$ from \cref{fancynotation1,fancynotation2}
by $\tv_n(f), v_n(f), u_n(g)$ and $\tu_n(g)$ for $f,g \in \OSym_1$. 
It is also convenient
to write simply $x$ in place of $x_1$ when working in rank 1, i.e.,
we identify $\OSym_1$ with the graded polynomial
superalgebra $\k[x]$ generated by the variable $x$ 
that is odd of degree 2 so that $x_1 \in \OSym_1$
is identified with $x \in \k[x]$.
So, for $f(x) \in \k[x]$, we have that
\begin{align}\label{thisisnotbad1}
\tv_n\big(f(x)\big)
&=f(x_{n+1}) \otimes 1 \in \tV^\ell_{n},
&v_n\big(f(x)\big)
&=f(x_{n+1}) \otimes 1 \in V^\ell_{n},\\
u_{n}\big(f(x)\big)&=
(-1)^{n'\parity(f)} 1 \otimes f(x_{n'+1}) \in U^\ell_{n},&
\tu_{n}\big(f(x)\big)&=
(-1)^{n'\parity(f)} 1 \otimes f(x_{n'+1}) \in \tU^\ell_{n}\label{thisisnotbad2}
\end{align}
for $n'$ defined from $\ell=n+1+n'$. For further motivation for the significance of the graded
$\big(\EOH_{n},\EOH_{n+1}\big)$-superbimodule $V_n^\ell$ and the 
graded  $\big(\EOH_{n+1},\EOH_n\big)$-superbimodule $U_n^\ell$, see \cref{indresrel,indresrel2}.

Applying a sequence of the isomorphisms from
\cref{itchyandscratchy,scratchy} 
in any way that makes sense, we obtain even degree 0 isomorphisms
\begin{align}\label{ukraine3}
c_{(1)^d}:
V_{n;(1^d)}^\ell&\stackrel{\sim}{\rightarrow}
V_{n}^\ell\otimes_{\EOH_{n+1}^\ell}
V_{n+1}^\ell\otimes_{\EOH_{n+2}^\ell}
\cdots \otimes_{\EOH_{n+d-1}^\ell}
V_{n+d-1}^\ell
\\\notag
v_{n;(1^d)}\Big(x_1^{\kappa_1} \cdots x_d^{\kappa_d}\Big)
&\mapsto 
(-1)^{\sum_{i=1}^d ((n+i)\#(d-i))\kappa_i}
v_n\big(x^{\kappa_1}\big)\otimes\cdots\otimes v_{n+d-1}\big(x^{\kappa_d}\big)
\\
\intertext{
of $\big(\EOH_n^\ell,\EOH_{n+d}^\ell\big)$-superbimodules, and}
b_{(1)^d}:
U_{(1^d);n}^\ell
&\stackrel{\sim}{\rightarrow}
U_{n+d-1}^\ell\otimes_{\EOH_{n+d-1}^\ell}
\cdots 
\otimes_{\EOH_{n+2}^\ell}
U_{n+1}^\ell
\otimes_{\EOH_{n+1}^\ell}
U_{n}^\ell
\label{ukraine2}\\\notag
u_{(1^d);n}\Big(x_d^{\kappa_d} \cdots x_1^{\kappa_1}\Big)&\mapsto
(-1)^{\sum_{i=1}^d (d-i) \kappa_i} 
u_{n+d-1}\big(x^{\kappa_d}\big)
\otimes\cdots\otimes u_n\big(x^{\kappa_1}\big)
\end{align}
of $\big(\EOH_{n+d}^\ell,\EOH_{n}^\ell\big)$-superbimodules.

The parity shifts incorporated into the definitions of the actions in
the next lemma have been
included to ensure that the actions agree with \cref{fromonh,fromonh2} below.

\begin{lemma}\label{lemma2}
Suppose that $\ell = n+d+n'$.
\begin{enumerate}
\item
There is a left action of $\ONH_d$ on $V_{n;(1^d)}^\ell$ 
making it into an $\big(\EOH_n^\ell\otimes\ONH_d,\EOH_{n+d}^\ell\big)$-superbimodule
so that $a \cdot v_{n;(1^d)}(f) = (-1)^{(n\#(d-1))\parity(a)}
v_{n;(1^d)}(a\cdot f)$ for $a \in \ONH_d$ and
$f \in \OPol_d$.
Moreover, the inclusion $\OSym_d \hookrightarrow \OPol_d$ induces an isomorphism
$V_{n;(d)}^\ell \stackrel{\sim}{\rightarrow}
(\omega\chi)_d \cdot \tV_{n;(1^d)}^\ell$
of $\big(\EOH_n^\ell,\EOH_{n+d}^\ell\big)$-superbimodules
taking $v_{n;(d)}(f) \mapsto v_{n;(1^d)}(f)$ for $f \in \OSym_d$. Hence, we have that
\begin{equation}\label{rescue}
V_{n;(1^d)}^\ell \simeq \bigoplus_{w \in S_d} (\Pi Q^2)^{\ell(w)}
V_{n;(d)}^\ell
\end{equation}
as graded $(\EOH_{n}^\ell, \EOH_{n+d}^\ell)$-superbimodules.
\item
There is a right action of $\ONH_d$ on $U_{(1^d);n}^\ell$ 
making it an
$\big(\EOH_{n+d}^\ell,\EOH_{n}^\ell\otimes \ONH_d\big)$-superbimodule
so that $u_{(1^d);n}(f) \cdot a = (-1)^{(d-1)\parity(a)}
u_{(1^d);n}(f \cdot a)$ for $a \in \ONH_d$ and $f \in \OPol_d$; the right action of
$\ONH_d$ on $\OPol_d$ being used here is the one from \cref{opolright}. Moreover,
the inclusion $\OSym_d \hookrightarrow \OPol_d$ induces 
an isomorphism $U_{(d);n}^\ell \stackrel{\sim}{\rightarrow}
U_{(1^d);n}^\ell
\cdot (\chi\omega)_d$ taking $u_{(d);n}(f) \mapsto u_{(1^d);n}(f)$
for $f \in \OSym_n$. Hence,
we have that
\begin{equation}\label{mission}
U_{(1^d);n}^\ell \simeq \bigoplus_{w \in S_d} (\Pi Q^2)^{\ell(w)}
U_{(d);n}^\ell
\end{equation}
as graded $(\EOH_{n+d}^\ell, \EOH_n^\ell)$-superbimodules.
\end{enumerate}
\end{lemma}

\begin{proof}
(1) 
Since $\tV^\ell_{n;(1^d)}=\OSym_{(n,1^d,n')} \otimes_{\OSym_\ell} R_\ell$
and $\OSym_{(n,1^d,n')} = \OSym_n \otimes \OPol_d \otimes \OSym_{n'}$,
the left action of $\ONH_d$ on $\OPol_d$ 
from \cref{secONH}
twisted by the automorphism $\operatorname{p}^{n+d-1}:\ONH_d\rightarrow \ONH_d$
induces a left action of $\ONH_d$ on $\tV^\ell_{n;(1^d)}$
such that $a \cdot v_{n;(1^d)}(f) = (-1)^{(n+d-1)\parity(a)}
v_{n;(1^d)}(a\cdot f)$ for $a \in \ONH_d$ and
$f \in \OPol_d$.
This supercommutes with the left action of $\EOH_n^\ell$
and commutes with the right action of $\EOH_{n+d}^\ell$,
so it makes $\tV^\ell_{n;(1^d)}$ into an
$\big(\EOH_n^\ell\otimes\ONH_d,\EOH_{n+d}^\ell\big)$-superbimodule.
We get the action of $\ONH_d$ on $V^\ell_{n;(1^d)}$ in the statement
of the lemma on incorporating the
additional sign of $(-1)^{(n\# d) \parity(f)}$ into the action, which comes from
the parity shift $\Pi^{n\# d}$ in the definition \cref{evil} of $V_{n;(1^d)}^\ell$.
Since $\OSym_d = (\omega\chi)_d \cdot \OPol_d$, 
the inclusion
$\OSym_d \hookrightarrow \OPol_d$ induce an isomorphism
$V^\ell_{n;(d)} \stackrel{\sim}{\rightarrow}
(\omega\chi)_d\cdot V^\ell_{n;(1^d)}$.
The decomposition \cref{rescue} follows from \cref{firstkeythm}.

\vspace{1mm}
\noindent
(2)
This is similar, starting from the right action of $\ONH_d$
on $\OPol_d$ discussed at the end of \cref{secONH}
composed with $\operatorname{p}^{d-1}$. 
The last assertions in (2)
follow because $\OSym_d = \OPol_d \cdot (\chi\omega)_d$,
and \cref{mission} follows from \cref{bubbleandsqueekier}.
\end{proof}

In view of \cref{lemma1,lemma2}, all of the odd Grassmannian bimodules
$\tV^\ell_{n;\alpha}$,
$V^\ell_{\alpha;n}$, $U^\ell_{\alpha;n}$  and $\tU^\ell_{\alpha;n}$
are isomorphic to 1-morphisms in the 2-supercategory $\OGBim_\ell$
introduced in the next important definition.

\begin{definition}\label{shy}
The category $\OGBim_\ell$
of {\em odd Grassmannian bimodules} is the full additive graded sub-$(Q,\Pi)$-2-supercategory
of the (weak) graded $(Q,\Pi)$-2-supercategory of 
graded superalgebras, graded superbimodules and superbimodule homomorphisms
consisting of objects that are the graded superalgebras $\EOH_n^\ell$
for  $0 \leq n \leq \ell$ plus a distinguished object that is a trivial (zero)
graded superalgebra,
and $1$-morphisms that are generated by the odd Grassmannian bimodules $V_{n}^\ell \in \Hom_{\OGBim_\ell}(\EOH_{n+1}^\ell,\EOH_{n}^\ell)$
and $U_{n}^\ell\in
\Hom_{\OGBim_\ell}(\EOH_{n}^\ell,\EOH_{n+1}^\ell)$ for $0 \leq n < \ell$.
\end{definition}

\begin{remark}\label{howstrict}
Although not a strict 2-supercategory, we often work with
$\OGBim_\ell$ as though it was in fact strict. More formally, when we do this, we 
are replacing it with its strictification. The latter can be realized as a 2-subcategory of the
strict graded 2-supercategory of graded supercategories, graded superfunctors and graded supernatural transformations
in such a way that the graded
2-superequivalence from $\OGBim_\ell$ takes
the graded
superalgebra $\EOH_n^\ell$
to the graded supercategory $\EOH_n^\ell\sMod$,
a graded $\big(\EOH_m^\ell,\EOH_n^\ell\big)$-superbimodule $M$
in $\OGBim_\ell$
to the graded superfunctor $M \otimes_{\EOH_n} -:
\EOH_n^\ell\sMod \rightarrow \EOH_m^\ell\sMod$,
and a graded superbimodule endomorphism
$f:M \rightarrow M'$ 
to the graded supernatural transformation 
$f \otimes \id:M\otimes_{\EOH_n}-\Rightarrow M' \otimes_{\EOH_n}-$.
\end{remark}

The next lemma gives explicit presentations for the
generating odd Grassmannian bimodules
$V^\ell_{n}$ and $U_{n}^\ell$. In formulating the result, 
we also incorporate an indeterminate $t$ into our notation,
working in the $R_\ell\lround t^{-1}\rround$-supermodules
$V^\ell_{n}\lround t^{-1}\rround$ and
$U^\ell_{n} \lround t^{-1}\rround$,
this being a natural extension of the generating function formalism developed already for odd symmetric functions.

\begin{lemma}\label{lemma3}
Suppose that $\ell=n+1+n'$.
\begin{enumerate}
\item
Let 
$V := \EOH_n^\ell \otimes_{R_\ell} R_\ell[x] \otimes_{R_\ell} \EOH_{n+1}^\ell$,  which is
the free graded $\big(\EOH_n^\ell,\EOH_{n+1}^\ell\big)$-superbimodule
on the graded $R_\ell$-supermodule $R_\ell[x]$.
For $f(x) \in \k[x] \subset R_\ell[x]$, we denote $1 \otimes f(x) \otimes 1 \in V$ by $v\big(f(x)\big)$.
Let $T$ be the sub-bimodule of $V$ generated by either of the following
equivalent relations\footnote{We mean the relations obtained by equating coefficients of powers of $t$ on both sides.}:
\begin{align}\label{newrels1}
v\big(x^r\big) \bar e^{(n+1)}(t)&=
(-1)^{n(r+1)} \bar e^{(n)}\big((-1)^{n+r} t\big)v\big((t-x) x^r\big),\\
\bar\eps^{(n)}(t) v\big(x^r\big) &= 
(-1)^{n(r+1)}  v\big(((-1)^{n+r} t-x)^{-1} x^r\big)\,
\bar \eps^{(n+1)}\big((-1)^{n+r}t\big),\label{newrelsfancy2} 
\end{align}
for $r \geq 0$.
There is an isomorphism of graded $\big(\EOH_n^\ell,\EOH_{n+1}^\ell\big)$-superbimodules 
 \begin{equation}\label{withlove}
V \:\big/\: T
 \stackrel{\sim}{\rightarrow} V_{n}^\ell,
 \qquad
 v\big(f(x)\big)+T\mapsto v_{n}\big(f(x)\big).
 \end{equation}
Moreover:
\begin{enumerate}
\item 
$V_{n}^\ell$ is free 
as a graded left $\EOH_n^\ell$-supermodule
with basis
$\{v_n(x^r)\:|\:0\leq r \leq n'\}$;
\item $V_{n}^\ell$ is free 
as a graded right $\EOH_{n+1}^\ell$-supermodule
with
basis
$\{v_n(x^r)\:|\:0\leq r \leq n\}$;
\item the vector $v_n(1)$ generates
$V_{n}^\ell$ as a graded $\big(\EOH_n^\ell,\EOH_{n+1}^\ell\big)$-superbimodule.
\end{enumerate}
\item
Let 
$U := \EOH_{n+1}^\ell \otimes_{R_\ell} R_\ell[x] \otimes_{R_\ell} \EOH_{n}^\ell$,  which is
the free graded $\big(\EOH_{n+1}^\ell,\EOH_{n}^\ell\big)$-superbimodule
on the graded $R_\ell$-supermodule $R_\ell[x]$.
For $f \in \k[x]\subseteq R_\ell[x]$, we denote $1 \otimes f \otimes 1 \in U$ by $u(f)$.
Let $S$ be the sub-bimodule of $U$ generated by either of the following
equivalent relations:
\begin{align}\label{rels1b}
\bar e^{(n+1)}(t)u\big(x^r\big) &=
(-1)^{n(r+1)} u\big((t-x) x^r\big) \bar e^{(n)}\big((-1)^{r+1}t\big),\\
\label{relsfancy3}
u\big(x^r\big) \bar\eps^{(n)}(t)&= 
(-1)^{n(r+1)} \bar\eps^{(n+1)}\big((-1)^{r+1} t\big)
u\big(((-1)^{r+1}t-x)^{-1} x^r\big),
\end{align}
for $r \geq 0$.
There is an isomorphism of graded $\big(\EOH_{n+1}^\ell,\EOH_{n}^\ell\big)$-superbimodules 
 \begin{equation}\label{withlove2}
U \:\big/\: S
 \stackrel{\sim}{\rightarrow} U_{n}^\ell,
 \qquad
u\big(f(x)\big)+S\mapsto u_{n}\big(f(x)\big).
 \end{equation}
Moreover:
\begin{enumerate}
\item $U_{n}^\ell$ is free 
as a right $\EOH_{n}^\ell$-supermodule
with
basis
$\{u_{n}(x^r)\:|\:0\leq r \leq n'\}$;
\item 
$U_{n}^\ell$ is free 
as a graded left $\EOH_{n+1}^\ell$-supermodule
with basis
$\{u_{n}(x^r)\:|\:0\leq r \leq n\}$;
\item the vector $u_{n}(1)$ generates
$U_{n}^\ell$ as a graded $\big(\EOH_{n+1}^\ell,\EOH_{n}^\ell\big)$-superbimodule.
\end{enumerate}
\end{enumerate}
\end{lemma}

\begin{proof}
Note to start with that (1a)--(1b) and (2a)--(2b) follow immediately from 
\cref{lemma01}.

\vspace{1mm}
\noindent
(1) 
In this paragraph, we prove the equivalence of the relations \cref{newrels1}
and \cref{newrelsfancy2} by some formal manipulation with power series.
Replacing $t$ by $(-1)^{n+r} t$, the relations \cref{newrels1} are
equivalent to
\begin{align}\label{polpol}
\bar e^{(n)}\big(t\big)v\big(\big((-1)^{n+r}t-x\big) x^r\big)&=
(-1)^{n(r+1)} v\big(x^r\big) \bar e^{(n+1)}\big((-1)^{n+r} t\big)
\end{align}
for all $r \geq 0$.
We first show that the relations \cref{polpol} are equivalent to the relations
\begin{align}\label{relsbad}
\bar e^{(n)}(t) v\big(x^r\big)
&= 
(-1)^{(n+1)r} v\big(t(t^2+x^2)^{-1} x^r\big) \bar e^{(n+1)}\big((-1)^{n+r}
                                 t\big)
-(-1)^{nr} v\big(x(t^2+x^2)^{-1} x^r\big) \bar e^{(n+1)}\big((-1)^{n+r+1}
                                 t\big)
\end{align}
for all $r \geq 0$.
To deduce \cref{polpol} from \cref{relsbad}, we take the left hand
side of \cref{polpol}, which equals
$(-1)^{n+r}\bar e^{(n)}(t)v\big(tx^r\big)-\bar e^{(n)}(t)v\big(x^{r+1}\big)$.
Then we use \cref{relsbad} to commute $\bar e^{(n)}(t)$ to the right
to obtain
\begin{multline*}
(-1)^{n(r+1)} v\big(t^2(t^2+x^2)^{-1}x^r\big) \bar  e^{(n+1)}\big((-1)^{n+r}
t \big)
+ (-1)^{(n+1)(r+1)} v\big(tx(t^2+x^2)^{-1}x^r\big) \bar
e^{(n+1)}\big((-1)^{n+r+1} t\big)\\
- (-1)^{(n+1)(r+1)} v\big(tx(t^2+x^2)^{-1}x^r\big) \bar
e^{(n+1)}\big((-1)^{n+r+1} t\big)+
(-1)^{n(r+1)} v\big(x^2(t^2+x^2)^{-1}x^r\big) \bar  e^{(n+1)}\big((-1)^{n+r}
t \big).
\end{multline*}
After making obvious cancellations, this is equal to the right hand
side of \cref{polpol}.
The reverse implication, that is, the deduction of \cref{relsbad} from
\cref{polpol} is similar: one starts with the right hand side of
\cref{relsbad} then uses \cref{polpol} to commute the $\bar
e^{(n+1)}(\pm t)$ term to the left.
So the relations \cref{newrels1} and
\cref{relsbad} are equivalent. Then we prove that \cref{relsbad} and
\cref{newrelsfancy2} are equivalent using the identities
\begin{align}\label{magic}
\eps^{(n)}(t) &= \frac{\sqi^{1-n} e^{(n)}(\sqi t)+\sqi^{n} e^{(n)}(-\sqi t)}{1 + \sqi},&
e^{(n)}(t) &= \frac{\sqi^{1-n} \eps^{(n)}(\sqi t)+\sqi^{n} \eps^{(n)}(-\sqi t)}{1 + \sqi}
\end{align}
where $\sqi\in \k$ denotes a square root of $-1$; these may be verified by equating coefficients on each side.
To pass from \cref{relsbad} to \cref{newrelsfancy2},
we replace $\eps^{(n)}(t)$ in \cref{newrelsfancy2} with this
linear combination of
$e^{(n)}(\pm \sqi t)$, then use \cref{relsbad} 
with $t$ replaced by $\pm \sqi t$ to
commute $e^{(n)}(\pm \sqi t)$ to the right. At the end,
a linear factor in the denominator
$t^2-x^2 = (t-x)(t+x)$ cancels, and after that one converts
back to $\eps^{(n)}(\pm t)$ using \cref{magic} again.
This calculation is quite lengthy but elementary. 
The argument can be reversed to obtain \cref{relsbad} from
\cref{newrelsfancy2}, hence, the equivalence.

Next, 
we first check that the images of the relations \cref{newrels1} and
\cref{newrelsfancy2} hold\footnote{Since \cref{newrels1} and
  \cref{newrelsfancy2} are equivalent, 
we really only need to check one of these. We check
  both because it is easy and explains how we discovered the
  relations in the first place.} for the
actions of $\EOH_n^\ell$ and $\EOH_{n+1}^\ell$ on $V_n^\ell$.
For \cref{newrels1}, it suffices to check that
$v_n(1) \bar e^{(n+1)}(t) = (-1)^n \bar e^{(n)}\big((-1)^n t\big)
v_n(t-x)$, i.e., the $r=0$ case,
for then we can act on the left with $x_1^r \in \ONH_1$ using 
\cref{lemma2} to deduce the more general formulae.
This follows because
\begin{equation}\label{old1}
\tv_n(1) \bar e^{(n+1)}(t) = \bar e^{(n)}(t) \tv_n(t-x)
\end{equation}
in $\tV_n^\ell[t] = 
\OSym_{(n,1,n')} \otimes_{\OSym_\ell} R_\ell[t]$
since, by the definition of the actions and \cref{stupid}, both sides 
are equal to $(t-x_1)\cdots (t-x_n) (t-x_{n+1}) \otimes 1$.
Similarly, for \cref{newrelsfancy2}, it suffices
to check that
$\bar \eps^{(n)}(t) v_n(1) = (-1)^n v_n\big(\big((-1)^n
t-x\big)^{-1}\big) \bar \eps^{(n+1)}\big((-1)^n t\big)$
or, equivalently,
$(-1)^n \bar \eps^{(n)}\big((-1)^n t\big) v_n(1) =
v_n\big((t-x)^{-1}\big)
\bar \eps^{(n+1)}(t)$.
This follows because
\begin{equation}\label{old2}
\bar \eps^{(n)} \tv_n(1) = \tv_n\big((t-x)^{-1}\big) \bar
\eps^{(n+1)}(t)
\end{equation}
in $\tV_n^\ell[t] = 
\OSym_{(n,1,n')} \otimes_{\OSym_\ell} R_\ell[t]$
since both sides are equal to $(t-x_n)\cdots (t-x_1) \otimes 1$ thanks to \cref{stupid}.


The relations check made in the previous paragraph 
implies that there is a well-defined
graded superbimodule homomorphism
$V / T
\rightarrow V_{n}^\ell$ taking $v\big(f(x)\big)+T$ to $v_{n}\big(f(x)\big)$
for all $f(x)\in \k[x]$. 
Moreover, this map is surjective by (1a)--(1b).
To show that it is an isomorphism, 
it suffices to show that
 $V/T$
 is generated as a right $\EOH_{n+1}^\ell$-module
 by the vectors $v(1)+T,v(x)+T,\dots, v(x^n)+T$.
 It is generated as a superbimodule by all 
 $ v(x^r)+T\:(r \geq 0)$.
 The relation \cref{newrelsfancy2} implies that 
 any $\bar a v\big(f(x)\big)$ for $a \in \OSym_n, f(x) \in \k[x]$ can
 be expanded as a linear combination of vectors of the form $v(g(x))
 \bar b$ for $b \in \OSym_{n+1},
 g(x) \in \k[x]$. 
 Hence, $V/T$ is generated just as a right $\EOH_{n+1}^\ell$-supermodule
 by the vectors $v(x^r)+T\:(r \geq 0)$.
 Now we let $V'$ be the right $\EOH_{n+1}^\ell$-submodule of $V$
 generated by $T$ and
the vectors $ v(1), v(x),\dots, v(x^n)$,
and complete the argument by
showing
that $v(x^{n+r}) \in V'$
by induction on $r=0,1,2,\dots$.
 The base $r=0$ is vacuous. 
 For the induction step, take $r > 0$. Consider the relation arising from the $t^{-1}$-coefficients in \cref{newrelsfancy2}. The left hand side is a polynomial, so this coefficient is zero on the 
 left hand side.
 Hence, the $t^{-1}$-coefficient of the right hand side belongs to
 $T\subseteq V'$.
Working out this coefficient explicitly reveals that it equals
 $\pm v(x^{n+1+r})$ plus a linear combination of terms of the form
$ v(x^s) \bar a$ for $0 \leq s \leq n+r$ and $a \in \OSym_{n+1}$
of positive degree. All of these ``lower terms" are in $V'$ by induction, hence, $v(x^{n+1+r}) \in V'$.
 
Finally, to establish (1c),
looking at the $t^n$-coefficients of \cref{newrels1}
shows that $v(x^{r+1})+T$ lies in the sub-bimodule generated
by $v(x^r)+T$ for any $r \geq 0$.
Hence, by another induction on $r$, the sub-bimodule 
of $V / T$ generated by $v(1)+T$ contains all $v(x^r)+T$.

\vspace{1mm}
\noindent
(2)
This follows by a similar argument to the proof of (1). We just
explain how to see that the relations \cref{rels1b,relsfancy3} both hold
for the actions of $\EOH_n^\ell$ and $\EOH_{n+1}^\ell$
on $U_n^\ell$.
For \cref{rels1b}, it suffices to check it in the case that $r=0$, then one can act on the right with $x^r$ 
using \cref{lemma2} to get the general result.
To prove it when $r=0$, it suffices to show that
$\bar e^{(n+1)}(t) u_n(1) = (-1)^n u_n(t-x)
\bar e^{(n)}(-t)$.
This follows because, in view of the signs in the definition
of the left and right actions in \cref{lemma0}(2) and also \cref{fancynotation2},
both sides are equal to
$\big(t-(-1)^{n'}x_{n'+1}\big) 
\big(t-(-1)^{n'}x_{n'+2}\big)
\cdots \big(t-(-1)^{n'}x_{\ell}\big)$.
Similarly, to prove \cref{relsfancy3},
it suffices to check the case
$r=0$, which amounts to showing that
$u_n(1) \bar \eps^{(n)}(t) = 
(-1)^{n+1} \bar\eps^{(n+1)}(-t) u_n\big((t+x)^{-1}\big)$.
This follows because both sides are equal to
$\big(t+(-1)^{n'} x_{\ell}\big)\cdots
\big(t+(-1)^{n'} x_{n'+2}\big)$.
\end{proof}

The final lemma in this section is an application of
the presentation for $V_n^\ell$ derived in \cref{lemma3}(1).
Recall the graded $R_\ell$-superalgebra automorphism
$\delta_n^\ell:\EOH_n^\ell \rightarrow
\EOH_n^\ell$
from \cref{deltadef}.
In terms of generating functions, we have that
\begin{align}
\delta_n^\ell\big(\bar e^{(n)}(t)\big)
&= (-1)^{\ell n} \bar e^{(n)}\big((-1)^\ell t\big) - (-1)^{\ell n}
t^{-1} \dot o^{(\ell)}\bar e^{(n)}\big((-1)^\ell t\big)  
+(-1)^{\ell n} t^{-1} \dot o^{(\ell)}\bar e^{(n)}\big((-1)^{\ell+1} t\big).\label{puky}
\end{align}
This follows from \cref{yuky} on
equating coefficients of $t$.

\begin{lemma}\label{crazycrazy}
There is a unique even degree 0 graded $R_\ell$-supermodule automorphism
$\phi_n^\ell:V_n^\ell \stackrel{\sim}{\rightarrow}
V_n^\ell$ such that
$\phi_n^\ell\big(v_n(1)\big) = v_n(1)$ and
$\phi_n^\ell\big(a v b\big) 
=\delta_n^\ell(a) \phi_n^\ell(v) \delta_{n+1}^\ell(b)$
for all $a \in \EOH_n^\ell, b \in \EOH_{n+1}^\ell$
and $v \in V_n^\ell$.
Moreover, we have that
\begin{align}\label{events}
\phi_n^\ell\big(v_n(x^r)\big) &= 
(-1)^{\ell r}v_n(x^r) + (-1)^{(\ell+1)(r-1)} \big(1-(-1)^r\big) \dot
                                o^{(\ell)} v_n(x^{r-1})
\end{align}
for all $r \geq 0$.
\end{lemma}

\begin{proof}
The uniqueness is clear since $V_n^\ell$ is a cyclic
superbimodule thanks to \cref{lemma3}(1c).
To prove existence,
consider the $\big(\EOH_n^\ell,
\EOH_{n+1}^\ell\big)$-superbimodule $V$ that is
$V_n^\ell$ with the left and right actions of $\EOH_n^\ell$ and $\EOH_{n+1}^\ell$
defined by twisting the usual actions with
the automorphisms $\delta_n^\ell$ and $\delta_{n+1}^\ell$,
respectively. We are trying to show that there is a superbimodule homomorphism
$\phi_n^\ell:V_n^\ell \rightarrow V$
satisfying \cref{events}.
We can check this using the presentation for the superbimodule 
$V_n^\ell$ from \cref{lemma3}(1) with the defining
relations \cref{newrels1}.
This reduces the proof to checking that the identity
\begin{equation}\label{goodtodo}
\phi_n^\ell\big(v_n(x^r)\big) \delta_{n+1}^\ell\big( \bar
e^{(n+1)}(t)\big)
=  (-1)^{n(r+1)}
\delta_n^\ell\Big(\bar e^{(n)}\big((-1)^{n+r} t \big) \Big)
\phi_n^\ell\big(v_n\big((t-x)x^r\big)\Big)
\end{equation}
is satisfied in the superbimodule 
$V_n^\ell$
for any $r \geq 0$.
Substituting the formulae from \cref{puky,events} into \cref{goodtodo} and
cancelling $(-1)^{\ell r+\ell (n+1)}$ from both sides,
this expands to the equation
\begin{multline*}
\Big(v_n(x^r) - (-1)^{\ell + r} \big(1-(-1)^r\big) \dot o^{(\ell)}v_n(x^{r-1}) 
\Big)
\Big(
\bar e^{(n+1)}\big((-1)^\ell t\big) -t^{-1} \dot o^{(\ell)} \bar e^{(n+1)}\big((-1)^\ell t\big)
+ t^{-1} \dot o^{(\ell)}\bar e^{(n+1)}\big((-1)^{\ell+1} t\big) 
\Big)\\\quad= 
\Big((-1)^{n(r+1)}\bar e^{(n)}\big((-1)^{\ell+n+r} t\big) -(-1)^{(n+1)r} t^{-1}\dot o^{(\ell)} \bar e^{(n)}\big((-1)^{\ell+n+r} t\big)+(-1)^{(n+1)r} t^{-1} \dot o^{(\ell)}\bar e^{(n)}\big((-1)^{\ell+n+r+1} t\big) 
\Big)\\\times\:
\Big(v_n\big(\big((-1)^\ell t - x)x^r\big) - (-1)^{r}  \big(1-(-1)^r\big) t \dot o^{(\ell)}
v_n(x^{r-1})
- (-1)^{\ell+r} \big(1+(-1)^r\big) \dot o^{(\ell)} v_n(x^r)
\Big).
\end{multline*}
Now we expand both sides using that
$\big(\dot o^{(\ell)}\big)^2 = 0$. Commuting $\dot o^{(\ell)}$ then
$\bar e^{(n+1)}(\pm t)$ to the left with \cref{newrels1},
the left hand side contributes the sum of the following four terms:
\begin{align*}
v_n(x^r)\bar e^{(n+1)}\big((-1)^\ell t\big)&=(-1)^{n(r+1)}
\bar e^{(n)}\big((-1)^{\ell+n+r} t\big)v_n\big(\big((-1)^\ell t - x)x^r\big), \\
 - (-1)^{\ell + r} \big(1-(-1)^r\big) \dot o^{(\ell)}
  v_n(x^{r-1})\bar e^{(n+1)}\big((-1)^\ell t\big)&=
-(-1)^{(n+1)r+\ell} \big(1-(-1)^r\big) \dot o^{(\ell)}
\bar e^{(n)}\big((-1)^{\ell+n+r+1} t\big)\\&\hspace{50mm}\times v_n\big(\big((-1)^\ell t -
                                                   x\big) x^{r-1}\big),\\
-(-1)^{n+r} t^{-1}\dot o^{(\ell)} v_n(x^r) \bar e^{(n+1)}\big((-1)^\ell
  t\big)&=
-(-1)^{(n+1)r} t^{-1}\dot o^{(\ell)}         \bar e^{(n)}\big((-1)^{\ell+n+r}
          t\big)v_n\big(((-1)^\ell t - x)x^r\big),
\\
(-1)^{n+r} t^{-1} \dot o^{(\ell)} v_n(x^r)  \bar e^{(n+1)}\big((-1)^{\ell+1} t\big)
 &
= -(-1)^{(n+1)r} t^{-1} \dot o^{(\ell)} \bar
                   e^{(n)}\big((-1)^{\ell+n+r+1} t\big)
                   v_n\big(((-1)^\ell t + x) x^r\big).
\end{align*}
Commuting $\dot o^{(\ell)}$ to the left, the right hand side
contributes the sum of the following five terms:
\begin{align*}
(-1)^{n(r+1)}
\bar e^{(n)}\big((-1)^{\ell+n+r} t\big)v_n\big(\big((-1)^\ell t - x)x^r\big),&\\
-(-1)^{(n+1)r} t^{-1}\dot o^{(\ell)} \bar e^{(n)}\big((-1)^{\ell+n+r} t\big)
v_n\big(\big((-1)^\ell t - x)x^r\big),&\\
(-1)^{(n+1)r} t^{-1} \dot o^{(\ell)} \bar e^{(n)}\big((-1)^{\ell+n+r+1}
  t\big) v_n\big(\big((-1)^\ell t - x)x^r\big),\\
-(-1)^{(n+1)r}  \big(1-(-1)^r\big) t \dot o^{(\ell)}\bar
  e^{(n)}\big((-1)^{\ell+n+r+1} t\big)v_n(x^{r-1}),
&\\
-(-1)^{(n+1)r+\ell}\big(1+(-1)^r\big) \dot o^{(\ell)} \bar
  e^{(n)}\big((-1)^{\ell+n+r+1} t\big) v_n(x^r).&
\end{align*}
From this, without making any  further commutations, it follows that the two sides are equal.
\end{proof}

\section{Rigidity \texorpdfstring{of $\OGBim_\ell$}{}}

In this section, we prove that the $2$-supercategory $\OGBim_\ell$
from \cref{shy} is  
{\em rigid} in the sense that all of its 1-morphisms have
left and right duals.
Given a formal Laurent series $f(t)$, we use the notation $\big[f(t)\big]_{t^r}$ 
to denote its $t^r$-coefficient.
Similarly, we use $\big[f(t) \big]_{t^{\leq r}}$,
$\big[f(t) \big]_{t^{\geq r}}$, etc. for the 
formal Laurent series obtained by keeping only the terms with the 
specified powers of $t$.
Note also the following elementary identity:
we have that
\begin{align}\label{why}
f(x) &= \left[(t- x)^{-1} f(t) \right]_{t^{-1}}
\end{align}
for any polynomial $f(x) \in \k[x]$.
For the first lemma, recall the notation \cref{thisisnotbad1,thisisnotbad2}.

\begin{lemma}\label{simpson}
Suppose that $\ell=n+1+n'$
and $0 \leq r,s \leq n$.
\begin{enumerate}
\item
$\displaystyle
v_n\big((t-x)^{-1} x^r\big)
= 
\sum_{p=r}^n v_n(x^p) \left[\bar\eps^{(n+1)}(t)\right]_{> t^p}
\bar\gamm^{(n+1)}(t) t^{r-p-1}
-\sum_{p=0}^{r-1} v_n(x^p) \left[\bar\eps^{(n+1)}(t)\right]_{\leq t^p}
\bar\gamm^{(n+1)}(t) t^{r-p-1}$.
\item
$\displaystyle
u_n\big((t-x)^{-1} x^s\big)
= 
\sum_{q=s}^n \bar\gamm^{(n+1)}(t)  \left[\bar\eps^{(n+1)}(t)\right]_{> t^q}
u_n(x^q)t^{s-q-1}
-\sum_{q=0}^{s-1} \bar\gamm^{(n+1)}(t) \left[\bar\eps^{(n+1)}(t)\right]_{\leq t^q}
u_n(x^q) t^{s-q-1}$.
\end{enumerate}
\end{lemma}

\begin{proof}
(1)
Remembering \cref{thisisnotbad1} and the definition of the right action of $\EOH_{n+1}^\ell$ from \cref{lemma0}(1), the identity obtained by applying
$\gamma_{n+1}$ to \cref{neweasypeasy}(2) implies for any $m \geq 0$ that
$$
v_n\big(x^{m+n+1}\big)
=
-\sum_{p=0}^n v_n\big(x^p\big)
\sum_{s=0}^m (-1)^{m+n+1-p-s} \bar\eps_{m+n+1-p-s}^{(n+1)}
\bar\gamm_s^{(n+1)}.
$$
We multiply this by $t^{r-m-n-2}$ and sum over $m \geq 0$ to obtain
$$
\sum_{m \geq 0} v_n\big(x^{m+n+1}\big)
t^{r-m-n-2}
= 
-\sum_{p=0}^n 
v_n\big(x^p\big)
\left(
\sum_{m \geq 0}
\sum_{s=0}^m (-1)^{m+n+1-p-s} \bar\eps^{(n+1)}_{m+n+1-p-s} \bar\eta^{(n+1)}_s t^{p-m-n-1}
\right)
t^{r-p-1}.
$$
The expression in brackets is equal to $\left[\bar\eps^{(n+1)}(t)\right]_{\leq t^p} \bar\gamm^{(n+1)}(t)$,
as may be checked by comparing $t^{p-m-n-1}$-coefficients for all $m \in \Z$.
Using this and \cref{newigr}, we obtain
\begin{align*}
v_n\big((t-x)^{-1} x^r\big)
&= \sum_{p=r}^n v_n\big(x^p\big) t^{r-p-1}
+ \sum_{m \geq 0} v_n\big(x^{m+n+1}\big)t^{r-m-n-2}\\
&= 
\sum_{p=r}^n v_n\big(x^p\big) \bar\eps^{(n+1)}(t) \bar \gamm^{(n+1)}(t) t^{r-p-1}
-\sum_{p=0}^n 
v_n\big(x^p\big)\left[\bar\eps^{(n+1)}(t)\right]_{\leq t^p} \bar\gamma^{(n+1)}(t)
t^{r-p-1}.
\end{align*}
It remains to write the first
$\bar\eps^{(n+1)}(t)$
as
$\left[\bar\eps^{(n+1)}(t)\right]_{> t^p}
+\left[\bar\eps^{(n+1)}(t)\right]_{\leq t^p}$
and make some obvious cancellations to obtain the desired formula.

\vspace{1mm}
\noindent
(2)
Remembering the sign in \cref{thisisnotbad2} 
and the matching sign in the definition of the left action of $\EOH_{n+1}^\ell$ from \cref{lemma0}(2), the identity obtained by
applying $\operatorname{p}^{n'} \circ \sig_{n'} \circ \smiley_{n+1}$ to 
\cref{neweasypeasy}(1)
gives
$$
u_n\big(x^{m+n+1}\big)
=
-\sum_{q=0}^n 
\sum_{r=0}^m 
(-1)^{m+n+1-q-r} 
\bar\gamm_r^{(n+1)}\bar\eps_{m+n+1-q-r}^{(n+1)}
u_n\big(x^q\big)
$$
for any $m \geq 0$.
We multiply this by $t^{s-m-n-2}$ and sum over $m \geq 0$ to obtain
$$
\sum_{m \geq 0} u_n\big(x^{m+n+1}\big)
t^{s-m-n-2}
= 
-\sum_{q=0}^n 
\left(
\sum_{m \geq 0}
\sum_{r=0}^m (-1)^{m+n+1-q-r} \bar\gamm^{(n+1)}_r \bar\eps^{(n+1)}_{m+n+1-q-r} t^{q-m-n-1}
\right)
u_n\big(x^q\big)
t^{s-q-1}.
$$
The expression in brackets is equal to $\bar\gamma^{(n+1)}(t)\left[\bar\eps^{(n+1)}(t)\right]_{\leq t^q} $.
Now the proof is completed as in (1).
\end{proof}

\begin{lemma}\label{violins}
Suppose that $\ell=n+1+n'$.
The element
$$
z :=
\sum_{\substack{r,s \geq 0 \\ r+s \leq n}} (-1)^{r+s}
v_n\big(x^{r}\big) \bar \eps^{(n+1)}_{n-r-s} \otimes u_n\big(x^{s}\big)
\in V_n^\ell\otimes_{\EOH_{n+1}^\ell} U_n^\ell
$$
is central in the sense that $az = za$ for all $a \in \EOH_n^\ell$.
\end{lemma}

\begin{proof}
By \cref{lemma3}(b), 
any element of $V_n^\ell \otimes_{\EOH_{n+1}^\ell} U_n^\ell$
can be written as $\sum_{p,q =0}^n v_n\big(x^p) f_{p,q} \otimes u_n\big(x^q\big)$ for unique $f_{p,q} \in \EOH_{n+1}^\ell$.
So there are unique 
$L_{p,q}(t) = \sum_{k=0}^n L_{p,q;k} t^k, 
R_{p,q}(t) =\sum_{k=0}^n R_{p,q;k} t^k 
\in \EOH_{n+1}^\ell [t]$
such that
\begin{align*}
\bar\eps^{(n)}(t) z &= 
\sum_{p,q =0}^n (-1)^{p+q} v_n\big(x^p) L_{p,q}(t) \otimes u_n\big(x^q\big),&
z \bar\eps^{(n)}(t)
&= 
\sum_{p,q =0}^n (-1)^{p+q} v_n\big(x^p) R_{p,q}(t) \otimes u_n\big(x^q\big).
\end{align*}
To prove the lemma,
it suffices to show that
$\bar\eps^{(n)}(t) z = z \bar\eps^{(n)}(t)$,
which we do by computing $L_{p,q}(t)$ and
$R_{p,q}(t)$ explicitly then checking that
$L_{p,q;k} = R_{p,q;k}$ for all $0 \leq p,q,k \leq n$.
For brevity, we adopt the convention that
$\eps^{(n+1)}_r = 0$ for $r < 0$; this allows the
restriction $r+s \leq n$ on the summation in the definition
of $z$ to be omitted.

To compute $L_{p,q}(t)$, we expand
$\bar\eps^{(n)}(t) z$ using \cref{newrelsfancy2}, \cref{simpson}(1) and \cref{newigr}:
\begin{align*}
\bar\eps^{(n)}(t) z &= 
\sum_{r,q \geq 0} (-1)^{r+q}
\bar\eps^{(n)}(t) v_n\big(x^{r}\big) \bar \eps^{(n+1)}_{n-r-q} \otimes u_n\big(x^{q}\big)\\
&=
\sum_{r,q \geq 0}(-1)^{r+q+n(r+1)} v_n\big(((-1)^{n+r}t- x)^{-1} x^r\big)\,
\bar \eps^{(n+1)}\big((-1)^{n+r}t\big)\bar \eps^{(n+1)}_{n-r-q} \otimes u_n\big(x^{q}\big)\\
&=
\sum_{r,q \geq 0}\sum_{p=r}^n v_n(x^p)\left(
(-1)^{r+q+n(r+1)+(n+r)(r-p-1)}
\left[\bar\eps^{(n+1)}\big((-1)^{n+r}t\big)\right]_{> t^p}
\bar \eps^{(n+1)}_{n-r-q}t^{r-p-1}\right) \otimes u_n\big(x^{q}\big)\\
&\qquad-
\sum_{r,q \geq 0}
\sum_{p=0}^{r-1} 
v_n(x^p)
\left(
(-1)^{r+q+n(r+1)+(n+r)(r-p-1)} \left[\bar\eps^{(n+1)}\big((-1)^{n+r}t\big)\right]_{\leq t^p}
\bar \eps^{(n+1)}_{n-r-q}t^{r-p-1}\right) \otimes u_n\big(x^{q}\big).
\end{align*}
Hence, using $r+q+n(r+1)+(n+r)(r-p-1)+p+q\equiv  
pr+pn+p+r\pmod{2}$ to simplify the sign,
we have that
\begin{multline}\label{adele}
L_{p,q}(t)
=
\sum_{r=0}^{p}
(-1)^{pr+pn+p+r}
\left[\bar\eps^{(n+1)}\big((-1)^{n+r}t\big)\right]_{> t^p}
\bar \eps^{(n+1)}_{n-r-q}t^{r-p-1}\\-
\sum_{r\geq p+1}
(-1)^{pr+pn+p+r}
\left[\bar\eps^{(n+1)}\big((-1)^{n+r}t\big)\right]_{\leq t^p}
\bar \eps^{(n+1)}_{n-r-q}t^{r-p-1}
\end{multline}
for $0 \leq p,q \leq n$.
Similarly, we compute $R_{p,q}(t)$ using \cref{relsfancy3}, \cref{simpson}(2) and \cref{newigr}:
\begin{align*}
z\bar\eps^{(n)}(t) &= 
\sum_{p,s \geq 0} (-1)^{p+s}
v_n\big(x^{p}\big) \bar \eps^{(n+1)}_{n-p-s} \otimes u_n\big(x^{s}\big)
\bar\eps^{(n)}(t) \\
&=
\sum_{p,s \geq 0 } (-1)^{p+s+n(s+1)} 
v_n\big(x^{p}\big) \bar \eps^{(n+1)}_{n-p-s} 
\bar\eps^{(n+1)}\big((-1)^{s+1} t\big)\otimes 
u_n\big(((-1)^{s+1}t-x)^{-1} x^s\big)\\
&=
\sum_{p,s \geq 0} \sum_{q=s}^n v_n\big(x^{p}\big) \left((-1)^{p+s+n(s+1)+(s+1)(s-q-1)} 
\bar \eps^{(n+1)}_{n-p-s} 
\left[\bar\eps^{(n+1)}\big((-1)^{s+1}t\big)\right]_{> t^q}
t^{s-q-1}\right)\otimes  u_n(x^q)
\\
&\qquad-\sum_{p,s \geq 0} \sum_{q=0}^{s-1} 
v_n\big(x^{p}\big) \left((-1)^{p+s+n(s+1)+(s+1)(s-q-1)} \bar \eps^{(n+1)}_{n-p-s} 
\left[\bar\eps^{(n+1)}\big((-1)^{s+1}t\big)\right]_{\leq t^q}
 t^{s-q-1}\right)\otimes u_n(x^q).
\end{align*}
From this, we get that
\begin{multline}\label{sinead}
R_{p,q}(t) = 
\sum_{s=0}^{q}
(-1)^{ns+qs+n+1} 
\bar \eps^{(n+1)}_{n-p-s} 
\left[\bar\eps^{(n+1)}\big((-1)^{s+1}t\big)\right]_{> t^q}
t^{s-q-1}\\
-
\sum_{s\geq q+1} (-1)^{ns+qs+n+1} 
\bar \eps^{(n+1)}_{n-p-s}
\left[\bar\eps^{(n+1)}\big((-1)^{s+1}t\big)\right]_{\leq t^q}
 t^{s-q-1}
\end{multline}
for $0 \leq p,q \leq n$.

Now we use \cref{adele,sinead} to compute the $t^k$-coefficients 
$L_{p,q;k}$ 
and $R_{p,q;k}$ for $0 \leq k \leq n$:
\begin{align}\label{rumour}
L_{p,q;k} &=
\sum_{r=0}^k (-1)^{(n+k)(r+k)} \bar\eps^{(n+1)}_{n+r-p-k} \bar\eps^{(n+1)}_{n-r-q}
-\sum_{r\geq p+1} (-1)^{(n+k)(r+k)} \bar\eps^{(n+1)}_{n+r-p-k} \bar\eps^{(n+1)}_{n-r-q},\\
R_{p,q;k} &= \sum_{s=0}^k (-1)^{(n+k)s} \bar\eps^{(n+1)}_{n-p-s} \bar\eps^{(n+1)}_{n+s-q-k}-
\sum_{s\geq q+1} (-1)^{(n+k)s} \bar\eps^{(n+1)}_{n-p-s} \bar\eps^{(n+1)}_{n+s-q-k}.\label{hasit}
\end{align}
The details of these two computations are very similar, so we just elaborate
on the first one.
Note that
$$
(-1)^{pr+pn+p+r}
\bar\eps^{(n+1)}\big((-1)^{n+r}t\big)
\bar \eps^{(n+1)}_{n-r-q}t^{r-p-1}
= 
\sum_{j\in\Z}
(-1)^{pr+pn+p+r+(n+r)j+n+1+j}
\bar\eps^{(n+1)}_{n+1-j}
\bar \eps^{(n+1)}_{n-r-q}t^{j+r-p-1}.
$$
To get a contribution of $t^k$ from this, 
we must have that $j = k+p-r+1$,
in which case
$$
(-1)^{pr+pn+p+r+(n+r)j+n+1+j}
\bar\eps^{(n+1)}_{n+1-j}
\bar \eps^{(n+1)}_{n-r-q}=
(-1)^{(n+k)(r+k)} 
\bar\eps^{(n+1)}_{n+r-p-k}
\bar \eps^{(n+1)}_{n-r-q}.
$$
Using these observations, it follows
 that
the $t^k$-coefficients from the
first summation in \cref{adele}
contribute
$\sum
(-1)^{(n+k)(r+k)}
\bar\eps^{(n+1)}_{n+r-p-k}
\bar \eps^{(n+1)}_{n-r-q}$
summing over $r$ with $0 \leq r \leq p$
such that $j := k+p-r+1$ satisfies $j > p$,
i.e., $0 \leq r \leq \min(k,p)$.
Similarly, the $t^k$-coefficients from the second
summation in \cref{adele} contribute
$-\sum
(-1)^{(n+k)(r+k)}
\bar\eps^{(n+1)}_{n+r-p-k}
\bar \eps^{(n+1)}_{n-r-q}$
summing over $r$ with $r \geq p+1$
such that $j := k+p-r+1$ satisfies $j \leq p$,
i.e., $r \geq \max(k,p)+1$.
Thus, we have shown that
$$
L_{p,q;k} =
\sum_{r=0}^{\min(k,p)} (-1)^{(n+k)(r+k)} \bar\eps^{(n+1)}_{n+r-p-k} \bar\eps^{(n+1)}_{n-r-q}
-\sum_{r\geq \max(k,p)+1} 
(-1)^{(n+k)(r+k)} \bar\eps^{(n+1)}_{n+r-p-k} \bar\eps^{(n+1)}_{n-r-q}.
$$
This is exactly as in \cref{rumour} when $k \leq p$.
To see that it is also 
equal to \cref{rumour} when $k > p$, one just has to cancel
the overlapping terms 
when $p+1 \leq r \leq k$
in the first and second summations from \cref{rumour}.

It remains to see that $L_{p,q;k} = R_{p,q;k}$ for every $0 \leq p,q,k \leq n$.
In fact, the first summation in \cref{rumour} is equal to the first summation in 
\cref{hasit}. This is easily seen on making the substitution $s=k-r$ in one of them. To see that the second summation in \cref{rumour} is equal to the second summation in \cref{hasit}, we substitute $m=r-p$ in \cref{rumour} and $m=s-q$ in \cref{hasit}, and the problem reduces to showing that
$$
\sum_{m \geq 1} (-1)^{(n+k)(m+p+k)} \bar\eps^{(n+1)}_{n-k+m} \bar\eps^{(n+1)}_{n-p-q-m}
=\sum_{m \geq 1} (-1)^{(n+k)(m+q)} \bar\eps^{(n+1)}_{n-p-q-m} \bar\eps^{(n+1)}_{n-k+m}
$$
for all $0 \leq p,q,k \leq n$.
In fact this equality already holds in $\OSym_{n+1}$ thanks to 
\cref{badday} or, rather, the identity obtained from that by applying the involution $\smiley_{n+1}$.
\end{proof}

\begin{theorem}\label{cupsandcaps}
Suppose that $\ell=n+1+n'$.
There are unique even degree 0 superbimodule homomorphisms
\begin{align}
\coev_n&:\EOH_n^\ell \rightarrow V_{n}^\ell \otimes_{\EOH_{n+1}^\ell}
U_{n}^\ell\notag\\\label{notsobad}
1 &\mapsto 
\sum_{\substack{r,s \geq 0 \\ r+s \leq n}} (-1)^{n-r-s}
v_n\big(x^{r}\big) \bar \eps^{(n+1)}_{n-r-s} \otimes u_n\big(x^{s}\big)\\\intertext{and}
\ev_n&:U_{n}^\ell \otimes_{\EOH_n^\ell}
V_{n}^\ell \rightarrow \EOH_{n+1}^\ell\notag
\\\label{reallynot}
u_n\big(x^{r}\big) \otimes v_n\big(x^{s}\big)
&\mapsto
\left\{
\begin{array}{ll}
\bar \gamm^{(n+1)}_{r+s-n}&\text{if $r+s\geq n$}\\
0&\text{otherwise,}
\end{array}\right.
\end{align}
the latter being true for all $r,s \geq 0$.
Moreover, the following compositions are identities:
\begin{equation}\label{zigzag1}
\begin{tikzcd}
U_n^\ell
\arrow[r,"\operatorname{can}"]
&
U_n^\ell\otimes_{\EOH_n^\ell} \EOH_n^\ell
\arrow[r,"\id\otimes\coev_n"]&
U_n^\ell\otimes_{\EOH_n^\ell} V_n^\ell \otimes_{\EOH_{n+1}^\ell}
U_n^\ell
\arrow[r,"\ev_n\otimes\id"]
&
\EOH_{n+1}^\ell\otimes_{EOH_{n+1}^\ell}
U_n^\ell \arrow[r,"\operatorname{can}"]&
U_n^\ell,
\end{tikzcd}
\end{equation}\begin{equation}\label{zigzag2}
\begin{tikzcd}
V_n^\ell
\arrow[r,"\operatorname{can}"]
&
\EOH_n^\ell\otimes_{\EOH_n^\ell} V_n^\ell
\arrow[r,"\coev_n\otimes\id"]&
V_n^\ell\otimes_{\EOH_{n+1}^\ell} U_n^\ell \otimes_{\EOH_{n}^\ell}
V_n^\ell
\arrow[r,"\id\otimes\ev_n"]
&
V_n^\ell\otimes_{\EOH_{n+1}^\ell}
\EOH_{n+1}^\ell
\arrow[r,"\operatorname{can}"]&
V_n^\ell.
\end{tikzcd}
\end{equation}
Hence, $\coev_n$ and $\ev_n$ are the unit and counit of an adjunction making $\left(U_{n}^\ell, V_{n}^\ell\right)$
into a dual pair 
of 1-morphisms in $\OGBim_\ell$.
\end{theorem}

Before proving the theorem, we write down several equivalent formulations of the
definitions of $\coev_n$ and $\ev_n$, assuming that such superbimodule homomorphisms do indeed exist.
For the unit of adjunction, the element $\coev_n(1)$ in the statement of \cref{cupsandcaps}
is equal to $(-1)^n z$ where $z$ is the central tensor from \cref{violins}.
In terms of generating functions, we have that
\begin{equation}
\coev_n(1) =
\left[v_n((t-x)^{-1})  \bar \eps^{(n+1)}(t) 
\otimes u_n\big((t- x)^{-1}\big)
\right]_{t^{-1}}.\label{sundoor1}
\end{equation}
This is easily checked by computing coefficients of $t$.
Using \cref{relsfancy3,newrelsfancy2} 
with $r=0$, we have that
\begin{align}
v_n\big((t-x)^{-1}\big)\bar\eps^{(n+1)}(t)
&= \bar\eps^{(n)}\big((-1)^{n}t\big)v_n\big((-1)^{n}\big),\label{crap1}\\\label{crap2}
\bar\eps^{(n+1)}(t) u_{n}\big((t-x)^{-1}\big)
&= u_n\big((-1)^{n}\big) \bar\eps^{(n)}(-t).
\end{align}
Hence, we can rewrite the right hand side of \cref{sundoor1}
to obtain
\begin{align}\label{fridoor}
\coev_n(1) &=
\left[v_n((t-x)^{-1})\otimes u_n\big((-1)^n \big)\bar\eps^{(n)}(-t) \right]_{t^{-1}}=
\left[\bar \eps^{(n)}\big((-1)^{n}t\big)v_n\big((-1)^n \big)\otimes u_n\big((t- x)^{-1}\big)
\right]_{t^{-1}}.
\end{align}
Equating coefficients in \cref{fridoor} gives two more formulae:
\begin{align}\label{Jackdoor}
\coev_n(1) &
=
\sum_{r=0}^n v_n\big(x^{r}\big) \otimes u_n(1) \bar \eps^{(n)}_{n-r}
= \sum_{s=0}^n
(-1)^{(n+1)s}
\bar \eps^{(n)}_{n-s} 
v_n(1)\otimes u_n\big(x^{s}\big).
\end{align}
For the counit $\ev_n$, we have the following two equivalent formulations:
\begin{align}\label{ev1}
\ev_n\left(u_n\big((t- x)^{-1}f(x)\big) \otimes 
v_n\big(g(x)\big)\right)
&= \left[\bar \gamm^{(n+1)}(t) f(t)g(t)\right]_{<t^{0}}\\
\ev_n\left(u_n\big(f(x)\big) \otimes v_n\big(g(x)(t- x)^{-1}\big)\right)
&= \left[\bar \gamm^{(n+1)}(t) f(t)g(t)\right]_{<t^{0}}\label{ev2}
\end{align}
for any $f(x),g(x)\in\k[x]$. 
This can be checked by 
assuming that $f(x)=x^r, g(x) = x^s$ then equating coefficients of $t$.

\begin{proof}[Proof of \cref{cupsandcaps}]
We split the proof up into six steps.

\vspace{1mm}
\noindent{\em Step one.}
We first construct the maps $\coev_n$ and $\ev_n$. 
For $\coev_n$, we define
$\coev_n(a) := (-1)^n az$ for any $a \in \EOH_n^\ell$, 
where $z$ is as in \cref{violins}.
This is an even degree 0 homomorphism of graded left $\EOH_n^\ell$-supermodules.
Since $(-1)^n az = (-1)^n za$ according to \cref{violins},
it is also a right $\EOH_n^\ell$-supermodule homomorphism, so it is a 
superbimodule homomorphism. 
Thus, we have defined the superbimodule homomorphism $\coev_n$, and 
\cref{notsobad} holds.
For $\ev_n$, 
$U_n^\ell \otimes_{\EOH_n^\ell} V_n^\ell$ is a free graded right $\EOH_{n+1}^\ell$-supermodule with basis 
$u_n(x^r) \otimes v_n(x^s)\:(0 \leq r \leq n', 0 \leq s \leq n)$
by \cref{lemma3}(1b) and (2a).
So there is a unique even degree 0 
graded right $\EOH_{n+1}^\ell$-supermodule homomorphism $\ev_n$ such that 
\cref{reallynot} holds for $0 \leq r \leq n'$ and $0 \leq s \leq n$.
It is not yet clear that \cref{reallynot} holds for other values of $r$ and $s$,
or that $\ev_n$ is a graded left $\EOH_{n+1}^\ell$-supermodule homomorphism.

\vspace{1mm}
\noindent
{\em Step two.}
Next we use induction to show that 
\cref{reallynot} also holds for $0 \leq r \leq n'$ and all
$s > n$. 
Fix a choice of $r$ with $0 \leq r \leq n'$.
We know by our definition that \cref{reallynot} holds for $0 \leq s \leq n$.
For the induction step, we take some $s \geq n$, assume that \cref{reallynot}
holds for this and all smaller values of $s$, and show that it also holds 
when $s$ is replaced by $s+1$.
The $m=0$ case of \cref{neweasypeasy}(2) shows that
$x_1^{n+1} = \sum_{p=0}^n (-1)^{n-p} x_1^p e^{(n+1)}_{n+1-p}$.
Multliplying on the left 
by $x_1^{s-n}$ then applying $\smiley_{n+1}$, we deduce that
 $x_{n+1}^{s+1} = \sum_{p=0}^n(-1)^{n-p} x_{n+1}^{p+s-n} \eps_{n+1-p}^{(n+1)}$.
 So
 $$
 u_n\big(x^r\big) \otimes v_n\big(x^{s+1}\big)
 = 
 \sum_{p=0}^n (-1)^{n-p} u_n\big(x^r\big)\otimes v_n\big(x^{p+s-n}\big)
 \bar\eps^{(n+1)}_{n+1-p}.
 $$
Now we apply the right supermodule homomorphism 
$\ev_n$ using the induction hypothesis
to see that
$$
\ev_n\Big(
 u_n\big(x^r\big) \otimes v_n\big(x^{s+1}\big)\Big)
 =  \sum_{p=\max(0,2n-r-s)}^n (-1)^{n-p} \bar \eta_{p+r+s-2n}^{(n+1)}
 \bar\eps^{(n+1)}_{n+1-p}.
 $$
This is equal to $\bar \eta_{r+s+1-n}^{(n+1)}$ by \cref{newigr},
which is what we wanted.

\vspace{1mm}
\noindent
{\em Step three.}
Since $\ev_n$ is a right supermodule homomorphism by definition, 
the composition of maps in \cref{zigzag1} makes sense 
and is a right $\EOH_n^\ell$-supermodule homomorphism\footnote{
However, \cref{zigzag2} does not make sense at this point
since $\id \otimes \ev_n$ is not defined until we can shown that $\ev_n$
is a left supermodule homomorphism.}.
To show that the composition is equal to the identity map,
it suffices to show that it
takes $u_n(f(x))$ to $u_n(f(x))$ for any $f(x) \in \k[x]$ of degree $\leq n'$,
since these elements generate $U_n^\ell$ as a right $\EOH_n^\ell$-supermodule
by \cref{lemma3}(2a).
Using \cref{sundoor1}, we apply the even map 
$\id \otimes \coev_n$ to $u_n(f(x))\otimes 1$ 
to obtain
$$
\left[u_n\big(f(x)\big)\otimes v_n\big((t-x)^{-1}\big)
\otimes \bar \eps^{(n+1)}(t) 
u_n\big((t-x)^{-1}\big)\right]_{t^{-1}}.
$$
By step two, \cref{ev2} holds for $f(x)$ of degree $\leq n'$.
We use it to apply $\ev_n \otimes \id$, then multiply out the tensor,
to obtain
$$
\left[\big[\bar \gamm^{(n+1)}(t) f(t)\big]_{<t^{0}}
\bar \eps^{(n+1)}(t) u_n\big(( t-x)^{-1}\big)\right]_{t^{-1}}.
$$
Now, \cref{crap2} shows that $\bar \eps^{(n+1)}(t) u_n\big((t-x)^{-1}\big)$ is a polynomial in $t$, so we can omit the inside square brackets. Then
the $\gamm$ and $\eps$ cancel by \cref{newigr}, leaving us with
$$
\left[f(t) u_n\big((t-x)^{-1}\big)\right]_{t^{-1}} = 
u_n\big(f(x)\big),
$$
where we applied \cref{why} for the final equality. Hence, the composition \cref{zigzag1}
is the identity.

\vspace{1mm}
\noindent
{\em Step four.}
We prove that $\ev_n$ is a left supermodule homomorphism.
 Since $v_n(1)$ generates $V_n^\ell$ as a superbimodule (\cref{lemma3}(1c))
 and $\ev_n$ is a right supermodule homomorphism, it suffices to show that
 $a \ev_n(u \otimes v_n(1)) = \ev_n(au \otimes v_n(1))$ for all
 $a \in \EOH_{n+1}^\ell$ and $u \in U_n^\ell$.
By step three, we have that
 $$
 a \left((\ev_n \otimes \id) \circ (\id \otimes\coev_n)(u \otimes 1)\right) = a(1 \otimes u)
  = (1 \otimes au) = 
  (\ev_n \otimes \id) \circ (\id \otimes\coev_n)(au \otimes 1)
 $$
 in $\EOH_{n+1}^\ell\otimes_{\EOH_{n+1}^\ell} U_n^\ell$.
 Using the formula for $\coev_n(1)$ from \cref{Jackdoor}, this shows that
$$
\sum_{s=0}^n
(-1)^{(n+1)s}
a \ev_n\big(u \otimes \bar \eps^{(n)}_{n-s} 
v_n(1)\big)\otimes u_n\big(x^{s}\big)
=
\sum_{s=0}^n
(-1)^{(n+1)s}
 \ev_n\big(au \otimes \bar \eps^{(n)}_{n-s} 
v_n(1)\big)\otimes u_n\big(x^{s}\big).
$$
By \cref{lemma3}(2b), $U_n^\ell$ is a free left $\EOH_{n+1}^\ell$-supermodule
with basis $u_n(x^s)\:(0 \leq s \leq n)$, so we can project to
$-\otimes u_n(x^n)$-components in the identity just proved
to obtain the 
desired equality $a \ev_n\big(u \otimes v_n(1)\big) = \ev_n\big(au \otimes v_n(1)\big)$.

\vspace{1mm}
\noindent
{\em Step five.}
We know already that \cref{reallynot} holds for $0 \leq r \leq n'$ and $s \geq 0$.
We now show that it holds for all remaining $r > n'$ and $s \geq 0$.
Take $r \geq n'$ and 
assume by induction that \cref{reallynot} holds for this value of $r$ and 
all $s \geq 0$. 
Equating $t^n$-coefficients in \cref{rels1b} gives
 that $u_n\big(x^{r+1}\big)
 = \bar o^{(n+1)} u_n\big(x^r\big)
 + (-1)^r u_n\big(x^r\big) \bar o^{(n)}$.
 Equating $t^n$-coefficients in \cref{newrels1} 
 (with $r$ replaced by $s$) gives that $v_n\big(x^{s+1}\big)= v_n\big(x^s\big) \bar o^{(n+1)}
 - (-1)^{n+s} \bar o^{(n)} v_n\big(x^s\big)$.
 Making these substitutions, it is then easy to check that
\begin{equation*}
u_n\big(x^{r+1}\big) \otimes v_n\big(x^s\big) = (-1)^{n+r+s+1} u_n\big(x^r\big) \otimes v_n(x^{s+1})
+ (-1)^{n+r+s} u_n\big(x^r\big) \otimes v_n\big(x^s\big) \bar o^{(n+1)}
+ \bar o^{(n+1)} u_n\big(x^r\big) \otimes v_n\big(x^s\big).
\end{equation*}
Now we compute $\ev_n\Big(u_n\big(x^{r+1}\big) \otimes v_n\big(x^s\big)\Big)$
by applying the superbimodule homomorphism $\ev_n$ to the right hand side of the equation just derived and using the induction hypothesis. If $r+s+1<n$ the right hand side evaluates to $0$
so $\ev_n\Big(u_n\big(x^{r+1}\big) \otimes v_n\big(x^s\big)\Big) = 0$ as required.
If $r+s+1=n$ the right hand side evaluates to $\bar \eta^{(n+1)}_{r+1+s-n}$
as required. If $r+s+1>n$ the right hand side evaluates to
$$
(-1)^{n+r+s+1} \bar \eta^{(n+1)}_{r+1+s-n} + (-1)^{n+r+s}  \bar\eta^{(n+1)}_{r+s-n}
\bar o^{(n+1)} + \bar o^{(n+1)} \bar\eta^{(n+1)}_{r+s-n}.
$$
If $n+r+s$ is odd then $\bar\eta^{(n+1)}_{r+s-n}$ and $\bar o^{(n+1)}$
commute by the image of the defining relation \cref{osym1b} under $\smiley$, so this simplifies to
the desired $\bar\eta^{(n+1)}_{r+1+s-n}$. 
If $n+r+s$ is even then $\bar\eta^{(n+1)}_{r+s-n} \bar o^{(n+1)} + \bar o^{(n+1)} \bar\eta^{(n+1)}_{r+s-n} = 2 \bar\eta^{(n+1)}_{r+1+s-n}$ by the image of 
the second relation from \cref{osym5}
under $\smiley$, so again the expression simplifies to 
$\bar\eta^{(n+1)}_{r+1+s-n}$.

\vspace{1mm}
\noindent
{\em Step six.}
It remains to check that the composition
\cref{zigzag2}, which makes sense as $\ev_n$ is a left supermodule homomorphism, is the identity.
We do this by showing 
that it takes $v_n(f(x))$ to $v_n(f(x))$ for any $f(x) \in \k[x]$.
The argument is similar to step three.
By \cref{sundoor1}, 
the map $\coev_n \otimes \id$ takes $1 \otimes v_n(f(x))$ to
$$
\left[v_n\big((t-x)^{-1}\big)\bar\eps^{(n+1)}(t) \otimes u_n\big((t- x)^{-1}\big) \otimes v_n\big(f(x)\big)
\right]_{t^{-1}}.
$$
Then we apply the even map $\id \otimes \ev_n$ using \cref{ev1} (whose validity relies on the conclusion of step five)
to get
$$\left[v_n\big((t-x)^{-1}\big)\bar \eps^{(n+1)}(t)
\big[\bar \gamm^{(n+1)}(t) f(t)\big]_{<t^{0}}\right]_{t^{-1}}.
$$
Since $v_n\big((t-x)^{-1}\big)\bar \eps^{(n+1)}(t)$ 
is a polynomial in $t$ by \cref{crap1}, 
we can omit the inside square brackets. After doing that,
$\eps$ and $\gamm$ cancel using \cref{newigr}, so we obtain
$$
\left[v_n\left((t-x)^{-1}\right)
 f(t)\right]_{t^{-1}} = v_n\big(f(x)\big),
$$
where we used \cref{why} for the final equality.
\end{proof}

For the following corollary, recall from \cref{buckley} that
$\EONH_n^\ell$ is a subalgebra of $\EONH_{n+1}^\ell$ in a natural way
for $0 \leq n < \ell$.

\begin{corollary}\label{indresrel}
For $0 \leq n < \ell$, the following diagrams of graded superfunctors commute up to even degree 0 isomorphisms:
$$
\begin{tikzcd}
\arrow[d,"(\overline{\omega\chi})_{n+1}-" left]
\EONH_{n+1}^\ell\sMod
\arrow[rr,"\!\!{(\Pi Q^{-2})^n \res_{\EONH_n^\ell}^{\EONH_{n+1}^\ell}}" above]&&
\EONH_{n}^\ell\sMod\arrow[d,"(\overline{\omega\chi})_n-" right]\\
\EOH_{n+1}^\ell\sMod\arrow[rr,"{V_{n}^\ell \otimes_{\EOH_{n+1}^\ell}\!\!\!\!\!-}" below]&&\EOH_{n}^\ell\sMod
\end{tikzcd}\qquad
\begin{tikzcd}
\arrow[rr,"(\Pi Q^{2})^n \ind_{\EONH_n^\ell}^{\EONH_{n+1}^\ell}" above]
\arrow[d,"(\overline{\omega\chi})_n-" left]
\EONH_{n}^\ell\sMod
&&\arrow[d,"(\overline{\omega\chi})_{n+1}-" right]\EONH_{n+1}^\ell\sMod
\\
\arrow[rr,"U_{n}^\ell \otimes_{\EOH_{n}^\ell}-\!\!" below]\EOH_{n}^\ell\sMod
&&\EOH_{n+1}^\ell\sMod
\end{tikzcd}
$$
(The vertical arrows are the equivalences 
of graded supercategories from \cref{veryslow}.)
\end{corollary}

\begin{proof}
Note that
$(\Pi Q^2)^n \ind_{\EONH_n^\ell}^{\EONH_{n+1}^\ell}$
is left adjoint to $(\Pi Q^{-2})^n \res_{\EONH_n^\ell}^{\EONH_{n+1}^\ell}$.
Also $U_n^\ell\otimes_{\EOH_n^\ell}-$ is left adjoint to
$V_n^\ell\otimes_{\EOH_{n+1}^\ell}-$ by \cref{cupsandcaps}.
Hence, using the uniqueness of left adjoints,
it suffices to prove that the first square commutes.
We show equivalently that the following commutes up to even degree 0 isomorphism:
$$
\begin{tikzcd}
\EONH_{n+1}^\ell\sMod
\arrow[rr,"{\res_{\EONH_n^\ell}^{\EONH_{n+1}^\ell}}" above]&&
\EONH_{n}^\ell\sMod\arrow[d,"(\overline{\omega\chi})_n-" right]\\
\arrow[u,"\OPol_{n+1} \otimes_{\OSym_{n+1}}\EOH_{n+1}^\ell\otimes_{\EOH_{n+1}^\ell}-" left]
\EOH_{n+1}^\ell\sMod\arrow[rr,"{\tV_{n}^\ell \otimes_{\EOH_{n+1}^\ell}\!\!\!\!\!-}" below]&&\EOH_{n}^\ell\sMod
\end{tikzcd}
$$
Here, we have removed the degree and parity shifts on the horizontal arrows
and we have replaced the equivalence 
$(\overline{\omega\chi})_{n+1}- \simeq
\Hom_{\ONH_{n+1}^\ell}\big(\ONH_{n+1}^\ell (\overline{\omega\chi})_{n+1},-\big)$
on the left hand vertical arrow with the 
quasi-inverse equivalence
$\ONH_{n+1}^\ell (\overline{\omega\chi})_{n+1}
\otimes_{\EOH_{n+1}^\ell}-
\simeq \OPol_{n+1}\otimes_{\OSym_{n+1}}\EOH_{n+1}^\ell
\otimes_{\EOH_{n+1}^\ell}-$.
To prove this new diagram commutes, it suffices to show that
$$
(\omega\chi)_n \OPol_{n+1} \otimes_{\OSym_{n+1}} \EOH_{n+1}^\ell
\simeq
\tV_{n}^\ell
$$
as $(\EOH_n^\ell,\EOH_{n+1}^\ell)$-superbimodules.
As the proof of \cref{lemma0}(1) in the case $d=0$ shows, the isomorphism
$\EOH_{n+1}^\ell \stackrel{\sim}{\rightarrow}
\OSym_{(n+1,n')} \otimes_{\OSym_\ell} R_\ell$
of \cref{EOH}(1) is an isomorphism of
 graded $(\EOH_{n+1}^\ell,\EOH_{n+1}^\ell)$-superbimodules.
So 
$(\omega\chi)_n \OPol_{n+1} \otimes_{\OSym_{n+1}} \EOH_{n+1}^\ell
\simeq 
(\omega\chi)_{n} \OPol_{n+1} \otimes_{\OSym_{n+1}}
\OSym_{(n+1,n')} \otimes_{\OSym_\ell} R_\ell$.
By \cref{newsnow}, 
we have that $(\omega\chi)_n \OPol_{n+1} \simeq
\OSym_{(n,1)}$, so
this is
$\simeq \OSym_{(n,1,n')}\otimes_{\OSym_\ell} R_\ell$,
which is exactly the definition of $\tV_{n}^\ell$.
\end{proof}

The next theorem, which we prove by twisting \cref{cupsandcaps} with
some automorphisms,
 gives the second adjunction. The explicit formulae for
this are not as nice as for the first adjunction.

\begin{theorem}\label{secondadjunction}
Suppose that $\ell=n+1+n'$.
There are unique even degree 0 superbimodule homomorphisms
\begin{align*}\tcoev_n&:
\EOH_{n+1}^\ell
\rightarrow \tU_{n}^\ell \otimes_{\EOH_{n}^\ell}
\tV_{n}^\ell\\
1 &\mapsto 
\sum_{s=0}^{n'}
(-1)^{\ell s+\binom{s}{2}}
    \big(\psi_{n+1}^\ell\big)^{-1}\Big(\bar\eps_{n'-s}^{(n')}\Big)
\tu_n(1) \otimes \tv_n(x^s)\\\intertext{and}\tev_n&:
\tV_{n}^\ell \otimes_{\EOH_{n+1}^\ell}
\tU_{n}^\ell \rightarrow \EOH_{n}^\ell\\
\!\!\!\tv_n(x^r)\otimes \tu_n(x^s)&\mapsto
\left\{
\begin{array}{ll}
\!\!(-1)^{n(r+s)+\binom{r}{2}+\binom{s+1}{2}}
\Big[\big(\psi_{n}^\ell\big)^{-1} \Big(
\bar\gamm_{r+s-n'}^{(n'+1)}
+(-1)^{n+1}\big(1-(-1)^s\big) \bar\eta_{r+s-n'-1}^{(n'+1)} \dot o^{(\ell)}
\Big)\Big]
&\text{if $r+s> n'$}\\
\!\!(-1)^{n(r+s)+\binom{r}{2}+\binom{s+1}{2}}
&\text{if $r+s=n'$}\\
0&\text{otherwise}
\end{array}\right.
\end{align*}
giving the unit and counit of an adjunction making
$\left(\tV_{n}^\ell, \tU_{n}^\ell\right)$
into another dual pair of 1-morphisms in $\OGBim_\ell$.
\end{theorem}

\begin{proof}
Let $\dagger := \phi_{n'}^\ell \circ *:\tU_n^\ell
\rightarrow V_{n'}^\ell$, where $\phi_{n'}^\ell$ is the map
from \cref{crazycrazy}.
Since $\psi_{n+1}^\ell = \delta_{n'}^\ell \circ
\big(\psi_{n'}^\ell\big)^{-1}$
and $\psi_n^\ell = \delta_{n'+1}^\ell \circ
\big(\psi_{n'+1}^\ell\big)^{-1}$
according to the definition \cref{deltadef},
\cref{crazycrazy} and \cref{lemma0b}(2) imply that
\begin{equation}\label{lemma0bchanged}
\left(\bar b_1 u \bar b_2\right)^{\dagger} = \psi_{n+1}^\ell(\bar b_1) u^{\dagger} \psi_{n}^\ell(\bar b_2)
\end{equation}
for $u \in \tU_n^\ell$, $b_1 \in \OSym_{n+1}$ and $b_2
\in \OSym_n$.
Now we define even degree 0 graded $R_\ell$-supermodule homomorphisms
$\tcoev_n$ and $\tev_n$ 
so that the following diagrams commute:
\begin{equation*}
{\begin{tikzcd}
\arrow[d,"\psi^\ell_{n+1}" left]
\EOH_{n+1}^\ell\arrow[r,"\tcoev_n" above]&
\tU_{n}^\ell \otimes_{\EOH_{n}^\ell}
\tV_{n}^\ell\arrow[d,"\dagger \otimes *" right]\\
\EOH_{n'}^\ell
\arrow[r,"\coev_{n'}" below] &V^\ell_{n'}\otimes_{\EOH_{n'+1}^\ell} U^\ell_{n'}
\end{tikzcd}}
\qquad\qquad
{\begin{tikzcd}
\tV_n^\ell\otimes_{\EOH_{n+1}^\ell} \tU_n^\ell
\arrow[r,"\tev_n" above]\arrow[d,"*\otimes \dagger" left]&
\EOH_n^\ell\arrow[d,"\psi_n^\ell" right]\\
U_{n'}^\ell\otimes_{\EOH_{n'}^\ell} V_{n'}^\ell
\arrow[r,"\ev_{n'}" below]
&\EOH_{n'+1}^\ell
\end{tikzcd}}
\end{equation*}
To see that the vertical maps $\dagger\otimes *$ and $* \otimes \dagger$ in these diagrams make sense, one needs to check that they are balanced,
which follows using 
\cref{lemma0b}(1) and \cref{lemma0bchanged}.
In fact, $\tcoev_n$ and $\tev_n$ defined
in this way are superbimodule homomorphisms.
This again follows using
\cref{lemma0b}(1) and \cref{lemma0bchanged}
since $\coev_{n'}$ and $\ev_{n'}$ are superbimodule homomorphisms.

The zig-zag identities
for $\tcoev_n$ and $\tev_n$
follow from their definitions using the zig-zag identities
\cref{zigzag1,zigzag2} 
for $\coev_{n'}$ and $\ev_{n'}$.
Hence they give the unit and counit of an adjunction.

It just remains to compute the explicit formulae for the maps given in 
the statement of the corollary.
To see that $\tcoev_n(1) = \sum_{s=0}^{n'}
(-1)^{\ell s + \binom{s}{2}}
 \Big[   \big(\psi_{n+1}^\ell\big)^{-1}\Big]\big(\bar\eps_{n'-s}^{(n')}\big)
\tu_n(1) \otimes \tv_n(x^s)$, the image of $1$ under the
southwest pair of maps in the diagram defining $\tcoev_n$
is $\sum_{s=0}^{n'} (-1)^{(n'+1)s} \bar\eps_{n'-s}^{(n')}
v_{n'}(1)\otimes u_{n'}(x^s)$ thanks to \cref{Jackdoor}.
This is also the image of 
$\sum_{s=0}^{n'}
(-1)^{\ell s+\binom{s}{2}}
\Big[    \big(\psi_{n+1}^\ell\big)^{-1}\big(\bar\eps_{n'-s}^{(n')}\big)\Big]
\tu_n(1) \otimes \tv_n(x^s)$
under the right hand map $\dagger \otimes *$ by
\cref{reallystarry1,reallystarry2,lemma0bchanged,events}.

To compute $\tev_n\big(\tv_n(x^r)\otimes \tu_n(x^s)\big)$, we use the diagram defining $\tev_n$.
From \cref{reallystarry1,reallystarry2,events},
one sees that 
$$
(* \otimes \dagger)\big(\tv_n(x^r)\otimes \tu_n(x^s)\big) = (-1)^{\binom{r}{2}+\binom{s+1}{2}+n(r+s)} u_{n'}(x^r)
\otimes 
\left[
v_{n'}(x^s)+(-1)^{n+1} \big(1-(-1)^s\big) v_{n'}(x^{s-1}) \dot o^{(\ell)}\right].
$$
Using \cref{reallynot}, the image of this under the homomorphism
$\big(\psi_n^\ell\big)^{-1} \circ \ev_{n'}$ is easily seen to be equal
to the formula for $\tev_n\big(\tv_n(x^r)\otimes \tu_n(x^s)\big)$ given in the statement of the lemma.
\end{proof}

\begin{corollary}\label{indresrel2}
For $0 \leq n < \ell$, the following diagrams of graded superfunctors commute up to even degree 0 isomorphisms:
$$
\begin{tikzcd}
\arrow[d,"(\overline{\omega\chi})_{n+1}-" left]
\EONH_{n+1}^\ell\sMod \arrow[rr,"\res_{\EONH_n^\ell}^{\EONH_{n+1}^\ell}" above]&&
\EONH_{n}^\ell\sMod\arrow[d,"(\overline{\omega\chi})_n-" right]\\
\EOH_{n+1}^\ell\sMod\arrow[rr,"\tV_n^\ell \otimes_{\EOH_{n+1}^\ell}\!\!\!\!\!-" below]&&\EOH_{n}^\ell\sMod
\end{tikzcd}\qquad
\begin{tikzcd}
\arrow[rr,"\coind_{\EONH_n^\ell}^{\EONH_{n+1}^\ell}" above]
\arrow[d,"(\overline{\omega\chi})_n-" left]
\EONH_{n}^\ell\sMod
&&\arrow[d,"(\overline{\omega\chi})_{n+1}-" right]\EONH_{n+1}^\ell\sMod
\\
\arrow[rr,"\tU_{n}^\ell \otimes_{\EOH_{n}^\ell}-\!\!" below]\EOH_{n}^\ell\sMod
&&\EOH_{n+1}^\ell\sMod
\end{tikzcd}
$$
(The vertical arrows are the equivalences 
of graded supercategories from \cref{veryslow}.)
\end{corollary}

\begin{proof}
The commutativity of the first diagram follows from \cref{indresrel}, then the
second follows using \cref{secondadjunction} and the uniqueness of right adjoints.
\end{proof}

The following generalizes \cref{forconsistency}.

\begin{corollary}
For $0 \leq n < \ell$, we have that
$
\ind_{\EONH_n^\ell}^{\EONH_{n+1}^\ell}
\simeq
(\Pi Q^{2})^{\ell-2n-1} \coind_{\EONH_n^\ell}^{\EONH_{n+1}^\ell}.
$
Hence, $\EONH_{n+1}^\ell$ is a graded Frobenius extension of
$\EONH_n^\ell$
of degree $2(\ell-2n-1)$ and parity $\ell-2n-1\pmod{2}$.
\end{corollary}

\begin{proof}
The second statement follows from the first statement by the general
theory of Frobenius extensions explained after \cref{generaltheory}.
To prove the first statement, 
\cref{indresrel,indresrel2} 
show that under the Morita equivalence $\ind_{\EONH_n^\ell}^{\EONH_{n+1}^\ell}$
corresponds to $(\Pi Q^2)^{-n} U_n \otimes_{\EOH_n^\ell}-$
and $\coind_{\EONH_n^\ell}^{\EONH_{n+1}^\ell}$ corresponds
to $\tU_n\otimes_{\EOH_n^\ell}-$.
Now the result follows because
$(\Pi Q^2)^{-n} U_n \simeq (\Pi Q^2)^{\ell-2n-1} \tU_n$ according to \cref{evil}.
\end{proof}

The final task in this section is to compute various mates of the
endomorphisms of $U_{(1^d);n}^\ell$ defined by the
action of $\ONH_d$ from \cref{lemma2}(2). Specifically,
we need to work out the endomorphisms in $\OGBim_\ell$
that correspond to the diagrams \cref{sigrel,phew} in the graphical calculus to be introduced later in the article.
Suppose that $\ell=n+d+n'$ for $d \geq 0$.
We define graded superalgebra homomorphisms
\begin{align}
\rho_{(1^d);n}:\ONH_d &\rightarrow 
\End_{\EOH_{n+d}^\ell\dash \EOH_n^\ell}
\big(U_{n+d-1}^\ell\otimes_{\EOH_{n+d-1}^\ell}
\cdots\otimes_{\EOH_{n+1}^\ell} U_{n}^\ell\big)^{\sop}\label{cam}\\\label{smith}
\lambda_{n;(1^d)}:\ONH_d&\rightarrow
\End_{\EOH_n^\ell\dash \EOH_{n+d}^\ell}
\big(V_n^\ell \otimes_{\EOH_{n+1}^\ell} \cdots\otimes_{\EOH_{n+d-1}^\ell}(V_{n+d-1}^\ell)\big)
\end{align}
as follows.
For $a \in \ONH_d$, $\rho_{(1^d);n}(a)$ is defined 
to be the top map the following diagram commute:
\begin{equation}\label{out}
\begin{tikzcd}
\arrow[d,"\big(b_{(1)^d}\big)^{-1}" left]
U_{n+d-1}^\ell \otimes_{\EOH_{n+d-1}^\ell}
\cdots\otimes_{\EOH_{n+1}^\ell}
U_{n}^\ell
\arrow[rr,"\rho_{(1^d);n}(a)" above]
&&
U_{n+d-1}^\ell \otimes_{\EOH_{n+d-1}^\ell}
\cdots\otimes_{\EOH_{n+1}^\ell}
U_{n}^\ell\\
U_{(1^d);n}^\ell
\arrow[rr,"u_{(1^d);n}(f) \mapsto (-1)^{\parity(a)\parity(f)} u_{(1^d);n}(f) \cdot a" below]&&
U_{(1^d);n}^\ell\arrow[u,"b_{(1)^d}" right]
\end{tikzcd}
\end{equation}
Also $\lambda_{n;(1^d)}(a)$ is the top map making
the following commute:
\begin{equation}\label{in}
\begin{tikzcd}
\arrow[d,"\big(c_{(1)^d}\big)^{-1}" left]
V_{n}^\ell \otimes_{\EOH_{n+1}^\ell}
\cdots\otimes_{\EOH_{n+d-1}^\ell}
V_{n+d-1}^\ell
\arrow[rr,"\lambda_{n;(1^d)}(a)" above]
&&
V_{n}^\ell \otimes_{\EOH_{n+1}^\ell}
\cdots\otimes_{\EOH_{n+d-1}^\ell}
V_{n+d-1}^\ell
\\
V_{n;(1^d)}^\ell
\arrow[rr,"v_{n;(1^d)}(f) \mapsto 
a\cdot v_{n;(1^d)}(f)" below]&&
V_{n;(1^d)}^\ell\arrow[u,"c_{(1)^d}" right]
\end{tikzcd}
\end{equation}
The vertical maps in \cref{out,in} come from \cref{ukraine2,ukraine3}.
We also remind the reader that
$u_{(1^d);n}(f) \cdot a = (-1)^{(d-1)\parity(a)} u_{(1^d);n}(f\cdot a)$
and $a\cdot v_{n;(1^d)}(f) = (-1)^{(n\# (d-1))\parity(a)} v_{n;(1^d)}(a \cdot f)$. 

\begin{lemma}\label{mate1}
For $0 \leq n < \ell$, the mate 
of the $\big(\EOH_{n+1}^\ell,\EOH_{n}^\ell\big)$-superbimodule endomorphism
$\rho_{(1);n}(x_1)$ 
under the adjunction from \cref{cupsandcaps},
that is, the composition
$$
\begin{tikzcd}
V_n^\ell\arrow[r,"\operatorname{can}"]&
\EOH_n^\ell\otimes_{\EOH_n^\ell} V_n^\ell\arrow[rr,"\coev_n\otimes\id"]&&
V_n^\ell \otimes_{\EOH_{n+1}^\ell} U_n^\ell\otimes_{\EOH_n^\ell} V_n^\ell\arrow[rrr,"{\id\otimes\rho_{(1);n}(x_1)\otimes\id}"]&&&
\phantom{x}
\end{tikzcd}\qquad\qquad\qquad
$$$$
\qquad\qquad\qquad\qquad\qquad\begin{tikzcd}
V_n^\ell \otimes_{\EOH_{n+1}^\ell}
U_n^\ell\otimes_{\EOH_n^\ell} V_n^\ell\arrow[rr,"\id\otimes\ev_n"]&&
V_n^\ell\otimes_{\EOH_{n+1}^\ell} \EOH_{n+1}^\ell\arrow[r,"\operatorname{can}"]&
V_n^\ell,
\end{tikzcd}
$$
is equal to $\lambda_{n;(1)}(x_1)$.
\end{lemma}

\begin{proof}
Since $V_{n}^\ell$ is generated as a superbimodule
by $v_n(1)$ by \cref{lemma3}(1c), it suffices to show that
$$
(\id \otimes \ev_n)\circ
(\id \otimes \rho_{(1);n}(x_1) \otimes \id)
\circ (\coev_n \otimes\id)\big(1 \otimes v_n(1)\big)
=
\lambda_{n;(1)}(x_1)\big(v_n(1)\big) \otimes 1.
$$
This is a calculation from the definitions.
The right hand side is
$v_n(x) \otimes 1$ by the definition \cref{in}.
For the left hand side, we first apply
$\coev_n \otimes\id$ using \cref{fridoor}
to get 
$$
\Big[
\bar\eps^{(n)}\big((-1)^{n} t \big)
v_n\big((-1)^n\big) \otimes u_n\big((t-x)^{-1}\big) \otimes v_n(1)
\Big]_{t^{-1}}.
$$
Then we apply the odd endomorphism
$\id \otimes \rho_{(1);n}(x_1) \otimes \id$
defined by \cref{out} to get
$$
\Big[
\bar\eps^{(n)}\big((-1)^{n+1} t \big)
v_n\big((-1)^n\big) \otimes u_n\big((t+x)^{-1}x\big) \otimes v_n(1)
\Big]_{t^{-1}}.
$$
Finally we apply
$\id \otimes \ev_n$ using \cref{ev1} with $t$ replaced by $-t$ to get
$$
\Big[
\bar\eps^{(n)}\big((-1)^{n+1} t \big)
v_n\big((-1)^n\big) \big[
\bar\gamm^{(n+1)}(-t)t
\big]_{<t^{0}}\Big]_{t^{-1}} \otimes 1.
$$
Computing the coefficients explicitly,
this is equal to
$v_n(1)\bar\gamm^{(n+1)}_1-(-1)^n\bar\eps_1^{(n)} v_n(1)$.
Since $\gamm_1^{(n+1)} = x_1+\cdots+x_{n+1}$
and $\eps_1^{(n)} = x_1+\cdots+x_n$ (and $(-1)^n$ cancels
when $\eps_1^{(n)}$ acts on $v_n(1)$ due to the parity shift) this is $v_n(x)\otimes 1$.
\end{proof}

\begin{lemma}\label{mate2}
For $0 < n < \ell$, the $\big(\EOH_n^\ell,\EOH_n^\ell\big)$-superbimodule 
endomorphism $\sigma_n$ that is defined by the composition
$$
\begin{tikzcd}
U_{n-1}^\ell\otimes_{\EOH_{n-1}^\ell} V_{n-1}^\ell
\arrow[r,"\operatorname{can}"]&
\EOH_n^\ell\otimes_{\EOH_n^\ell} U_{n-1}^\ell\otimes_{\EOH_{n-1}^\ell} V_{n-1}^\ell\arrow[rrr,"\coev_n \otimes \id\otimes\id"]&&&\phantom{x}
\end{tikzcd}\qquad\qquad\qquad\qquad\qquad$$$$\begin{tikzcd}
V_n^\ell\otimes_{\EOH_{n+1}^\ell} U_n^\ell \otimes_{\EOH_n^\ell}
U_{n-1}^\ell\otimes_{\EOH_{n-1}^\ell} V_{n-1}^\ell\arrow[rrr,"\id \otimes \rho_{(1^2);n-1}(\tau_1) \otimes \id"]&&&
V_n^\ell\otimes_{\EOH_{n+1}^\ell} U_n^\ell \otimes_{\EOH_n^\ell}
U_{n-1}^\ell\otimes_{\EOH_{n-1}^\ell} V_{n-1}^\ell
\end{tikzcd}$$$$\qquad\qquad\qquad\qquad\qquad\qquad\quad\begin{tikzcd}
\phantom{x}\arrow[rrr,"\id\otimes\id\otimes\ev_{n-1}"]&&&
V_n^\ell\otimes_{\EOH_{n+1}^\ell} U_n^\ell \otimes_{\EOH_n^\ell}
\EOH_n^\ell\arrow[r,"\operatorname{can}"]&
V_n^\ell\otimes_{\EOH_{n+1}^\ell} U_n^\ell
\end{tikzcd}
$$
maps
$u_{n-1}\big(x^r\big) \otimes v_{n-1}\big(x^s\big)$
to \begin{equation}\label{Ian}
(-1)^{nr+rs+r+s+n+1} v_n\big(x^s\big) \otimes u_n\big(x^r\big)
+
\sum_{p=0}^n \sum_{q=0}^{r+s-n}
(-1)^{nq+pq+rq+r+q}
 v_n\big(x^p\big) \otimes u_n\big(x^{q}\big)
\bar\eps^{(n)}_{n-p} \bar \gamm^{(n)}_{r+s-n-q}
\end{equation}
for any $r \geq 0$ and $0 \leq s \leq n-1$.
\end{lemma}

\begin{proof}
First, we show that 
$\rho_{(1^2);n-1}(\tau_1)
\Big(u_n\big((t- x)^{-1}\big)
\otimes u_{n-1}\big(x^r\big)\Big)$
equals
\begin{equation}\label{signtrouble}
-u_{n}\big((t+x)^{-1} x^r\big) \otimes u_{n-1}\Big(\big(t-(-1)^r x\big)^{-1}\Big)
+ \sum_{q=0}^{r-1} (-1)^{q+r+rq}
u_{n}\big((t+x)^{-1}x^{q}\big)\otimes u_{n-1}\big(x^{r-q-1}\big).
\end{equation}
To do this, according to the definition \cref{out},
we first use the inverse of the map $b_{(1)^2}$ from \cref{ukraine2} to pass to
$U_{(1^2);n-1}^\ell$.
This maps
$u_n\big((t- x)^{-1}\big)\otimes
u_{n-1}\big(x^r\big)$
to
$(-1)^r u_{(1^2);n-1}\big((t-x_2)^{-1} x_1^r\big)$.
The application of $\rho_{(1^2);n-1}(\tau_1)$
takes this to
$-u_{(1^2);n-1}\big((t+x_2)^{-1} x_1^r \cdot \tau_1\big)$.
This we can compute with \cref{trulycrazy}
to get
$$
-u_{(1^2);n-1}\Big((t+x_2)^{-1} x_2^r \big(t+(-1)^r x_1\big)^{-1}\Big)
- \sum_{q=0}^{r-1} (-1)^{rq}
u_{(1)^2;n-1}\Big((t+x_2)^{-1}x_2^{q}x_1^{r-q-1}\Big).
$$
After that we apply $b_{(1)^2}$ to 
obtain the vector 
in $U_{n}^\ell\otimes_{\EOH_n^\ell} U_{n-1}^\ell$
displayed in \cref{signtrouble}.

Now to prove the lemma, we again calculate
with generating functions.
Start with the vector $u_{n-1}\big(x^r\big) \otimes v_{n-1}\big(x^s\big)$ for $0 \leq s \leq n-1$ (this assumption on $s$ will be crucial shortly).
By \cref{sundoor1}, the map
$\coev_n \otimes \id \otimes \id$ takes it to 
\begin{equation}\label{crv}
\Big[ v_n((t-x)^{-1}) \bar \eps^{(n+1)}(t) \otimes
u_n\big((t- x)^{-1}\big)\otimes
u_{n-1}\big(x^r\big) \otimes v_{n-1}\big(x^s\big)
\Big]_{t^{-1}}.
\end{equation}
Then we apply the odd homomorphism
$\id \otimes \rho_{(1^2);n-1}(\tau_1) \otimes \id$
using \cref{signtrouble} to obtain
\begin{multline}\label{tricky}
\Big[ v_n((t+x)^{-1}) \bar \eps^{(n+1)}(-t) \otimes
u_{n}\big((t+x)^{-1} x^r\big) \otimes u_{n-1}\Big(\big(t-(-1)^r x\big)^{-1}\Big)
\otimes v_{n-1}\big(x^s\big)
\Big]_{t^{-1}}\\
+
 \sum_{q=0}^{r-1} (-1)^{q+r+rq+1}
 \Big[ v_n((t+x)^{-1}) \bar \eps^{(n+1)}(-t) \otimes
u_{n}\big((t+x)^{-1}x^{q}\big)\otimes u_{n-1}\big(x^{r-q-1}\big)
\otimes v_{n-1}\big(x^s\big)
\Big]_{t^{-1}}.
\end{multline}
It just remains to apply $\id\otimes\id\otimes \ev_{n-1}$.
We treat the two terms in \cref{tricky} 
separately. For the first term, we have 
that 
$\ev_{n-1}\Big(u_{n-1}\Big(\big(t- (-1)^r x\big)^{-1}\Big)
\otimes v_{n-1}\big(x^s\big)\Big)
=(-1)^{rs+r} \left[\bar \gamm^{(n)}\big((-1)^r t\big) t^s\right]_{<t^{0}}$
by \cref{ev1} (with $t$ replaced by $(-1)^{r}t$).
The assumption that $s \leq n-1$ 
means that we can omit the truncation to $<t^{0}$ here.
So the first term contributes
$$
(-1)^{rs+r}
\Big[ v_n((t+x)^{-1}) \otimes \bar \eps^{(n+1)}(-t) u_{n}\big((t+x)^{-1} x^r\big)\otimes
\bar \gamm^{(n)}\big((-1)^r t\big) t^s
\Big]_{t^{-1}}.
$$
Now we use \cref{relsfancy3} to rewrite this as
$$
(-1)^{nr+rs+r+n+1}
\Big[ v_n((t+x)^{-1}) \otimes 
u_n\big(x^r\big) \bar\eps^{(n)}\big((-1)^{r} t\big)
\bar \gamm^{(n)}\big((-1)^r t\big) t^s \otimes 1
\Big]_{t^{-1}}.
$$
The $\eps$ and $\gamm$ cancel by the infinite Grassmannian relation 
to leave
$$
(-1)^{nr+rs+r+n+1}
\Big[ v_n((t+x)^{-1}t^s) \otimes 
u_n\big(x^r\big) \otimes 1
\Big]_{t^{-1}}
\stackrel{\cref{why}}{=}
(-1)^{nr+rs+r+s+n+1}
 v_n\big(x^s\big) \otimes u_n\big(x^r\big)\otimes 1.
$$
It remains to consider the term obtained by applying 
$\id\otimes\id\otimes \ev_{n-1}$
to the second term from \cref{tricky}. Using \cref{reallynot,relsfancy3}, this contributes
\begin{multline*}
\sum_{q=0}^{r+s-n} (-1)^{q+r+rq+1}
 \Big[ v_n((t+x)^{-1}) \otimes
\bar \eps^{(n+1)}(-t) u_{n}\big((t+x)^{-1}x^{q}\big)
\otimes \bar \gamm^{(n)}_{r+s-q-n}
\Big]_{t^{-1}}=\\
\sum_{q=0}^{r+s-n} (-1)^{q+r+rq+nq+n}
\Big[ v_n\big((t+x)^{-1}\big)\otimes u_n\big(x^q\big)\bar \eps^{(n)}\big((-1)^{q}t\big)\bar\gamm^{(n)}_{r+s-q-n}\otimes 1\Big]_{t^{-1}}.
\end{multline*}
It remains to work out the $t^{-1}$-coefficient explicitly to complete the proof.
\end{proof}

\begin{lemma}\label{mate3}
For $0 < n < \ell$, the mate of $\rho_{(1^2);n-1}(\tau_1)$
under the adjunction from \cref{cupsandcaps}, that is, the composition
$$
\begin{tikzcd}
V_{n-1}^\ell\otimes_{\EOH_{n}^\ell} V_{n}^\ell\arrow[r,"\operatorname{can}"]&
\EOH_{n-1}^\ell\otimes_{\EOH_{n-1}^\ell}V_{n-1}^\ell\otimes_{\EOH_{n}^\ell} V_{n}^\ell
\arrow[rrr,"\coev_{n-1}\otimes\id\otimes \id"]&&&\phantom{x}
\end{tikzcd}\qquad\qquad\qquad\qquad\qquad
$$
$$
\begin{tikzcd}
V_{n-1}^\ell\otimes_{\EOH_n^\ell} U_{n-1}^\ell
\otimes_{\EOH_{n-1}^\ell}V_{n-1}^\ell\otimes_{\EOH_{n}^\ell} V_{n}^\ell
\arrow[rr,"\id\otimes\sigma_n\otimes \id"]&&
V_{n-1}^\ell\otimes_{\EOH_n^\ell} V_{n}^\ell
\otimes_{\EOH_{n+1}^\ell}U_{n}^\ell\otimes_{\EOH_{n}^\ell} V_{n}^\ell
\end{tikzcd}
$$
$$
\qquad\qquad\qquad\qquad\qquad
\begin{tikzcd}
\phantom{x}\arrow[rrr,"\id\otimes\id \otimes \ev_{n}"]&&&
V_{n-1}^\ell\otimes_{\EOH_n^\ell} V_{n}^\ell
\otimes_{\EOH_{n+1}^\ell} \EOH_{n+1}^\ell\arrow[r,"\operatorname{can}"]&
V_{n-1}^\ell\otimes_{\EOH_n^\ell} V_{n}^\ell
\end{tikzcd}
$$ where $\sigma_n$ is the superbimodule homomorphism\footnote{Diagrammatically, we have rotated through $180^\circ$ by rotating by $90^\circ$ twice, see \cref{sigrel,phew}.}
 defined in \cref{mate2},
is equal to $\lambda_{n-1;(1^2)}(\tau_1)$.
\end{lemma}

\begin{proof}
By \cref{lemma3}(1b)--(1c),
$V_{n-1}^\ell\otimes_{\EOH_n^\ell}
V_{n}^\ell$ is generated as an
$\big(\EOH_{n-1}^\ell,\EOH_{n+1}^\ell\big)$-superbimodule
by the vectors $v_{n-1}(1)\otimes v_n(x^s)\:\:(0 \leq s \leq n)$.
Therefore it suffices to show that
$$
(\id\otimes\id \otimes \ev_{n})
\circ (\id \otimes \sigma_n \otimes \id) \circ (\coev_{n-1}
\otimes \id \otimes \id)
\big(1 \otimes v_{n-1}(1)\otimes v_n(x^s)\big) =
\lambda_{n-1;(1^2)}(\tau_1) \big( v_{n-1}(1)\otimes v_n(x^s)\big) \otimes 1
$$
for $0 \leq s \leq n$.
The right hand side may be computed directly 
from \cref{in,ukraine3}.
It equals 
$$
(-1)^{n} c_{(1)^2}\Big(v_{n-1;(1^2)}\big(\tau_1 \cdot x_2^s\big)\Big) \otimes 1
\stackrel{\cref{rank22}}{=}  (-1)^{n+1} c_{(1)^2}\Big( v_{n-1;(1^2)}\big(\gamm^{(2)}_{s-1}\big)\Big) \otimes 1.
$$
So to complete the proof we must show that
\begin{multline}\label{todo}
\!\!\left(c_{(1)^2}\otimes\id\right)^{\!-1}\!\!
\Big( (\id\otimes\id \otimes \ev_{n})\
\!\circ (\id \otimes \sigma_n \otimes \id) \circ (\coev_{n-1}
\otimes \id \otimes \id)
\big(1 \otimes v_{n-1}(1)\otimes v_n(x^s)\big)\Big)\\ = 
(-1)^{n+1} v_{n-1;(1^2)}\big(\gamm^{(2)}_{s-1}\big)\otimes 1.
 \end{multline}
To compute the left hand side, we first use \cref{Jackdoor}
to get
$$
(\coev_{n-1}
\otimes \id \otimes \id)
\big(1\otimes v_{n-1}(1)\otimes v_n(x^s)\big)
=
\sum_{r=0}^{n-1}(-1)^{nr} \bar \eps_{n-1-r}^{(n-1)}
v_{n-1}(1) \otimes 
u_{n-1}\big(x^r\big) \otimes v_{n-1}(1)
\otimes v_n(x^s).
$$
Then we apply $\id \otimes \sigma_n \otimes \id$
using \cref{Ian}, noting also that $\sigma_n$ is {\em odd}. 
Since $r \leq n-1$ in this expression,
the summation over $q$ on the right hand side of \cref{Ian} is actually
an empty sum, so zero, and we get simply
$$
(-1)^{n+1} \sum_{r=0}^{n-1}
\bar \eps_{n-1-r}^{(n-1)}
v_{n-1}(1) \otimes
v_n(1) \otimes u_n\big(x^r\big)
\otimes v_n\big (x^s\big ).
$$
Next we apply $\id\otimes\id\otimes\ev_n$.
We must have that $r+s \geq n$ so $r \geq n-s$, and the final expression is
$$
 (-1)^{n+1} \sum_{r=n-s}^{n-1}
\bar \eps_{n-1-r}^{(n-1)}
v_{n-1}(1) \otimes
v_n(1) 
\bar\gamm^{(n+1)}_{r+s-n}
= (-1)^{n+1}\sum_{r=0}^{s-1} \bar\eps_r^{(n-1)}
v_{n-1}(1) \otimes
v_n(1) 
\bar\gamm^{(n+1)}_{s-1-r}.
$$
Applying $\big(c_{(1)^2}\otimes\id\big)^{-1}$ as is required for \cref{todo}, we get
$$
(-1)^{n+1} \sum_{r=0}^{s-1} \bar\eps_r^{(n-1)}
v_{n-1;(1^2)}(1)
\bar\gamm^{(n+1)}_{s-1-r}.
$$
There is a sign change of $(-1)^r$
due to the parity shift $(n-1)\# 2 \equiv 1 \pmod{2}$. Also 
we have that $\sum_{r=0}^{s-1} (-1)^r \eps_r^{(n-1)} \gamm^{(n+1)}_{s-1-r}
= \sig_{n-1}\big(\gamm^{(2)}_{s-1}\big)$ by a similar argument to the proof of \cref{notsoeasypeasy}.
So this is $(-1)^{n+1} v_{n-1;(1^2)}\big(\gamm^{(2)}_{s-1}\big)$
exactly as in \cref{todo}.
\end{proof}

\section{Singular Rouquier complex}\label{derby}

Throughout the section, we fix $\ell \in \N$.
The graded $(Q,\Pi)$-2-supercategory
$\OGBim_\ell$ 
categorifies the locally unital $\Z[q,q^{-1}]^\pi$-algebra that is
the image of $\mathbf{U}_{q,\pi}(\sl_2)$
in its representation on
$\mathbf{V}(-\ell)$, notation as at the end of \cref{leavingdepoe}.
To make this statement precise, let
$K_0(\EOH_n^\ell)$ be the Grothendieck group of 
$\EOH_n^\ell$; recall this means the split Grothendieck group
of the category $\EOH_n^\ell\Upsmod$.
Since $\EOH_n^\ell$ is positively graded with
degree 0 component that is the ground field $\k$,
this is nothing more than the free $\Z[q,q^{-1}]^\pi$-module
generated by the isomorphism class $[\EOH_n^\ell]$
of the regular module, with the actions of $\pi$ and $q$
induced by the parity and degree shift functors $\Pi$ and $Q$, respectively.
So we can identify
\begin{align}
\mathbf{V}(-\ell)
&\equiv
\bigoplus_{n=0}^\ell
K_0(\EOH_n^\ell),
&
b_n^\ell &\equiv [\EOH_n^\ell].\label{theid}
\end{align}
The following matches up the action of generators of 
$\mathbf{U}_{q,\pi}(\sl_2)$
on $\mathbf{V}(-\ell)$ with 
endomorphisms of the Grothendieck group
induced by tensoring with odd Grassmannian bimodules.

\begin{theorem}\label{K0}
Under the identification \cref{theid} of
$\bigoplus_{n=0}^\ell
K_0(\EOH_n^\ell)$
with
$\mathbf{V}(-\ell)$, 
the 
$\Z[q,q^{-1}]^\pi$-module endomorphisms induced by 
tensoring with odd Grassmannian bimodules
correspond
to endomorphisms defined 
by actions of elements of $\mathbf{U}_{q,\pi}(\sl_2)$ according to the following dictionary:
\begin{enumerate}
\item
 $\big[U_n^\ell\otimes_{\EOH_n^\ell}-\big]\equiv q^n E 1_{2n-\ell}$
 and, more generally, $\big[U_{(d);n}^\ell\otimes_{\EOH_n^\ell}-\big]
\equiv q^{nd} E^{(d)}1_{2n-\ell}$;
\item $\big[V_n^\ell\otimes_{\EOH_{n+1}^\ell}-\big]
\equiv q^{\ell-3n-1} 1_{2n-\ell} F$
 and $\big[V_{n;(d)}^\ell\otimes_{\EOH_{n+d}^\ell}-\big]\equiv
 q^{d(\ell-3n-2d+1)} 1_{2n-\ell} F^{(d)}$.
\end{enumerate}  
Also, for $-\ell \leq k \leq \ell$
with $k \equiv \ell\pmod{2}$,
the map $T:1_{-k} \mathbf{V}(-\ell)
\rightarrow 1_k \mathbf{V}(-\ell)$
from \cref{oddreflection}
corresponds to the $\Z[q,q^{-1}]^\pi$-module
homomorphism
\begin{align}\label{fromR}
T:K_0(\EOH_n^\ell) &\rightarrow K_0(\EOH_{n'}^\ell),&
\big[\EOH_n^\ell\big]&\mapsto
(-1)^n (\pi q^2)^{\binom{n+1}{2}+nk} q^{-nk}
\big[\EOH_{n'}^\ell\big]
\end{align}
where $n:=\frac{\ell-k}{2}$ and $n':=\frac{\ell+k}{2}$.
\end{theorem}

\begin{proof}
(1) 
Note here we are assuming implicitly that $0 \leq n \leq \ell-1$
and
$0 \leq n \leq \ell-d$, respectively
so that $U_n^\ell$ and $U_{(d);n}^\ell$
are defined.
By \cref{lemma3}(2b), we have that
$\big[U_n^\ell \otimes_{\EOH_n^\ell} \EOH_n^\ell\big]
=
q^n [n+1]_{q,\pi} \big[\EOH_{n+1}^\ell\big].$
Also $E b_n^\ell =
[n+1]_{q,\pi} b_n^\ell$.
It follows that
$\big[U_n^\ell\otimes_{\EOH_n^\ell}-\big]$ and 
$q^n E 1_{2n-\ell}$
define the same endomorphisms of $\mathbf{V}(-\ell)$.
For the more general assertion, take $d \geq 1$.
By \cref{lemma2}(2) and \cref{poincare}, we have that
$\big[U_{(1^d);n}^\ell\otimes_{\EOH_n^\ell}-\big]
= q^{\binom{d}{2}} [d]^!_{q,\pi}
\big[U_{(d);n}^\ell\otimes_{\EOH_n^\ell}-\big]$.
Also $U_{(1^d);n}^\ell
\simeq U^\ell_{n+d-1}\otimes_{\EOH_{n+d-1}^\ell}
\cdots\otimes_{\EOH_{n+1}^\ell} U_n^\ell$
by \cref{ukraine2},
so we deduce using the special case already treated that
\begin{align*}
q^{\binom{d}{2}} [d]^!_{q,\pi}
\big[U_{(d);n}^\ell\otimes_{\EOH_n^\ell}-\big]
&=
q^{nd+\binom{d}{2}} E^d 1_{2n-\ell}
= 
q^{nd+\binom{d}{2}} [d]^!_{q,\pi}
E^{(d)} 1_{2n-\ell}. 
\end{align*}
Cancelling $q^{\binom{d}{2}} [d]^!_{q,\pi}$
gives the required conclusion.

\vspace{1mm}
\noindent
(2) Again we are assuming that 
$0 \leq n \leq \ell-1$ and $0 \leq n \leq \ell-d$, respectively. 
The first step is to show that
$\big[\tV_n^\ell\otimes_{\EOH_{n+1}^\ell}-\big]
= q^{\ell-n-1}\pi^{n} 1_{2n-\ell} F$,
which follows from \cref{lemma3}(1a) like in the proof of (1).
Hence, since $V_n^\ell
= (\Pi Q^{-2})^n \tV_n^\ell$, 
we get that
$\big[V_n^\ell\otimes_{\EOH_{n+1}^\ell}-\big]
= q^{\ell-3n-1} 1_{2n-\ell} F$.
The passage from this to 
the more general
result about $V_{n;(d)}^\ell$ 
follows in a similar way to the argument given in (1).

\vspace{1mm}
\noindent
Now consider the final statement about $T$.
Take $k$ and $n=\frac{\ell-k}{2}, n'=\frac{\ell+k}{2}$
as in the statement of the theorem.
We saw in \cref{rain2} that $T(b_n^\ell) = 
(-1)^n \pi^{\binom{n}{2}+nn'} q^{n+n n'} b_{n'}^\ell$.
Using the identification \cref{theid},
it follows that
$
T\Big(\big[\EOH_n^\ell\big]\Big)
=
(-1)^n \pi^{\binom{n}{2}+nn'} q^{n+n n'}
\big[\EOH_{n'}^\ell\big].$
On replacing $n'$ by $k+n$,  
this becomes the formula in the statement of the theorem.
\end{proof}

The goal in the remainder section is to 
categorify $T:1_{-k} \mathbf{V}(-\ell)
\rightarrow 1_k \mathbf{V}(-\ell)$
for all $-\ell \leq k \leq \ell$ with $k \equiv \ell\pmod{2}$.
Throughout, we let $n := \frac{\ell-k}{2}$ and $n' := \frac{\ell+k}{2}=n+k$
so that $2n-\ell=-k$ and $2n'-\ell=k$.
Since $n+n'=\ell$, \cref{EOH}(4) shows that 
the graded superalgebras
$\EOH_n^\ell$ and $\EOH_{n'}^\ell$ are isomorphic.

\begin{definition}\label{weekend}
For $0 \leq d \leq n$, let
\begin{align}
C_{d} &:= 
\left\{
\begin{array}{ll}
U^\ell_{(k+d);n-d} \otimes_{\EOH_{n-d}^\ell}
V_{n-d;(d)}^\ell
&\text{if $d \geq -k$,}\\
0&\text{otherwise}.
\end{array}
\right.
\end{align}
The {\em singular Rouquier complex} for odd Grassmannian bimodules
is 
the following sequence of graded
$\big(\EOH_{n'}^\ell,\EOH_n^\ell\big)$-superbimodules 
and even degree 0 superbimodule homomorphisms
in $\OGBim_\ell$:
\begin{equation}\label{src}
0 \stackrel{\partial_{n+1}}{\longrightarrow} C_n 
\stackrel{\partial_n}{\longrightarrow} \cdots 
\stackrel{\partial_{d+1}}{\longrightarrow} C_d \stackrel{\partial_d}{\longrightarrow} C_{d-1}\stackrel{\partial_{d-1}}{\longrightarrow}\cdots\stackrel{\partial_2}{\longrightarrow} C_1 
\stackrel{\partial_1}{\longrightarrow} C_0 \stackrel{\partial_0}{\longrightarrow} 0
\end{equation}
where $\partial_d = 0$ unless $\max(0,-k) < d \leq n$, in which case
$\partial_d:C_d \rightarrow C_{d-1}$ is the even degree 0
superbimodule homomorphism defined by the composition
$$
\!\begin{tikzcd}
U^\ell_{(k+d);n-d} \otimes_{\EOH_{n-d}^\ell} V_{n-d;(d)}^\ell
\arrow[r,"\operatorname{inc}" above]
&
U^\ell_{(k+d-1,1);n-d} \otimes_{\EOH_{n-d}^\ell} 
V_{n-d;(1,d-1)}^\ell
\arrow[rrr,"{c_{(k+d-1),(1)}' \otimes\; c_{(1),(d-1)}}" above]&&&\phantom{psi[mgjokurhx}
\end{tikzcd}\!
$$

\vspace{-4mm}

$$
\!\begin{tikzcd}
U^\ell_{(k+d-1);n-d+1}\otimes_{\EOH_{n-d+1}^\ell}
U^\ell_{n-d}\otimes_{\EOH_{n-d}^\ell} V_{n-d}^\ell \otimes_{\EOH_{n-d+1}^\ell} V_{n-d+1;(d-1)}^\ell
\arrow[rr,"\id \otimes \ev_{n-d} \otimes \id" above]&&\phantom{x;oiagrhjqnrdgore'pj}
\end{tikzcd}\!
$$

\vspace{-4mm}

$$
\!\begin{tikzcd}
U^\ell_{(k+d-1);n-d+1}\otimes_{\EOH_{n-d+1}^\ell}\EOH_{n-d+1}^\ell\otimes_{\EOH_{n-d+1}^\ell} 
V_{n-d+1;(d-1)}^\ell
\arrow[r,"\operatorname{can}" above]&
U^\ell_{(k+d-1);n-d+1}\otimes_{\EOH_{n-d+1}^\ell}
V_{n-d+1;(d-1)}^\ell.
\end{tikzcd}\!
$$
\end{definition}

\begin{theorem}\label{SRC}
The singular Rouquier complex \cref{src} is a chain complex 
with homology that is zero in all except for the top ($n$th) homological degree.
Moreover, 
as a graded 
$\big(\EOH_{n'}^\ell,\EOH_n^\ell)$-superbimodule
the top homology is $\simeq(\Pi Q^2)^{\binom{n+1}{2}+nk}\EOH_{n'}^\ell$
viewed as a graded left $\EOH_{n'}^\ell$-supermodule by the natural action
and as a graded right $\EOH_{n}^\ell$-supermodule
by restricting the natural right action of $\EOH_{n'}^\ell$ along
some graded superalgebra isomorphism
$\EOH_{n}^\ell\stackrel{\sim}{\rightarrow}
\EOH_{n'}^\ell$.
\end{theorem}

To formulate a corollary, let
$K^b(\EOH_n^\ell\psmod)$ be the bounded homotopy supercategory
of the graded supercategory of finitely generated projective
graded left $\EOH_n^\ell$-supermodules;
in the definition of this 
we require that differentials and chain homotopies are even of degree 0 but
chain maps between cochain complexes can be constructed using 
arbitrary morphisms in $\EOH_n^\ell\psmod$.
By Euler characteristic (e.g., see \cite{Rose}),
the triangulated Grothendieck group of the underlying ordinary category
is identified with 
$K_0(\EOH_n^\ell)$, hence, 
via \cref{theid}, with $1_{-k} \mathbf{V}(-\ell)$.

\begin{corollary}\label{weirdshift}
The graded superfunctor 
$K^b(\EOH_n^\ell\psmod) \rightarrow K^b(\EOH_{n'}^\ell\psmod)$
defined by
tensoring with the singular Rouquier complex
\cref{src} (viewed now as a cochain complex) 
then taking the total complex
is an equivalence of triangulated graded supercategories. 
The induced  $\Z[q,q^{-1}]^\pi$-module isomorphism
$1_{-k} \mathbf{V}(-\ell)\stackrel{\sim}{\rightarrow}
1_k \mathbf{V}(-\ell)$
at the level of Grothendieck groups
is equal to $q^{nk} T$ for $T$ as in \cref{fromR}.
\end{corollary}

\begin{proof}
The theorem shows that the singular Rouquier complex
is quasi-isomorphic to the cochain complex which is the graded
superbimodule 
$(\Pi Q^2)^{\binom{n+1}{2}+nk}\EOH_{n'}^\ell$
described in \cref{SRC}
in cohomological degree $-n$, and zero elsewhere.
So it defines an equivalence of
triangulated graded supercategories
$D^-(\EOH_n^\ell\sMod) \rightarrow D^-(\EOH_{n'}^\ell\sMod)$
between the bounded-above derived categories.
Since $D^-(\EOH_n^\ell\sMod)$
is equivalent to $K^-(\EOH_n^\ell\psMod)$ and similarly
for $\EOH_{n'}^\ell$, we deduce that the functor arising from tensoring with the singular Rouquier complex defines
an equivalence of triangulated graded 
supercategories $K^-(\EOH_n^\ell\psMod)
\rightarrow K^-(\EOH_{n'}^\ell\psMod)$.
The first part of the corollary follows on
restricting this equivalence to $K^b(\EOH_n^\ell\psmod)$.
The second part follows using also (\ref{fromR}) because 
the functor takes the cochain complex that is $\EOH_n^\ell$ concentrated in cohomological degree 0 
to a cochain complex with the same Euler characteristic as 
$(\Pi Q^2)^{\binom{n+1}{2}+nk}
\EOH_{n'}^\ell$ concentrated in cohomological degree $-n$.
\end{proof}

The remainder of the section is devoted to the proof of \cref{SRC}, which will be carried out with a series of lemmas.
We assume for simplicity of notation that $k \geq 0$, although with obvious modifications the arguments work for negative $k$ too (in that case, the terms $C_0, \dots, C_{-k-1}$ of the complex are zero).

\begin{lemma}\label{isdifferential}
We have that $\partial_{d-1} \circ \partial_d = 0$ for $d=1,\dots,n+1$,
hence, \cref{src} is a chain complex.
\end{lemma}

\begin{proof}
By the super interchange law, $\partial_{d-1} \circ \partial_d$
factorizes as the composition first of the embedding
$$
\begin{tikzcd}
C_d = U_{(k+d);n-d}^\ell \otimes_{\EOH_{n-d}^\ell}
V_{n-d;(d)}
\arrow[r,"\operatorname{inc}" above]
&U_{(k+d-2,2);n-d}^\ell \otimes_{\EOH_{n-d}^\ell}
V_{n-d;(2,d-2)}
\arrow[rrr,"b_{(k+d-2),(2)}\otimes\;c_{(2),(d-2)}" above]
&&&\phantom{s}
\end{tikzcd}
$$

\vspace{-4mm}

$$
\begin{tikzcd}
\qquad\qquad\qquad U_{(k+d-2);n-d+2}^\ell\otimes_{\EOH_{n-d+2}^\ell}
U_{(2);n-d}^\ell \otimes_{\EOH_{n-d}^\ell}
V^\ell_{n-d;(2)}\otimes_{\EOH_{n-d+2}^\ell}
V^\ell_{n-d+2;(d-2)}
\end{tikzcd}
$$
then the map $\id\otimes(\partial\circ\iota)\otimes\id$ 
from there to 
$C_{d-2}=U_{(k+d-2);n-d+2}^\ell\otimes_{\EOH_{n-d+2}^\ell}
V^\ell_{n-d+2;(d-2)}$,
where 
$$
\begin{tikzcd}
\iota:U_{(2);n-d}^\ell \otimes_{\EOH_{n-d}^\ell}
V_{n-d;(2)}^\ell\arrow[r,"\operatorname{inc}" above]&
U_{(1,1);n-d}^\ell\otimes_{\EOH_{n-d}^\ell}
V_{n-d;(1,1)}^\ell
\arrow[rr,"b_{(1),(1)}\otimes\;c_{(1),(1)}" above]&&\phantom{xwoiprgeqjrgpqjj}
\end{tikzcd}
$$

\vspace{-4mm}

$$
\phantom{oiergnoweirjgf}U_{n-d+1}^\ell\otimes_{\EOH_{n-d+1}^\ell}
U_{n-d}^\ell\otimes_{\EOH_{n-d}^\ell}
V_{n-d}^\ell
\otimes_{\EOH_{n-d+1}^\ell}
V_{n-d+1}^\ell,
$$

$$
\begin{tikzcd}
\partial:\!U_{n-d+1}^\ell\!\otimes_{\EOH_{n-d+1}^\ell}\!\!\!\!
U_{n-d}^\ell\otimes_{\EOH_{n-d}^\ell}\!\!\!
V_{n-d}^\ell
\!\otimes_{\EOH_{n-d+1}^\ell}\!\!\!
V_{n-d+1}^\ell
\arrow[rr,"\id\otimes\ev_{n-d}\otimes\id" above]&&
U_{n-d+1}^\ell
\!\otimes_{\EOH_{n-d+1}^\ell}\!\!\!\!
\EOH_{n-d+1}^\ell\!\otimes_{\EOH_{n-d+1}^\ell}\!\!\!
V_{n-d+1}^\ell
\end{tikzcd}
$$

\vspace{-4mm}

$$
\begin{tikzcd}
\phantom{aoeikrerfefg;rgja;porgi}
U_{n-d+1}^\ell
\otimes_{\EOH_{n-d+1}^\ell}
V_{n-d+1}^\ell
\arrow[rr,"\ev_{n-d+1}" above]&&
\EOH_{n-d+2}^\ell.\phantom{rgwtwrthtgrw}
\end{tikzcd}
$$
Thus, we are reduced to proving that $\partial$ takes vectors in the image
of $\iota$ to zero.
By \cref{lemma2}, the image of $\iota$
is equal to the image of the projection
$\rho_{(1^2);n-d}\big((\chi\omega)_2\big)
\otimes \lambda_{n-d;(1^2)}\big((\omega\chi)_2\big)$.
This projection equals
$$
\rho_{(1^2);n-d}(x_1 \tau_1)
\otimes \lambda_{n-d;(1^2)}(\tau_1 x_1)
=\big(
\rho_{(1^2);n-d}(\tau_1) \otimes  
\lambda_{n-d;(1^2)}(\tau_1)\big)
\circ\big(
\rho_{(1^2);n-d}(x_1)
\otimes \lambda_{n-d;(1^2)}(x_1)\big).
$$
Finally, to complete the proof,
we observe that
$\partial \circ \big(\rho_{(1^2);n-d}(\tau_1) \otimes \lambda_{n-d;(1^2)}(\tau_1)\big) = 0$
because
$$
\partial \circ \big(\rho_{(1^2);n-d}(\tau_1) \otimes \id\otimes \id\big)
=
\partial \circ  \big(\id \otimes \id \otimes \lambda_{n-d;(1^2)}(\tau_1)\big)
$$
thanks to \cref{mate3},
and 
$\lambda_{n-d;(1^2)}(\tau_1) \circ \lambda_{n-d;(1^2)}(\tau_1)
= \lambda_{n-d;(1^2)}\big(\tau_1^2\big) = 0$.
\end{proof}

Now we need to understand the ``numerology" of \cref{src}.
In fact, the combinatorial \cref{numerology} derived long ago
is just what we need for this. Recall the definitions
of $b_{m,n}(r), c_{m,n}(r) \in \Z[q,q^{-1}]^\pi$
made in the statement of that lemma.
The following shows that $c_{n+k,n}(d)$ is 
the graded superrank of $C_d$ either as a 
free graded 
right $\EOH_n^d$-supermodule or a free graded left $\EOH_{n'}^d$-supermodule. 
We will also see in a bit that $b_{n+k,n}(d)$ is the graded rank of
$\im \partial_{d}$ for $d=0,1,\dots,n$.


\begin{lemma}\label{depoebay}
The vectors
\begin{equation}\label{thebasis}
\bigg\{
u_{(k+d);n-d}\big(\bar \sigma^{(k+d)}_\lambda\big)
\otimes v_{n-d;(d)}\big(\bar \sigma^{(d)}_\mu\big)
\:\bigg|\:
(\lambda,\mu) \in \GPar{(k+d)}{n}\times \GPar{d}{(n-d)}
\bigg\}
\end{equation}
give a basis for $C_d$ as a free right $\EOH_n^\ell$-supermodule.
Hence, 
as a graded right $\EOH_n^\ell$-supermodule,
$C_d$ is free
of graded superrank $c_{n+k,n}(d)$.
It is also free as a graded left $\EOH_{n'}^\ell$-superbimodule
with the same graded superrank.
\end{lemma}

\begin{proof}
\cref{lemma01} implies that it is free as a graded right $\EOH_n^\ell$-supermodule with basis
\cref{thebasis}.
The formula for its graded superrank then follows using
\cref{qbinomialformula}. It is also free as a graded left $\EOH_{n'}^\ell$-supermodule
thanks to \cref{lemma01} again. Since $\EOH_{n'}^\ell
\cong \EOH_n^\ell$, its graded superrank for $\EOH_{n'}^\ell$
is the same as for $\EOH_n^\ell$.
\end{proof}

Recall that $\EOH_n^\ell$ is positively graded 
with degree 0 component isomorphic
the ground field $\k$.
We apply the functor
$-\otimes_{\EOH_n^\ell} \k$ to \cref{src} to obtain the chain complex
\begin{equation}\label{src2}
0 \stackrel{\overline{\partial}_{n+1}}{\longrightarrow} \overline{C}_n 
\stackrel{\overline{\partial}_n}{\longrightarrow} \cdots 
\stackrel{\overline{\partial}_{d+1}}{\longrightarrow} \overline{C}_d \stackrel{\overline{\partial}_d}{\longrightarrow} \overline{C}_{d-1}\stackrel{\overline{\partial}_{d-1}}{\longrightarrow}\cdots\stackrel{\overline{\partial}_2}{\longrightarrow} \overline{C}_1 
\stackrel{\overline{\partial}_1}{\longrightarrow} \overline{C}_0 \stackrel{\overline{\partial}_0}{\longrightarrow} 0
\end{equation}
of graded left $\ROH_{n'}^\ell$-supermodules.
\cref{depoebay} implies that 
$\overline{C}_d$ is of graded superdimension $c_{n+k,n}(d)$.

\begin{lemma}\label{mainstep}
Suppose we are given that $\dim \im \bar\partial_d \geq |b_{n+k,n}(d)|$
for $d=1,\dots,n$, 
where $|b_{n+k,n}(d)|$ denotes the natural number obtained by
applying the evaluation
map $\Z[q,q^{-1}]^\pi \rightarrow \Z, q \mapsto 1, \pi \mapsto 1$
to $b_{n+k,n}(d)$.
Then \cref{SRC} is true.
\end{lemma}

\begin{proof}
We first consider the specialized complex \cref{src2},
showing that $\im \overline{\partial}_d = \ker \overline{\partial}_{d-1}$
and that it is of graded superdimension $b_{n+k,n}(d)$
for each $d=1,\dots,n$. This follows by induction
on $d$, defining $\overline{\partial}_{-1}$ to be the zero map
so that we can start the induction at $d=0$. The induction base 
holds because $b_{n+k,n}(0) = 0$.
For the induction step, take $0 \leq d < n$ and assume
that $\im\overline{\partial}_d = \ker \overline{\partial}_{d-1}$
is of graded superdimension $b_{n+k,n}(d)$.
We have that
$
\overline{C}_d = \overline{B}'_d \oplus \overline{Z}_d
$
where $\overline{Z}_d := \ker \overline{\partial}_d$
and $\overline{B}'_d \simeq \im \overline{\partial}_d$ is a complementary graded superspace.
By induction, $\gsdim \overline{B}'_d = b_{n+k,n}(d)$.
We have that
$\dim \overline{C}_d = \dim \im \overline{\partial}_d+\dim \ker \overline{\partial}_d$ so, using \cref{numerology} for the final equality, get that
\begin{align*}
|c_{n+k,n}(d)| = |b_{n+k,n}(d)| \!+\! \dim \ker \overline{\partial}_d
\geq |b_{n+k,n}(d)|\!+\! \dim \im \overline{\partial}_{d+1}
\geq |b_{n+k,n}(d)|\!+\!|b_{n+k,n}(d+1)| = |c_{n+k,n}(d)|.
\end{align*}
This means that equality holds throughout, thereby proving
that $\im \overline{\partial}_{d+1} = \ker \overline{\partial}_d$.
The same sequence of equalities without evaluating
at $q=\pi = 1$ now gives that
$\gsdim \im \overline{\partial}_{d+1} = b_{n+k,k}(d+1)$,
and the argument is complete.

Next we show that
$\im \partial_d = \ker \partial_{d-1}$
and that it is free as a graded right $\EOH_n^\ell$-supermodule
of graded superrank $b_{n+k,n}(d)$
for each $d=1,\dots,n$. This is a similar induction to the one in the previous paragraph.
For the induction step, we take $0 \leq d < n$
and assume that we have shown already that
$\im \partial_d$ is free of 
graded superrank $b_{n+k,n}(d)$. Consider the short exact sequence
 $$
 0 \longrightarrow Z_d \longrightarrow C_d \longrightarrow \im \partial_d \longrightarrow 0
 $$
 where $Z_d := \ker \partial_d$.
 Since $\im \partial_d$ is free, this short exact sequence splits, 
 so we have that 
 $
 C_d = B_d' \oplus Z_d
 $
 where $B'_d \simeq \im \partial_d$ is a complement to $Z_d$ in $C_d$
 as a graded right $\EOH_n^\ell$-supermodule.
Moreover, $\overline{Z}_d$ from the previous paragraph
 is $Z_d \otimes 1$. As it is a summand of $C_d$, which is free,
 we deduce that $Z_d$ is a free graded right $\EOH_n^\ell$-supermodule
 with $\gsrank Z_d = \gsdim \overline{Z}_d = b_{n+k,k}(d+1)$.
The map $\partial_{d+1}:C_{d+1}\rightarrow Z_d$
 is surjective because $\id\otimes\partial_{d+1}:\overline{C}_{d+1}
\rightarrow \overline{Z}_d$ is surjective according to the previous paragraph. 
 We deduce that $\im \partial_{d+1} = \ker \partial_d$
is free of graded superrank $b_{n+k,k}(d+1)$, and the argument 
is complete.

So now we have shown that $\im \partial_d = \ker \partial_{d-1}$
is free as a graded right $\EOH_n^\ell$-supermodule
of graded superrank $b_{n+k,n}(d)$ for $d=1,\dots,n$.
The same is true as a graded left $\EOH_{n'}^\ell$-supermodule
since $\EOH_{n'}^\ell \cong \EOH_n^\ell$ and all of the numerology is the same.

To complete the proof of \cref{SRC}, it just remains
to prove the assertion about the top degree homology.
As $\EOH_{n'}^\ell \cong \EOH_n^\ell$, it suffices to show that it is 
free 
of graded superrank
$(\pi q^2)^{\binom{n+1}{2}+nk}$ both as a graded right $\EOH_n^\ell$- and as a graded left $\EOH_{n'}^\ell$-supermodule.
We have already shown that the image of $\partial_n$ is free
of graded superrank $b_{n+k,k}(n)$. Hence, since $C_n$ is free of
graded superrank $c_{n+k,n}(n)$ by \cref{depoebay}, 
we deduce that $\ker \partial_n$
is free of graded superrank $c_{n+k,n}(n)-b_{n+k,k}(n)$.
Thus, we are reduced to showing that
$c_{n+k,n}(n)-b_{n+k,k}(n) = (\pi q^2)^{\binom{n+1}{2}+nk}.$
This follows from the following calculation:
\begin{align*}
c_{n+k,n}(n) - b_{n+k,k}(n) &=
(\pi q^{-2})^{\binom{n}{2}}
q^{(n+k)n}\sqbinom{2n+k}{n}_{q,\pi}\\
&\qquad\qquad-
(\pi q^{-2})^{\binom{n}{2}}
\sum_{s=0}^{n-1} (\pi q^{2})^{(n+k)(n-s-1)}
q^{(n+k-1)(s+1)}
\sqbinom{n+k+s}{s+1}_{q,\pi}\\
&=
(\pi q^{-2})^{\binom{n}{2}}
\sum_{s=0}^{n} (\pi q^{2})^{(n+k)(n-s)}
q^{(n+k-1)s}
\sqbinom{n+k+s-1}{s}_{q,\pi}\\&\qquad\qquad-
(\pi q^{-2})^{\binom{n}{2}}
\sum_{s=1}^{n} (\pi q^{2})^{(n+k)(n-s)}
q^{(n+k-1)s}
\sqbinom{n+k+s-1}{s}_{q,\pi}\\
&=
(\pi q^{2})^{(n+k)n-\binom{n}{2}}
= (\pi q^2)^{\binom{n+1}{2}+nk}.
\end{align*}
Here, we have used the definitions in \cref{numerology} for the first equality and \cref{numerologyc} for the second.
\end{proof}

\begin{remark}
From \cref{K0}(1)--(2),
tensoring with the $d$th term
$U^\ell_{(k+d);n-d} \otimes_{\EOH_{n-d}^\ell} V_{n-d;(d)}^\ell$
in the singular Rouquier complex 
corresponds at the level of Grothendieck groups to the 
$\Z[q,q^{-1}]^\pi$-module homomorphism
$1_{-k}\mathbf{V}(-\ell)
\rightarrow 1_k\mathbf{V}(-\ell)$
defined by the action of
$$
q^{(n-d)(k+d)+d(\ell-3(n-d)-2d+1)}
E^{(k+d)}F^{(d)}1_{2n-\ell}
= q^{nk} 
\left(q^d E^{(k+d)}F^{(d)}
1_{2n-\ell}\right).
$$
From the original definition of the map $T$ in
\cref{oddreflection},
it follows that  multiplication by the
Euler characteristic of \cref{src}
corresponds to 
$q^{nk}T: 1_{-k} \mathbf{V}(-\ell)
\rightarrow 1_k \mathbf{V}(-\ell)$.
Given that the homology of the complex
vanishes in all but the top degree, 
it then follows by \cref{fromR} that the
top homology is 
$(\pi q^2)^{\binom{n+1}{2}+nk}\left[\EOH_{n'}^\ell\right]$
which, reassuringly, agrees with the final assertion of \cref{SRC}.
\end{remark}

\begin{proof}[Proof of \cref{SRC}]
Fix $d$ with $1 \leq d \leq n$.
Define an {\em initial pair} to be
$(\lambda,\mu) \in \GPar{(k+d)}{n}\times\GPar{d}{(n-d)}$
such that 
$\lambda_{k+d} = d-1-s$ and $\mu_{d-s} = n-d$
for some $0 \leq s \leq d-1$. This condition is illustrated by the following picture:
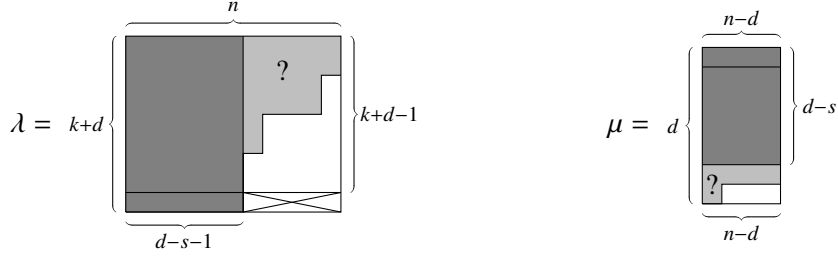
\begin{align*}
\lambda &= 
\begin{tikzpicture}[anchorbase,scale=1.3]
 \draw[-] (1,0) to (1,1.8);
 \draw[-] (1,.2) to (2,.2);
  \fill [lightgray] (-.2,1.8) rectangle (2,1.4);
  \fill [lightgray] (-.2,1.4) rectangle (1.8,1);
  \fill [lightgray] (-.2,1) rectangle (1.2,.6);
  \fill [gray] (-.2,0) rectangle (1,1.8);
  \draw [-] (1,0) to (1,1.8);
  \draw[-] (-.2,0) to (1,0) to (1,.6) to (1.2,.6) to (1.2,1) to (1.8,1) to (1.8,1.4) to (2,1.4);
 \draw[-] (-.2,0) to (2,0) to (2,1.8) to (-.2,1.8) to (-.2,0);
 \draw[-] (-.2,.2) to (1,.2);
  \draw [decorate, decoration = {calligraphic brace}] (-.2,1.9) --  (2,1.9);
\node at (.9,2.1) {$\scriptstyle n$};
  \draw [decorate, decoration = {calligraphic brace}] (-.3,0) --  (-.3,1.8);
\node at (-.6,.9) {$\scriptstyle k+d$};
  \draw [decorate, decoration = {calligraphic brace,mirror}] (-.2,-.1) --  (1,-.1);
	\node at (0.4,-.3) {$\scriptstyle d-s-1$};
	\node at (1.4,1.45) {?};
  \draw [decorate, decoration = {calligraphic brace,mirror}] (2.1,.2) --  (2.1,1.8);
\node at (2.5,1) {$\scriptstyle k+d-1$};
	\draw[-] (1,0) to (2,.2);
	\draw[-] (1,0.2) to (2,0);
\end{tikzpicture}
&
\mu &=
\begin{tikzpicture}[anchorbase,scale=1.3]
 \fill [lightgray] (0,0) rectangle (.2,.4);
 \fill [lightgray] (.2,.2) rectangle (.8,.4);
 \fill [gray] (0,.4) rectangle (.8,1.6);
\draw[-] (0,1.4) to (.8,1.4);
  \draw[-] (0,0) to (.8,0) to (.8,1.6) to (0,1.6) to (0,0);
 \draw [-] (0,.4) to (.8,.4);
 \draw[-] (0,0) to (.2,0) to (.2,.2) to (.8,.2);
	\node at (.1,.2) {?};
 \draw [decorate, decoration = {calligraphic brace}] (0,1.7) --  (.8,1.7);
	\node at (0.4,1.9) {$\scriptstyle n-d$};
   \draw [decorate, decoration = {calligraphic brace,mirror}] (0,-.1) --  (.8,-.1);
\node at (.4,-.3) {$\scriptstyle n-d$};
\draw [decorate, decoration = {calligraphic brace}] (-.1,0) --  (-.1,1.6);
	\node at (-0.3,.8) {$\scriptstyle d$};
 \draw [decorate, decoration = {calligraphic brace,mirror}] (.9,.4) --  (.9,1.6);
	\node at (1.2,1) {$\scriptstyle d-s$};
\end{tikzpicture}
\end{align*}
Let $I$ be the set of all initial pairs. Note that
\begin{equation}\label{siz}
|I| = |b_{n+k,n}(d)|.
\end{equation}
To see this, the number of $(\lambda,\mu)\in I$
with $\lambda_{k+d} = d-1-s$ and $\mu_{d-s} = n-d$
is $|\GPar{(k+d-1)}{(n-d+s+1)} \times \GPar{s}{(n-d)}|$, which is
$\binom{n+k+s}{n-d+s+1}\binom{n-d+s}{s}$. 
Summing\footnote{This is all we need here, but with a little more care using also \cref{qbinomialformula}, 
this argument can be used to show that
the coefficient of $q^{2r} \pi^r $ in $b_{n+k,n}(d)$
is equal to
the number of $(\lambda,\mu) \in I$ with $|\lambda|+|\mu| = 2r$, explaining the 
definition of $b_{n+k,n}(d)$ itself rather than merely
its evaluation at $q=\pi=1$.} over $s=0,1,\dots,d-1$
gives $|b_{n+k,n}(d)|$ by the original definition of this natural number.

Also define a {\em terminal pair}
to be 
$(\kappa,\nu) \in \GPar{(k+d-1)}{n}\times\GPar{(d-1)}{(n-d+1)}$
such that 
$\kappa_{k+d-1} \geq d-1-s$ and $\nu_{d-s} <\nu_{d-s-1}= n-d+1$
for some $0 \leq s \leq d-1$:
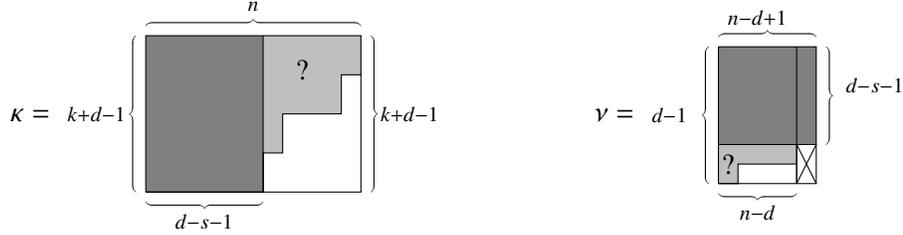
\begin{align*}
\kappa &= 
\begin{tikzpicture}[anchorbase,scale=1.3]
 \draw[-] (1,.2) to (1,1.8);
  \fill [lightgray] (-.2,1.8) rectangle (2,1.4);
  \fill [lightgray] (-.2,1.4) rectangle (1.8,1);
  \fill [lightgray] (-.2,1) rectangle (1.2,.6);
  \fill [gray] (-.2,.2) rectangle (1,1.8);
  \draw [-] (1,.2) to (1,1.8);
  \draw[-] (-.2,.2) to (1,.2) to (1,.6) to (1.2,.6) to (1.2,1) to (1.8,1) to (1.8,1.4) to (2,1.4);
 \draw[-] (-.2,.2) to (2,.2) to (2,1.8) to (-.2,1.8) to (-.2,.2);
  \draw [decorate, decoration = {calligraphic brace}] (-.2,1.9) --  (2,1.9);
\node at (.9,2.1) {$\scriptstyle n$};
  \draw [decorate, decoration = {calligraphic brace}] (-.3,.2) --  (-.3,1.8);
\node at (-.7,1) {$\scriptstyle k+d-1$};
  \draw [decorate, decoration = {calligraphic brace,mirror}] (-.2,.1) --  (1,.1);
	\node at (0.4,-.1) {$\scriptstyle d-s-1$};
	\node at (1.4,1.45) {?};
  \draw [decorate, decoration = {calligraphic brace,mirror}] (2.1,.2) --  (2.1,1.8);
\node at (2.5,1) {$\scriptstyle k+d-1$};
\end{tikzpicture}
&
\nu &=
\begin{tikzpicture}[anchorbase,scale=1.3]
 \fill [lightgray] (0,0) rectangle (.2,.4);
 \fill [lightgray] (.2,.2) rectangle (.8,.4);
 \fill [gray] (0,.4) rectangle (1,1.4);
 \draw[-] (0,0) to (1,0) to (1,1.4) to (0,1.4) to (0,0);
 \draw[-] (.8,0) to (.8,1.4);
 \draw [-] (0,.4) to (1,.4);
 \draw[-] (0,0) to (.2,0) to (.2,.2) to (.8,.2);
 	\node at (.1,.2) {?};
 \draw [decorate, decoration = {calligraphic brace}] (0,1.5) --  (1,1.5);
	\node at (0.4,1.7) {$\scriptstyle n-d+1$};
   \draw [decorate, decoration = {calligraphic brace,mirror}] (0,-.1) --  (.8,-.1);
\node at (.4,-.3) {$\scriptstyle n-d$};
\draw [decorate, decoration = {calligraphic brace}] (-.1,0) --  (-.1,1.4);
	\node at (-0.5,.7) {$\scriptstyle d-1$};
 \draw [decorate, decoration = {calligraphic brace,mirror}] (1.1,.4) --  (1.1,1.4);
	\node at (1.6,1) {$\scriptstyle d-s-1$};
	\draw[-] (.8,0) to (.8,.4);
    \draw[-] (.8,0) to (1,.4);
    \draw[-] (1,0) to (.8,.4);
\end{tikzpicture}
\end{align*}
Let $T$ be the set of all terminal pairs.
Our final combinatorial observation
is that there is a bijection \begin{equation}\label{fd}
f:I \stackrel{\sim}{\rightarrow} T
\end{equation} 
taking
$(\lambda,\mu) \in I$ to $(\lambda^-,\mu^+) \in T$
where $\lambda^-$ is obtained from $\lambda$ by removing the
bottom row of its Young diagram, and 
$\mu^+$ is obtained from $\mu$ by 
adding one box to the end of the first $(d-s)$ rows
then removing completely the top row.

Now we are going to make an explicit computation of
the differential 
$\bar\partial_d$ in terms of the basis
for $\overline{C}_d = C_d \otimes_{\EOH_n^\ell} \k$
consisting of the vectors
\begin{equation}
w(\lambda,\mu) := 
u_{n-d;(k+d)}\big(\bar \sigma^{(k+d)}_\lambda\big)
\otimes v_{n-d;(d)}\big(\bar \sigma^{(d)}_\mu\big)
\otimes 1 \in 
U_{(k+d);n-d}^\ell\otimes_{\EOH_{n-d}^\ell}
V_{n-d;(d)}^\ell \otimes_{\EOH_n^\ell} \k 
\end{equation}
for $(\lambda,\mu) \in 
\GPar{(k+d)}{n}\times \GPar{d}{(n-d)}$; cf. \cref{depoebay}.
Order pairs $(\kappa,\nu) \in \GPar{(k+d-1)}{n} \times \GPar{(d-1)}{(n-d+1)}$
so that $(\kappa',\nu') < (\kappa,\nu)$
if either $|\kappa'| < |\kappa|$,
or $|\kappa'| = |\kappa|$ and $\nu' <_{\lex} \nu$.
We claim for $(\lambda,\mu) \in I$ that
\begin{equation}\label{triangularplace}
\overline{\partial}_d(w(\lambda,\mu))
= \pm w(\lambda^-,\mu^+) + \text{
(a linear combination of $w(\kappa,\nu)$ for $(\kappa,\nu) < (\lambda^-,\mu^+)$)}.
\end{equation}
Given the claim, it follows by
\cref{siz,fd} that
$\dim \im \overline{\partial}_d \geq |b_{n+k,n}(d)|$, so that the theorem 
follows by \cref{mainstep}.

It remains to prove \cref{triangularplace}.
Take $(\lambda,\mu) \in 
\GPar{(k+d)}{n}\times \GPar{d}{(n-d)}$
and consider $\overline{\partial}_d(w(\lambda,\mu))$.
According to 
the definition of $\overline{\partial}_d$, we have to apply three
different maps to $w(\lambda,\mu)$ arising from $b_{(d+k-1),(1)}$,
$c_{(1),(d-1)}$ and $\ev_{n-d}$. We apply these maps one by one.
\begin{itemize}
\item
First, the map $b_{(d+k-1),(1)}$ comes from the embedding
$$
\OSym_{k+d} \hookrightarrow \OSym_{(k+d-1,1)}
\stackrel{\sim}{\rightarrow}
\OSym_{k+d-1} \otimes \k[x].
$$
\cref{thepierineededlater} shows that this embedding takes
$s_\lambda^{(k+d)}$
to $\sum_{\kappa} \pm s_{\kappa}^{(k+d-1)}
\otimes x^{|\lambda|-|\kappa|}$
summing over all
$\kappa \in \GPar{(k+d-1)}{n}$ whose Young diagram is 
obtained by removing boxes from the bottoms of different columns of
the Young diagram of $\lambda$,  
necessarily including all
$\lambda_{k+d}$ 
boxes on its $(k+d)$th row.
We say simply ``$\kappa$ obtained by removing a row strip from $\lambda$''
for this from now on.
\item
Next, we apply the map $c_{(1),(d-1)}$, which comes from the embedding
$$
\OSym_d \hookrightarrow \OSym_{(1,d-1)}
\stackrel{\sim}{\rightarrow} \k[x] \otimes \OSym_{d-1},
$$
plus some extra signs due to the parity shift.
The version of Pieri obtained by applying $\gamma_d$
to \cref{thepierineededlater} (using also \cref{littlebears})
shows that this
takes $\sigma_\mu^{(d)}$
to $\sum_{\delta} \pm x^{|\mu|-|\delta|} \otimes 
\sigma_{\delta}^{(d-1)}$
summing over 
 $\delta\in \GPar{(d-1)}{(n-d)}$ whose Young diagram is obtained by removing 
 boxes from the bottoms of different columns of the Young diagram
of $\mu$, necessarily including all $\mu_d$ boxes on its $d$th row,
to obtain the Young diagram of partition $\delta$.
We say simply ``$\delta$ obtained by removing a row strip from $\mu$'' for this from now on.
\item
So far, remembering \cref{fancynotation2},
we have shown that
$\big(b_{(k+d-1),(1)}\otimes\; c_{(1),(d-1)}\otimes \id \big)\circ (\operatorname{inc} \otimes \id)$
takes $w(\lambda,\mu)$
to 
$$
\sum_{(\kappa,\delta)} \pm u_{(k+d-1);n-d+1}\big(\sigma_\kappa^{(k+d-1)}\big)
\otimes u_{n-d}\big(x^{|\lambda|-|\kappa|}\big) \otimes v_{n-d}\big(x^{|\mu|-|\delta|}\big)
\otimes v_{n-d+1;(d-1)}\big(\sigma_\delta^{(d-1)}\big)\otimes 1
$$
summing over $(\kappa,\delta) \in \GPar{(k+d-1)}{n}\times\GPar{(d-1)}{(n-d)}$
such that $\kappa$ is obtained by removing a row strip from 
$\lambda$ and $\delta$ is obtained by removing a row strip from $\mu$.
Then we use the definition in
\cref{cupsandcaps} to
apply $(\operatorname{can}\otimes \id) \circ (\id \otimes \ev_{n-d}\otimes \id \otimes \id)$, giving
$$
\overline{\partial}_d(w(\lambda,\mu))
=
\sum_{(\kappa,\delta)} \pm u_{(k+d-1);n-d+1}\big(\sigma_\kappa^{(k+d-1)}\big)
\otimes \bar\gamm^{(n-d+1)}_{(|\lambda|-|\kappa|)+(|\mu|-|\delta|)-
(n-d)}
v_{n-d+1;(d-1)}\big(\sigma_\delta^{(d-1)}\big)\otimes 1.
$$
summing over all $(\kappa,\delta) \in \GPar{(k+d-1)}{n}\times
\GPar{(d-1)}{(n-d)}$
obtained by removing a row strip from $(\lambda,\mu)
\in \GPar{(k+d)}{n}\times\GPar{d}{(n-d)}$.
\end{itemize}
It remains to commute the
elements $\bar\gamm^{(n-d+1)}_{(|\lambda|-|\kappa|)+(|\mu|-|\delta|)-
(n-d)}$ to the right hand side in this expression.
In view of \cref{lemma01}(1) 
and degree considerations, this will produce
some linear combination of basis vectors of the form $w(\kappa,\nu)$
for $\nu\in\GPar{(d-1)}{(n-d+1)}$ with $|\nu| =
|\mu|+(|\lambda|-|\kappa|)-(n-d)$.
We just need to show that $w(\lambda^-,\mu^+)$ appears with
coefficient $\pm 1$ and all other $w(\kappa,\nu)$
that arise satisfy $(\kappa,\nu) < (\lambda^-,\mu^+)$.
This is clearly the case if $|\kappa| < |\lambda^-|$,
so we may assume from now on that $\kappa$, like $\lambda^-$,
is obtained from $\lambda$
by removing the minimal number of boxes, i.e., just
its bottom row. So we have that $\kappa = \lambda^-$
and $|\lambda|-|\kappa| = \lambda_{k+d}$, which equals $d-s-1$ for 
a unique $0 \leq s < d$.
Also let $p := (d-s+1)+(|\mu|-|\delta|)
- (n-d)$ for short and consider
$$
\bar\gamm^{(n-d+1)}_{p}
v_{n-d+1;(d-1)}\big(\sigma_\delta^{(d-1)}\big)\otimes 1.
$$
By \cref{alreadybasis}, we have that
$\gamm^{(n-d+1)}_p = s^{(n-d+1)}_{(p)}$
plus a linear combination of other $s^{(n-d+1)}_\tau$
for partitions $\tau$ with $|\tau|=p$ and $\h(\tau) > 1$.
Also $\sigma_{(1^p)}^{(d-1)} = \eps_p^{(d-1)}$. 
Using \cref{trains}, we deduce that
\begin{equation}\label{dishy}
\bar\gamm^{(n-d+1)}_{p}
v_{n-d+1;(d-1)}\big(\sigma_\delta^{(d-1)}\big)\otimes 1
=
\pm v_{n-d+1;(d-1)}\big(\sigma_{\delta}^{(d-1)} \eps_p^{(d-1)}\big) \otimes 1
+ (*)
\end{equation}
where $(*)$
is a linear combination of
terms of the form
$v_{n-d+1;(d-1)}\big(\sigma_\delta^{(d-1)} \sigma^{(d-1)}_\tau\big)\otimes 1$
for partitions $\tau$ with $|\tau|=p$ and $\tau_1 > 1$.
We can compute all of these products of dual Schur polynomials
by conjugating with $\smiley_{d-1}$ and 
using the odd Littlewood-Richardson rule; see \cref{discussionhere}.
Remembering also that
$v_{n-d+1;(d-1)}\big(\sigma_\nu^{(d-1)}\big)\otimes 1 = 0$
unless $\nu \in \GPar{(d-1)}{(n-d+1)}$ by \cref{lemma01}(1) again,
we obtain a linear combination of 
basis vectors
$v_{n-d+1;(d-1)}\big(\sigma_{\nu}^{(d-1)}\big)\otimes 1$
for $\nu\in \GPar{(d-1)}{(n-d+1)}$ obtained from $\delta$
by adding $p$ boxes in particular ways.
If $p < d-s-1$ then we cannot have $\nu = \mu^+$
since that has $d-s+1$ boxes in the rightmost column,
whereas that column is empty in $\delta$.
Now suppose that $p = d-s-1$; then $\delta$ is $\mu$ with 
all $(n-d)$ boxes in its first row removed.
In this case, 
the leading term
$\pm v_{n-d+1;(d-1)}\Big(\sigma_{\delta}^{(d-1)}\eps_p^{(d-1)} \big) \otimes 1$ computed via the odd Littlewood-Richardson rule
does produce $\pm v_{n-d+1;(d-1)}\big(\sigma_{\mu^+}^{(d-1)}\big)\otimes 1$
when a column strip of $p$ boxes is added at the top right of the Young diagram of $\delta$. All other basis vectors coming from this leading term
are of the form
$v_{n-d+1;(d-1)}\big(\sigma_{\nu}^{(d-1)}\big)\otimes 1$
for $\nu <_{\lex} \mu^+$.
The basis vectors coming from the lower terms
$v_{n-d+1;(d-1)}\big(\sigma_\delta^{(d-1)}\sigma^{(d-1)}_\tau \big)\otimes 1$
for $\tau$ with $|\tau|=p$ and $\tau_1 > 1$
must also all be of the form
$v_{n-d+1;(d-1)}\big(s_{\nu}^{(d-1)}\big)\otimes 1$
for $\nu <_{\lex} \mu^+$ since when $\tau_1 > 1$ the odd Littlewood-Richardson rule
does not allow all $p$ boxes to be added to the same column of the
Young diagram of $\delta$.
\end{proof}

\section{Non-degeneracy of the odd 2-category
\texorpdfstring{$\UU(\sl_2)$}{U(sl(2))}}\label{catsec}

In this section, we will use the 
string calculus for strict graded monoidal supercategories and 2-supercategories
adopting all of the conventions from \cite{BE}. In particular, 
$f \circ g$ is vertical composition ($f$ on top of $g$) and $f \otimes g$ or simply $fg$ is horizontal composition
($f$ to the left of $g$).
The following definition originated in \cite{EL} and was reformulated in the present terms in \cite{BE2}.

\begin{definition}\label{km2cat}
The {\em odd $\mathfrak{sl}_2$ 2-category}
is the strict graded $2$-supercategory
$\UU(\sl_2)$ with object set $\Z$,
generating 1-morphisms 
$E 1_k =1_{k+2} E:k \rightarrow k\!+\!2$ and
$1_k F = F 1_{k+2}:k\!+\!2 \rightarrow k$ for each $k \in \Z$
whose identity 2-morphisms are represented graphically
by $\Big\uparrow \:{\red\scriptstyle k}={\red \scriptstyle k+2}\:\Big\uparrow$ and 
${\red\scriptstyle k} \:\Big\downarrow=\Big\downarrow \:{\red \scriptstyle k+2}$, respectively,
and generating 2-morphisms
\begin{align}\label{solid1}
&
\begin{tikzpicture}[anchorbase]
	\draw[->] (0.08,-.3) to (0.08,.4);
     \opendot{0.08,0.05};
   \node at (.3,.1) {$\red\scriptstyle{k}$};
\end{tikzpicture}
:E 1_k \Rightarrow E 1_k
&&
\begin{tikzpicture}[anchorbase]
	\draw[->] (0.2,-.3) to (-0.2,.4);
	\draw[->] (-0.2,-.3) to (0.2,.4);
   \node at (.35,.1) {$\red\scriptstyle{k}$};
\end{tikzpicture}:E^2 1_k \Rightarrow E^2 1_k
&&
\begin{tikzpicture}[anchorbase]
	\draw[<-] (0.4,0.3) to[out=-90, in=0] (0.1,-0.2);
	\draw[-] (0.1,-0.2) to[out = 180, in = -90] (-0.2,0.3);
  \node at (0.55,-0.1) {$\red\scriptstyle{k}$};
\end{tikzpicture}
:FE 1_k\Rightarrow 1_k
&&
\begin{tikzpicture}[anchorbase]
	\draw[<-] (0.4,-0.2) to[out=90, in=0] (0.1,0.3);
	\draw[-] (0.1,0.3) to[out = 180, in = 90] (-0.2,-0.2);
  \node at (0.55,0.2) {$\red\scriptstyle{k}$};
\end{tikzpicture}
:1_{k} \Rightarrow EF 1_{k}
\end{align}
which are odd of degree 2, odd of degee $-2$,
even of degree $k+1$, and even of degree $1-k$, respectively.
Then there are three families of relations.
First we have the odd nil-Hecke relations (in the standard formulation
rather than the modified version from \cref{ONH1,ONH3,ONH5,ONH2,ONH4,ONH6}):
\begin{align}
\begin{tikzpicture}[anchorbase,scale=.85]
	\draw[->] (0.28,.4) to[out=90,in=-90] (-0.28,1.1);
	\draw[->] (-0.28,.4) to[out=90,in=-90] (0.28,1.1);
	\draw[-] (0.28,-.3) to[out=90,in=-90] (-0.28,.4);
	\draw[-] (-0.28,-.3) to[out=90,in=-90] (0.28,.4);
   \node at (0.5,0.4) {$\red\scriptstyle{k}$};
\end{tikzpicture}
  &=0&
    \begin{tikzpicture}[anchorbase]
	\draw[<-] (0.45,.8) to (-0.45,-.4);
	\draw[->] (0.45,-.4) to (-0.45,.8);
        \draw[-] (0,-.4) to[out=90,in=-90] (-.45,0.2);
        \draw[->] (-0.45,0.2) to[out=90,in=-90] (0,0.8);
   \node at (.4,0.2) {$\red\scriptstyle{k}$};
\end{tikzpicture}
&=
\begin{tikzpicture}[anchorbase]
	\draw[<-] (0.45,.8) to (-0.45,-.4);
	\draw[->] (0.45,-.4) to (-0.45,.8);
        \draw[-] (0,-.4) to[out=90,in=-90] (.45,0.2);
        \draw[->] (0.45,0.2) to[out=90,in=-90] (0,0.8);
   \node at (0.7,0.2) {$\red\scriptstyle{k}$};
\end{tikzpicture}&
\begin{tikzpicture}[anchorbase]
	\draw[<-] (0.25,.6) to (-0.25,-.2);
	\draw[->] (0.25,-.2) to (-0.25,.6);
      \opendot{0.15,0.42};
   \node at (0.4,0.2) {$\red\scriptstyle{k}$};
\end{tikzpicture}+        
\begin{tikzpicture}[anchorbase]
	\draw[<-] (0.25,.6) to (-0.25,-.2);
	\draw[->] (0.25,-.2) to (-0.25,.6);
      \opendot{-0.13,-0.02};
   \node at (0.4,0.2) {$\red\scriptstyle{k}$};
\end{tikzpicture}
=
\begin{tikzpicture}[anchorbase]
	\draw[<-] (0.25,.6) to (-0.25,-.2);
	\draw[->] (0.25,-.2) to (-0.25,.6);
      \opendot{0.15,-0.02};
   \node at (.3,0.2) {$\red\scriptstyle{k}$};
\end{tikzpicture}
+
\begin{tikzpicture}[anchorbase]
	\draw[<-] (0.25,.6) to (-0.25,-.2);
	\draw[->] (0.25,-.2) to (-0.25,.6);
      \opendot{-0.13,0.42};
   \node at (.3,0.2) {$\red\scriptstyle{k}$};
\end{tikzpicture}
&=
    \begin{tikzpicture}[anchorbase]
      	\draw[->] (0.08,-.2) to (0.08,.6);
	\draw[->] (-0.2,-.2) to (-0.2,.6);
   \node at (0.3,0.2) {$\red\scriptstyle{k}$};
      \end{tikzpicture}
\label{nearlydone}\end{align}
Next we have the {\em right adjunction relations}
asserting that $Q^{-k-1} 1_{k}F $ is
right dual to $E 1_k$ in the $(Q,\Pi)$-envelope of $\UU(\sl_2)$
(cf. \cref{oddadjunction}(1)):
\begin{align}\label{rightadj}
\begin{tikzpicture}[anchorbase]
  \draw[->] (0.3,0) to (0.3,.4);
	\draw[-] (0.3,0) to[out=-90, in=0] (0.1,-0.3);
	\draw[-] (0.1,-0.3) to[out = 180, in = -90] (-0.1,0);
	\draw[-] (-0.1,0) to[out=90, in=0] (-0.3,0.3);
	\draw[-] (-0.3,0.3) to[out = 180, in =90] (-0.5,0);
  \draw[-] (-0.5,0) to (-0.5,-.4);
   \node at (0.5,.1) {$\red\scriptstyle{k}$};
\end{tikzpicture}
&=
\begin{tikzpicture}[anchorbase]
  \draw[->] (0,-0.4) to (0,.4);
   \node at (0.2,0.1) {$\red\scriptstyle{k}$};
\end{tikzpicture}&
\begin{tikzpicture}[anchorbase]
  \draw[->] (0.3,0) to (0.3,-.4);
	\draw[-] (0.3,0) to[out=90, in=0] (0.1,0.3);
	\draw[-] (0.1,0.3) to[out = 180, in = 90] (-0.1,0);
	\draw[-] (-0.1,0) to[out=-90, in=0] (-0.3,-0.3);
	\draw[-] (-0.3,-0.3) to[out = 180, in =-90] (-0.5,0);
  \draw[-] (-0.5,0) to (-0.5,.4);
   \node at (-0.7,0.1) {$\red\scriptstyle{k}$};
\end{tikzpicture}
&=
\begin{tikzpicture}[anchorbase]
  \draw[<-] (0,-0.4) to (0,.4);
   \node at (-0.2,0.1) {$\red\scriptstyle{k}$};
\end{tikzpicture}
\end{align}
Finally there are some {\em inversion relations}. 
To formulate these, we first
introduce new 2-morphisms
\begin{align}\label{sigrel}
&
\begin{tikzpicture}[anchorbase]
	\draw[<-] (0.3,-.3) to (-0.3,.4);
	\draw[->] (-0.3,-.3) to (0.3,.4);
   \node at (.4,.05) {$\red\scriptstyle k$};
\end{tikzpicture}
:=
\begin{tikzpicture}[anchorbase]
	\draw[->] (0.3,-.5) to (-0.3,.5);
\draw[->] (-0.7,.4) to[out=-90,in=180] (-.35,-.4) to[out=0,in=180] (.35,.4) to[out=0,in=90] (.7,-.5);
   \node at (.9,.05) {$\red\scriptstyle k$};
\end{tikzpicture}
:E F 1_k \rightarrow F E 1_k.
\end{align}
Then, denoting powers of the dot generator by labelling it with a natural number,  we require that the following (not necessarily homogeneous) 2-morphisms are isomorphisms:
\begin{align}\label{inv1}
\left(\:
\displaystyle\begin{tikzpicture}[baseline = 0]
	\draw[<-] (0.3,-.3) to (-0.3,.4);
	\draw[->] (-0.3,-.3) to (0.3,.4);
   \node at (.4,.05) {$\red\scriptstyle{k}$};
\end{tikzpicture}
\qquad
\begin{tikzpicture}[baseline = 0]
	\draw[<-] (0.4,-.1) to[out=90, in=0] (0.1,0.4);
	\draw[-] (0.1,0.4) to[out = 180, in = 90] (-0.2,-.1);
  \node at (0.6,0.2) {$\red\scriptstyle{k}$};
\end{tikzpicture}
\qquad
\begin{tikzpicture}[baseline = 0]
	\draw[<-] (0.4,-.1) to[out=90, in=0] (0.1,0.4);
	\draw[-] (0.1,0.4) to[out = 180, in = 90] (-0.2,-.1);
  \node at (0.6,0.2) {$\red\scriptstyle{k}$};
     \opendot{-0.15,0.2};
\end{tikzpicture}\qquad
\cdots\qquad
\begin{tikzpicture}[baseline = 0]
	\draw[<-] (0.5,-.1) to[out=90, in=0] (0.15,0.4);
	\draw[-] (0.15,0.4) to[out = 180, in = 90] (-0.2,-.1);
  \node at (0.7,0.2) {$\red\scriptstyle{k}$};
      \node at (0.15,0.2) {$\scriptstyle{k-1}$};
     \opendot{-0.15,0.2};
\end{tikzpicture}
\right)^T&:
E F 1_k\stackrel{\sim}{\rightarrow}
F E 1_k \oplus 1_k^{\oplus k}&&
\text{for $k \geq 0$}\\
\left(\:
\displaystyle\begin{tikzpicture}[baseline = 0]
	\draw[<-] (0.3,-.3) to (-0.3,.4);
	\draw[->] (-0.3,-.3) to (0.3,.4);
   \node at (.4,-.05) {$\red\scriptstyle{k}$};
\end{tikzpicture}
\qquad
\begin{tikzpicture}[baseline = 0]
	\draw[<-] (0.4,0.2) to[out=-90, in=0] (0.1,-.3);
	\draw[-] (0.1,-.3) to[out = 180, in = -90] (-0.2,0.2);
  \node at (0.6,-0.25) {$\red\scriptstyle{k}$};
      \end{tikzpicture}\qquad
      \begin{tikzpicture}[baseline = 0]
	\draw[<-] (0.4,0.2) to[out=-90, in=0] (0.1,-.3);
	\draw[-] (0.1,-.3) to[out = 180, in = -90] (-0.2,0.2);
      \opendot{.38,0};
  \node at (0.6,-0.25) {$\red\scriptstyle{k}$};
      \end{tikzpicture}
      \qquad\cdots\qquad
      \begin{tikzpicture}[baseline = 0]
	\draw[<-] (0.5,0.2) to[out=-90, in=0] (0.1,-.3);
	\draw[-] (0.1,-.3) to[out = 180, in = -90] (-0.3,0.2);
  \node at (0.65,-0.25) {$\red\scriptstyle{k}$};
      \node at (0.1,0) {$\scriptstyle{-k-1}$};
      \opendot{.47,0};
      \end{tikzpicture}\right)
&:E F 1_k \oplus  1_k^{\oplus (-k)}
\stackrel{\sim}{\rightarrow}
 F E 1_k&&\text{for $k \leq 0$.}\label{inv2}
\end{align}
The morphisms depicted in \cref{inv1,inv2} 
represent a $(k+1)\times 1$ matrix
and a $1\times (1-k)$ matrix of 2-morphisms
in $\UU(\sl_2)$, respectively, i.e., they are 2-morphisms in the additive envelope of $\UU(\sl_2)$. Saying that they are isomorphisms means that 
there are some 
further generating $2$-morphisms in $\UU(\sl_2)$ which provide
the matrix entries of two-sided inverses to these morphisms.
\end{definition}

The defining relations \cref{solid1,nearlydone,rightadj,inv1,inv2} look quite innocent but they imply many further relations.
In order to record some of these, we first 
need to introduce some further shorthand for 
generating $2$-morphisms whose existence is provided by the inversion relation. First, we have the leftward crossing and the leftward cups and caps \begin{align}\label{solid2}
\begin{tikzpicture}[baseline = 0]
	\draw[->] (0.3,-.3) to (-0.3,.4);
	\draw[<-] (-0.3,-.3) to (0.3,.4);
   \node at (.45,.1) {$\red\scriptstyle{k}$};
\end{tikzpicture}:F E 1_k \Rightarrow E F 1_k
&&
\begin{tikzpicture}[baseline = 0]
	\draw[-] (0.4,0.3) to[out=-90, in=0] (0.1,-0.2);
	\draw[->] (0.1,-0.2) to[out = 180, in = -90] (-0.2,0.3);
  \node at (0.55,-0.1) {$\red\scriptstyle{k}$};
\end{tikzpicture}
:1_{k} EF\Rightarrow 1_{k}
&&
\begin{tikzpicture}[baseline = 0]
	\draw[-] (0.4,-0.2) to[out=90, in=0] (0.1,0.3);
	\draw[->] (0.1,0.3) to[out = 180, in = 90] (-0.2,-0.2);
  \node at (0.55,0.2) {$\red\scriptstyle{k}$};
\end{tikzpicture}
:1_{k} \Rightarrow FE 1_{k}
\end{align}
which are defined as follows.
\begin{itemize}
\item
We let
$\begin{tikzpicture}[baseline = 0,scale=.8]
	\draw[->] (0.3,-.3) to (-0.3,.4);
	\draw[<-] (-0.3,-.3) to (0.3,.4);
   \node at (.45,.1) {$\red\scriptstyle{k}$};
\end{tikzpicture}:F E 1_k \Rightarrow E F 1_k
$
be the {\em negation} of the leftmost entry of the $1 \times (k+1)$ matrix that is the two-sided inverse of \cref{inv1}
if $k \geq 0$, 
or the {\em negation} of the top entry of the $(1-k)\times 1$ matrix that is the two-sided inverse of \cref{inv2} if $k \leq 0$; cf. \cite[(2.8)--(2.9)]{BE2}.
\item
We let 
$\begin{tikzpicture}[baseline = 0,scale=.8]
	\draw[-] (0.4,0.3) to[out=-90, in=0] (0.1,-0.2);
	\draw[->] (0.1,-0.2) to[out = 180, in = -90] (-0.2,0.3);
  \node at (0.5,-0.1) {$\red\scriptstyle{k}$};
\end{tikzpicture}
$
be the rightmost entry of the $1\times (k+1)$ matrix that is the
two-sided inverse of \cref{inv1} if $k > 0$,
or $(-1)^{k+1}
\begin{tikzpicture}[anchorbase,scale=.75]
	\draw[-] (0.28,.6) to[out=240,in=90] (-0.28,-.1);
	\draw[<-] (-0.28,.6) to[out=300,in=90] (0.28,-0.1);
   \node at (-.05,-.1) {$\scriptstyle{-k}$};
	\draw[-] (0.28,-0.1) to[out=-90, in=0] (0,-0.4);
	\draw[-] (0,-0.4) to[out = 180, in = -90] (-0.28,-0.1);
      \node at (.7,0.1) {$\red\scriptstyle{k}$};
      \opendot{0.28,-0.1};
      \end{tikzpicture}$ if $k \leq 0$; cf. \cite[(2.10)]{BE2}.
\item We let
\begin{tikzpicture}[baseline = 0,scale=.8]
	\draw[-] (0.4,-0.2) to[out=90, in=0] (0.1,0.3);
	\draw[->] (0.1,0.3) to[out = 180, in = 90] (-0.2,-0.2);
  \node at (0.55,0.2) {$\red\scriptstyle{k}$};
\end{tikzpicture}
be the bottom entry of the $(1-k) \times 1$ matrix that is the two-sided inverse of \cref{inv2}
if $k < 0$, or
$(-1)^{k+1}
\begin{tikzpicture}[anchorbase,scale=.75]
	\draw[-] (0.28,-.6) to[out=120,in=-90] (-0.28,.1);
	\draw[<-] (-0.28,-.6) to[out=60,in=-90] (0.28,.1);
   \node at (.5,-.1) {$\red\scriptstyle{k}$};
	\draw[-] (0.28,0.1) to[out=90, in=0] (0,0.4);
	\draw[-] (0,0.4) to[out = 180, in = 90] (-0.28,0.1);
      \node at (-0.05,0.1) {$\scriptstyle{k}$};
    \opendot{-.25,0.1};
\end{tikzpicture}$ if $k \geq 0$; cf. \cite[(2.11)]{BE2}
\end{itemize}
Finally, we have the downward dot and the downward crossing, which are
the right mates of the upward dot and the upward crossing:
\begin{align}\label{phew}
\begin{tikzpicture}[anchorbase]
	\draw[<-] (0.08,-.4) to (0.08,.4);
     \opendot{0.08,0};
   \node at (-.21,.15) {$\red\scriptstyle{k}$};
\end{tikzpicture}
&:=
\begin{tikzpicture}[anchorbase]
  \draw[->] (0.3,0) to (0.3,-.4);
	\draw[-] (0.3,0) to[out=90, in=0] (0.1,0.3);
	\draw[-] (0.1,0.3) to[out = 180, in = 90] (-0.1,0);
	\draw[-] (-0.1,0) to[out=-90, in=0] (-0.3,-0.3);
	\draw[-] (-0.3,-0.3) to[out = 180, in =-90] (-0.5,0);
  \draw[-] (-0.5,0) to (-0.5,.4);
   \opendot{-0.1,0};
    \node at (-0.67,0) {$\red\scriptstyle{k}$};
\end{tikzpicture}
:1_k F \Rightarrow 1_k F&
\begin{tikzpicture}[anchorbase]
	\draw[<-] (0.3,-.4) to (-0.3,.4);
	\draw[<-] (-0.3,-.4) to (0.3,.4);
   \node at (-.35,.1) {$\red\scriptstyle{k}$};
\end{tikzpicture}
&:=
\begin{tikzpicture}[anchorbase,scale=.8]
\draw[->] (1.3,.4) to (1.3,-1.2);
\draw[-] (-1.3,-.4) to (-1.3,1.2);
\draw[-] (.5,1.1) to [out=0,in=90] (1.3,.4);
\draw[-] (-.35,.4) to [out=90,in=180] (.5,1.1);
\draw[-] (-.5,-1.1) to [out=180,in=-90] (-1.3,-.4);
\draw[-] (.35,-.4) to [out=-90,in=0] (-.5,-1.1);
\draw[-] (.35,-.4) to [out=90,in=-90] (-.35,.4);
        \draw[-] (-0.35,-.5) to[out=0,in=180] (0.35,.5);
        \draw[->] (0.35,.5) to[out=0,in=90] (0.8,-1.2);
        \draw[-] (-0.35,-.5) to[out=180,in=-90] (-0.8,1.2);
   \node at (-1.5,0) {$\red\scriptstyle{k}$};
\end{tikzpicture}
:1_k F^2 \Rightarrow 1_k F^2.
\end{align}
The following table summarizes the parity and degree information about all of the 2-morphisms defined thus far.
\begin{equation}\label{table}
\begin{array}{|c|c|c||c|c|c|}
\hline
\text{Generator}&\text{Degree}&\text{Parity}&\text{Generator}&\text{Degree}&\text{Parity}\\\hline
\begin{tikzpicture}[baseline = 0,scale=.8]
	\draw[->] (0.08,-.3) to (0.08,.3);
\opendot{.08,0};
\end{tikzpicture}
{\red\scriptstyle k}
&2&\bar 1&
\begin{tikzpicture}[baseline = 0,scale=.8]
	\draw[<-] (0.08,-.3) to (0.08,.3);
     \opendot{0.08,0.05};
     \end{tikzpicture}
{\red\scriptstyle k}
&2&\bar 1\\
\begin{tikzpicture}[baseline = 0,scale=.8]
	\draw[->] (0.3,-.3) to (-0.3,.3);
	\draw[->] (-0.3,-.3) to (0.3,.3);
   \node at (.4,.05) {$\red\scriptstyle{k}$};
\end{tikzpicture}
&-2&\bar 1&
\begin{tikzpicture}[baseline = 0,scale=.8]
	\draw[<-] (0.3,-.3) to (-0.3,.3);
	\draw[->] (-0.3,-.3) to (0.3,.3);
   \node at (.4,.05) {$\red\scriptstyle{k}$};
\end{tikzpicture}&0&\bar 1\\
\begin{tikzpicture}[baseline = 0,scale=.8]
	\draw[<-] (0.3,-.3) to (-0.3,.3);
	\draw[<-] (-0.3,-.3) to (0.3,.3);
   \node at (.4,.05) {$\red\scriptstyle{k}$};
\end{tikzpicture}
&-2&\bar 1&
\begin{tikzpicture}[baseline = 0,scale=.8]
	\draw[->] (0.3,-.3) to (-0.3,.3);
	\draw[<-] (-0.3,-.3) to (0.3,.3);
   \node at (.4,.05) {$\red\scriptstyle{k}$};
\end{tikzpicture}
&0&\bar 1\\
\begin{tikzpicture}[baseline = 0,scale=.8]
	\draw[<-] (0.4,0.4) to[out=-90, in=0] (0.1,-0.1);
	\draw[-] (0.1,-0.1) to[out = 180, in = -90] (-0.2,0.4);
\node at (0.65,0.2) {$\red\scriptstyle{k}$};
\end{tikzpicture}
&k+1&\0&
\begin{tikzpicture}[baseline = 0,scale=.8]
	\draw[-] (0.4,0.4) to[out=-90, in=0] (0.1,-0.1);
	\draw[->] (0.1,-0.1) to[out = 180, in = -90] (-0.2,0.4);
  \node at (0.65,0.2) {$\red\scriptstyle{k}$};
\end{tikzpicture}
&1-k&\bar k+\1\\
\begin{tikzpicture}[baseline = 0,scale=.8]
	\draw[<-] (0.4,-0.1) to[out=90, in=0] (0.1,0.4);
	\draw[-] (0.1,0.4) to[out = 180, in = 90] (-0.2,-0.1);
   \node at (0.65,0.2) {$\red\scriptstyle{k}$};
\end{tikzpicture}
&1-k&\0&
\begin{tikzpicture}[baseline = 0,scale=.8]
	\draw[-] (0.4,-0.1) to[out=90, in=0] (0.1,0.4);
	\draw[->] (0.1,0.4) to[out = 180, in = 90] (-0.2,-0.1);
  \node at (0.65,0.2) {$\red\scriptstyle{k}$};
\end{tikzpicture}
&k+1&\bar k+\1\\
\hline
\end{array}
\end{equation}
The following relations
are derived from the defining relations
in \cite{BE2}.
\begin{itemize}
\item {\em Downward odd nil-Hecke relations}; cf. \cite[(3.7),(3.9),(3.5)--(3.6)]{BE2}.
\begin{align}
\begin{tikzpicture}[anchorbase,scale=.85]
	\draw[-] (0.28,.4) to[out=90,in=-90] (-0.28,1.1);
	\draw[-] (-0.28,.4) to[out=90,in=-90] (0.28,1.1);
	\draw[<-] (0.28,-.3) to[out=90,in=-90] (-0.28,.4);
	\draw[<-] (-0.28,-.3) to[out=90,in=-90] (0.28,.4);
   \node at (-0.5,0.4) {$\red\scriptstyle{k}$};
\end{tikzpicture}
  &=0&
    \begin{tikzpicture}[anchorbase]
	\draw[->] (0.45,.8) to (-0.45,-.4);
	\draw[<-] (0.45,-.4) to (-0.45,.8);
        \draw[<-] (0,-.4) to[out=90,in=-90] (-.45,0.2);
        \draw[-] (-0.45,0.2) to[out=90,in=-90] (0,0.8);
   \node at (-.7,0.2) {$\red\scriptstyle{k}$};
\end{tikzpicture}
&=
\begin{tikzpicture}[anchorbase]
	\draw[->] (0.45,.8) to (-0.45,-.4);
	\draw[<-] (0.45,-.4) to (-0.45,.8);
        \draw[<-] (0,-.4) to[out=90,in=-90] (.45,0.2);
        \draw[-] (0.45,0.2) to[out=90,in=-90] (0,0.8);
   \node at (-0.5,0.2) {$\red\scriptstyle{k}$};
\end{tikzpicture}&
\begin{tikzpicture}[anchorbase]
	\draw[->] (0.25,.6) to (-0.25,-.2);
	\draw[<-] (0.25,-.2) to (-0.25,.6);
      \opendot{0.15,0.42};
   \node at (-0.4,0.2) {$\red\scriptstyle{k}$};
\end{tikzpicture}+        
\begin{tikzpicture}[anchorbase]
	\draw[->] (0.25,.6) to (-0.25,-.2);
	\draw[<-] (0.25,-.2) to (-0.25,.6);
      \opendot{-0.13,-0.02};
   \node at (-0.4,0.2) {$\red\scriptstyle{k}$};
\end{tikzpicture}
=\begin{tikzpicture}[anchorbase]
	\draw[->] (0.25,.6) to (-0.25,-.2);
	\draw[<-] (0.25,-.2) to (-0.25,.6);
      \opendot{0.15,-0.02};
   \node at (-.3,0.2) {$\red\scriptstyle{k}$};
\end{tikzpicture}
+
\begin{tikzpicture}[anchorbase]
	\draw[->] (0.25,.6) to (-0.25,-.2);
	\draw[<-] (0.25,-.2) to (-0.25,.6);
      \opendot{-0.13,0.42};
   \node at (-.3,0.2) {$\red\scriptstyle{k}$};
\end{tikzpicture}
&=
 -   \begin{tikzpicture}[anchorbase]
      	\draw[<-] (0.08,-.2) to (0.08,.6);
	\draw[<-] (-0.2,-.2) to (-0.2,.6);
   \node at (-0.35,0.2) {$\red\scriptstyle{k}$};
      \end{tikzpicture}
\label{nearlydone2}\end{align}
\item {\em Left adjunction relations}; cf. \cite[(6.6)]{BE2}.
\begin{align}\label{leftadj}
\begin{tikzpicture}[anchorbase]
  \draw[-] (0.3,0) to (0.3,.4);
	\draw[-] (0.3,0) to[out=-90, in=0] (0.1,-0.3);
	\draw[-] (0.1,-0.3) to[out = 180, in = -90] (-0.1,0);
	\draw[-] (-0.1,0) to[out=90, in=0] (-0.3,0.3);
	\draw[-] (-0.3,0.3) to[out = 180, in =90] (-0.5,0);
  \draw[->] (-0.5,0) to (-0.5,-.4);
   \node at (-0.7,.1) {$\red\scriptstyle{k}$};
\end{tikzpicture}
&=
\begin{tikzpicture}[anchorbase]
  \draw[<-] (0,-0.4) to (0,.4);
   \node at (-0.2,0.1) {$\red\scriptstyle{k}$};
\end{tikzpicture}&
\begin{tikzpicture}[anchorbase]
  \draw[-] (0.3,0) to (0.3,-.4);
	\draw[-] (0.3,0) to[out=90, in=0] (0.1,0.3);
	\draw[-] (0.1,0.3) to[out = 180, in = 90] (-0.1,0);
	\draw[-] (-0.1,0) to[out=-90, in=0] (-0.3,-0.3);
	\draw[-] (-0.3,-0.3) to[out = 180, in =-90] (-0.5,0);
  \draw[->] (-0.5,0) to (-0.5,.4);
   \node at (0.5,0.1) {$\red\scriptstyle{k}$};
\end{tikzpicture}
&=
(-1)^{k+1} \begin{tikzpicture}[anchorbase]
  \draw[->] (0,-0.4) to (0,.4);
   \node at (0.2,0.1) {$\red\scriptstyle{k}$};
\end{tikzpicture}
\end{align}
Recalling \cref{oddadjunction}(1), these imply that
$Q^{k+1} \Pi^{k+1}E 1_k$ is right dual to $1_k F$ in the $(Q,\Pi)$-envelope
 of $\UU(\sl_2)$.
\item {\em Infinite Grassmannian relation}; cf. \cite[(5.3)--(5.7)]{BE2}.
Recall that $R$ is the largest supercommutative quotient of
$\OSym$ described explicitly in \cref{dumbc1}.
For each $k \in \Z$, there is a graded 
superalgebra homomorphism\footnote{There is some freedom in defining
  $\beta_k$---it is unique only up to an automorphism of $R$.
The specific choice here has been made to facilitate \cref{counterclockwisebubbles}.}
\begin{align}\label{hatbeta}
\beta_k:R &\rightarrow \End_{\UU(\sl_2)}(1_k),\\\notag
\dot \eps_r
&\mapsto
(-1)^{(k+1)r}\!\!\!\begin{tikzpicture}[anchorbase]
  \draw[<-] (0,0.4) to[out=180,in=90] (-.2,0.2);
  \draw[-] (0.2,0.2) to[out=90,in=0] (0,.4);
 \draw[-] (-.2,0.2) to[out=-90,in=180] (0,0);
  \draw[-] (0,0) to[out=0,in=-90] (0.2,0.2);
   \node at (0.4,0.2) {$\red\scriptstyle{k}$};
   \opendot{-0.2,0.2};
   \node at (-0.7,0.2) {$\scriptstyle{r+k-1}$};
\end{tikzpicture}
\text{ if $r \geq 1-k$,}\\\notag
\dot \eta_r
&\mapsto
(-1)^{(k+1)r}\!\!\! \begin{tikzpicture}[anchorbase]
  \draw[->] (0.2,0.2) to[out=90,in=0] (0,.4);
  \draw[-] (0,0.4) to[out=180,in=90] (-.2,0.2);
\draw[-] (-.2,0.2) to[out=-90,in=180] (0,0);
  \draw[-] (0,0) to[out=0,in=-90] (0.2,0.2);
   \node at (-0.35,0.2) {$\red\scriptstyle{k}$};
   \opendot{0.2,0.2};
   \node at (0.7,0.2) {$\scriptstyle{r-k-1}$};
\end{tikzpicture}
\text{ if $r\geq k+1$.}
\end{align}
Following Lauda's convention from \cite{Lauda, Lauda2},
we introduce new shorthands for 
endomorphisms of $1_k$ called ``fake bubbles", which are certain bubbles decorated by a negative number of dots. 
The clockwise fake bubble 
decorated by $r+k-1$ dots on its left boundary for $r \leq -k$
denotes $(-1)^{(k+1)r}\beta_k(\dot \eps_r)$ if $r \geq 0$, and it is defined to be $0$ if $r < 0$.
The counterclockwise fake bubble decorated by $r-k-1$
dots on its right boundary for $r \leq k$ denotes
$(-1)^{(k+1)r}\beta_k(\dot \eta_r)$ if $r \geq 0$, and it is defined to be $0$ if $r < 0$.
\item {\em Centrality of the odd bubble}; cf. \cite[(7.15)]{BE2}. 
The ``odd bubble"
$\begin{tikzpicture}[anchorbase,scale=.8]
  \draw[-] (0,0.3) to[out=180,in=90] (-.2,0.1);
  \draw[-] (0.2,0.1) to[out=90,in=0] (0,.3);
 \draw[-] (-.2,0.1) to[out=-90,in=180] (0,-0.1);
  \draw[-] (0,-0.1) to[out=0,in=-90] (0.2,0.1);
  \draw[-] (0.14,-0.05) to (-0.14,0.23);
  \draw[-] (0.14,0.23) to (-0.14,-0.05);
   \node at (0.4,0.1) {$\red\scriptstyle{k}$};
\end{tikzpicture}$ is shorthand for 
$\begin{tikzpicture}[anchorbase,scale=1]
  \draw[<-] (0,0.3) to[out=180,in=90] (-.2,0.1);
  \draw[-] (0.2,0.1) to[out=90,in=0] (0,.3);
 \draw[-] (-.2,0.1) to[out=-90,in=180] (0,-0.1);
  \draw[-] (0,-0.1) to[out=0,in=-90] (0.2,0.1);
  \opendot{-0.2,0.1};
    \node at (-.03,0.1) {$\scriptstyle{k}$};
   \node at (0.4,.13) {$\red\scriptstyle{k}$};
\end{tikzpicture}
$
if $k \geq 0$ or 
$\begin{tikzpicture}[anchorbase,scale=1]
   \node at (0.4,0.13) {$\red\scriptstyle{k}$};
  \draw[->] (0.2,0.1) to[out=90,in=0] (0,.3);
  \draw[-] (0,0.3) to[out=180,in=90] (-.2,0.1);
\draw[-] (-.2,0.1) to[out=-90,in=180] (0,-0.1);
  \draw[-] (0,-0.1) to[out=0,in=-90] (0.2,0.1);
   \opendot{0.2,0.1};
      \node at (-.06,0.13) {$\scriptstyle{-k}$};
\end{tikzpicture}$
if $k \leq 0$.
So it is
$(-1)^{k+1} \beta_k(\dot o) \in \End_{\UU(\sl_2)}(1_k)$.
These are odd 2-morphisms whose square is zero.
Moreover, they are strictly central:
\begin{align}\label{centrality}
\begin{tikzpicture}[anchorbase]
  \draw[-] (0,0.2) to[out=180,in=90] (-.2,0);
  \draw[-] (0.2,0) to[out=90,in=0] (0,.2);
 \draw[-] (-.2,0) to[out=-90,in=180] (0,-0.2);
  \draw[-] (0,-0.2) to[out=0,in=-90] (0.2,0);
  \draw[-] (0.14,-0.15) to (-0.14,0.13);
  \draw[-] (0.14,0.13) to (-0.14,-0.15);
   \node at (.7,0) {$\red\scriptstyle{k}$};
	\draw[->] (.5,-.4) to (.5,.4);
\end{tikzpicture}
&=
\begin{tikzpicture}[anchorbase]
  \draw[-] (0,0.2) to[out=180,in=90] (-.2,0);
  \draw[-] (0.2,0) to[out=90,in=0] (0,.2);
 \draw[-] (-.2,0) to[out=-90,in=180] (0,-0.2);
  \draw[-] (0,-0.2) to[out=0,in=-90] (0.2,0);
  \draw[-] (0.14,-0.15) to (-0.14,0.13);
  \draw[-] (0.14,0.13) to (-0.14,-0.15);
   \node at (.4,0) {$\red\scriptstyle{k}$};
	\draw[->] (-.5,-.4) to (-.5,.4);
\end{tikzpicture}
&
\begin{tikzpicture}[anchorbase]
  \draw[-] (0,0.2) to[out=180,in=90] (-.2,0);
  \draw[-] (0.2,0) to[out=90,in=0] (0,.2);
 \draw[-] (-.2,0) to[out=-90,in=180] (0,-0.2);
  \draw[-] (0,-0.2) to[out=0,in=-90] (0.2,0);
  \draw[-] (0.14,-0.15) to (-0.14,0.13);
  \draw[-] (0.14,0.13) to (-0.14,-0.15);
   \node at (.7,0) {$\red\scriptstyle{k}$};
	\draw[<-] (.5,-.4) to (.5,.4);
\end{tikzpicture}
=
\begin{tikzpicture}[anchorbase]
  \draw[-] (0,0.2) to[out=180,in=90] (-.2,0);
  \draw[-] (0.2,0) to[out=90,in=0] (0,.2);
 \draw[-] (-.2,0) to[out=-90,in=180] (0,-0.2);
  \draw[-] (0,-0.2) to[out=0,in=-90] (0.2,0);
  \draw[-] (0.14,-0.15) to (-0.14,0.13);
  \draw[-] (0.14,0.13) to (-0.14,-0.15);
   \node at (.4,0) {$\red\scriptstyle{k}$};
	\draw[<-] (-.5,-.4) to (-.5,.4);
\end{tikzpicture}
\end{align}
\item {\em Pitchfork relations}; cf. \cite[(2.4)--(2.5), (7.5)--(7.6)]{BE2}.
\begin{align}
\begin{tikzpicture}[anchorbase,scale=.8]
\draw[<-](.6,.3) to (.1,-.3);
	\draw[<-] (0.6,-0.3) to[out=140, in=0] (0.1,0.2);
	\draw[-] (0.1,0.2) to[out = -180, in = 90] (-0.2,-0.3);
  \node at (0.7,0) {$\red\scriptstyle{k}$};
\end{tikzpicture}&=
\begin{tikzpicture}[anchorbase,scale=.8]
\draw[<-](-.5,.3) to (0,-.3);
	\draw[<-] (0.3,-0.3) to[out=90, in=0] (0,0.2);
	\draw[-] (0,0.2) to[out = -180, in = 40] (-0.5,-0.3);
  \node at (0.5,0) {$\red\scriptstyle{k}$};
\end{tikzpicture}&
\begin{tikzpicture}[anchorbase,scale=.8]
\draw[->](-.5,-.2) to (0,.4);
	\draw[<-] (0.3,0.4) to[out=-90, in=0] (0,-0.1);
	\draw[-] (0,-0.1) to[out = 180, in = -40] (-0.5,0.4);
  \node at (0.5,0.1) {$\red\scriptstyle{k}$};
\end{tikzpicture}&=
\begin{tikzpicture}[anchorbase,scale=.8]
\draw[->](.6,-.2) to (.1,.4);
	\draw[<-] (0.6,0.4) to[out=-140, in=0] (0.1,-0.1);
	\draw[-] (0.1,-0.1) to[out = 180, in = -90] (-0.2,0.4);
  \node at (0.7,0.1) {$\red\scriptstyle{k}$};
\end{tikzpicture}
&\begin{tikzpicture}[anchorbase,scale=.8]
\draw[->](.6,.3) to (.1,-.3);
	\draw[<-] (0.6,-0.3) to[out=140, in=0] (0.1,0.2);
	\draw[-] (0.1,0.2) to[out = -180, in = 90] (-0.2,-0.3);
  \node at (0.7,0) {$\red\scriptstyle{k}$};
\end{tikzpicture}&=
\begin{tikzpicture}[anchorbase,scale=.8]
\draw[->](-.5,.3) to (0,-.3);
	\draw[<-] (0.3,-0.3) to[out=90, in=0] (0,0.2);
	\draw[-] (0,0.2) to[out = -180, in = 40] (-0.5,-0.3);
  \node at (0.5,0) {$\red\scriptstyle{k}$};
\end{tikzpicture}&
\begin{tikzpicture}[anchorbase,scale=.8]
\draw[<-](-.5,-.2) to (0,.4);
	\draw[<-] (0.3,0.4) to[out=-90, in=0] (0,-0.1);
	\draw[-] (0,-0.1) to[out = 180, in = -40] (-0.5,0.4);
  \node at (0.5,0.1) {$\red\scriptstyle{k}$};
\end{tikzpicture}&=
\begin{tikzpicture}[anchorbase,scale=.8]
\draw[<-](.6,-.2) to (.1,.4);
	\draw[<-] (0.6,0.4) to[out=-140, in=0] (0.1,-0.1);
	\draw[-] (0.1,-0.1) to[out = 180, in = -90] (-0.2,0.4);
  \node at (0.7,0.1) {$\red\scriptstyle{k}$};
\end{tikzpicture}\\
\begin{tikzpicture}[anchorbase,scale=.8]
\draw[<-](.6,.3) to (.1,-.3);
	\draw[-] (0.6,-0.3) to[out=140, in=0] (0.1,0.2);
	\draw[->] (0.1,0.2) to[out = -180, in = 90] (-0.2,-0.3);
  \node at (0.7,0) {$\red\scriptstyle{k}$};
\end{tikzpicture}&=
\begin{tikzpicture}[anchorbase,scale=.8]
\draw[<-](-.5,.3) to (0,-.3);
	\draw[-] (0.3,-0.3) to[out=90, in=0] (0,0.2);
	\draw[->] (0,0.2) to[out = -180, in = 40] (-0.5,-0.3);
  \node at (0.5,0) {$\red\scriptstyle{k}$};
\end{tikzpicture}&
\begin{tikzpicture}[anchorbase,scale=.8]
\draw[->](-.5,-.2) to (0,.4);
	\draw[-] (0.3,0.4) to[out=-90, in=0] (0,-0.1);
	\draw[->] (0,-0.1) to[out = 180, in = -40] (-0.5,0.4);
  \node at (0.5,0.1) {$\red\scriptstyle{k}$};
\end{tikzpicture}&=
\begin{tikzpicture}[anchorbase,scale=.8]
\draw[->](.6,-.2) to (.1,.4);
	\draw[-] (0.6,0.4) to[out=-140, in=0] (0.1,-0.1);
	\draw[->] (0.1,-0.1) to[out = 180, in = -90] (-0.2,0.4);
  \node at (0.7,0.1) {$\red\scriptstyle{k}$};
\end{tikzpicture}&
\begin{tikzpicture}[anchorbase,scale=.8]
\draw[->](.6,.3) to (.1,-.3);
	\draw[-] (0.6,-0.3) to[out=140, in=0] (0.1,0.2);
	\draw[->] (0.1,0.2) to[out = -180, in = 90] (-0.2,-0.3);
  \node at (0.7,0) {$\red\scriptstyle{k}$};
\end{tikzpicture}&=
-\begin{tikzpicture}[anchorbase,scale=.8]
\draw[->](-.5,.3) to (0,-.3);
	\draw[-] (0.3,-0.3) to[out=90, in=0] (0,0.2);
	\draw[->] (0,0.2) to[out = -180, in = 40] (-0.5,-0.3);
  \node at (0.5,0) {$\red\scriptstyle{k}$};
\end{tikzpicture}&
\begin{tikzpicture}[anchorbase,scale=.8]
\draw[<-](-.5,-.2) to (0,.4);
	\draw[-] (0.3,0.4) to[out=-90, in=0] (0,-0.1);
	\draw[->] (0,-0.1) to[out = 180, in = -40] (-0.5,0.4);
  \node at (0.5,0.1) {$\red\scriptstyle{k}$};
\end{tikzpicture}&=
-\begin{tikzpicture}[anchorbase,scale=.8]
\draw[<-](.6,-.2) to (.1,.4);
	\draw[-] (0.6,0.4) to[out=-140, in=0] (0.1,-0.1);
	\draw[->] (0.1,-0.1) to[out = 180, in = -90] (-0.2,0.4);
  \node at (0.7,0.1) {$\red\scriptstyle{k}$};
\end{tikzpicture}
\end{align}
\item {\em Dot slides}; cf. \cite[(2.3),(4.3)--(4.4)]{BE2}.
\begin{align}\label{ds}
\begin{tikzpicture}[anchorbase,scale=.8]
	\draw[<-] (0.4,0) to[out=90, in=0] (0.1,0.5);
	\draw[-] (0.1,0.5) to[out = 180, in = 90] (-0.2,0);
      \opendot{-0.17,0.25};
  \node at (0.6,0.35) {$\red\scriptstyle{k}$};
\end{tikzpicture}
&=
\begin{tikzpicture}[anchorbase,scale=.8]
      	\draw[<-] (0.4,0) to[out=90, in=0] (0.1,0.5);
	\draw[-] (0.1,0.5) to[out = 180, in = 90] (-0.2,0);
  \node at (-0.4,0.35) {$\red\scriptstyle{k}$};
      \opendot{0.35,0.3};
\end{tikzpicture}
&\begin{tikzpicture}[anchorbase,scale=.8]
	\draw[<-] (0.4,0.5) to[out=-90, in=0] (0.1,0);
	\draw[-] (0.1,0) to[out = 180, in = -90] (-0.2,0.5);
      \opendot{-0.16,0.2};
  \node at (0.55,0.15) {$\red\scriptstyle{k}$};
\end{tikzpicture}
&=
\begin{tikzpicture}[anchorbase,scale=.8]
     \draw[<-] (0.4,0.5) to[out=-90, in=0] (0.1,0);
	\draw[-] (0.1,0) to[out = 180, in = -90] (-0.2,0.5);
  \node at (-0.35,0.15) {$\red\scriptstyle{k}$};
      \opendot{.36,.2};
\end{tikzpicture}\\
\begin{tikzpicture}[anchorbase,scale=.8]
      	\draw[-] (0.4,0) to[out=90, in=0] (0.1,0.5);
	\draw[->] (0.1,0.5) to[out = 180, in = 90] (-0.2,0);
  \node at (-0.4,0.4) {$\red\scriptstyle{k}$};
      \opendot{0.35,0.3};
\end{tikzpicture}
&=
(-1)^{k}
\begin{tikzpicture}[anchorbase,scale=.8]
	\draw[-] (0.4,0) to[out=90, in=0] (0.1,0.5);
	\draw[->] (0.1,0.5) to[out = 180, in = 90] (-0.2,0);
      \opendot{-0.15,0.25};
  \node at (0.6,0.35) {$\red\scriptstyle{k}$};
\end{tikzpicture}
+
2\:
\begin{tikzpicture}[anchorbase,scale=.8]
  \draw[-] (.44,0.65) to (0.16,0.37);
  \draw[-] (0.44,.37) to (0.16,0.65);
  \draw[-] (0.3,0.3) to[out=180,in=-90] (0.1,0.5);
  \draw[-] (0.5,0.5) to[out=-90,in=0] (0.3,0.3);
 \draw[-] (0.1,0.5) to[out=90,in=180] (0.3,0.7);
  \draw[-] (0.3,0.7) to[out=0,in=90] (0.5,0.5);
	\draw[-] (0.6,-0.3) to[out=90, in=0] (0.3,0.2);
	\draw[->] (0.3,0.2) to[out = 180, in = 90] (0,-0.3);
   \node at (-0.1,0.2) {$\red\scriptstyle{k}$};
\end{tikzpicture}
&
\begin{tikzpicture}[anchorbase,scale=.8]
     \draw[-] (0.4,0.5) to[out=-90, in=0] (0.1,0);
	\draw[->] (0.1,0) to[out = 180, in = -90] (-0.2,0.5);
  \node at (0.55,0.15) {$\red\scriptstyle{k}$};
      \opendot{-.16,.2};
\end{tikzpicture}
&=
(-1)^{k}\begin{tikzpicture}[anchorbase,scale=.8]
	\draw[-] (0.4,0.5) to[out=-90, in=0] (0.1,0);
	\draw[->] (0.1,0) to[out = 180, in = -90] (-0.2,0.5);
      \opendot{0.36,0.2};
  \node at (-0.4,0.15) {$\red\scriptstyle{k}$};
\end{tikzpicture}
+
2\:
\begin{tikzpicture}[anchorbase,scale=.8]
  \draw[-] (.45,-0.35) to (0.17,-0.63);
  \draw[-] (0.45,-.63) to (0.17,-0.35);
  \draw[-] (0.31,-0.7) to[out=180,in=-90] (.11,-.5);
  \draw[-] (0.11,-.5) to[out=90,in=180] (0.31,-0.3);
 \draw[-] (.31,-.3) to[out=0,in=90] (0.51,-0.5);
  \draw[-] (.51,-0.5) to[out=-90,in=0] (0.31,-.7);
	\draw[-] (0.6,0.3) to[out=-90, in=0] (0.3,-0.2);
	\draw[->] (0.3,-0.2) to[out = 180, in = -90] (0,0.3);
   \node at (-0.15,-0.15) {$\red\scriptstyle{k}$};
\end{tikzpicture}
\end{align}
\item {\em Almost pivotal structure\footnote{Here, we have corrected
 a sign error in \cite[(1.28)]{BE2}, and another one is corrected
 in \cref{notasbroad} below. These mistakes were uncovered in \cite{DEL}.}}; cf. \cite[(1.27)--(1.28)]{BE2}.
\begin{align}
\begin{tikzpicture}[anchorbase]
	\draw[<-] (0.08,-.4) to (0.08,.4);
     \opendot{0.08,0};
   \node at (-.25,.15) {$\red\scriptstyle{k}$};
\end{tikzpicture}
&=2\:\begin{tikzpicture}[anchorbase]
  \draw[-] (0,0.3) to[out=180,in=90] (-.2,0.1);
  \draw[-] (0.2,0.1) to[out=90,in=0] (0,.3);
 \draw[-] (-.2,0.1) to[out=-90,in=180] (0,-0.1);
  \draw[-] (0,-0.1) to[out=0,in=-90] (0.2,0.1);
  \draw[-] (0.14,-0.05) to (-0.14,0.23);
  \draw[-] (0.14,0.23) to (-0.14,-0.05);
\node at (-.4,.1) {$\red\scriptstyle k$};
	\draw[<-] (0.38,-.4) to (0.38,.4);
\end{tikzpicture}
-
\begin{tikzpicture}[anchorbase]
  \draw[-] (0.3,0) to (0.3,.4);
	\draw[-] (0.3,0) to[out=-90, in=0] (0.1,-0.3);
	\draw[-] (0.1,-0.3) to[out = 180, in = -90] (-0.1,0);
	\draw[-] (-0.1,0) to[out=90, in=0] (-0.3,0.3);
	\draw[-] (-0.3,0.3) to[out = 180, in =90] (-0.5,0);
  \draw[->] (-0.5,0) to (-0.5,-.4);
   \node at (-.65,0) {$\red\scriptstyle{k}$};
   \opendot{-0.1,0};
\end{tikzpicture}
&
\begin{tikzpicture}[anchorbase]
	\draw[<-] (0.3,-.4) to (-0.3,.4);
	\draw[<-] (-0.3,-.4) to (0.3,.4);
   \node at (-.35,.1) {$\red\scriptstyle{k}$};
\end{tikzpicture}
&
=-\begin{tikzpicture}[anchorbase,scale=.8]
\draw[->] (-1.3,.4) to (-1.3,-1.2);
\draw[-] (1.3,-.4) to (1.3,1.2);
\draw[-] (-.5,1.1) to [out=180,in=90] (-1.3,.4);
\draw[-] (.35,.4) to [out=90,in=0] (-.5,1.1);
\draw[-] (.5,-1.1) to [out=0,in=-90] (1.3,-.4);
\draw[-] (-.35,-.4) to [out=-90,in=180] (.5,-1.1);
\draw[-] (-.35,-.4) to [out=90,in=-90] (.35,.4);
        \draw[-] (0.35,-.5) to[out=180,in=0] (-0.35,.5);
        \draw[->] (-0.35,.5) to[out=180,in=90] (-0.8,-1.2);
        \draw[-] (0.35,-.5) to[out=0,in=-90] (0.8,1.2);
   \node at (-1.55,0) {$\red\scriptstyle{k}$};
\end{tikzpicture}\label{crossingcyclicity}
\end{align}
\item {\em Bubble slides}; cf. \cite[(7.10)]{BE2}.
\begin{align}
\begin{tikzpicture}[anchorbase]
  \draw[<-] (-.1,0.2) to[out=180,in=90] (-.3,0);
  \draw[-] (0.1,0) to[out=90,in=0] (-0.1,.2);
 \draw[-] (-.3,0) to[out=-90,in=180] (-0.1,-0.2);
  \draw[-] (-.1,-0.2) to[out=0,in=-90] (0.1,0);
   \node at (-.85,0) {$\scriptstyle{r+k+1}$};
      \opendot{-0.3,0};
\end{tikzpicture}\:\:\:
\begin{tikzpicture}[anchorbase]
	\draw[->] (0.08,-.4) to (0.08,.4);
\node at (0.33,0) {${\red\scriptstyle k}$};
\end{tikzpicture}
&=
\sum_{s \geq 0}
(2s+1)
\begin{tikzpicture}[anchorbase]
  \draw[<-] (0,0.2) to[out=180,in=90] (-.2,0);
  \draw[-] (0.2,0) to[out=90,in=0] (0,.2);
 \draw[-] (-.2,0) to[out=-90,in=180] (0,-.2);
  \draw[-] (0,-0.2) to[out=0,in=-90] (0.2,0);
   \node at (-.9,0) {$\scriptstyle{r+k-2s-1}$};
      \opendot{-0.2,0};
	\draw[->] (-1.68,-.2) to (-1.68,.6);
   \node at (-1.95,0.3) {$\scriptstyle{2s}$};
      \opendot{-1.68,.3};
\node at (-.8,.4) {${\red\scriptstyle k}$};
\end{tikzpicture}
\end{align}
\item {\em Curl relations}; cf. \cite[(5.21)]{BE2}.
\begin{align}\label{curlrel}
\begin{tikzpicture}[anchorbase]
	\draw[<-] (0,0.5) to (0,0.3);
	\draw[-] (0,0.3) to [out=-90,in=180] (.3,-0.2);
	\draw[-] (0.3,-0.2) to [out=0,in=-90](.5,0);
	\draw[-] (0.5,0) to [out=90,in=0](.3,0.2);
	\draw[-] (0.3,.2) to [out=180,in=90](0,-0.3);
	\draw[-] (0,-0.3) to (0,-0.5);
   \node at (0.3,-0.02) {$\scriptstyle{r}$};
      \opendot{.2,-.18};
       \node at (.3,0.4) {$\red\scriptstyle{k}$};
\end{tikzpicture}&=
\sum_{s\geq 0} (-1)^{s}
\begin{tikzpicture}[anchorbase]
	\draw[->] (-0.3,-.5) to (-0.3,.5);
   \node at (-0.54,0.2) {$\scriptstyle{s}$};
      \opendot{-.3,0.2};
  \draw[<-] (1,0) to[out=180,in=90] (.8,-0.2);
  \draw[-] (1.2,-0.2) to[out=90,in=0] (1,0);
 \draw[-] (.8,-0.2) to[out=-90,in=180] (1,-.4);
  \draw[-] (1,-.4) to[out=0,in=-90] (1.2,-0.2);
   \node at (.3,-0.2) {$\scriptstyle{r-s-1}$};
     \opendot{.8,-0.2};
\node at (.4,.2) {$\red\scriptstyle k$};
     \end{tikzpicture}
\end{align}
\item {\em Alternating braid relation}; cf. \cite[(7.20)]{BE2}.
\begin{align}
\label{notasbroad}
\begin{tikzpicture}[anchorbase]
	\draw[<-] (0.45,.8) to (-0.45,-.4);
	\draw[->] (0.45,-.4) to (-0.45,.8);
        \draw[<-] (0,-.4) to[out=90,in=-90] (-.45,0.2);
        \draw[-] (-0.45,0.2) to[out=90,in=-90] (0,0.8);
   \node at (.5,.2) {$\red\scriptstyle{k}$};
\end{tikzpicture}
-
\begin{tikzpicture}[anchorbase]
	\draw[<-] (0.45,.8) to (-0.45,-.4);
	\draw[->] (0.45,-.4) to (-0.45,.8);
        \draw[<-] (0,-.4) to[out=90,in=-90] (.45,0.2);
        \draw[-] (0.45,0.2) to[out=90,in=-90] (0,0.8);
   \node at (.6,.2) {$\red\scriptstyle{k}$};
\end{tikzpicture}
&=
\displaystyle\sum_{q,r,s \geq 0}
(-1)^{k+r+s}
\begin{tikzpicture}[anchorbase]
	\draw[<-] (0.3,0.7) to[out=-90, in=0] (0,0.3);
	\draw[-] (0,0.3) to[out = 180, in = -90] (-0.3,0.7);
    \node at (.6,0) {$\red\scriptstyle{k}$};
  \draw[-] (0.2,0) to[out=90,in=0] (0,0.2);
  \draw[<-] (0,0.2) to[out=180,in=90] (-.2,0);
\draw[-] (-.2,0) to[out=-90,in=180] (0,-0.2);
  \draw[-] (0,-0.2) to[out=0,in=-90] (0.2,0);
   \opendot{-0.2,0};
    \node at (-1,0) {$\scriptstyle{-q-r-s-3}$};
   \opendot{0.23,0.43};
     \node at (0.43,0.43) {$\scriptstyle{r}$};
	\draw[-] (0.3,-.7) to[out=90, in=0] (0,-0.3);
	\draw[->] (0,-0.3) to[out = 180, in = 90] (-0.3,-.7);
    \opendot{0.25,-0.5};
    \node at (.4,-.5) {$\scriptstyle{s}$};
	\draw[->] (-.98,-0.7) to (-.98,0.7);
   \opendot{-0.98,0.5};
   \node at (-1.15,0.5) {$\scriptstyle{q}$};
\end{tikzpicture}
-
\displaystyle\sum_{q,r,s \geq 0}
(-1)^{k+r+s}
\begin{tikzpicture}[anchorbase]
	\draw[-] (0.3,0.7) to[out=-90, in=0] (0,0.3);
	\draw[->] (0,0.3) to[out = 180, in = -90] (-0.3,0.7);
    \node at (1.4,-0.32) {$\red\scriptstyle{k}$};
  \draw[->] (0.2,0) to[out=90,in=0] (0,0.2);
  \draw[-] (0,0.2) to[out=180,in=90] (-.2,0);
\draw[-] (-.2,0) to[out=-90,in=180] (0,-0.2);
  \draw[-] (0,-0.2) to[out=0,in=-90] (0.2,0);
   \opendot{0.2,0};
   \node at (.96,0) {$\scriptstyle{-q-r-s-3}$};
   \opendot{-0.23,0.43};
      \node at (-0.43,0.43) {$\scriptstyle{r}$};
	\draw[<-] (0.3,-.7) to[out=90, in=0] (0,-0.3);
	\draw[-] (0,-0.3) to[out = 180, in = 90] (-0.3,-.7);
   \opendot{-0.25,-0.5};
    \node at (-.4,-.5) {$\scriptstyle{s}$};
	\draw[->] (.98,-0.7) to (.98,0.7);
   \opendot{0.98,0.5};
    \node at (1.15,0.5) {$\scriptstyle{q}$};
\end{tikzpicture}
\end{align}
\end{itemize}

The following theorem is an odd analog of \cite[Th.~4.12]{Lauda2}.
Our proof is shorter since we are using the more efficient 
presentation of \cref{km2cat} (although afterwards there is still work to do to determine the
images of the leftward cups and caps).

\begin{theorem}\label{actiontheorem}
Fix $\ell \geq 0$.
There is a graded 2-superfunctor
$\Psi_\ell:\UU(\sl_2) \rightarrow \OGBim_\ell$
with the following properties.
\begin{enumerate}
\item
On objects,
$\Psi_\ell$ takes 
$2n-\ell$
to the graded superalgebra
$\EOH_n^\ell$ for $0 \leq n \leq \ell$.
All other objects of $\UU(\sl_2)$ go to the trivial graded superalgebra.
\item
On generating 1-morphisms,
$\Psi_\ell$ takes 
$E 1_{2n-\ell}$ to
the graded superbimodule
$Q^{-n} U_{n}^\ell$
and $1_{2n-\ell} F$ to 
the graded superbimodule $Q^{3n-\ell+1} V_{n}^\ell$, respectively,
both for $0 \leq n < \ell$.
All other generating 1-morphisms necessarily go to trivial graded superbimodules.
\item
On generating 2-morphisms,
$\Psi_\ell$ takes
\begin{align*}
\begin{tikzpicture}[anchorbase,scale=.8]
	\draw[->] (0.08,-.3) to (0.08,.4);
     \opendot{0.08,0.05};
   \node at (.65,.1) {$\red\scriptstyle{2n-\ell}$};
\end{tikzpicture}&\mapsto\Big(\rho_{(1);n}(x_1):Q^{-n}U_n^\ell \rightarrow
Q^{-n}U_n^\ell\Big)&&(0 \leq n \leq \ell-1)\\
\begin{tikzpicture}[anchorbase,scale=.8]
	\draw[->] (0.2,-.3) to (-0.2,.4);
	\draw[->] (-0.2,-.3) to (0.2,.4);
   \node at (.65,.1) {$\red\scriptstyle{2n-\ell}$};
\end{tikzpicture}&\mapsto
\Big(-\rho_{(1^2);n}(\tau_1):Q^{-2n-1} U_{n+1}^\ell\otimes_{\EOH_{n+1}^\ell} U_n^\ell 
\rightarrow Q^{-2n-1}U_{n+1}^\ell \otimes_{\EOH_{n+1}^\ell} U_n^\ell
\Big)&&(0 \leq n \leq \ell-2)\\
\begin{tikzpicture}[anchorbase,scale=.8]
	\draw[<-] (0.4,0.3) to[out=-90, in=0] (0.1,-0.2);
	\draw[-] (0.1,-0.2) to[out = 180, in = -90] (-0.2,0.3);
  \node at (0.8,-0.1) {$\red\scriptstyle{2n-\ell}$};
\end{tikzpicture}
&\mapsto \Big(\coev_n
: \EOH_n^\ell\rightarrow
Q^{2n-\ell + 1} V_{n}^\ell \otimes_{\EOH_{n+1}^\ell} U_{n}^\ell\Big)&&(0\leq n\leq\ell-1)\\
\begin{tikzpicture}[anchorbase,scale=.8]
	\draw[<-] (0.4,-0.2) to[out=90, in=0] (0.1,0.3);
	\draw[-] (0.1,0.3) to[out = 180, in = 90] (-0.2,-0.2);
  \node at (.99,0.2) {$\red\scriptstyle{2n-\ell+2}$};
\end{tikzpicture}
&\mapsto
\Big(\ev_{n}: Q^{2n-\ell+1} U_{n}^\ell \otimes_{\EOH_{n}^\ell}
 V_{n}^\ell \rightarrow \EOH_{n+1}^\ell\Big)
&&(0 \leq n \leq \ell-1)
\end{align*}
Here, $\rho_{(1^d);n}(a)$ is the superbimodule endomorphism 
from \cref{out},
and $\ev_n$ and $\coev_n$ are as in \cref{cupsandcaps}\footnote{We
  mean the same underlying linear maps---in some places we have
  applied some degree shifts (but no parity shifts) so that they are
  not now being viewed as homomorphisms between exactly the same {\em graded} superbimodules as before.}.
All other generating 2-morphisms are taken to zero.
\end{enumerate}
\end{theorem}

\begin{proof}
Note to start with that the assignments in (3) are superbimodule
homomorphisms of the correct degrees and parities; cf. \cref{table}.
Viewing $\OGBim_\ell$ as a strict graded 2-supercategory
as explained in \cref{howstrict}, we will simply construct 
$\Psi_\ell$ as a strict
graded 2-superfunctor by checking that the defining relations from 
\cref{km2cat} are all satisfied. 
There are three sets of relations,
\cref{nearlydone}, \cref{rightadj} and \cref{inv1}--\cref{inv2}.

The right adjunction relations \cref{rightadj} follow immediately from \cref{cupsandcaps}.

Let us check the odd nil-Hecke relations
from \cref{nearlydone}. The
formulation of these relations in \cref{nearlydone}
differs by signs from the formulation in \cref{ONH1,ONH2,ONH3,ONH4,ONH5,ONH6}.
This discrepancy is explained by the signs in formula \cref{ukraine2}.
To be clear about this, 
for $a \in \ONH_d$,
let $\rho_{(1^d);n}(a)$ be the
$\big(\EOH_{n+d}^\ell,\EOH_n^\ell\big)$-superbimodule
endomorphism from \cref{out} viewed now as an endomorphism
of the degree-shifted $Q^{-nd - \binom{d}{2}} U^\ell_{n+d-1}\otimes_{\EOH_{n+d-1}^\ell}
\cdots\:\:\otimes_{\EOH_{n+1}^\ell}  U^\ell_{n}$,
for $0 \leq d \leq n$ and $0 \leq n \leq \ell-d$.
The definition from (3) implies more generally that
\begin{align}\label{inhand1}
\Psi_\ell\left(
\begin{tikzpicture}[anchorbase,scale=.8]
	\draw[->] (1,-.3) to (1,.4);
	\node at (.5,.05) {$\cdots$};
	\node at (-.5,.05) {$\cdots$};
	\draw[->] (-1,-.3) to (-1,.4);
	\draw[->] (0.0,-.3) to (0.0,.4);
	\node at (0,-.5) {$\stringnumber{i}$};
	\node at (1,-.5) {$\stringnumber{1}$};
	\node at (-1,-.5) {$\stringnumber{d}$};
     \opendot{0.0,0.05};
   \node at (1.5,.1) {$\red\scriptstyle{2n-\ell}$};
\end{tikzpicture}
\right)
&= (-1)^{i-1} \rho_{(1^d);n}(x_i),\\\label{inhand2}
\Psi_\ell\left(
\begin{tikzpicture}[anchorbase,scale=.8]
	\draw[->] (1,-.3) to (1,.4);
	\node at (.5,.05) {$\cdots$};
	\node at (-.8,.05) {$\cdots$};
	\draw[->] (-1.3,-.3) to (-1.3,.4);
	\draw[->] (0.0,-.3) to (-0.3,.4);
	\draw[->] (-.3,-.3) to (0.0,.4);
	\node at (0.03,-.5) {$\stringnumber{j}$};
	\node at (-0.33,-.5) {$\stringnumber{j+1}$};
	\node at (1,-.5) {$\stringnumber{1}$};
	\node at (-1.3,-.5) {$\stringnumber{d}$};
   \node at (1.5,.1) {$\red\scriptstyle{2n-\ell}$};
\end{tikzpicture}
\right)
&= -(-1)^{j-1} \rho_{(1^d);n}(\tau_j).
\end{align}
We check \cref{inhand2}, leaving the easier \cref{inhand1} to the reader.
We must show that
\begin{equation*}
\id\otimes\cdots\otimes \rho_{(1^2);n+j-1}(\tau_1)\otimes\cdots\otimes \id
= (-1)^{j-1} \rho_{(1^d);n}(\tau_j).
\end{equation*}
We do this by checking that both sides take the same value on
$u_{n+d-1}\big(x^{\kappa_d}\big)\otimes\cdots \otimes u_n\big(x^{\kappa_1}\big)$
for any $\kappa \in \N^d$.
Let $\sum_{\kappa'}$ denote summation over 
$\kappa'\in\N^d$ with $\kappa_i' = \kappa_i$ for $i \neq j,j+1$.
Suppose that
$x_{j+1}^{\kappa_{j+1}} x_j^{\kappa_j}
 \cdot \tau_j
= \sum_{\kappa'} c_{\kappa'}x_{j+1}^{\kappa_{j+1}'} x_j^{\kappa_j'}$.
Then, using \cref{ukraine2}, the left hand side gives
$$
\sum_{\kappa'}
(-1)^{1+(\kappa_j-\kappa'_{j})+\kappa_{j}+\cdots+\kappa_d}
c_{\kappa'}
u_{n+d-1}\big(x^{\kappa_d'}\big)\otimes\cdots \otimes
u_n(x^{\kappa_1'}\big)
$$
and the right hand side gives
$$
(-1)^{j-1}\sum_{\kappa'}
(-1)^{d-1+(d-j)(\kappa_j - \kappa_j')+
(d-j-1)(\kappa_{j+1}-\kappa_{j+1}') + 
\kappa_{j}+\cdots+\kappa_d} c_{\kappa'}
u_{n+d-1}\big(x^{\kappa'_d}\big)\otimes\cdots 
\otimes u_n\big(x^{\kappa'_1}\big).
$$
These are equal because
$\kappa_j - \kappa_j'+\kappa_{j+1}-\kappa_{j+1}'=1$
whenever $c_{\kappa'} \neq 0$.

With \cref{inhand1,inhand2} in hand,  the relations \cref{nearlydone} are easily checked.
For example, to check the length three braid relation,
we must show that
$$
\rho_{(1^3);n}(\tau_2) \circ 
\big(-\rho_{(1^3);n}(\tau_1)\big) \circ 
\rho_{(1^3);n}(\tau_2)
=
\big(-\rho_{(1^3);n}(\tau_1)\big) \circ \rho_{(1^3);n}(\tau_2) \circ \big(-\rho_{(1^3);n}(\tau_1)\big).
$$
Translating to $U_{(1^3);n}^\ell$
using the isomorphism $b_{(1)^3}$, the 
left hand side becomes the map $u_{(1^3);n}(f) \mapsto (-1)^{\parity(f)}
u_{(1^3);n}(f \cdot \tau_2 \tau_1 \tau_2)$
and the right hand side becomes 
$u_{(1^3);n}(f) \mapsto -(-1)^{\parity(f)}
u_{(1^3);n}(f \cdot \tau_1 \tau_2 \tau_1)$.
These are equal due to the sign in the 
relation \cref{ONH4}.
To check the third relation in \cref{nearlydone},
we must show that
$$
\rho_{(1^2);n}(x_1)\circ \big(-\rho_{(1^2);n}(\tau_1)\big)+ 
\rho_{(1^2);n}(\tau_1) \circ \rho_{(1^2);n}(x_2)
= \rho_{(1^2);n}(\id).
$$
The left hand side corresponds to the map
$u_{(1^2);n}(f) \mapsto u_{(1^2);n}(f \cdot \tau_1 x_1 -f \cdot x_2 \tau_1)$,
which equals $u_{(1^2);n}(f)$ by \cref{ONH6}.

Next we check \cref{inv1,inv2}.
There is nothing to do if $\ell = 0$ (the zero map is an isomorphism between zero superbimodules!) so assume that $\ell > 0$.
Take a weight $k=2n-\ell$ of $\mathbf{V}(-\ell)$
for some $0 \leq n \leq \ell$.
Setting $n' := \ell-n-1$, we have that $k = n-n'-1$.
For brevity, we will ignore grading shifts since they play no role in this place, i.e., we work with ordinary rather than
graded superbimodules.

We first check \cref{inv1}, so  $k \geq 0$ or, equivalently,
$n \geq n'+1$.
We need to show that the superbimodule homomorphism $f$
defined by the $(n-n')\times 1$ matrix
$$
\left(\:
\sigma_n
\quad
\ev_{n-1}
\quad
\cdots
\quad
\ev_{n-1} \circ \big(\rho_{(1);n-1}(x)^{n-n'-2} \otimes \id\big)
\right)^T\!:
U_{n-1} \otimes_{\EOH_{n-1}^\ell} \!V_{n-1}^\ell
\rightarrow
V_n^\ell\otimes_{\EOH_{n+1}^\ell} U_n^\ell
\oplus (\EOH_n^\ell)^{\oplus (n-n'-1)}
$$
is an isomorphism where $\sigma_n$, the image of the rightward crossing,
is the superbimodule homomorphism described explicitly in \cref{mate2},
or the zero map in the extremal case $n=\ell, n'=-1$.
By \cref{lemma3}, 
the domain of $f$ is free as a right $\EOH_n^\ell$-supermodule
with basis 
$\big\{u_{n-1}(x^r) \otimes v_{n-1}(x^s)\:\big|\:
0 \leq r \leq n'+1, 0 \leq s \leq n-1\big\}$,
and the codomain of $f$ is
free as a right $\EOH_n^\ell$-supermodule with basis
$\big\{v_n(x^s)\otimes u_n(x^r)\:\big|\:0 \leq r \leq n',
0 \leq s \leq n\big\} \cup \big\{b_1,\dots,b_{n-n'-1}\big\}$
where $b_i$ is the identity element in the $i$th copy of $\EOH_n^\ell$.
Both of these sets are of size $nn'+2n$, so it suffices 
to show that $f$ is {\em surjective}.
To prove this, since $\EOH_n^\ell$ is graded local, 
it is enough to show that the homomorphism
$\bar f := f \otimes 1$
obtained by applying the functor $-\otimes_{\EOH_n^\ell} \k$
is surjective.
Let $uv(r,s)$ denote 
$u_{n-1}(x^r)\otimes v_{n-1}(x^s) \otimes 1$ and
$vu(s,r)$ denote $v_n(x^s) \otimes u_n(x^r) \otimes 1$.
Thus, the domain of $\bar f$ has linear basis
$\{uv(r,s)\:|\:
0 \leq r \leq n'+1, 0 \leq s \leq n-1\}$
and the codomain has linear basis
$\{vu(s,r)\:|\:0 \leq r \leq n',
0 \leq s \leq n\} \cup \{b_1\otimes 1,\dots, b_{n-n'-1}\otimes 1\}$.
By \cref{Ian} and \cref{cupsandcaps}, we have that
\begin{equation}\label{howitgoes}
\bar f\big(uv(r,s)\big) = 
\pm vu(n,r+s-n)
\pm b_{n-r-s}\otimes 1
\pm vu(s,r)
\end{equation}
for $0 \leq r\leq n'+1$ and $0 \leq s \leq n-1$,
where the first term should be interpreted as zero if $r+s-n < 0$
and the middle term should be interpreted as zero
if $n-r-s < 1$ or $n-r-s > n-n'-1$.
Note also that the last term $vu(s,r)$ is zero for $r > n'$ by degree considerations.
The argument is completed with the following observations.
\begin{itemize}
\item
We get the vectors
$vu(n,r)$ for $0 \leq r \leq n'$ from the images of
the basis vectors $uv(n'+1,s)$ for $n-n'-1\leq s \leq n-1$.
Indeed, in the formula \cref{howitgoes}
for $\bar f(uv(n'+1,s))$ for these values of $s$,
the second and third terms on the right hand side are both zero.
\item
Modulo the span of vectors already obtained, 
we get the vectors $b_1\otimes 1,\dots, b_{n-n'-1}\otimes 1$
from the images of the basis vectors
$uv(n'+1,s)$ for $0 \leq s \leq n-n'-2$.
Indeed, in the formula for $f(uv(n'+1,s))$ for these values of $s$
the third term on the right hand side is zero.
\item
Modulo the span of vectors already obtained,
we get the vectors
$vu(s,r)$
for $0 \leq r \leq n'$ and $0 \leq s \leq n-1$
from the images of the remaining basis vectors $uv(r,s)$
for these values of $r$ and $s$.
\end{itemize}

Now consider \cref{inv2}, so $k \leq 0$ and $n'\geq n-1$.
We need to show that the superbimodule homomorphism $f$
defined by the $1\times (n'-n+2)$ matrix
$$
\left(\:
\sigma_n
\quad
\coev_n
\quad\cdots\quad
\big(\id \otimes \rho_{(1);n}(x)^{n'-n}\big) \circ \coev_n
\right):
U_{n-1} \otimes_{\EOH_{n-1}^\ell}\! V_{n-1}^\ell
\oplus (\EOH_n^\ell)^{\oplus (n'-n+1)}\rightarrow
V_n^\ell\otimes_{\EOH_{n+1}^\ell} U_n^\ell
      $$ 
      is an isomorphism,
      where $\sigma_n$ is as in \cref{mate2} or the zero map
      in the extremal case $n=0,n'=\ell-1$.
By \cref{lemma3}, 
the domain of $f$ is free as a right $\EOH_n^\ell$-supermodule
with basis $\big\{u_{n-1}(x^r) \otimes v_{n-1}(x^s)\:\big|\:
0 \leq r \leq n'+1, 0 \leq s \leq n-1\big\}
\cup \big\{b_1,\dots,b_{n'-n+1}\big\}$,
where $b_i$ is the identity element in the $i$th copy of $\EOH_n^\ell$,
and the codomain of $f$ is
free as a right $\EOH_n^\ell$-supermodule with basis
$\big\{v_n(x^s)\otimes u_n(x^r)\:\big|\:0 \leq r \leq n',
0 \leq s \leq n\big\}$.
Both of these sets are of size $nn'+n+n'+1$, so it suffices to 
see that $f$ is surjective. Again, we apply $-\otimes_{\EOH_n^\ell} \k$
and show that the resulting map $\bar f := f\otimes 1$ is surjective.
Let $uv(r,s) := u_{n-1}(x^r)\otimes v_{n-1}(x^s) \otimes 1$
and $vu(s,r) := v_n(x^s)\otimes u_n(x^r) \otimes 1$ for short.
So the domain of $\bar f$ has linear basis
$\{uv(r,s)\:|\:
0 \leq r \leq n'+1, 0 \leq s \leq n-1\}
\cup \{b_1\otimes 1,\dots,b_{n'-n+1}\otimes 1\}$
and the codomain has linear basis
$\{vu(s,r)\:|\:0 \leq r \leq n',
0 \leq s \leq n\}$.
By \cref{Ian} and \cref{Jackdoor}, we have that
\begin{align}\label{minimap}
\bar f\big(uv(r,s)\big) &= 
\pm vu(n,r+s-n)
\pm vu(s,r),
&
\bar f(b_i\otimes 1) &=
vu(n,i-1),
\end{align}
for $0 \leq r\leq n'+1$, $0 \leq s \leq n-1$
and $1 \leq i \leq n'-n+1$,
interpreting $vu(r+s-n)$ as zero if $r+s-n < 0$
and $s$.
The proof is completed by the following.
\begin{itemize}
\item
We get $vu(n,r)$ for $0 \leq r \leq n'-n$ from the
images of the vectors $b_i\otimes 1$ for $i=1,\dots,n'-n+1$.
\item
We get $vu(n,r)$ for $n'-n+1\leq r \leq n'$ from the images of the
vectors $uv(n'+1,r)$ for $0 \leq r \leq n-1$. This uses the observation
that $vu(r,n'+1) = 0$.
\item Modulo the span of vectors already obtained, we get the remaining $vu(s,r)$ for $0 \leq s \leq n-1$
and $0 \leq r \leq n'$ from the images of the vectors
$uv(r,s)$ for the same values of $r$ and $s$.
\end{itemize}
\end{proof}

In the next theorem, we give explicit descriptions of the images of the leftward cups and caps under the graded 2-superfunctor $\Psi_\ell$ from \cref{actiontheorem}, that is, the superbimodule homomorphisms
\begin{align}
\coev'_n := \Psi_\ell\left(
\begin{tikzpicture}[anchorbase,scale=.8]
	\draw[-] (0.4,0.3) to[out=-90, in=0] (0.1,-0.2);
	\draw[->] (0.1,-0.2) to[out = 180, in = -90] (-0.2,0.3);
  \node at (0.99,-0.1) {$\red\scriptstyle{2n-\ell+2}$};
\end{tikzpicture}\right)&:
\EOH_{n+1}^\ell\rightarrow
Q^{2n-\ell + 1} U_n^\ell \otimes_{\EOH_n^\ell} V_n^\ell,
\\
\ev'_n :=
\Psi_\ell\left(
\begin{tikzpicture}[anchorbase,scale=.8]
	\draw[-] (0.4,-0.2) to[out=90, in=0] (0.1,0.3);
	\draw[->] (0.1,0.3) to[out = 180, in = 90] (-0.2,-0.2);
  \node at (.8,0.2) {$\red\scriptstyle{2n-\ell}$};
\end{tikzpicture}
\right)&: Q^{2n-\ell+1} V_n^\ell
\otimes_{\EOH_{n+1}^\ell} U_{n}^\ell \rightarrow \EOH_n^\ell
\end{align}
for $0 \leq n \leq \ell-1$ (these maps are zero for all other $n$). 
Let $n'$ be defined so that $\ell=n+1+n'$. 
Then, by \cref{table},
$\coev_n'$
is a superbimodule homomorphism of 
degree $n'-n$ and parity $n'-n\pmod{2}$,
and
$\ev_n'$
is of degree $n-n'$ and parity $n-n'\pmod{2}$.
Recall also the maps $\tev_n$ and $\tcoev_n$ from \cref{secondadjunction}.

\begin{theorem}\label{leftwardscupsandcaps}
For $\ell=n+1+n'$ as above, we have that
\begin{align*}
\coev_n' &= (-1)^{\binom{n}{2}+(n+1)n'} 
(p_{n'}^{-1} \otimes q_n) \circ \tcoev_n,&
\ev_n' &= (-1)^{\binom{n+1}{2}+(n+1)n'} \tev_n \circ (q_n^{-1} \otimes p_{n'}),
\end{align*}
where
$p_{n'}:Q^{-n} U_n^\ell \rightarrow \tU_n^\ell$
and $q_n:\tV_n^\ell \rightarrow Q^{3n-\ell+1} V_n^\ell$
are the superbimodule isomorphisms that are the identity maps
on the underlying vector spaces.
Moreover:
\begin{align}\label{bubbly1}
\displaystyle\Psi_\ell\left(\begin{tikzpicture}[anchorbase,scale=.8]
  \draw[<-] (0,0.4) to[out=180,in=90] (-.2,0.2);
  \draw[-] (0.2,0.2) to[out=90,in=0] (0,.4);
 \draw[-] (-.2,0.2) to[out=-90,in=180] (0,0);
  \draw[-] (0,0) to[out=0,in=-90] (0.2,0.2);
   \node at (0.85,0.2) {$\red\scriptstyle{n-n'+1}$};
   \opendot{-0.2,0.2};
   \node at (-0.9,0.24) {$\scriptstyle{n-n'+r}$};
\end{tikzpicture}\right)(1) &= 
\sum_{s=0}^r
(-1)^{(n'+1)r+(n+1)s+\binom{s}{2}}
\Big[\big(\psi_{n+1}^\ell\big)^{-1}\big(\bar\eps_{r-s}^{(n')}\big) \Big]\bar\eta_{s}^{(n+1)}
&&\text{if $r \geq n'-n$,}\\
\label{bubbly2}
\displaystyle\Psi_\ell\left(\begin{tikzpicture}[anchorbase,scale=.8]
  \draw[->] (0.2,0.2) to[out=90,in=0] (0,.4);
  \draw[-] (0,0.4) to[out=180,in=90] (-.2,0.2);
\draw[-] (-.2,0.2) to[out=-90,in=180] (0,0);
  \draw[-] (0,0) to[out=0,in=-90] (0.2,0.2);
   \node at (-0.85,0.2) {$\red\scriptstyle{n-n'-1}$};
   \opendot{0.2,0.2};
   \node at (0.9,0.22) {$\scriptstyle{n'-n+r}$};
\end{tikzpicture}\right)(1) &= 
\sum_{s=0}^r (-1)^{(n'+s)r+(n+1)s+\binom{s}{2}}
\Big[\big(\psi_{n}^\ell\big)^{-1} \big(\bar\gamm_{r-s}^{(n'+1)}\big)\Big]\bar\eps_s^{(n)}&&\text{if $r \geq n-n'$.}
\end{align}
\end{theorem}

\begin{proof}
Recalling \cref{evil}, the map $p_{n'}$ is of degree $n-2n'$ and parity $n'\pmod{2}$,
and $q_n$ is of degree $n-\ell+1$ and parity $n\pmod{2}$.
The inverse of the map $p_{n'}^{-1} \otimes q_n$
is $(-1)^{nn'} p_{n'} \otimes q_n^{-1}$.
With this in mind, 
we let \begin{align*}
\widehat{\coev}_n &:= (-1)^{\binom{n}{2}+n'}
\big(p_{n'} \otimes q_n^{-1}\big) \circ \coev_{n}',
&
\widehat{\ev}_n &:= (-1)^{\binom{n+1}{2}+n'} \ev_n'
\circ \big(q_n \otimes p_{n'}^{-1}\big).
\end{align*}
These are both even of degree 0.
To prove the first part of the lemma, we must show that
$\widehat{\coev}_n = \tcoev_n$
and $\widehat{\ev}_n = \tev_n$.

We first show that
$\widehat{\coev}_n$ and $\widehat{\ev}_n$
are the counit and unit of an adjunction.
This follows from the left adjunction relations \cref{leftadj}:
\begin{align*}
\big(\id\otimes \widehat{\ev}_n\big) 
\circ \big(\widehat{\coev}_n \otimes \id\big)
&=
(-1)^{\binom{n}{2}+\binom{n+1}{2}} (\id \otimes \ev_n') \circ \big(\id \otimes q_n \otimes p_{n'}^{-1}\big)
\circ \big(p_{n'}\otimes q_n^{-1}\otimes \id\big)
\circ (\coev_n' \otimes \id)\\
&=
(-1)^{n+n'} (\id \otimes \ev_n') \circ \big(p_{n'} \otimes \id \otimes \id\big)
\circ \big(\id\otimes \id\otimes p_{n'}^{-1}\big)
\circ (\coev_n' \otimes \id)\\
&=
(-1)^{\ell-1} p_{n'} \circ (\id \otimes \ev_n') 
\circ (\coev_n' \otimes \id) \circ p_{n'}^{-1}
= \id,\\
\big(\widehat{\ev}_n\otimes\id\big) 
\circ \big(\id\otimes \widehat{\coev}_n \big)
&=
(-1)^{\binom{n}{2}+\binom{n+1}{2}} (\ev_n' \otimes \id) \circ \big(q_n \otimes p_{n'}^{-1}\otimes\id\big)
\circ \big(\id\otimes p_{n'}\otimes q_n^{-1}\big)
\circ (\id\otimes \coev_n')\\
&=
( \ev_n'\otimes\id) \circ \big(\id\otimes\id\otimes q_{n}^{-1}\big)
\circ \big(q_n\otimes \id\otimes \id\big)
\circ (\id\otimes \coev_n')\\
&=
q_{n}^{-1} \circ (\ev_n'\otimes\id) 
\circ (\id\otimes \coev_n') \circ q_{n}
= \id.
\end{align*}
Here, for brevity, we are implicitly assuming that $\OGBim_\ell$ 
is strict as in \cref{howstrict}.

So now we have two adjunctions
making $(\tV^\ell_n, \tU_n^\ell)$
into a dual pair, one $A_1$ with unit $\widehat{\ev}_n$ and counit
 $\widehat{\coev}_n$ just constructed,
and the other $A_2$ with unit $\tev_n$ and counit
$\tcoev_n$ coming from \cref{secondadjunction}.
Any such adjunction $A$ induces a  degree 0 even 
$\big(\EOH_n^\ell,\EOH_{n+1}^\ell\big)$-superbimodule
isomorphism
$\alpha:\tU^\ell_n\stackrel{\sim}{\rightarrow}\Hom_{\EOH_{n}^\ell\dash}(\tV_n^\ell, \EOH_{n}^\ell)$.
So from $A_1$ and $A_2$
we get isomorphisms $\alpha_1$ and $\alpha_2$,
hence, an even degree 0 automorphism 
$\alpha_2^{-1} \circ \alpha_1$ of $\tU^\ell_n$.
By \cref{lemma3}(2c), $\tU^\ell_n$ is cyclic generated by the vector 
$\tu_n(1)$. Moreover, this vector spans the (one-dimensional) graded component of $\tU^\ell_n$ of lowest degree.
So we must have that
$\alpha_2^{-1} \circ \alpha_1 = c_n \id$
for $c_n \in \k^\times$.

The argument in the previous paragraph
shows that
$\widehat{\coev}_n = c_n\; \tcoev_n$
and $\widehat{\ev}_n = c_n^{-1}\; \tev_n$
for some $c_n \in \k^\times$.
To complete the proof of the first part of the lemma, we must show that $c_n = 1$. To see this, we will show that
\begin{align}
\displaystyle\Psi_\ell\left(\begin{tikzpicture}[anchorbase,scale=.8]
  \draw[<-] (0,0.4) to[out=180,in=90] (-.2,0.2);
  \draw[-] (0.2,0.2) to[out=90,in=0] (0,.4);
 \draw[-] (-.2,0.2) to[out=-90,in=180] (0,0);
  \draw[-] (0,0) to[out=0,in=-90] (0.2,0.2);
   \node at (0.85,0.2) {$\red\scriptstyle{n-n'+1}$};
   \opendot{-0.2,0.2};
   \node at (-0.9,0.24) {$\scriptstyle{n-n'+r}$};
\end{tikzpicture}\right)(1) &= c_n 
\sum_{s=0}^r
(-1)^{(n'+1)r+(n+1)s+\binom{s}{2}}
\Big[\big(\psi_{n+1}^\ell\big)^{-1}\big(\bar\eps_{r-s}^{(n')}\big) \big]\bar\eta_{s}^{(n+1)}
&&\text{if $r \geq n'-n$,}\label{firstofthelast}\\\label{secondofthelast}
\displaystyle\Psi_\ell\left(\begin{tikzpicture}[anchorbase,scale=.8]
  \draw[->] (0.2,0.2) to[out=90,in=0] (0,.4);
  \draw[-] (0,0.4) to[out=180,in=90] (-.2,0.2);
\draw[-] (-.2,0.2) to[out=-90,in=180] (0,0);
  \draw[-] (0,0) to[out=0,in=-90] (0.2,0.2);
   \node at (-0.85,0.2) {$\red\scriptstyle{n-n'-1}$};
   \opendot{0.2,0.2};
   \node at (0.9,0.22) {$\scriptstyle{n'-n+r}$};
\end{tikzpicture}\right)(1) &= 
c_n^{-1} \sum_{s=0}^r(-1)^{(n'+s)r+(n+1)s+\binom{s}{2}}
\Big[\big(\psi_{n}^\ell\big)^{-1} \big(\bar\gamm_{r-s}^{(n'+1)}\big)\Big]\bar\eps_s^{(n)}
&&\text{if $r \geq n-n'$.}
\end{align}
Given this, taking $r=0$ in one of these equations and using that the
``bottom bubbles"
$\begin{tikzpicture}[anchorbase,scale=.8]
  \draw[<-] (0,0.4) to[out=180,in=90] (-.2,0.2);
  \draw[-] (0.2,0.2) to[out=90,in=0] (0,.4);
 \draw[-] (-.2,0.2) to[out=-90,in=180] (0,0);
  \draw[-] (0,0) to[out=0,in=-90] (0.2,0.2);
   \node at (0.85,0.2) {$\red\scriptstyle{n-n'+1}$};
   \opendot{-0.2,0.2};
   \node at (-0.64,0.24) {$\scriptstyle{n-n'}$};
\end{tikzpicture}$ 
and
\begin{tikzpicture}[anchorbase,scale=.8]
  \draw[->] (0.2,0.2) to[out=90,in=0] (0,.4);
  \draw[-] (0,0.4) to[out=180,in=90] (-.2,0.2);
\draw[-] (-.2,0.2) to[out=-90,in=180] (0,0);
  \draw[-] (0,0) to[out=0,in=-90] (0.2,0.2);
   \node at (-0.85,0.2) {$\red\scriptstyle{n-n'-1}$};
   \opendot{0.2,0.2};
   \node at (0.66,0.22) {$\scriptstyle{n'-n}$};
\end{tikzpicture}
are identities if $n \geq n'$ or $n \leq n'$, respectively,
gives that $c_n=1$, and the lemma follows.

In this paragraph, we prove \cref{firstofthelast}.
We need to apply
$\ev_n \circ \big(\rho_{(1);n}(x)^{n-n'+r}\otimes\id\big)
\circ \coev_n'$ to $1 \in \EOH_{n+1}^\ell$
using that $\coev_n' = (-1)^{\binom{n}{2}+(n+1)n'}
c_n\;\big(p_{n'}^{-1} \otimes q_n\big) \circ \tcoev_{n}$.
Applying $\tcoev_n$ to 1 using the second formula for that in \cref{secondadjunction} gives
$$
\sum_{s=0}^{n'}(-1)^{\ell s
  +\binom{s}{2}}\Big[\big(\psi_{n+1}^\ell\big)^{-1} \big(\bar\eps_{n'-s}^{(n')}\big)\Big]\tu_n(1)\otimes \tv_n(x^s).
$$
Then we scale by 
$(-1)^{\binom{n}{2}+(n+1)n'} c_n$ and apply
$(p_{n'}^{-1} \otimes q_n)$ to get
$$
c_n \sum_{s=0}^{n'}(-1)^{\binom{n}{2}+(n+1)n'+\ell s
  +\binom{s}{2}+ns+n'(n'-s)}\Big[\big(\psi_{n+1}^\ell\big)^{-1}
\big(\bar\eps_{n'-s}^{(n')}\big)\Big] u_n(1)\otimes v_n(x^s).
$$
This is $\coev'_n(1)$.
Then we apply $\rho_{(1);n}(x)^{n-n'+r}\otimes\id$
(the dots on the left boundary of the bubble)
using \cref{out}
to get
$$
c_n \sum_{s=0}^{n'}(-1)^{\binom{n}{2}+(n+1)n'+\ell s
  +\binom{s}{2}+ns+n'(n'-s)+(n-n'+r)(n'-s)+\binom{n-n'+r}{2}}\Big[\big(\psi_{n+1}^\ell\big)^{-1} \big(\bar\eps_{n'-s}^{(n')}\big)\Big]\tu_n(x^{n-n'+r})\otimes \tv_n(x^s).
$$
Finally we apply $\ev_n$ using the formula from \cref{reallynot} 
to obtain
$$
c_n \sum_{s=n'-r}^{n'}
(-1)^{\binom{n}{2}+(n+1)n'+\ell s +
  \binom{s}{2}+ns+n'(n'-s)+(n-n'+r)(n'-s)+\binom{n-n'+r}{2}}
\Big[\big(\psi_{n+1}^\ell\big)^{-1}
\big(\bar\eps_{n'-s}^{(n')}\big)\Big]
 \bar\gamm^{(n+1)}_{r+s-n'}.
$$
It remains to reindex the summation replacing $s$ by $s+(n'-r)$
and to simplify the signs to obtain \cref{firstofthelast}.

To prove \cref{secondofthelast},
we 
first note that
$\begin{tikzpicture}[anchorbase,scale=.8]
  \draw[->] (0.2,0.2) to[out=90,in=0] (0,.4);
  \draw[-] (0,0.4) to[out=180,in=90] (-.2,0.2);
\draw[-] (-.2,0.2) to[out=-90,in=180] (0,0);
  \draw[-] (0,0) to[out=0,in=-90] (0.2,0.2);
   \node at (-0.85,0.2) {$\red\scriptstyle{n-n'-1}$};
   \opendot{0.2,0.2};
   \node at (0.9,0.22) {$\scriptstyle{n'-n+r}$};
\end{tikzpicture}=
(-1)^{\binom{n'-n+r}{2}}\begin{tikzpicture}[anchorbase,scale=.8]
  \draw[->] (0.2,0.2) to[out=90,in=0] (0,.4);
  \draw[-] (0,0.4) to[out=180,in=90] (-.2,0.2);
\draw[-] (-.2,0.2) to[out=-90,in=180] (0,0);
  \draw[-] (0,0) to[out=0,in=-90] (0.2,0.2);
   \node at (0.95,0.2) {$\red\scriptstyle{n-n'-1}$};
   \opendot{-0.2,0.2};
   \node at (-0.8,0.22) {$\scriptstyle{n'-n+r}$};
\end{tikzpicture}
$
by \cref{ds} and the super interchange law.
Also $\Psi_\ell\Big(\begin{tikzpicture}[anchorbase,scale=.8]
	\draw[<-] (0.08,-.3) to (0.08,.4);
     \opendot{0.08,0.05};
   \node at (-.7,-.1) {$\red\scriptstyle{n-n-'1}$};
\end{tikzpicture}\Big)=\lambda_{n;(1)}(x_1)$
by the definition \cref{phew}, \cref{actiontheorem} and \cref{mate1}.
Now we calculate by applying
$(-1)^{\binom{n'-n+r}{2}}\ev'_n \circ \big(\lambda_{n;(1)}(x)^{n'-n+r}\otimes\id\big)
\circ \coev_n$ to $1 \in \EOH_{n}^\ell$
using that $\ev_n' = (-1)^{\binom{n+1}{2}+(n+1)n'}
c_n^{-1}\;\tev_n\circ \big(q_{n}^{-1} \otimes p_{n'}\big)$.
By \cref{Jackdoor}, we have that
$$
\coev_n(1) = 
\sum_{s=0}^n v_n\big(x^{s}\big) \otimes u_n(1) \bar \eps^{(n)}_{n-s}.
$$
Then we scale by 
$(-1)^{\binom{n'-n+r}{2}+\binom{n+1}{2}+(n+1)n'}
c_n^{-1}$ and apply $\lambda_{n;(1)}(x)^{n'-n+r}\otimes\id$ to obtain
$$
c_n^{-1} \sum_{s=0}^n (-1)^{\binom{n'-n+r}{2}+\binom{n+1}{2}+(n+1)n'}
v_n\big(x^{n'-n+r+s}\big)\otimes u_n(1) \bar \eps^{(n)}_{n-s}.
$$
Next $q_n^{-1} \otimes p_{n'}$ gives
$$
c_n^{-1} \sum_{s=0}^n (-1)^{\binom{n'-n+r}{2}+\binom{n+1}{2}+(n+1)n'+n'(n'+r+s)}
\tv_n\big(x^{n'-n+r+s}\big)\otimes \tu_n(1) \bar \eps^{(n)}_{n-s}.
$$
Finally we apply $\tev_n$ using the formula in 
\cref{secondadjunction} to obtain
$$
c_n^{-1} \sum_{s=n-r}^n (-1)^{\binom{n'-n+r}{2}+\binom{n+1}{2}+(n+1)n'+n'(n'+r+s)+n(n'-n+r+s)+\binom{n'-n+r+s}{2}}
\Big[\big(\psi_{n}^\ell\big)^{-1} \Big(
\bar\gamm_{r+s-n}^{(n'+1)}\Big)\Big] \bar \eps^{(n)}_{n-s}.
$$
It remains to replace $s$ by $n-s$ and simplify the sign to obtain \cref{secondofthelast}.
\end{proof}

The formulae for the positively 
dotted bubbles in \cref{bubbly1,bubbly2} 
are rather complicated.
To simplify, one can apply the homomorphism $\alpha_n^\ell$
from \cref{alphamap}, to obtain the following. 

\begin{corollary}\label{counterclockwisebubbles}
For $\ell = n+n'$ and $k=n-n'$,
the graded superalgebra homomorphism defined by the composition
$$
\begin{tikzcd}
R\arrow[r,"\beta_k"]
&\End_{\UU(\sl_2)}(1_k)
\arrow[r,"\Psi_\ell"]& \End_{\EOH_n\dash\EOH_n}(\EOH_n)
\arrow[r,"\theta \mapsto \theta(1)"]&\EOH_n
\arrow[r,"\alpha_n^\ell"]&
R_{n'}
\end{tikzcd}
$$
is equal to the canonical quotient map $R \twoheadrightarrow R_{n'},
\dot c
\mapsto \dot c^{(n')}$.
\end{corollary}

\begin{proof}
We note first that this composition is indeed a graded superalgebra homomorphism. Now we use \cref{bubbly1,bubbly2} to show 
for all $r \geq 1$ that it takes
$\dot \eps_r \mapsto \dot \eps_r^{(n')}$
in the case $k \geq 0$
and $\dot\eta_r \mapsto \dot \eta_r^{(n')}$
in the case $k \leq 0$.
The arguments are similar in the two cases, so we just give 
the details for $k \geq 0$, i.e., $n \geq n'$.
If $\ell = 0$ the result is trivial, so we may assume $\ell > 0$,
hence, $n \geq 1$.
Remembering
the definition of $\beta_k(\dot \eps_r)$ from \cref{hatbeta}, 
we apply \cref{bubbly1} with $n$ replaced by $n-1$ to get that
$$
\Psi_\ell(\beta_k(\dot\eps_r))(1)
=
(-1)^{(n-n'+1)r}\sum_{s=0}^r (-1)^{(n'+1)r+ns+\binom{s}{2}}
\Big[\big(\psi_{n}^\ell\big)^{-1}\big(\bar\eps_{r-s}^{(n')}\big)\Big]\bar\gamm_{s}^{(n)}.
$$
From \cref{po2}, it follows that $\alpha_n^\ell\Big(\big(\psi_{n}^\ell\big)^{-1}\big(\bar e_{r}^{(n')}\big)\Big) = 
(-1)^{nr} \dot e_{r}^{(n')}$, hence,
$\alpha_n^\ell\Big(\big(\psi_{n}^\ell\big)^{-1}\big(\bar \eps_{r}^{(n')}\big)\Big) = 
(-1)^{nr} \dot \eps_{r}^{(n')}$.
So
$\alpha_n^\ell\big(\Psi_\ell(\beta_k(\eps_r))(1)\big)
= (-1)^{(n-n'+1)r+(n'+1)r+nr}  \dot\eps_r^{(n')} = \dot \eps_r^{(n')}$.
\end{proof}

The results so far in this section have an application to prove the
{\em non-degeneracy} of $\UU(\sl_2)$, which was conjectured in \cite{EL, BE2}.
This asserts that the 2-morphism spaces in $\UU(\sl_2)$ have the expected
graded dimensions. The result may be formulated as follows.
For any $k,\ell \in \Z$
and 1-morphisms $X, Y \in \Hom_{\UU(\sl_2)}(k,\ell)$ (i.e., words
consisting of $m$ letters $E$ and $n$ letters $F$
such that $\ell = k + 2m-2n$)
we view the 2-morphism
space
$\Hom_{\UU(\sl_2)}(X,Y)$ as a graded right $R$-supermodule
so that $\dot c \in R$
acts by horizontally composing on the right with $\beta_{k}(\dot c)$.

\begin{theorem}\label{nondegthm}
For $k, \ell \in \Z$
and
$X, Y \in \Hom_{\UU(\sl_2)}(k,\ell)$,
the 2-morphism space
$\Hom_{\UU(\sl_2)}(X, Y)$
is free as a graded right $R$-supermodule
with basis given by a set of representatives for 
equivalence classes of decorated
reduced $(X,Y)$-matchings
in the sense defined in \cite[Sec.8]{BE2}.
In particular, $\beta_k:R \rightarrow \End_{\UU(\sl_2)}(1_k)$
is an isomorphism for all $k \in \Z$.
\end{theorem}

\begin{proof}
The ``easy'' step in the proof is to show that 
$\Hom_{\UU(\sl_2)}(X, Y)$ is spanned as a right $R$-supermodule
by the 2-morphisms that are the
representatives for 
equivalence classes of decorated
reduced $(X,Y)$-matchings. This is proved by exhibiting
an explicit straightening algorithm 
going by induction on the number of crossings. See \cite[Th.~8.1]{BE2}, which simply cites \cite[Prop.~3.11]{KL3} as 
the argument is the same as in the purely even setting, or \cite{DEL} for a more systematic treatment.
Note the straightening algorithm requires all of the relations described above, including the alternating braid relation.

The ``hard" step is to establish the linear independence.
By a standard reduction, which is again the same as in the ordinary even setting
as in \cite[Rem.~3.16]{KL3}, it suffices to treat the case that
$X = Y = E^d$ for some $d \geq 0$.
In this case, the
decorated reduced $(X,Y)$-matchings consist of
$d$ strings oriented from bottom to top decorated with some dots 
close to the top boundary. 
We index them by pairs $(\kappa,w)$ for $\kappa \in \N^d$
and $w \in S_d$.
For such a pair, the corresponding 2-morphism $f(\kappa,w)$
has $\kappa_i$ dots at the top of the $i$th string, 
with the strings below arranged so that they represent some reduced expression for
 $w$.
 Consider some linear relation
$$
f := \sum_{\kappa\in \N^d,w \in S_d}
f(\kappa,w) \beta_k(\dot c_{\kappa,w}) = 0
$$
for $\dot c_{\kappa,w} \in R$.
Each $\dot c_{\kappa,w}$ is an $\k$-linear combination
of basis vectors 
$\dot h_\lambda$ of $R$ for $\lambda$ in some finite set $P_\kappa$
of partitions.
Pick $0\leq n \leq \ell$ with
$k = 2n-\ell$ in such a way that
$n$ and $\ell-n$ are both very large
relative to $|\kappa|$ and $|\lambda|$ for all 
$\lambda \in P_\kappa,\kappa\in \N^n$ with $\dot c_{\kappa,w} \neq 0$
for some $w \in S_n$.
Then we apply the 2-superfunctor $\Psi_\ell$ to $f$ to obtain
the relation
$$
\Psi_\ell(f) = \sum_{\kappa\in \N^d,w \in S_d}
\Psi_\ell(f(\kappa,w)) \Psi_\ell(\beta_k(\dot c_{\kappa,w})) = 0
$$
in $\End_{\EOH_{n+d}^\ell\dash\EOH_n^\ell}\big(U_{n+d-1}\otimes_{\EOH_{n+d-1}^\ell}\cdots\otimes_{\EOH_{n+1}^\ell} U_n^\ell\big)$.
Conjugating with the isomorphism $b_{(1)^d;n}$
from \cref{ukraine2},
we get from $\Psi_\ell(f)$ a superbimodule
endomorphism $\tilde f = 0 $ of 
$U_{(1^d);n}^\ell$.
Using \cref{inhand1,inhand2}, it follows that
$$
\tilde f = \sum_{\kappa\in\N^d,w \in S_d}
\pm x^\kappa \tau_w \otimes
\Psi_\ell(\beta_k(\dot c_{\kappa,w}))
$$
for some signs, where this is being
viewed as an endomorphism of the free right $\EOH_n^\ell$-superbimodule
$U_{(1^d);n}^\ell$
using the right action of $\ONH_d$ from \cref{lemma2}(2).
By the large choice of $n$ and $\ell$,
the endomorphisms defined by each $x^\kappa \tau_w$ 
are linearly independent; 
cf. the proof of \cref{ONHbasis}.
We deduce that 
$\Psi_\ell\big(\beta_k(\dot c_{\kappa,w})\big) = 0$
for all $\kappa$ and $w$.

It remains to show that 
$\Psi_\ell\big(\beta_k(\dot c_{\kappa,w})\big) = 0$
implies that $\dot c_{\kappa,w} = 0$ for sufficiently large $n$ and $\ell$.
Assume that 
$\Psi_\ell\big(\beta_k(\dot c_{\kappa,w})\big) = 0$.
Remembering 
that $\dot c_{\kappa,w}$ is an $\k$-linear combination
of $\dot h_\lambda$ for $\lambda$ with $|\lambda|$ small,
this follows on evaluating at $1 \in \EOH_n^\ell$ then applying
the homomorphism $\alpha_n^\ell:\EOH_n^\ell \rightarrow R_{\ell-n}$.
The point here is that by 
\cref{counterclockwisebubbles}
we have that 
$$
\alpha_n^\ell\Big(\Psi_\ell\big(\beta_k(\dot h_\lambda)\big)(1)\Big)
= \dot h_\lambda^{(\ell-n)}.
$$ 
These elements of $R_{\ell-n}$ are linearly independent
for small $\lambda$, so we can conclude that the coefficients 
of all $\dot h_\lambda$ in $\dot c_{\kappa,w}$ are zero.
\end{proof}

The following corollary is well known; see also 
\cite[Th.~11.7]{BE2} for the 
explicit definition of the isomorphism.
We just note a different convention for $(q,\pi)$-integers
is used 
in \cite[Sec.~9]{BE2} compared to \cref{qinteger}.
This accounts for the difference in the defining
relation \cite[(9.2)]{BE2} for $U_{q,\pi}(\sl_2)$ 
compared to the relation \cref{fridaylate} being used for it here.

\begin{corollary}\label{sgr}
The split Grothendieck ring $K_0\big(\gsKarunderline(\UU(\sl_2))\big)$
is isomorphic as
a $\Z[q,q^{-1}]^\pi$-algebra
to the integral form
$\mathbf{U}_{q,\pi}(\sl_2)$ of $U_{q,\pi}(\sl_2)$ defined at the end of \cref{leavingdepoe}.
Under the isomorphism,
the isomorphism classes of the 1-morphisms
$E 1_k$ and $F 1_k$ correspond to the elements
of $\mathbf{U}_{q,\pi}(\sl_2)$ denoted by the same notation.
\end{corollary}

\begin{proof}
See \cite[Th.~12.1]{BE2}, which explains how to deduce this from the non-degeneracy
given by \cref{nondegthm}.
\end{proof}

\begin{remark}
\cref{nondegthm} is not new---it was already been established in \cite{DEL} by a completely different technique.
Also a version of \cref{sgr} already appeared in \cite{EL}.
The proof of \cref{nondegthm}
given here is in the same 
spirit as the proof of non-degeneracy of the ordinary $\sl_2$ 2-category given in \cite[Prop.~8.2]{Lauda} and the more general proof of non-degeneracy for $\sl_n$ given in \cite{KL3}. 
\end{remark}

For $d \geq 1$ and $k \in \Z$, 
there are graded superalgebra homomorphisms
\begin{align}
\label{fromonh}
\rho_d^{(k)}:\ONH_d &\rightarrow \End_{\UU(\sl_2)}(E^d1_k)^{\sop},&&\\\notag
x_i &\mapsto 
(-1)^{i-1}
\begin{tikzpicture}[anchorbase,scale=.8]
	\draw[->] (1,-.3) to (1,.4);
	\node at (.5,.05) {$\cdots$};
	\node at (-.5,.05) {$\cdots$};
	\draw[->] (-1,-.3) to (-1,.4);
	\draw[->] (0.0,-.3) to (0.0,.4);
	\node at (0,-.5) {$\stringnumber{i}$};
	\node at (1,-.5) {$\stringnumber{1}$};
	\node at (-1,-.5) {$\stringnumber{d}$};
     \opendot{0.0,0.05};
   \node at (1.2,.1) {$\red\scriptstyle{k}$};
\end{tikzpicture},
&\tau_j &\mapsto
-(-1)^{j-1}
\begin{tikzpicture}[anchorbase,scale=.8]
	\draw[->] (1,-.3) to (1,.4);
	\node at (.5,.05) {$\cdots$};
	\node at (-.8,.05) {$\cdots$};
	\draw[->] (-1.3,-.3) to (-1.3,.4);
	\draw[->] (0.0,-.3) to (-0.3,.4);
	\draw[->] (-.3,-.3) to (0.0,.4);
	\node at (0.03,-.5) {$\stringnumber{j}$};
	\node at (-0.35,-.5) {$\stringnumber{j+1}$};
	\node at (1,-.5) {$\stringnumber{1}$};
	\node at (-1.3,-.5) {$\stringnumber{d}$};
   \node at (1.2,.1) {$\red\scriptstyle{k}$};
\end{tikzpicture}\\
\lambda_d^{(k)}:\ONH_d &\rightarrow \End_{\UU(\sl_2)}(1_k F^d),&&\label{fromonh2}\\\notag
x_i &\mapsto 
(-1)^{d-i}
\begin{tikzpicture}[anchorbase,scale=.8]
	\draw[<-] (1,-.3) to (1,.4);
	\node at (.5,.05) {$\cdots$};
	\node at (-.5,.05) {$\cdots$};
	\draw[<-] (-1,-.3) to (-1,.4);
	\draw[<-] (0.0,-.3) to (0.0,.4);
	\node at (0,.6) {$\stringnumber{i}$};
	\node at (-1,.6) {$\stringnumber{1}$};
	\node at (1,.6) {$\stringnumber{d}$};
     \opendot{0.0,0.07};
   \node at (-1.2,.1) {$\red\scriptstyle{k}$};
\end{tikzpicture},
&\tau_j &\mapsto
-(-1)^{d-j} 
\begin{tikzpicture}[anchorbase,scale=.8]
	\draw[<-] (1,-.3) to (1,.4);
	\node at (.5,.05) {$\cdots$};
	\node at (-.8,.05) {$\cdots$};
	\draw[<-] (-1.3,-.3) to (-1.3,.4);
	\draw[<-] (0.0,-.3) to (-0.3,.4);
	\draw[<-] (-.3,-.3) to (0.0,.4);
	\node at (0.06,.6) {$\stringnumber{j+1}$};
	\node at (-0.38,.6) {$\stringnumber{j}$};
	\node at (1,.6) {$\stringnumber{d}$};
	\node at (-1.3,.6) {$\stringnumber{1}$};
   \node at (-1.5,.1) {$\red\scriptstyle{k}$};
\end{tikzpicture}
\end{align}
This follows from the relations \cref{nearlydone,nearlydone2},
with the signs in \cref{fromonh,fromonh2}
accounting for the difference between these and our preferred relations for $\ONH_n$
from \cref{ONH1,ONH2,ONH3,ONH4,ONH5,ONH6}.
Another consequence of \cref{nondegthm} is that 
both $\rho_{d}^{(k)}$ and $\lambda_d^{(k)}$
are {\em injective}.

\begin{remark}\label{fromonhrem}
On comparing with \cref{inhand1,inhand2}, it follows that 
the composition of $\rho_d^{(2n-\ell)}$
with the homomorphism 
$\End_{\UU(\sl_2)}(E^d 1_k)^{\sop}\rightarrow
\End_{\EOH_{n+d}^\ell\dash \EOH_n^\ell}\big(Q^{-nd-\binom{d}{2}} U_{n+d-1}^\ell\otimes_{\EOH_{n+d-1}^\ell}
\cdots\otimes_{\EOH_{n+1}^\ell} U_n^\ell\big)^\sop$
induced by the 2-superfunctor
$\Psi_\ell$ is equal to
the anti-homomorphism $\rho_{(1^d);n}$ from \cref{cam} (up to a degree
shift).
One can check similarly starting from \cref{mate1} that the composition
of $\lambda_d^{(2n-\ell)}$
with the homomorphism induced by $\Psi_\ell$
is equal to the homomorphism $\lambda_{n;(1^d)}$ from \cref{smith} (up
to degree shift).
\end{remark}

To conclude the section, we explain how to define divided
powers.
In $\GSKar(\UU(\sl_2))$,
there are 1-morphisms
\begin{align}\label{divpowerdef}
E^{(d)} 1_k
:= 
Q^{\binom{d}{2}}\Big( E^d 1_k, 
\rho^{(k)}_d((\chi\omega)_d)\Big)&:
k \rightarrow k+2d,\\\label{divpowerdef2}
1_{k}F^{(d)} :=
Q^{\binom{d}{2}}
\Big(
1_k F^d ,\lambda^{(k)}_d((\omega\chi)_d)\Big)
&:k+2d \rightarrow k.
\end{align} 
For example, $E^{(2)} 1_k = Q \left(E^2 1_k, \begin{tikzpicture}[anchorbase,scale=.6]
	\draw[<-] (0.25,.6) to (-0.25,-.2);
	\draw[->] (0.25,-.2) to (-0.25,.6);
      \opendot{0.15,-0.02};
   \node at (0.45,0.2) {$\red\scriptstyle{k}$};
\end{tikzpicture}\!\right)$ and
$1_k F^{(2)} = Q \left(1_k F^2, 
-\begin{tikzpicture}[anchorbase,scale=.6]
	\draw[->] (0.25,.6) to (-0.25,-.2);
	\draw[<-] (0.25,-.2) to (-0.25,.6);
      \opendot{-0.11,0.02};
   \node at (-0.45,0.2) {$\red\scriptstyle{k}$};
\end{tikzpicture}\;\right)$.
By \cref{borisagain} 
plus \cref{idempotentstar}, we have that
\begin{align}\label{divpowers}
E^d 1_k &\simeq 
\bigoplus_{w \in S_d}
Q^{2\ell(w) - \binom{d}{2}}
\Pi^{\ell(w)}
E^{(d)} 1_k,& 1_k F^d &\simeq \bigoplus_{w \in S_d} 
Q^{2\ell(w) - \binom{d}{2}} \Pi^{\ell(w)}1_k F^{(d)}.
\end{align}
In view of \cref{poincare}, it follows that $E^{(d)} 1_k$
and $F^{(d)} 1_k$ categorify the divided powers 
\cref{snow}, i.e., the isomorphism classes of the former 1-morphisms
under the isomorphism from 
\cref{sgr} give the latter elements of $\mathbf{U}_{q,\pi}(\sl_2)$.
Note $E^{(d)} 1_k$ and $1_k F^{(d)}$ are 
obtained by upshifting
the bottom degree summands of $E^d 1_k$ and $1_k F^d$.
It is also useful to have available the following, which are
downshifts of the top degree summands:
\begin{align}\label{divpowerdeftop}
\overline{E}^{(d)} 1_k
:= 
Q^{-\binom{d}{2}}\Big( E^d 1_k, 
\rho^{(k)}_d((\omega\chi)_d)\Big)&:
k \rightarrow k+2d,\\\label{divpowerdeftop2}
1_k \overline{F}^{(d)}:=
Q^{-\binom{d}{2}}
\Big(
1_k F^d ,\lambda^{(k)}_d((\chi\omega)_d)\Big)
&:k+2d \rightarrow k.
\end{align} 
For example, $\overline{E}^{(2)} 1_k = Q^{-1} \left(E^2 1_k, \begin{tikzpicture}[anchorbase,scale=.6]
	\draw[<-] (0.25,.6) to (-0.25,-.2);
	\draw[->] (0.25,-.2) to (-0.25,.6);
      \opendot{0.11,0.38};
   \node at (0.45,0.2) {$\red\scriptstyle{k}$};
\end{tikzpicture}\!\right)$ and
$1_k \overline{F}^{(2)} = Q^{-1} \left(1_k F^2, 
-\begin{tikzpicture}[anchorbase,scale=.6]
	\draw[->] (0.25,.6) to (-0.25,-.2);
	\draw[<-] (0.25,-.2) to (-0.25,.6);
      \opendot{-0.13,0.42};
   \node at (-0.45,0.2) {$\red\scriptstyle{k}$};
\end{tikzpicture}\;\right)$.
These categorify the elements $\overline{E}^{(d)} 1_k$
and $1_k \overline{F}^{(d)}$
of $\mathbf{U}_{q,\pi}(\sl_2)$ from \cref{snowtilde} since,
by \cref{borisagain} 
plus \cref{idempotentstar} again, we have that
\begin{align}\label{divpowerstilde}
E^d 1_k &\simeq 
\bigoplus_{w \in S_d}
Q^{\binom{d}{2}-2\ell(w)}
\Pi^{\ell(w)}
\overline{E}^{(d)} 1_k,& 
 1_k F^d &\simeq 
\bigoplus_{w \in S_d}
Q^{\binom{d}{2}-2\ell(w)}
\Pi^{\ell(w)}
1_k \overline{F}^{(d)}.
\end{align}
Mirroring \cref{snowier}, we have that
\begin{align}\label{snowiest}
\overline{E}^{(d)} 1_k &\simeq \Pi^{\binom{d}{2}} E^{(d)} 1_k,
&
1_k \overline{F}^{(d)} &\simeq \Pi^{\binom{d}{2}} 1_k F^{(d)}.
\end{align}
This follows because the idempotents
$(\omega\chi)_d$ and $(\chi \omega)_d$ are conjugate
as discussed after \cref{newsnow}\footnote{It could also be deduced from \cref{divpowers,divpowerstilde}
using Krull-Schmidt, but we prefer the argument given since it constructs the isomorphism explicitly.}.

\begin{lemma}\label{haverightadjs}
In $\GSKar(\UU(\sl_2))$, 
the 1-morphism 
$Q^{-d(k+d)}\Pi^{\binom{d}{2}} 1_k F^{(d)}$ is right dual to $E^{(d)} 1_k$,
and the 1-morphism
$Q^{d(k+d)} \Pi^{d(k+d)+\binom{d}{2}} E^{(d)} 1_{k-2d}$ is
 right dual to $F^{(d)} 1_k$.
 \end{lemma}

\begin{proof}
We first show that 
$Q^{-d(k+d)} 1_k F^{(d)}$ is right dual to $E^{(d)} 1_k$
in the $(Q,\Pi)$-envelope $\GSKar(\UU(\sl_2))$.
By \cref{rightadj}, 
$Q^{-k-1} 1_k F$ is right dual to $E 1_k$.
Hence, $Q^{-d(k+d)} Q^{-\binom{d}{2}}
1_k F^d$ is right dual to $Q^{\binom{d}{2}} E^d 1_k$.
By definition, $E^{(d)} 1_k$ is the summand of
$Q^{\binom{d}{2}} E^d 1_k$ defined by the idempotent
$Q^{\binom{d}{2}} \rho_d^{(k)}\big((\chi\omega)_d\big)$
and $Q^{-d(k+d)} 1_k \overline{F}^{(d)}$ is the summand of
$Q^{-d(k+d)} Q^{-\binom{d}{2}} 1_k F^d$ defined by the 
the idempotent
$Q^{-d(k+d)} Q^{-\binom{d}{2}} \lambda_d^{(k)}
\big((\chi\omega)_d\big)$.
Now we observe using \cref{oddadjunction}(2),
\cref{phew} and \cref{idempotentstar} 
that $Q^{-d(k+d)}Q^{-\binom{d}{2}}\lambda_d^{(k)}\big((\chi\omega)_d\big)$
is the right mate of $Q^{\binom{d}{2}} \rho_d^{(k)}\big((\chi\omega)_d\big)$.
Hence, we get that
$Q^{-d(k+d)} 1_k \overline{F}^{(d)}$ is right dual to
$E^{(d)} 1_k$. It remains to apply
\cref{snowiest} to pass from
$Q^{-d(k+d)} 1_k \overline{F}^{(d)}$
to
$Q^{-d(k+d)} \Pi^{\binom{d}{2}}1_k F^{(d)}$.

The proof that
$Q^{d(k+d)}\Pi^{d(k+d)+\binom{d}{2}} E^{(d)} 1_{k}$ is
 right dual to $1_k F^{(d)}$ is similar using \cref{leftadj} instead of \cref{rightadj}.
 By \cref{leftadj} and \cref{oddadjunction}(1),
 $Q^{k+1}\Pi^{k+1} E 1_k$ is right dual to $1_k F$.
Hence,
$Q^{d(k+dx)-\binom{d}{2}} \Pi^{d(k+d)} E^d 1_k$ is right dual to
$Q^{\binom{d}{2}} 1_k F^d$.
Since \cref{crossingcyclicity} is more complicated than
\cref{phew}, it is no longer true that
the idempotent
$Q^{d(k+d)-\binom{d}{2}}\Pi^{d(k+d)} 
\rho_d^{(k)}\big((\chi\omega)_d\big)$
is {\em equal} to the right mate of the idempotent
$Q^{\binom{d}{2}}\lambda_d^{(k)}\big((\chi\omega)_d\big)$,
but these two idempotents are conjugate via even degree 0 units.
This follows by the Krull-Schmidt theorem applied to the finite-dimensional algebra that is the even degree 0 component of $\ONH_d \otimes R$.
Hence, $Q^{d(k+d)} \Pi^{d(k+d)} \overline{E}^{(d)} 1_k$ is right dual to
$1_k F^{(d)}$.
It remains to appeal to \cref{snowiest} one more time.
\end{proof}

\section{Some graded 2-representation theory}\label{2Repsec}

In this section, we develop some 2-representation theory of
the $\sl_2$ 2-category  $\UU(\sl_2)$ from
\cref{km2cat}. We work throughout in the graded setting, but all the 
definitions and results here have analogs with the $\Z$-grading forgotten.
The following is modelled on \cite[Def.~5.1.1]{Rou}.

\begin{definition}\label{2repdef}
By a {\em graded 2-representation} $\cV$ of $\UU(\sl_2)$, we mean
a strict graded 2-superfunctor
$\cV:\UU(\sl_2)\rightarrow \GSCAT$.
Decoding the definition, $\cV$ consists of the following data:
\begin{itemize}
\item a graded supercategory $\cV$ with a 
given decomposition 
into {\em weight subcategories}
$\cV = \coprod_{k \in \Z} \cV_k$ (or
$\cV = \bigoplus_{k \in \Z} \cV_k$ when $\cV$ is additive);
\item graded superfunctors $E:\cV \rightarrow \cV$
and $F:\cV \rightarrow \cV$ 
such that $E|_{\cV_k}:\cV_k\rightarrow \cV_{k+2}$
and $F|_{\cV_k}:\cV_k \rightarrow \cV_{k-2}$ for each $k \in \Z$;
\item  graded supernatural transformations
$x:E \Rightarrow E$ and $\tau:E^2 \Rightarrow E^2$
which are odd of degrees 2
and $-2$, respectively;
\item (inhomogeneous) graded supernatural transformations
$\eta:\Id \Rightarrow FE$ and $\eps:EF \Rightarrow \Id$
whose 
restrictions
$\eta: \Id_{\cV_k} \Rightarrow FE|_{\cV_k}$
and
$\eps:E F|_{\cV_{k+2}} \Rightarrow \Id_{\cV_{k+2}}$
are even of degrees $k+1$
and $-k-1$, respectively.
\end{itemize}
Then there are the axioms:
\begin{itemize}
\item
the relations from \cref{nearlydone} hold:
$\tau\circ \tau = 0$, $(\tau E) \circ (E \tau) \circ (\tau E) = (E \tau) \circ (\tau E) \circ (E \tau )$ and
$(Ex)\circ \tau+(xE) \circ \tau = (xE)\circ \tau +\tau \circ (Ex) = E^2$;
\item
$\eta$ and $\eps$ satisfy the zig-zag relations:
$(F\eps) \circ (\eta F) = F$
and $(\eps E) \circ (E \eta) = E$ (equivalently,
they define units and counits of adjunctions
making $Q^{-k-1} F|_{\cV_{k+2}}$ into a 
right adjoint to $E|_{\cV_k}$ for each $k \in \Z$);
\item
letting $\sigma := (FE\eps) \circ (F\tau F) \circ (\eta EF):EF\Rightarrow FE$
be the image of the rightward crossing under $\cV$,
the following inhomogeneous matrices of 
supernatural transformations are isomorphisms:
\begin{align*}
\left(\:\sigma
\qquad
\eps
\qquad
\eps\circ (xF)
\qquad
\cdots\qquad
\eps \circ (x F)^{k-1}
\right)^T&:
E F|_{\cV_k}\Rightarrow
F E|_{\cV_k} \oplus \Id_{\cV_k}^{\oplus k}&&
\text{for $k \geq 0$}\\
\left(\:
\sigma\qquad
\eta
\qquad
(F x) \circ \eta
      \qquad\cdots\qquad
(F x)^{-k-1} \circ \eta
\right)
&:E F|_{\cV_k} \oplus  \Id_{\cV_k}^{\oplus (-k)}
\Rightarrow
 F E|_{\cV_k}&&\text{for $k \leq 0$}
 \end{align*}
\end{itemize}
\end{definition}

There are natural notions of (full) {\em sub-2-representations} (which are called ``invariant ideals" in \cite[$\S$4.2]{BD}),
{\em quotient 2-representations}, and
{\em morphisms} of graded 2-representations. 
The latter definition, which is the super analog of \cite[Def.~2.3]{Rou},
is equivalent to the following, which is similar to the formulation adopted in
\cite[Sec.~5.2.1]{CR}; the terminology being used is the same as in
\cite[Def.~4.6]{BD} (and actually goes back to Ben Webster).

\begin{definition}\label{sef}
Let $\cV$ and $\cW$ be two graded 2-representations of $\UU(\sl_2)$.
A {\em strongly equivariant graded superfunctor}
$\Omega:\cV\rightarrow \cW$
is a graded superfunctor such that $\Omega|_{\cV_k}:\cV_k \rightarrow
\cW_k$ for each $k \in \Z$,
plus a degree 0 even graded
supernatural isomorphism
$\zeta:E \Omega \stackrel{\sim}{\Rightarrow} \Omega E$,
such that the following holds
\begin{itemize}
\item
the supernatural transformation $(F \Omega \eps) \circ (F \zeta F) \circ
(\eta \Omega F):\Omega F \Rightarrow F \Omega$ is invertible;
\item
we have that $(\Omega x) \circ \zeta = \zeta \circ (x \Omega)$;
\item we have that $(\Omega \tau) \circ (\zeta E)\circ (E \zeta) = (\zeta E) \circ (E \zeta) \circ (\tau \Omega)$.
\end{itemize}
A {\em strongly equivariant graded superequivalence} is
a strongly equivariant graded superfunctor which is also a superequivalence of
supercategories. 
\end{definition}

\begin{remark}
For strongly equivariant graded superequivalences, the first axiom in \cref{sef} actually
holds automatically; see \cite[Rem.~4.8]{BD} where this is explained (in the purely even setting).
Also in \cite{BD}, the diagrammatic interpretation of these definitions
is discussed, which we still find helpful.
\end{remark}

\begin{remark}
There is an obvious way to make the composition
of two strongly
equivariant graded superfunctors into a strongly equivariant graded superfunctor in its own
right.
Also the identity functor $\Id$ is strongly equivariant with
$\zeta := 1_{E}$.
So there is a category $\mathpzc{Rep}(\UU(\sl_2))$
consisting of graded
2-representations and strongly equivariant graded superfunctors.
\end{remark}

Usually, the graded supercategories $\cV_k$ in
a graded 2-representation $\cV$ will have some extra structure,
such as being additive or $(Q,\Pi)$-complete.
We are mainly interested here in what we call 
{\em graded Karoubian 2-representations}.
By definition, this means a graded 2-representation $\cV$ such that,
for each $k \in \Z$, the weight subcategory
$\cV_k$ is additive and $(Q,\Pi)$-complete,
and the underlying ordinary category
$\underline{\cV}_k$ is idempotent complete.
Any graded 2-representation $\cV$ can be upgraded to a Karoubian graded
2-representation by passing to its graded super Karoubi envelope $\GSKar(\cV)$.

Given a graded Karoubian 2-representation
$\cV$, the underlying graded 2-superfunctor
from $\UU(\sl_2)$ to $\cV$
extends canonically to a graded
2-superfunctor from the
graded super Karoubi envelope $\GSKar(\UU(\sl_2))$ to $\cV$.
The direct sum over all $k \in \Z$ of the images 
under this graded 2-superfunctor
of the 1-morphisms 
$E^{(d)} 1_k$ and $F^{(d)} 1_k$ 
from \cref{divpowerdef,divpowerdef2}
give graded superfunctors
\begin{equation}
E^{(d)}, F^{(d)}:\cV \rightarrow \cV.
\end{equation}
By \cref{divpowers}, we have that
\begin{align}\label{divpowersapp}
E^d &\simeq 
\bigoplus_{w \in S_d}
\Pi^{\ell(w)}
Q^{2\ell(w) - \binom{d}{2}}
E^{(d)},& 
F^d &\simeq 
\bigoplus_{w \in S_d}
\Pi^{\ell(w)}
Q^{2\ell(w) - \binom{d}{2}}
F^{(d)}.
\end{align}
\cref{haverightadjs} implies that 
$Q^{-d(k+d)}\Pi^{\binom{d}{2}} 
F^{(d)}|_{\cV_{k+2d}}$ is right adjoint to
$E^{(d)}|_{\cV_k}$
and $Q^{d(k+d)}\Pi^{d(k+d)+\binom{d}{2}}E^{(d)}|_{\cV_{k-2d}}$
is right adjoint to $F^{(d)}|_{\cV_{k}}$,
with 
units and counits of adjunction that are defined by images of 2-morphisms in $\GSKar(\UU(\sl_2))$.

A graded 2-representation $\cV$
is said to be {\em integrable} if $E$ and $F$ are locally nilpotent, i.e.,
for any $k \in \Z$ and any $M \in \cV_k$
there is some $n \geq 0$ such that $E^n M = F^n M = 0$.
Also, for $\ell \in \N$, a {\em lowest weight object of weight $-\ell$}
means an object $M \in \cV_{-\ell}$ such that $FM = 0$.

\begin{example}\label{aishaonthephone}
Suppose that $\ell \in \N$. By
\cref{actiontheorem}, there is an integrable
graded Karoubian 2-representation 
\begin{equation}
\EOH^\ell\psmod
 := 
\bigoplus_{n=0}^\ell \EOH_n^\ell\psmod
\end{equation}
with the weight $k$ subcategory
$(\EOH^\ell\psmod)_k$ equal to
$\EOH_n^\ell\psmod$ if $k = 2n-\ell$
for $0 \leq n \leq \ell$,
or the trivial (zero) graded supercategory otherwise.
Other data is as follows.
\begin{itemize}
\item
The graded superfunctors $E$ and $F$ 
are $Q^{-n} U_n^\ell\otimes_{\EOH_n^\ell}-$
on the weight subcategory $\EOH_n^\ell\psmod$ 
and $Q^{3n+1-\ell} V_n^\ell\otimes_{\EOH_{n+1}^\ell}-$
on the weight subcategory $\EOH_{n+1}^\ell\psmod$, 
respectively,
assuming $0 \leq n < \ell$. On all other weight subcategories, $E$ and $F$ are zero. 
\item
The graded supernatural transformations $x$  and $\tau$ are defined 
by the supernatural transformations $\rho_{(1);n}(x_1) \otimes \id$
viewed as elements 
 $\sEnd\big(Q^{-n} U_n^\ell\otimes_{\EOH_n^\ell}-\big)_{2,\1}$ and
the supernatural transformations
$-\rho_{(1^2);n}(\tau_1) \otimes \id$ viewed as elements of 
 $\sEnd\big(Q^{-2n-1} U_{n+1}^\ell\otimes_{\EOH_{n+1}^\ell}
Q^{-n} U_n^\ell\otimes_{\EOH_n^\ell}-\big)_{-2,\1}$, respectively, 
for all admissible $n$.
\item
The graded supernatural transformations $\eta$ and $\eps$ are given by the appropriate 
counit and unit from
\cref{cupsandcaps}.
\item
The homomorphisms induced by \cref{fromonh,fromonh2} 
are equal to \cref{cam,smith} thanks to \cref{fromonhrem}.
\item
For $0 \leq n \leq n+d \leq \ell$,
we have that
$E^{(d)}|_{\EOH_n^\ell\psmod}
\simeq Q^{-dn} U_{(d);n}^\ell \otimes_{\EOH_n^\ell}-$
and $F^{(d)}|_{\EOH_{n+d}^\ell\psmod}
\simeq Q^{-d(\ell-3n-2d+1)}V_{n;(d)}^\ell \otimes_{\EOH_{n+d}^\ell}-$;
cf. \cref{K0}.
\end{itemize}
We point out also by \cref{K0} that 
$K_0(\EOH^\ell\Upsmod\big)$
is naturally identified with the $\mathbf{U}_{q,\pi}(\sl_2)$-module 
$\mathbf{V}(-\ell)$, and $\EOH_0^\ell$ is a lowest weight object of weight $-\ell$.
\end{example}

Now we come to one of the key constructions introduced by
Rouquier in \cite{Rou} in the purely even case, the construction of 
{\em cyclotomic quotients}.
For any $\ell \in \Z$,
there is a graded 2-representation $\cR(\ell)$
with
\begin{equation}
\cR(\ell)_k := \mathpzc{Hom}_{\UU(\sl_2)}(\ell,k)
\end{equation}
for $k \in \Z$,
viewed as a graded 2-representation of the graded 2-supercategory
$\UU(\sl_2)$ in an obvious way.
For example, the graded superfunctor 
$E|_{\cR(\ell)_k}:\cR(\ell)_k \rightarrow \cR(\ell)_{k+2}$
is defined by horizontally composing on the left
with the 1-morphism $E 1_k$, and the 
supernatural transformation $x:E|_{\cR(\ell)_k}\Rightarrow E|_{\cR(\ell)_k}$
is induced by the 2-endomorphism
$\begin{tikzpicture}[anchorbase,scale=.8]
	\draw[->] (0.08,-.3) to (0.08,.4);
     \opendot{0.08,0.05};
   \node at (.35,.1) {$\red\scriptstyle{k}$};
\end{tikzpicture} :E 1_k \Rightarrow E 1_k$.
The graded 2-representation $\cR(\ell)$
has the following universal property.

\begin{lemma}\label{Lop}
Given any graded 2-representation $\cV$ and any $M \in \cV_\ell$
there is a canonical strongly equivariant graded superfunctor
$\omega_M:\cR(\ell) \rightarrow \cV$ taking the
object $1_\ell$ of $\cR(\ell)_{\ell}$ to $M$.
\end{lemma}

We define the {\em universal
graded 2-representation of lowest weight $-\ell \in \Z$},
denoted $\cV(-\ell)$, to be the quotient 2-representation
$\cR(-\ell) / \cI$,
where $\cI$ here is the sub-2-representation of $\cR(-\ell)$
generated by $\Big\downarrow \:{\red \scriptstyle -\ell}$ (the identity endomorphism of the object $F 1_{-\ell}$).
We denote the lowest weight object of $\cV(-\ell)_{-\ell}$
arising from the object $1_{-\ell} \in \cR(-\ell)_{-\ell}$ 
by $\overline{1}_{-\ell}$, and call this the {\em canonical 
lowest weight object}. It is a generating object for $\cV(-\ell)$.
The identity endomorphism of
$\overline{1}_{-\ell}$ is equal to the image of
the bottom bubble
$\begin{tikzpicture}[anchorbase]
  \draw[<-] (0,0.4) to[out=180,in=90] (-.2,0.2);
  \draw[-] (0.2,0.2) to[out=90,in=0] (0,.4);
 \draw[-] (-.2,0.2) to[out=-90,in=180] (0,0);
  \draw[-] (0,0) to[out=0,in=-90] (0.2,0.2);
   \node at (0.4,0.2) {$\red\scriptstyle{-\ell}$};
   \opendot{-0.2,0.2};
   \node at (-.7,0.2) {$\scriptstyle{-\ell-1}$};
\end{tikzpicture}$, i.e., the image of $1 \in R$
under the homomorphism \cref{hatbeta}. 
If $\ell < 0$, this bubble is not a fake bubble, so it
belongs to $\cI$.
This shows that  the graded supercategory
$\cV(-\ell)$ is trivial if $\ell < 0$.
Thus, $\cV(-\ell)$ is only interesting if $\ell \in \N$,
i.e., it is a dominant weight for $\sl_2$.
The following, the universal property of $\cV(-\ell)$,
follows immediately from \cref{Lop} and
the universal property of quotients.

\begin{lemma}\label{Lup}
Let $\cV$ be any graded 2-representation
of $\UU(\sl_2)$, $\ell \in \N$ and 
$M \in\cV_{-\ell}$ be a lowest weight object.
The superfunctor $\omega_M:\cR(-\ell)\rightarrow \cV$
from \cref{Lop} induces a strongly equivariant graded superfunctor
$\Omega_M:\cV(-\ell)\rightarrow \cV$
taking $\overline{1}_{-\ell}$ to $M$.
\end{lemma}

There is a more sophisticated version of \cref{Lup}, which is analogous to
\cite[Prop.~5.6]{Rou}. To formulate this, we need one more preliminary lemma.

\begin{lemma}\label{endofid}
The homomorphism\footnote{In fact, $\beta_{-\ell}$ is
itself an isomorphism thanks
to \cref{nondegthm}, but this is not relevant for the present lemma.} 
 $\beta_{-\ell}:R \rightarrow \End_{\cR(-\ell)}(1_{-\ell})$
from \cref{hatbeta} induces an isomorphism
$\bar{\beta}_{-\ell}: 
R_\ell\stackrel{\sim}{\rightarrow} \End_{\cV(-\ell)}(\overline{1}_{-\ell})$.
\end{lemma}

\begin{proof}
The bubble
$\begin{tikzpicture}[anchorbase]
  \draw[<-] (0,0.4) to[out=180,in=90] (-.2,0.2);
  \draw[-] (0.2,0.2) to[out=90,in=0] (0,.4);
 \draw[-] (-.2,0.2) to[out=-90,in=180] (0,0);
  \draw[-] (0,0) to[out=0,in=-90] (0.2,0.2);
   \node at (0.5,0.2) {$\red\scriptstyle{-\ell}$};
   \opendot{-0.2,0.2};
   \node at (-.7,0.2) {$\scriptstyle{r-\ell-1}$};
\end{tikzpicture}$ belongs to $\cI$ for $r > \ell$.
Up to a sign, the composition of $\beta_{-\ell}$ with the canonical  
 map $\End_{\cR(-\ell)}(1_{-\ell}) \twoheadrightarrow \End_{\cV(-\ell)}(\overline{1}_{-\ell})$
takes $\dot \eps_r \in R$ to the 
the image of this bubble, which is zero.
We deduce
that this  homomorphism factors through the quotient $R_\ell$ of $R$
to induce
$\bar\beta_{-\ell}$.
Moreover, $\bar\beta_{-\ell}$ is surjective 
since $\beta_{-\ell}$ is surjective by the ``easy" part of \cref{nondegthm}.

To show that $\bar\beta_{-\ell}$ is also injective, we use the following
diagram of graded supercategories and superfunctors:
$$
\begin{tikzcd}
\END_{\UU(\sl_2)}(-\ell)\arrow[d,"\operatorname{Ev}_{\overline{1}_{-\ell}}" left]
\arrow[r,"\Psi_\ell" above]&\EOH_0^\ell \bisMod \EOH_0^\ell\arrow[d,"-\otimes_{\EOH_0^\ell} \EOH_0^\ell" right]\\
\cV(-\ell)_{-\ell}\arrow[r,"\Omega_{\EOH_0^\ell}" below]&\EOH_0^\ell\sMod
\end{tikzcd}
$$
Here, the top map comes from \cref{actiontheorem},
the left hand vertical superfunctor 
is given by evaluating 
on the object $\overline{1}_{-\ell}$,
and the right hand vertical superfunctor is given by
tensoring with the lowest weight object $\EOH_0^\ell$.
The way the bottom superfunctor 
$\Omega_{\EOH_0^\ell}$ is defined in \cref{Lup}
ensures that this diagram commutes strictly. 
It follows that the middle square 
in the following diagram commutes:
$$
\begin{tikzcd}
\arrow[d,"\operatorname{can}" left,twoheadrightarrow]
R\arrow[r,"\beta_{-\ell}" above]&\End_{\cR(-\ell)_{-\ell}}(1_{-\ell})
\arrow[d,"\operatorname{can}" right]
\arrow[r,"\Psi_\ell" above]&\End_{\EOH_0^\ell\dash\EOH_0^\ell}(\EOH_0^\ell)
\arrow[d,"\theta\mapsto \theta\otimes \id" right]\\
R_\ell
\arrow[r,twoheadrightarrow,"\bar\beta_{-\ell}" below]&\End_{\cV(-\ell)_{-\ell}}(\overline{1}_{-\ell})\arrow[r,"\Omega_{\EOH_0^\ell}" below]&\End_{\EOH_0^\ell\dash}(\EOH_0^\ell)\arrow[r,"\phi \mapsto \phi(1)" below]&\EOH_0^\ell 
\arrow[r,"\sim" above,"\alpha_0^\ell" below]&R_\ell
\end{tikzcd}
$$
\cref{counterclockwisebubbles} 
shows that 
the composition $R \rightarrow R_\ell$
around the northeast boundary of this diagram
is equal to the canonical quotient map.
Hence, the composition $R_\ell\rightarrow R_\ell$
of the three maps at the bottom of the diagram is the identity.
This implies that $\bar\beta_{-\ell}$ is injective.
\end{proof}

Any morphism space
$\Hom_{\cR(-\ell)}(X,Y)$ in $\cR(-\ell)$
can be viewed as a right $R$-supermodule
so that $\ocirc c \in R$ acts by horizontally composing on the right
with $\beta_{-\ell}(\ocirc c)$.
This induces a structure of right $R_\ell$-supermodule
on any morphism space $\Hom_{\cV(-\ell)}(X,Y)$; cf.
the first paragraph of the proof of \cref{endofid}.
Given a graded $R_\ell$-superalgebra $A$,
we let $\cV(-\ell)\otimes_{R_\ell} A$
be the graded supercategory with the same
objects as $\cV(-\ell)$ and 
morphism spaces $\Hom_{\cV(-\ell)\otimes_{R_\ell} A}(X,Y)
:= \Hom_{\cV(-\ell)}(X,Y)\otimes_{R_\ell} A$.
This is naturally a graded 2-representation
of $\UU(\sl_2)$ in its own right.

\begin{theorem}\label{structthm}
Let $\cV$ be any graded 2-representation
of $\UU(\sl_2)$, $\ell \in \N$ and 
$M \in\cV_{-\ell}$ be any lowest weight object.
The strongly equivariant graded 
superfunctor $\Omega_M:\cV(-\ell)\rightarrow
\cV$ from \cref{Lup} extends to a {\em fully faithful}
strongly equivariant graded superfunctor
$\Omega_M\otimes \id:\cV(-\ell)\otimes_{R_\ell} 
\End_{\cV}(M) \rightarrow \cV$.
\end{theorem}

\begin{proof}
Let $A:= \End_{\cV}(M)$ for short.
The graded superfunctor $\Omega_M$ extends to
$\Omega_M\otimes \id$ by the universal property of tensor product.
To see that the resulting graded superfunctor is fully faithful,
we must show that it defines an isomorphism
$\Hom_{\cV(-\ell)\otimes_{R_\ell} A}(X,Y)
\stackrel{\sim}{\rightarrow}
\Hom_{\cV}(X,Y)$ for objects of any weight subcategory of
$\cV(-\ell)\otimes_{R_\ell} A$.
This is clear if $X=Y = \overline{1}_{-\ell}$.
The result in general then follows by the (now standard) technique 
explained in the proof of \cite[Lem.~5.4, Prop.~5.6]{Rou}.
\end{proof}

\begin{corollary}\label{babykid}
For $\ell \in \N$, let $\EOH^\ell\psmod$ be the graded Karoubian 2-representation from \cref{aishaonthephone},
and let $\GSKar(\cV(-\ell))$ be the graded super Karoubi envelope of
$\cV(-\ell)$, which is another graded
Karoubian 2-representation.
The strongly equivariant graded superfunctor
$\Omega_{\EOH_0^\ell}:\cV(-\ell) \rightarrow \EOH^\ell\psmod$
associated to the lowest weight object
$\EOH_0^\ell$
induces a strongly equivariant graded superequivalence
$\Xi_\ell:\GSKar(\cV(-\ell)) \rightarrow \EOH^\ell\psmod$.
\end{corollary}

\begin{proof}
In view of \cref{endofid,structthm},
$\Omega_{\EOH_0^\ell}$ is fully faithful. 
This extends by the universal property of the graded super Karoubi envelope
to give 
a fully faithful strongly equivariant graded superfunctor
$\Xi_\ell:\GSKar(\cV(-\ell)) \rightarrow \EOH^\ell\psmod$.
To see that $\Xi_\ell$ is a graded superequivalence, it remains to check that it is dense.
This follows because 
$$
E^{(n)} \EOH_0^\ell \simeq 
U_{(n);0}^\ell \otimes_{\EOH_0^\ell} \EOH_0^\ell
\simeq \EOH_n^\ell,
$$
the last isomorphism following
 since 
$U_{(n);0}^\ell$ is free of rank 1 as a graded left $\EOH_n^\ell$-supermodule
by \cref{lemma01}(2).
\end{proof}

We record one more basic lemma, which is analogous to the first part of \cite[Lem.~5.2]{Rou}.

\begin{lemma}\label{sillylemma}
Let $\cV$ be an integrable Karoubian graded 2-representation of $\UU(\sl_2)$.
Let $N$ be an object of $\cV_k$ for some $k \in \Z$.
If 
$\Hom_{\cV_k}(E^n M, N) = 0$ for all $\ell \in \N$,
$n \geq 0$ such that $k = 2n-\ell$
and all lowest weight objects $M \in \cV_{-\ell}$, then $N = 0$.
\end{lemma}

\begin{proof}
Suppose that $N \neq 0$.
By integrability, there exists $n \geq 0$ such that 
$F^n N \neq 0$ and $F^{n+1} N = 0$.
This means that $M := F^n N$ is a non-zero lowest weight object
of $\cV_{-\ell}$ for  $\ell=k-2n \in \N$.
By assumption, we have that $\Hom_{\cV_k}(E^n M, N) = 0$.
Hence, by adjunction, $$
\End_{\cV_{-\ell}}(M)=
\Hom_{\cV_{-\ell}}(M, F^n N)
\simeq \Hom_{\cV_k}(E^n M, N)
 = 0.
 $$
It follows that $1_{M} = 0$, so $M = 0$, which is a contradiction.
\end{proof}

\begin{remark}
There is more still to be done here. For example, Rouquier continues in \cite[Sec.~5.1.4]{Rou} to construct
a Jordan-H\"older series in an arbitrary integrable 
Karoubian 2-representation, and this result assuredly
carries over to our setting.
There is also a good theory of
{\em locally finite Abelian} 2-representations of $\UU(\sl_2)$,
including an analog of \cite[Prop.~5.20]{CR} which implies that
the irreducible objects of such a 2-representation 
can be given the structure of a crystal in the sense of Kashiwara.
It would be worthwhile to extend \cite[Th.~5.27]{CR} (which is 
a special case of Rouquier's ``control by $K_0$'' from \cite[Th.~5.22]{Rou}) 
to this setting. This would pave the way to more applications involving representations of the supergroup $Q(n)$ and 
the Lie superalgebra $\mathfrak{q}_n(\C)$. 
In the ordinary case, an alternative approach by-passing control by $K_0$ was developed in \cite{BSW}, which we expect 
should also have an interesting and non-trivial super analog.
Another direction we would like to investigate further is 
to extend \cref{actiontheorem,nondegthm} from odd $\sl_2$ to 
the super Kac-Moody 2-category 
associated to 
``odd $\mathfrak{so}_{2n+1}$", 
thereby giving an odd analog of the 2-representation
of $\mathfrak{sl}_n$ constructed 
in \cite{KL3}.
\end{remark}

\section{The odd analog of the Rickard complex}\label{ricksec}

Let $\cV$ be a graded 2-supercategory. 
The notation $\Ch^b(\cV)$ denotes the graded supercategory of bounded cochain complexes
and chain maps in $\cV$; 
differentials in a cochain complex are assumed to be even of degree 0
but we allow 
chain maps whose components are inhomogeneous.
Also $K^b(\cV)$ is the homotopy category, which
is a graded supercategory
with the same objects as
$\Ch^b(\cV)$ and morphisms that are chain homotopy equivalence classes of chain maps; 
chain homotopies are again required 
to be even of degree 0.
If $\cV$ is an 
integrable graded Karoubian 2-representation of $\UU(\sl_2)$
as in the previous section,
both $\Ch^b(\cV)$ and $K^b(\cV)$ are themselves integrable graded Karoubian 2-representations
of $\UU(\sl_2)$ in a natural way.

Fix $k \in \Z$.
The {\em odd Rickard complex} $\Theta_k$,
so-called because it is the odd analog of the complex in \cite[Sec.~6.2]{CR} 
which was
introduced originally by Rickard in the context of symmetric groups,
is the following 
cochain complex in
$\Ch\big(\mathpzc{Hom}_{\GSKar(\UU(\sl_2))}(-k,k)\big)$:
$$
\left\{
\begin{array}{rl}
\cdots \rightarrow
Q^d E^{(k+d)} F^{(d)}1_{-k} \stackrel{\partial^{-d}}{\rightarrow}
Q^{d-1}E^{(k+d-1)} F^{(d-1)} 1_{-k}\rightarrow
\cdots \rightarrow E^{(k)} 1_{-k} \rightarrow 0\rightarrow\cdots&\text{if $k \geq 0$}\\
\cdots \rightarrow
Q^d E^{(k+d)} F^{(d)}1_{-k} \stackrel{\partial^{-d}}{\rightarrow}
Q^{d-1}E^{(k+d-1)} F^{(d-1)} 1_{-k}\rightarrow
\cdots \rightarrow Q^{-k}F^{(-k)} 1_{-k} \rightarrow 0 \rightarrow \cdots
&\text{if $k \leq 0$},
\end{array}\right.
$$
where in both cases $E^{(k+d)} F^{(d)} 1_{-k}$
is in cohomological degree $-d$.
The differential
$$
\partial^{-d}:
Q^d E^{(k+d)} F^{(d)}1_{-k} \rightarrow
Q^{d-1} E^{(k+d-1)} F^{(d-1)} 1_{-k}
$$ is
the composition first of the ``inclusion"
of $Q^d E^{(k+d)} F^{(d)} 1_{-k}
\rightarrow Q^{k+3d-2} E^{(k+d-1)} E F F^{(d-1)} 1_{-k}$ 
as a summand\footnote{The idempotent endomorphism defining 
$Q^{k+3d-2} E^{(k+d-1)} E F F^{(d-1)} 1_{-k}$ as a summand of $Q^d E^{k+d} F^d 1_{-k}$
decomposes as the sum of two mutually orthogonal idempotents, one of which
is the idempotent defining $Q^{d-1} E^{(k+d)} F^{(d)} 1_{-k}$.} 
of $E^{(k+d-1)} E F F^{(d-1)} 1_{-k}$,
then 
$Q^{k+3d-2} E^{(k+d-1)} \eps F^{(d-1)}:
Q^{k+3d-2} E^{(k+d-1)} E F F^{(d-1)} 1_{-k}
\rightarrow
Q^{d-1} E^{(k+d-1)} F^{(d-1)} 1_{-k}$.
Note this is even of degree 0 as required.
The following checks that it is a cochain complex.

\begin{lemma}\label{kentucky}
We have that $\partial^{-d+1} \circ \partial^{-d} = 0$ for all $d$.
\end{lemma}

\begin{proof}
Ignoring gradings for brevity, 
it suffices to show that 
the composition 
$$
E^{(2)} F^{(2)}  
1_{-k-2d+4} 
\stackrel{\operatorname{inc}}{\longrightarrow} 
E^2 F^2 1_{-k-2d+4}
\stackrel{E \eps F}{\longrightarrow} E 1_{-k-2d+4}
F \stackrel{\eps}{\longrightarrow} 1_{-k-2d+4}
$$
is zero.
The identity endomorphism of
$E^{(2)} F^{(2)} 1_{-k-2d+4}$ 
is $$
\rho_2^{(-k-2d)}(x_1 \tau_1) \lambda_2^{(-k-2d)} (\tau_1 x_1)
= \big(\rho_2^{(-k-2d)}(\tau_1) \lambda_2^{(-k-2d)}(\tau_1)\big)
\circ \big(\rho_2^{(-k-2d)}(x_1) \lambda_2^{(-k-2d)}(x_1)\big).
$$
The composition of this with
$\eps \circ (E \eps F)$ is zero:
$$
\begin{tikzpicture}[anchorbase,scale=1]
\draw[->] (-.1,-.5) to (-.1,.2) [out=90,in=-90] to (.2,.7) to [out=90,in=-90] (.2,.8) to [out=90,in=180](.4,1) to [out=0,in=90] (.6,.8) to[out=-90,in=90] (.6,.3) to [out=-90,in=90] (.9,-.2) to [out=-90,in=90] (0.9,-.5);
\draw[->] (.2,-.5) to (.2,.2) [out=90,in=-90] to (-.1,.7) to [out=90,in=-90] (-.1,.8) to [out=90,in=180] (.4,1.3) to [out=0,in=90]
(.9,.8) to [out=-90,in=90] (.9,.3) to [out=-90,in=90] (.6,-.2) to [out=-90,in=90] (.6,-.5);
\opendot{.2,-.2};
\opendot{.6,-.3};
   \node at (1.5,.45) {$\red\scriptstyle{-k-2d+4}$};
\end{tikzpicture}
=
\begin{tikzpicture}[anchorbase,scale=1]
\draw[->] (.2,-.5) to [out=90,in=-90] (.2,.8) to [out=90,in=180] (.4,1) to[out=0,in=90] (.6,.8) to [out=-90,in=90] (.9,.3) to[out=-90,in=90] (.6,-.2) to [out=-90,in=90] (.6,-.5);
\draw[->] (-.1,-.5) to[out=90,in=-90] (-.1,.8) to [out=90,in=180] (.4,1.3) to [out=0,in=90] (.9,.8) to[out=-90,in=90] (.6,.3) to [out=-90,in=90] (.9,-.2) to [out=-90,in=90] (0.9,-.5);
   \node at (1.5,.45) {$\red\scriptstyle{-k-2d+4}$};
\opendot{.2,-.2};
\opendot{.6,-.3};
\end{tikzpicture} = 0.
$$
\end{proof}

\begin{remark}
Note \cref{kentucky} plus \cref{actiontheorem} implies \cref{isdifferential}.
So the proof of that lemma was actually unnecessary
(as, by association, 
was \cref{mate3}) but we included it to make \cref{derby} 
independent of the subsequent material.
\end{remark}

Suppose now that $\cV$ is an integrable graded Karoubian 2-representation
of $\UU(\sl_2)$.
Given any cochain complex $C \in \Ch^b(\cV_{-k})$,
we can apply the complex of graded superfunctors that is the image under 
$\cV$ of the odd Rickard complex $\Theta_k$ to obtain a double complex.
The associated total complex is again bounded thanks to the integrability assumption.
This construction defines a graded superfunctor 
$\Ch^b(\cV_{-k}) \rightarrow \Ch^b(\cV_k)$. Passing to the quotient
$K^b(\cV)$ of $\Ch^b(\cV)$,
we obtain from this a graded superfunctor 
\begin{equation}
\cV(\Theta_k):K^b(\cV_{-k}) \rightarrow K^b(\cV_k).
\end{equation}

\begin{lemma}\label{nearly}
Let $\cV$ be the graded 2-representation $\EOH^\ell\psmod$ from
\cref{aishaonthephone}. The image of the odd Rickard complex 
$\Theta_k$ under $\cV$ recovers the graded superfunctor defined by tensoring with singular Rouqiuer complex
from \cref{src} shifted globally in degree by an application of $Q^{-nk}$.
\end{lemma}

\begin{proof}
This follows using the explicit identification of the
divided powers $E^{(d)}$ and $F^{(d)}$ as
endofunctors of $\cV$ explained in \cref{aishaonthephone}.
We just check that the degree shifts match correctly.
Let $n = \frac{\ell-k}{2}$ and $d$ be as in \cref{weekend}, so $k = \ell-2n$.
In the $-d$th cohomological degree in the odd Rickard complex,
we have 
$Q^d E^{(k+d)} F^{(d)}1_{-k}$. 
In the 2-representation $\cV$, this acts by tensoring with the graded
superbimodule
$Q^d 
\big(Q^{-(\ell-2n+d)(n-d)}
U_{(k+d);n-d}^\ell \big)\otimes_{\EOH_{n-d}^\ell} \big(Q^{-d(\ell-3(n-d)-2d+1)}
V_{n-d;(d)}^\ell\big)$, where
the degree shifts are as described in \cref{aishaonthephone}.
The total grading shift here simplifies to
$Q^{-n k}$,
so this is equal to the graded superbimodule
$U_{(k+d);n-d} \otimes_{\EOH_{n-d}^\ell} V_{n-d;(d)}$
in the $d$th homological degree of the singular Rouquier complex
shifted by $Q^{-nk}$.
\end{proof}

\begin{corollary}\label{dearly}
For $\ell \in \N$, 
$\big(\GSKar(\cV(-\ell)\big))(\Theta_k):K^b\big(\GSKar(\cV(-\ell))_{-k}\big) \rightarrow K^b\big(\GSKar(\cV(-\ell))_k\big)$ is a graded superequivalence
inducing $T:1_{-k} \mathbf{V}(-\ell)\stackrel{\sim}{\rightarrow}
1_k \mathbf{V}(-\ell)$ at the level of the Grothendieck groups.
\end{corollary}

\begin{proof}
This follows from \cref{nearly} together with \cref{weirdshift} and \cref{babykid}.
\end{proof}

The proof of the following theorem is based on the argument in \cite[Th.~5.18]{Rou},
the main step really being \cite[Lem.~5.5]{Rou}.
This was itself a generalization of \cite[Th.~6.4]{CR}
which
constructed equivalences between 
bounded derived categories of locally finite Abelian 2-representations.

\begin{theorem}\label{bigt}
Let $\cV$ be an integrable graded Karoubian 2-representation
of $\UU(\sl_2)$.
For $k \in \Z$, the graded 
superfunctor $\cV(\Theta_k):K^b(\cV_{-k}) \rightarrow K^b(\cV_k)$ induced by the odd Rickard complex
is a graded superequivalence.
\end{theorem}

\begin{proof}
By \cref{haverightadjs}, the 1-morphism $Q^d E^{(k+d)} F^{(d)} 1_{-k}$ has a right dual 
in $\GSKar(\UU(\sl_2))$. 
Hence, we can form the right dual
$\Theta^k$ to $\Theta_k$, which is a cochain complex in
$\Ch\big(\mathpzc{Hom}_{\GSKar(\UU(\sl_2))}(k,-k)\big)$.
The 1-morphism in
the $d$th cohomological degree of $\Theta^k$
is the right dual 
of the 1-morphism
in the $(-d)$th cohomological degree of $\Theta_k$,
and the differentials in $\Theta^k$ are the right mates
of the corresponding differentials in $\Theta_k$.
Let
$\Theta^k \circ \Theta_k$ and $\Theta_k \circ \Theta^k$
be the total complexes associated to the double complexes obtained by
composing these cochain complexes. 
The complex $\Theta_k$ is bounded above, and $\Theta^k$ is bounded below,
but neither is bounded. Consequently,
in each cohomological degree,
the total complexes 
$\Theta^k \circ \Theta_k$ and $\Theta_k \circ \Theta^k$
involve {\em infinite} direct sums of 1-morphisms in 
$\GSKar(\UU(\sl_2))$,
so in fact, one needs to pass to a completion of this graded $(Q,\Pi)$-supercategory 
for it to make sense. This does not cause issues since, on a given object in an
 integrable graded Karoubian
 2-representation, the superfunctors arising from 
 all but finitely many of the summands of these infinite 
 direct sums are zero.

Like $\Theta_k$, the complex $\Theta^k$ defines a graded superfunctor
denoted $\cV(\Theta^k):K^b(\cV_k) \rightarrow
K^b(\cV_{-k})$. Moreover, $\cV(\Theta^k)$ is
right adjoint to $\cV(\Theta_k)$, with
counit and unit of adjunction denoted
\begin{align*}
\cV(\eps):&
\cV(\Theta_k)\circ\cV(\Theta^k)
\Rightarrow \Id_{K^b(\cV_k)},
&\cV(\eta):&\Id_{K^b(\cV_{-k})}
\Rightarrow 
\cV(\Theta^k) \circ \cV(\Theta_k).
\end{align*}
This is explained in more detail in \cite[Sec.~4.1.4]{CR}.
As the notation $\cV(\eps)$ and
$\cV(\eta)$ suggests, if we identify
$\cV(\Theta_k)\circ\cV(\Theta^k)$ with $\cV(\Theta_k \circ \Theta^k)$
and
$\cV(\Theta^k) \circ \cV(\Theta_k)$ with
$\cV(\Theta^k \circ \Theta_k)$ then
these even degree 0 supernatural transformations are induced by
corresponding chain maps denoted simply by
$\eps:\Theta_k\circ\Theta^k\Rightarrow 1_k$
and $\eta:1_{-k}\rightarrow \Theta^k\circ \Theta_k$
between cochain complexes in the completion of $\GSKar(\UU(\sl_2))$.
Although not needed here, these chain maps can be seen quite explicitly; the 
matrix coefficients of their components
are 2-morphisms in $\GSKar(\UU(\sl_2))$ that arise from the counits
and units defining the duality between
the 1-morphisms $Q^d E^{(k+d)} F^{(d)} 1_{-k}$ and their right duals.

To prove the theorem, it suffices to show that 
$\cV(\eps)$ and $\cV(\eta)$ are isomorphisms.
We just explain the argument to see this in the case of
$\cV(\eps)$, since the case of $\cV(\eta)$ is similar.
Since a chain map is an isomorphism in $K^b(\cV_k)$
if and only if its cone is zero in $K^b(\cV_k)$, 
the even degree 0 graded supernatural transformation $\cV(\eps)$ is an isomorphism if and only if
$\Cone(\cV(\eps)_C) \cong 0$ in $K^b(\cV_k)$ 
for all $C \in K^b(\cV_k)$.
Now we observe that
$\Cone(\cV(\eps)_C) = \cV(Z)(C)$
where $Z := \Cone(\eps)$ is the cone of 
$\eps:\Theta_k\circ\Theta^k\Rightarrow 1_k$.
Thus, it suffices to show that
the graded superfunctor $\cV(Z):K^b(\cV_k)\rightarrow K^b(\cV_k)$
is zero.

Consider $K^b(\cV)$ as an integrable Karoubian graded 2-representation in its own right.
In this paragraph, we show that 
$\cV(Z)(E^n C) = 0$
in $K^b(\cV_k)$
for all $\ell \in \N$, $n \geq 0$ 
such that $k=2n-\ell$, 
and all lowest weight objects $C \in K^b(\cV_{-\ell})$. 
To see this, we apply \cref{Lup} (with $\cV$ replaced by 
$K^b(\cV)$)
to get a
strongly equivariant graded superfunctor $\Omega_C:\cV(-\ell)\rightarrow K^b(\cV)$
taking $\bar 1_{-\ell}$ to $C$. 
This extends to a strongly equivariant graded superfunctor
$\widehat{\Omega}_C:\GSKar(\cV(-\ell)) \rightarrow K^b(\cV)$ 
by the universal property of
graded super Karoubi envelope.
Let $K^b(\widehat{\Omega}_C):
K^b(\GSKar(\cV(-\ell)_k)) \rightarrow K^b(\cV_k)$
be the graded superfunctor defined by applying 
$\widehat{\Omega}_C$ to a complex in 
$K^b(\GSKar(\cV(-\ell)_k))$ to obtain a double complex then taking the associated total complex.
Since $\widehat{\Omega}_C$ is strongly equivariant,
we have that
\begin{equation}\label{lastthing}
K^b(\widehat{\Omega}_C)
\circ (\GSKar(\cV(-\ell)))(Z)\circ \operatorname{inc} \simeq
\cV(Z) \circ \widehat{\Omega}_C,
\end{equation}
where $\operatorname{inc}:\GSKar(\cV(-\ell)) 
\rightarrow K^b(\GSKar(\cV(-\ell))$ is the canonical graded superfunctor
sending objects to complexes concentrated in cohomological degree 0.
By \cref{dearly},
$(\GSKar(\cV(-\ell)))(Z) = 0$
in $K^b(\GSKar(\cV(-\ell)_k))$,
hence, the graded superfunctor on the left hand side of \cref{lastthing}
takes $E^n \bar 1_{-\ell}$ to 0.
So the graded superfunctor on the right hand side takes
$E^n \bar 1_{-\ell}$ to 0 too.
Since we have that 
$\widehat{\Omega}_C (E^n \bar 1_{-\ell}) \simeq E^n C$,
it follows that
$\cV(Z)(E^n C) = 0$ as required.

To complete the proof, 
we let $\cV(Z)^\vee$ be a right adjoint to
$\cV(Z):K^b(\cV_k) \rightarrow K^b(\cV_k)$, which exists by the general 
discussion in \cite[Sec.~4.1.4]{CR} again.
We must show that $\cV(Z)(D) = 0$ for any
$D \in K^b(\cV_k)$, which we do by showing that
$\cV(Z)^\vee \big(\cV(Z) (D)\big) = 0$;
this is sufficient since 
 it implies that $\End_{K^b(\cV_k)}\big(\cV(Z)(D)\big) = 0$.
Using \cref{sillylemma}, we just need to show that
$$
\Hom_{K^b(\cV_k)}\Big(E^n C, 
\cV(Z)^\vee \big(\cV(Z) (D)\big)\Big) = 0
$$
for $C$ and $n$ as in the previous paragraph.
This follows because by adjunction we have that
$$
\Hom_{K^b(\cV_k)}\Big(E^n C, 
\cV(Z)^\vee \big( \cV(Z) (D)\big)\Big) 
\simeq
\Hom_{K^b(\cV_k)}\big(\cV(Z)(E^n C), 
\cV(Z)(D)\big) 
$$
which is zero by the previous paragraph.
\end{proof}

\section{Application to representations of spin symmetric groups}\label{sapps}

\cref{bigt} can be applied to obtain graded superequivalences 
between homotopy/derived categories of
supermodules over the
cyclotomic quiver Hecke superalgebras from \cite{KKT,KKO1,KKO2}.
In explaining this, we will mainly cite \cite[Sec.~8]{KKO2} which presents the results needed to do this
rather concisely. However, 
we need to 
reverse the roles of $E$ and $F$
compared to \cite{KKO2} to be consistent with our convention for $\UU(\sl_2)$
in \cref{catsec}, in which we preferred lowest weight modules 
to highest weight modules.

Fix a Cartan superdatum $(A, P, \Pi, \Pi^\vee)$ as in \cite[Sec.~4.1]{KKO2}. So:
\begin{itemize}
\item
$I$ is an index set
with given decomposition $I = I_{\even}\sqcup I_{\odd}$;
\item
$A = (a_{i,j})_{i,j \in I}$ is a symmetrizable
Cartan matrix such that $a_{i,j} \in 2\Z$ for all $i \in I_{\odd}, j \in I$;
\item
$P$ is the weight lattice;
\item
 $\Pi = \{\alpha_i\:|\:i \in I\}$ is the set of simple roots;
 \item $\Pi^\vee = \{h_i\:|\:i \in I\}$ is the set of simple coroots.
 \end{itemize}
Let $d_i\:(i \in I)$ be positive integers chosen
so that $d_i a_{i,j} = d_j a_{j,i}$ for all $i,j \in I$.
Let $P^+$ be the corresponding set of dominant weights
and $Q^+ := \bigoplus_{i \in I} \N \alpha_i$ be the non-negative
part of the root lattice.
Finally, let $W < \Aut(P)$ be the Weyl group.

Let $\Bbbk = \bigoplus_{d \geq 0} \Bbbk_d$ be a positively graded commutative ground ring with $\Bbbk_0 = \k$ (our usual algebraically closed ground field) and $\dim_\k 
\Bbbk_d < \infty$ for all $d$. We view $\Bbbk$ as a purely even 
graded $\k$-superalgebra.
Given any $\alpha \in Q^+$, there is a corresponding {\em quiver Hecke superalgebra} $R_\alpha$ which is defined by generators and relations
as in \cite[Sec.~8.1]{KKO2}; the definition depends on an 
additional choice of parameters as explained in \cite{KKO2}.
Let $R^\lambda_\alpha$ be the deformed cyclotomic quotient
from \cite[Def.~8.10]{KKO2} associated to a dominant weight
$\lambda \in P^+$ and a choice of 
monic polynomials $a_i^\lambda \:(i \in I)$ as in \cite[(8.12)]{KKO2}.
We are interested in the graded $(Q,\Pi)$-supercategory
\begin{equation}
R^\lambda\psmod := \bigoplus_{\alpha \in Q^+} 
R^\lambda_\alpha\psmod.
\end{equation}
The constructions in \cite[Sec.~8.3]{KKO2} make $R^\lambda\psmod$
into a ``supercategorification" of the 
integrable lowest weight module $V(-\lambda)$ for
the covering quantum group $U_{q,\pi}(\mathfrak{g})$
with the given Cartan superdatum. 
From this, it can be seen that
$R^\lambda\psmod$ has the structure of a graded
2-representation of the corresponding graded Kac-Moody 2-supercategory
as defined in \cite{BE2}, with the Grothendieck group 
$K_0(R^\lambda\Upsmod)$
being identified with the Kostant
$\Z[q,q^{-1}]^\pi$-form for
$V(-\lambda)$.

To be more precise, we focus now on some fixed $i \in I$
and consider the corresponding $\sl_2$-subalgebra
of $U_{q,\pi}(\mathfrak{g})$. In this generality,
we actually need to work now with $q_i := q^{d_i}$ and the grading shift functor
$Q_i := Q^{d_i}$ rather than $q$ and $Q$ used in previous sections. 
This means that when $d_i > 1$
definitions such as \cref{2repdef} 
earlier in the paper
should be modified by replacing $Q$ with $Q_i$
and scaling all degrees by $d_i$ too, e.g., $x$ and $\tau$
are now of degrees $2 d_i$ and $-2 d_i$ rather than of degrees 2 and $-2$.
Since the $\Z$- and $\Z/2$-gradings are independent this does not cause any problems.
There are graded superfunctors 
\begin{align*}
E_i&:
R^\lambda \psmod \rightarrow R^\lambda \psmod,
&
F_i&:R^\lambda \psmod \rightarrow R^\lambda \psmod.
\end{align*}
In terms of the induction and restriction functors denoted
$F_i^\lambda$  and $E_i^\lambda$ in \cite[Sec.~8.3]{KKO2},
our $E_i$ is 
$F_i^\lambda = \bigoplus_{\alpha \in Q_+} F_i^\lambda|_{R^\lambda_\alpha\psmod}$
and our $F_i$ is
$\bigoplus_{\alpha \in Q^+}
Q_i^{\langle h_i, \alpha-\lambda\rangle - 1} E_i^\lambda|_{R^\lambda_{\alpha}\psmod}$.
As well as switching the roles of $E$ and $F$
we have incorporated an additional grading shift
into the restriction functors compared to \cite{KKO2}.
This is needed because \cite{KKO2} does not follow the standard
conventions for covering quantum groups.
It ensures that the graded supernatural transformations
$\eps:E_i F_i|_{R^\lambda_{\alpha+\alpha_i}\psmod} \Rightarrow \Id_{R^\lambda_{\alpha+\alpha_i}\psmod}$,
$\eta:\Id_{R^\lambda_\alpha\psmod} \Rightarrow F_i E_i|_{R^\lambda_\alpha\psmod}$
defined on a graded supermodule by exactly the same underlying functions as for the natural adjunction between restriction and induction
are of the correct degree to match the degrees of the rightward cups and caps
in \cref{table} (also now scaled by $d_i$).
Also in \cite[Sec.~8.3]{KKO2}, one finds the definition of graded supernatural transformations
$x:E_i \Rightarrow E_i$ of degree $2d_i$ and 
$\tau:E_i^2 \Rightarrow E_i^2$
of degree $-2d_i$, both of which are even if $i \in I_{\even}$
and odd if $i \in I_{\odd}$.
(A further complication is that the language of supercategory, superfunctor and supernatural transformation is used differently in \cite{KKO2} compared to
here, but the appropriate translation is easy to make; see the table at the end of the introduction in \cite{BE}.)

This construction makes $R^\lambda\psmod$ into
a graded integrable Karoubian 2-representation of the ordinary $\sl_2$
2-category from \cite{Lauda,Rou} 
if $i$ is even, or of the odd $\sl_2$ 2-category
$\UU(\sl_2)$ as in \cref{2repdef}
if $i$ is odd (with the modified convention for degrees when $d_i >1$).
The last statement is not stated explicitly in \cite{KKO2}---the relevant place
is \cite[Th.~8.13]{KKO2} but one has to work through the proof which
goes back to \cite[Th.~5.2]{KK} to see that the isomorphisms are given by the 
appropriate matrices of supernatural transformations needed to check 
the difficult relations \cref{inv1,inv2}.

\begin{theorem}
In the above setup, for $\alpha \in Q^+$ such that $V(-\lambda)_{\alpha-\lambda} \neq 0$, the even or odd Rickard complex $\Theta_{\langle h_i, \lambda-\alpha\rangle}$ induces a graded superequivalence 
$K^b\big(R^\lambda_{\alpha}\psmod\big)
 \rightarrow K^b\big(R^\lambda_{\alpha-\langle h_i, \alpha-\lambda\rangle \alpha_i}\psmod\big)$.
\end{theorem}

\begin{proof}
This follows from \cite[Th.~5.18]{Rou} 
if $i$ is even, with the graded superequivalence being induced by the
even analog of the Rickard complex,
or from \cref{bigt} if $i$ is odd.
\end{proof}

There is also a ``dual version" of this theorem
with $R^\lambda\psmod$ replaced with
\begin{equation}
R^\lambda\smod := \bigoplus_{\alpha \in Q^+}
R^\lambda_\alpha\smod.
\end{equation}
The underlying ordinary category 
is a locally finite Abelian $(Q,\Pi)$-category.
The results from \cite[Sec.~8.3]{KKO2} 
show that this categorifies 
the dual Kostant $\Z[q,q^{-1}]^\pi$-form
for $V(-\lambda)$.
For fixed $i \in I$ again, 
$R^\lambda\smod$
can be made into a graded
2-representation of the even or odd $\sl_2$ 2-category
exactly as above.

\begin{theorem}\label{dereq}
In the above setup, for $\alpha \in Q^+$ such that $V(-\lambda)_{\alpha-\lambda} \neq 0$, 
the even or
odd Rickard complex $\Theta_{\langle h_i, \lambda-\alpha\rangle}$
induces a graded superequivalence 
$
D^b\big(R^\lambda_{\alpha}\smod\big)
 \rightarrow D^b\big(R^\lambda_{\alpha-\langle h_i, \alpha-\lambda\rangle \alpha_i}\smod\big).
 $
 \end{theorem}

 \begin{proof}
 Like in the previous theorem, the even or odd Rickard complex
$\Theta_{\langle h_i, \lambda-\alpha\rangle}$
induces a graded superequivalence 
$K^b\big(R^\lambda_{\alpha}\smod\big)
 \rightarrow K^b\big(R^\lambda_{\alpha-\langle h_i, \alpha-\lambda\rangle \alpha_i}\smod\big)$.
The result for derived categories follows since they are localizations of these homotopy categories.
 \end{proof}

For a graded superalgebra $A$, we write $A \otimes C_1$ for the graded
superalgebra obtained by tensoring with the rank one Clifford superalgebra
generated by an odd degree 0 involution.
There is also a variation of \cref{dereq} with 
$R^\lambda\smod$ replaced by 
\begin{equation}
R^\lambda\otimes C_1 \smod := \bigoplus_{\alpha \in Q^+}
R^\lambda_\alpha\otimes C_1\smod.
\end{equation}
This can be made into a graded 2-representation which also categorifies
the dual Kostant $\Z[q,q^{-1}]^\pi$-form
for $V(-\lambda)$, just as $R^\lambda\smod$ did earlier.
In particular, for each $i \in I$, 
we can make
$R^\lambda\otimes C_1 \smod$ into a graded
2-representation of the even or odd $\sl_2$ 2-category
exactly as above. This follows by the construction explained in the next paragraph.

There is a general notion of the
{\em Clifford twist} $\cA^\CT$ of a graded 
supercategory $\cA$, which goes back to \cite[Lem.~2.3]{KKT}.
By definition, this is the graded supercategory 
whose objects are pairs $(X, \phi)$ for $X \in \cA$ and an
odd degree 0 involution $\phi \in \End_{\cA}(X)$.
A morphism $f:(X,\phi) \rightarrow (Y,\theta)$
is a morphism $f:X\rightarrow Y$ in $\cA$ such that 
$\theta \circ f = (-1)^{\parity(f)} \phi \circ f$.
Degree and parity of morphisms in $\cA^\CT$ are induced by the ones for $\cA$.
There are obvious ways to define the
Clifford twist $F^\CT:\cA^\CT \rightarrow \cB^\CT$
of a graded superfunctor $F:\cA \rightarrow \cB$,
and also the Clifford twist $\alpha^\CT:F^\CT \Rightarrow G^\CT$
of a graded supernatural transformation $\alpha:F \Rightarrow G$
between two graded superfunctors.
This makes $\CT$ into a strict graded 2-superfunctor
$\CT:\GSCAT\rightarrow \GSCAT$.
Now if  $\cV$ is any graded 2-representation of the even or the odd
2-supercategory $\UU(\sl_2)$, its Clifford twist $\cV^\CT$ can be made into a graded 2-representation
in its own right, with the required
graded superfunctors $E$ and $F$ on $\cV^\CT$ being the Clifford twists of the ones
for $\cV$, and all of the required graded supernatural transformations being the Clifford twists of the one for $\cV$ too. If $\cV$ is integrable and Karoubian then so is $\cV^\CT$.

\begin{theorem}\label{dereq2}
In the above setup, for $\alpha \in Q^+$ such that $V(-\lambda)_{\alpha-\lambda} \neq 0$, 
the even or
odd Rickard complex $\Theta_{\langle h_i, \lambda-\alpha\rangle}$
induces a graded superequivalence 
$D^b\big(R^\lambda_{\alpha}\otimes C_1\smod\big)
 \rightarrow D^b\big(R^\lambda_{\alpha-\langle h_i, \alpha-\lambda\rangle \alpha_i}\otimes C_1\smod\big).$
 \end{theorem}

 \begin{proof}
 This follows by the same arguments as \cref{dereq}.
\end{proof}

Assume henceforth that the characteristic of the
ground field $\k$ is $p = 2l+1 > 2$,
and that the Cartan superdatum fixed above is of type 
$A_{2l}^{(2)}$, with the shortest simple root $\alpha_0$ being odd and all other simple roots being even.
We consider the cyclotomic quiver Hecke superalgebras $R^\lambda_\alpha$
for $\alpha \in Q_+$ and $\lambda := \Lambda_0$, 
taking the ring $\Bbbk$ 
to be the ground field $\k$, and all other choices made as 
explained in \cite[Sec.~3.1]{KLi}.
Let 
$R^\lambda_\alpha \otimes C_1$
be the superalgebra tensor product of $R^\lambda_\alpha$ with the
rank one Clifford superalgebra generated by an odd involution.
Now we forget both the $\Z$- and $\Z/2$-gradings
on $R^\lambda_\alpha$ and $R^\lambda_\alpha \otimes C_1$
to view them as ordinary finite-dimensional algebras.
For such an algebra $A$, we write $A\fdmod$ for the Abelian category
of finite-dimensional left $A$-modules and $D^b(A\fdmod)$ for its ordinary
bounded derived category.

In view of \cite[Lem.~3.1.39]{KLi}, 
the following proves \cite[Conj.~1.3.1]{KLi}.

\begin{theorem}
Suppose that $\alpha, \beta \in Q^+$ are
such that $\alpha-\lambda$ and $\beta - \lambda$ are weights of 
$V(-\lambda)$ in the same $W$-orbit.
The categories
$D^b\big(R^\lambda_\alpha\fdmod\big)$
and
$D^b\big(R^\lambda_\beta\fdmod\big)$ are equivalent
as are $D^b\big(R^\lambda_\alpha\otimes C_1\fdmod\big)$
and
$D^b\big(R^\lambda_\beta\otimes C_1 \fdmod\big)$.
\end{theorem}

\begin{proof}
Since the simple reflections generate $W$,
it suffices to prove the theorem in the special case that
$\alpha-\lambda$ is a weight of $V(-\lambda)$
and $\beta = \alpha-\langle h_i, \alpha-\lambda\rangle \alpha_i$ for some $i \in I$.
The graded superequivalences in \cref{dereq,dereq2} are obtained by taking the derived tensor product
with the complex of graded superbimodules
arising from the appropriate Rickard complex.
Similarly, the quasi-inverse graded superequivalences are obtained from the
right adjoint of this complex. Now we are forgetting both the $\Z$- and $\Z/2$-gradings,
viewing these complexes of graded superbimodules as complexes of ordinary bimodules.
The resulting complexes define functors 
between the ordinary derived categories.
Since they are quasi-inverse with all gradings present, they are obviously quasi-inverse without these gradings.
\end{proof}

\begin{corollary}\label{CBroue}
Brou\'e's Abelian Defect Group Conjecture holds for double covers of symmetric and alternating groups over any algebraically closed field of positive characteristic.
\end{corollary}

\begin{proof}
See \cite[Th.~5.4.13]{KLi} where this is deduced from \cite[Conj.~1.3.1]{KLi}.
\end{proof}

\end{document}